\documentclass[paper=a4,fontsize=10pt,twoside,DIV=12,BCOR=18mm,bibliography=totoc,headings=optiontohead,chapterprefix=true,
]{scrbook}

\PassOptionsToPackage{tbtags}{amsmath}

\usepackage{tikz}
\usetikzlibrary{matrix,arrows,calc}
\usepackage{tikz-cd}

\usepackage{ifxetex}

\ifxetex
\usepackage{xltxtra} 
\usepackage{polyglossia}
\setdefaultlanguage{spanish}
\defaultfontfeatures{Mapping=tex-text}
\else
\usepackage[utf8]{inputenc}
\usepackage[spanish.lcroman]{babel}
\usepackage[T1]{fontenc}
\usetikzlibrary{cd}
\usetikzlibrary{babel}
\fi 

\DeclareMathAlphabet{\mathpzc}{OT1}{pzc}{m}{it}

\usepackage{microtype}
\usepackage{mathtools}
\usepackage{csquotes}
\usepackage[shortcuts]{extdash}
\usepackage{amsmath}
\usepackage{amssymb}
\usepackage{amsfonts}
\usepackage{stmaryrd}
\usepackage{color}
\usepackage{enumitem}
\setenumerate[1]{leftmargin=*,labelindent=\parindent,label=(\alph*)}
\setenumerate[2]{leftmargin=*,labelindent=\parindent,label=(\roman*)}
\usepackage{xparse}
\usepackage[hidelinks]{hyperref}
\usepackage{multirow}

\usepackage{xcolor}
\definecolor{mygray}{rgb}{0.95,0.95,0.92}
\makeatletter\newenvironment{remark}{
  \vspace{1em}
  \noindent\begin{lrbox}{\@tempboxa}\begin{minipage}{\textwidth}\small
      \noindent\textbf{Observación:}}{\end{minipage}\end{lrbox}
  \colorbox{mygray}{\usebox{\@tempboxa}}
  \vspace{1em}
}\makeatother
\definecolor{codegreen}{rgb}{0,0.6,0}
\definecolor{codegray}{rgb}{0.5,0.5,0.5}
\definecolor{codepurple}{rgb}{0.58,0,0.82}
\definecolor{backcolour}{rgb}{0.95,0.95,0.92}

\usepackage{listings}

\lstdefinestyle{mystyle}{
    backgroundcolor=\color{mygray},   
    keywordstyle=\color{red},
    numberstyle=\tiny\color{codegray},
    stringstyle=\color{codepurple},
    basicstyle=\ttfamily\footnotesize,
    breakatwhitespace=false,         
    breaklines=true,                 
    captionpos=b,                    
    keepspaces=true,                 
    numbers=left,                    
    numbersep=5pt,                  
    showspaces=false,                
    showstringspaces=false,
    showtabs=false,                  
    tabsize=2
}
 
\usepackage[ngerman,shadow,color=blue!30,textsize=footnotesize,colorinlistoftodos,disable]{todonotes}
\presetkeys{todonotes}{fancyline}{}

\deffootnote[1em]{1em}{1em}{\textsuperscript{\thefootnotemark}\ }
\numberwithin{equation}{chapter}
\usepackage[thmmarks,hyperref,amsmath]{ntheorem}
\usepackage[nameinlink,spanish]{cleveref}
\theoremstyle{plain}
\theoremseparator{:}
\theoremheaderfont{\normalfont\bfseries}
\theorembodyfont{\itshape}
\newtheorem{thm}{Teorema}[chapter]
\newtheorem{conj}[thm]{Conjetura}
\newtheorem{lem}[thm]{Lema}
\newtheorem{cor}[thm]{Corolario}
\newtheorem{prop}[thm]{Proposición}
\newtheorem{con}[thm]{Conjetura}
\theorembodyfont{\normalfont}
\newtheorem{defi}[thm]{Definición}

\newtheorem{ex}[thm]{Ejemplo}
\theorembodyfont{\small}
\newtheorem{ejer}{Ejercicio}[chapter]
\theoremnumbering{Roman}
\theorembodyfont{\itshape}
\newtheorem{introthm}{Teorema}
\theoremstyle{nonumberplain}
\theoremheaderfont{\itshape}
\theorembodyfont{\normalfont}
\theoremsymbol{$\Box$}
\newtheorem{proof}{Demostración}

\usepackage[style=alphabetic,backend=biber,maxnames=5,maxalphanames=5,backref=true,backrefstyle=two+]{biblatex}
\DeclareFieldFormat*{title}{\emph{#1}}
\DeclareFieldFormat*{journaltitle}{#1}
\bibliography{literatura}

\crefname{section}{sección}{secciones}
\crefname{footnote}{pie de página}{pies de página}

\setlist[description]{%
  topsep=30pt,               
  itemsep=5pt,               
  font={\normalfont\textit}, 
}
\newlist{arabiclist}{enumerate}{2}
\setlist[arabiclist]{leftmargin=*,labelindent=\parindent,label=(\arabic*)}
\setlist[arabiclist,2]{label=(\roman*)}

\makeatletter
\let\C\@undefined
\let\G\@undefined
\makeatother
\newcommand{\isom}{\simeq}
\newcommand{\complexi}{\mathrm i}
\newcommand{\R}{\mathbb{R}}
\newcommand{\C}{\mathbb{C}}
\newcommand{\Q}{\mathbb{Q}}
\newcommand{\Qbar}{\overline\Q}
\newcommand{\F}{\mathbb{F}}
\newcommand{\Fp}{\mathbb{F}_p}
\newcommand{\Z}{\mathbb{Z}}
\newcommand{\Ncero}{\mathbb N_{\ge0}}
\newcommand{\Nuno}{\mathbb N_{\ge1}}

\newcommand{\Zl}{\mathbb{Z}_{\ell}}
\newcommand{\Zp}{\mathbb{Z}_{p}}
\renewcommand{\O}{\mathcal O}

\newcommand{\Qp}{\mathbb{Q}_{p}}
\newcommand{\Qpbar}{\Qbar_p}
\newcommand{\LL}{\Lambda}
\newcommand{\mut}{\widetilde{\mu}}
\newcommand{\lat}{\widetilde{\lambda}}
\newcommand{\nut}{\widetilde{\nu}}
\newcommand{\w}{\omega}
\newcommand{\M}{\mathfrak{M}}
\newcommand{\p}{\mathfrak{p}}

\newcommand{\Kp}{K_{\p}}
\newcommand{\e}{\mathrm e}
\newcommand{\G}[1]{\mathrm G_{#1}}
\newcommand{\GL}{\mathrm{GL}}
\newcommand{\GQ}{\G\Q}
\newcommand{\Bcris}{\mathrm B_{\mathrm{cris}}}
\newcommand{\BdR}{\mathrm B_{\mathrm{dR}}}
\newcommand{\nr}{\mathrm{nr}}
\newcommand{\ab}{\mathrm{ab}}

\newcommand{\HL}{\mathrm H}
\newcommand{\Hf}{\HL_{\mathrm f}}
\DeclareMathOperator{\Sel}{Sel}
\newcommand{\betti}{\mathrm{B}}
\newcommand{\et}{\text{\normalfont ét}}
\newcommand{\dR}{\mathrm{dR}}
\newcommand{\GQp}{\G{\Qp}}
\renewcommand{\Re}{\operatorname{Re}}
\newcommand{\Frob}{\mathrm{Frob}}
\DeclareMathOperator{\Hom}{Hom}
\DeclareMathOperator{\Prin}{Prin}
\DeclareMathOperator{\Div}{Div}
\DeclareMathOperator{\Cl}{Cl}
\DeclareMathOperator{\End}{End}
\DeclareMathOperator{\Aut}{Aut}
\DeclareMathOperator{\Gal}{Gal}
\DeclareMathOperator{\rg}{rg}
\DeclareMathOperator{\Sop}{sop}
\DeclareMathOperator{\Log}{Log}
\DeclareMathOperator{\Rad}{rad}
\DeclareMathOperator{\Norm}{\mathit{N}}
\DeclareMathOperator{\Frac}{Frac}
\DeclareMathOperator{\coker}{coker}
\newcommand{\abs}[1]{\left\vert#1\right\vert}
\newcommand{\Clog}[1]{\widetilde{\mathcal{C}\ell}_{#1}}
\DeclareMathOperator{\id}{id}
\DeclareMathOperator{\Spec}{Spec}
\DeclareMathOperator{\rk}{rg}
\DeclareMathOperator{\im}{im}
\DeclareMathOperator{\charideal}{car}
\newcommand{\labeledarrow}[1]{\stackrel{#1}{\longrightarrow}}
\newcommand{\isomarrow}{\labeledarrow{\isom}}
\newcommand{\integrald}{\,\mathrm d}
\newcommand{\Rc}{\mathcal{R}}
\newcommand{\Uc}{\mathcal{U}}
\newcommand{\Jc}{\mathcal{J}}

\DeclareMathOperator{\Pl}{Pl}

\newcommand{\ejercicios}{\vspace*{2.5ex plus 1ex minus 1ex}\penalty-200\noindent{\sffamily{\bfseries{Ejercicios}}}\vspace*{1.4ex plus .4ex minus .4ex}\nopagebreak[4]}

\makeatletter
\appto\inlineextras@spanish{\renewcommand\lim{\qopname \relax m{lím}}}
\appto\blockextras@spanish{\renewcommand\lim{\qopname \relax m{lím}}}
\def\varlim@#1#2{%
  \vtop{\m@th\ialign{##\cr
    \hfil$#1\operator@font \lim$\hfil\cr
    \noalign{\nointerlineskip\kern1.5\ex@}#2\cr
    \noalign{\nointerlineskip\kern-\ex@}\cr}}%
}
\def\varcolim@#1#2{%
  \vtop{\m@th\ialign{##\cr
    \hfil$#1\operator@font co\!\lim$\hfil\cr
    \noalign{\nointerlineskip\kern1.5\ex@}#2\cr
    \noalign{\nointerlineskip\kern-\ex@}\cr}}%
}
\def\varinjlim{%
  \mathop{\mathpalette\varcolim@{\rightarrowfill@\textstyle}}\nmlimits@
}
\def\varprojlim{%
  \mathop{\mathpalette\varlim@{\leftarrowfill@\textstyle}}\nmlimits@
}
\appto\inlineextras@spanish{\renewcommand\mod{\qopname \relax m{mód}}}
\appto\blockextras@spanish{\renewcommand\mod{\qopname \relax m{mód}}}
\appto\inlineextras@spanish{\renewcommand\max{\qopname \relax m{máx}}}
\appto\blockextras@spanish{\renewcommand\max{\qopname \relax m{máx}}}
\makeatother

\NewDocumentCommand{\tensop}{m}
 {%
  \mathbin{\mathop{\otimes}\displaylimits_{#1}}%
 }
\NewDocumentCommand{\tensor}{t_}
 {%
  \IfBooleanTF{#1}
   {\tensop}
   {\otimes}%
 }

 \DeclareMathOperator{\Ann}{ann}

\usepackage{esindex}
\usepackage{imakeidx}
\makeatletter
\newcommand*{\define}[2][]{%
 \def\iearg{#1}%
  \ifdefempty{\iearg}%
    {\def\ie@arg{#2}}%
    {\def\ie@arg{#1}}%
  \esindex[def]{\ie@arg}%
  \emph{#2}%
}
\newcommand*{\importante}[2][]{%
 \def\iearg{#1}%
  \ifdefempty{\iearg}%
    {\def\ie@arg{#2}}%
    {\def\ie@arg{#1}}%
  \esindex[def]{\ie@arg}%
  #2%
}
\newcommand*{\definen}[2][]{%
 \def\iearg{#1}%
  \ifdefempty{\iearg}%
    {\def\ie@arg{#2}}%
    {\def\ie@arg{#1}}%
  \esindex[def]{\ie@arg}%
}
\makeatother
\makeindex[name=def, title=Índice, columns=2, options=-s indexstyle, intoc]

\title{Torres infinitas sorprendentes}
\subtitle{\Large Una introducción a la Teoría de Iwasawa}
\author{\textsc{Michael {Fütterer} \& José {Villanueva}}}
\date{}
\publishers{\small\today}

\begin{document}

\frontmatter


\maketitle

\chapter*{Introducción}  

En un sentido amplio, el objetivo de la teoría de números es estudiar como se comportan los
números enteros o algebraicos y las ecuaciones que los involucran en diferentes situaciones. Muchas veces esta conducta puede ser descrita por estructuras algebraicas. Por ejemplo, el grupo de
clases de ideales de un campo de números describe la descomposición en factores primos
únicos en su anillo de enteros (o más bien su ausencia). Otro ejemplo es el grupo de Selmer
de una curva elíptica sobre un campo de números, que tiene que ver con la discrepancia entre
la existencia de puntos globales en la curva (es decir, definidos sobre el campo de números)
y locales (es decir, definidos sobre una completación). Por eso, estos grupos son de mucho
interés pero su estudio es intrincado.

Este patrón se puede extender a otras
situaciones, la más general siendo la de un \define{motivo} -- un término que no vamos a
explicar en este texto, pero es quizás el objeto de interés más general de la geometría
aritmética, y un campo de números o una curva elíptica son ejemplos de motivos.

La idea pionera de \textsc{Kenkichi Iwasawa} (1917--1998) era estudiar estas conductas no
sobre un solo campo de números $K$, sino sobre todos los campos de una torre infinita
$K_1\subseteq K_2\subseteq\dotsm$ de estos. Aunque esto parece hacer todo aún más
complicado, en realidad lleva a una teoría rica y fecunda. Este proceso analiza los grupos
que uno quiere entender como módulos sobre un cierto anillo, el \emph{álgebra de Iwasawa}
$\LL$, y así los hace más manejables -- de manera similar a como el dominio entero de los
enteros $p$-ádicos $\Zp$ se comporta mejor que los anillos finitos $\Z/p^r\Z$. Con estos
métodos, Iwasawa consiguió demostrar teoremas importantes sobre los grupos de clases de
campos de números, como por ejemplo el siguiente.

\begin{introthm}[Iwasawa, 1959]\label{thm:iwasawa-intro}
    Sea $K$ un campo de números y $p$ un primo. Para cada $r\in\Nuno$ sea $K_r/K$ una extensión
  tal que $\Gal(K_r/K)\isom\Z/p^r\Z$ y $K_r\subseteq K_{r+1}$. Escribimos $p^{e_r}$ como la
  máxima potencia de $p$ que divide el orden del grupo de clases de $K_r$. Entonces existen
  constantes $\mu,\lambda\in\Ncero$ y $\nu\in\Z$ tal que
  \[ e_r=\mu p^r+\lambda r + \nu \quad\text{para }r\gg0. \] 
\end{introthm}

Otro objeto en el centro de interés de la teoría de números moderna son las funciones $L$,
por ejemplo la función zeta de Riemann, la función zeta de Dedekind de un campo de números o
la función $L$ de Hasse y Weil de una curva elíptica (en general, se puede definir una tal
función para cualquier motivo). Desde hace mucho tiempo existía evidencia de que estas funciones
tienen alguna conexión con los grupos que mencionamos anteriormente. Por ejemplo, la
fórmula analítica de números de clases conecta los grupos de clases de un campo de
números a su función zeta de Dedekind, o la conjetura de Birch y Swinnerton-Dyer conecta
el grupo de Selmer de una curva elíptica a la función $L$ de Hasse y Weil. Otro resultado en
este estilo es el \define{criterio de Kummer}:

\begin{introthm}[Kummer, 1850]\label{thm:kummer-crit}
  Sea $h_p$ el número de clases del campo ciclotómico $\mathbb Q(\mu_p)$. Entonces
  \begin{align*}
    p \mid h_p & \iff p \mid \zeta(1-n) \text{ para algún }n\ge 1\text{ par} \\
               &\iff p \text{ divide uno de } \zeta(-1),\zeta(-3),\dotsc,\zeta(4-p).
  \end{align*}
  Aquí, $\zeta$ es la
  función zeta de Riemann.\footnote{De hecho se sabe que los valores en los enteros
    negativos de la función zeta de Riemann son racionales, que es un fenómeno habitual,
    aunque no trivial, de muchas funciones $L$. Decimos que $p\mid\frac a b$ para un número
    racional $\frac a b$ con $(a,b)=1$ si $p\mid a$.}
\end{introthm}

La razón por la cual \textsc{Ernst Eduard Kummer} (1810--1893) se interesó en esto es que
fue capaz de demostrar un caso especial del Último Teorema de Fermat para exponentes $p$ con
$p\nmid h_p$. Su demostración bonita es muy instructiva para entender la teoría básica de
campos ciclotómicos, que juegan un papel importante en la teoría de Iwasawa, y la esbozamos
en el \cref{ejer:fermat}.

Una de las ideas revolucionarias de Iwasawa fue usar su teoría de torres de extensiones
infinitas y el álgebra de Iwasawa para estudiar estas conexiones sorprendentes entre las
funciones $L$ con grupos de clases o sus análogos y buscar una explicación profunda para
ellas. Esto es viable porque no sólo ciertos valores especiales de funciones $L$ son
racionales (o al menos algebraicos), sino además varían de manera $p$-ádicamente
continua. Estos asombrosos fenómenos llevan a la existencia de un análogo $p$-ádico de
muchas funciones $L$. Este análogo es \enquote{más algebraico} en el sentido que puede ser
visto como elemento en el álgebra de Iwasawa $\LL$. Iwasawa intuyó que estas \emph{funciones
  $L$ $p$-ádicas} son como un eslabón intermedio entre las funciones $L$ clásicas
(complejas) y los grupos de origen aritmético \dots\ esto es la \emph{Conjetura Principal de
  Iwasawa}.

De manera más precisa, en el caso clásico esto significa lo siguiente. El análogo $p$-ádico
de la función zeta de Riemann fue construido por primera vez por \textsc{Tomio Kubota}
(*1930) y \textsc{Heinrich-Wolfgang Leopoldt} (1927--2011), y luego con métodos más
novedosos por Iwasawa. Estos últimos métodos ven este análogo $p$-ádico como un objeto
(esencialmente) en el álgebra de Iwasawa que está conectado a la función compleja de Riemann
por una \emph{fórmula de interpolación}:

\begin{introthm}[Kubota/Leopoldt, 1964; Iwasawa, 1969]\label{thm:palf-intro}
  Existe un único elemento $\mu\in\LL[1/h]$, donde $\LL$ es el álgebra de Iwasawa y
  $h\in\LL$ es un elemento regular, con la propiedad de que
  \[ \kappa^{1-n}(\mu) = (1-p^{n-1})\zeta(1-n) \] para cada $n\in\Nuno$, donde los
  $\kappa^{1-n}\colon\LL\rightarrow\Zp$ con $n\in\Nuno$ son una familia de morfismos canónicos.
\end{introthm}

Para $n\ge1$ sea $X_n$ un cierto subgrupo\footnote{Más precisamente, tomamos el subgrupo de
  la $p$-parte donde la conjugación compleja actúa por $-1$.} del grupo de clases del campo
$K_n:=\mathbb Q(\mu_{p^n})$ y sea $X=\varprojlim_n X_n$. Entonces $X$ es un módulo sobre el
álgebra de Iwasawa $\LL$. La Conjetura Principal de Iwasawa, demostrada por \textsc{Barry
  Mazur} (*1937) y \textsc{Andrew Wiles} (*1953), conecta la función zeta de Riemann usando
su análogo $p$-ádico, con estos grupos de clases:

\begin{introthm}[Mazur/Wiles, 1984]\label{thm:mc-intro}
  Existe un morfismo de $\LL$-módulos \[ \LL/(h\mu)\rightarrow X \] con núcleo y
  conúcleo finito. 
\end{introthm}

El poder de esta afirmación podría no ser obvio a primera vista. Sin embargo, es una de las
relaciones más profundas entre las funciones $L$ y la aritmética. El criterio de Kummer del
\cref{thm:kummer-crit} es una consecuencia de la Conjetura Principal, como son otros resultados
sobre los grupos de clases como los teoremas de Stickelberger y Herbrand/Ribet.

Este texto invita al lector que conozca los fundamentos de la teoría de números
algebraica, al fascinante mundo de la Teoría de Iwasawa. Vamos a conocer el álgebra de
Iwasawa (capítulos \labelcref{sec:algebra-de-iwasawa} y \labelcref{sec:modulos-iwasawa}) y
su teoría de torres de extensiones y usar esto para demostrar su \cref{thm:iwasawa-intro} y
algunos resultados más sobre grupos de clases (\cref{sec:grupos-de-clases}). Luego vamos
a construir la función zeta $p$-ádica de Riemann del \cref{thm:palf-intro} y las funciones
$L$ $p$-ádicas de Dirichlet (\cref{sec:palf}). Después vamos explicar la Conjetura Principal
de Iwasawa y sus implicaciones (\cref{sec:mc}). La demostración de la
Conjetura Principal es mucho más difícil y por eso no podemos decir mucho sobre ella en este
texto. Finalmente, en el \cref{sec:generalizaciones} vamos a describir algunas de las áreas de
investigación activa en la Teoría de Iwasawa que muestran como todas estas ideas pueden ser
generalizadas a nuevos terrenos.

\section*{Guía}

Los conocimientos necesarios para entender este texto son modestos. Asumimos que el lector tiene conocimientos básicos del álgebra conmutativa, la teoría de Galois, la teoría de números algebraica y la topología. En el \cref{chap:preambulo} recapitulamos algunos resultados importantes de dichas áreas y damos referencias. Los capítulos
\ref{sec:algebra-de-iwasawa}--\ref{sec:mc} constituyen la parte principal de este texto, en el que explicamos en detalle y de la manera más autocontenida posible los aspectos de la Teoría de Iwasawa clásica que nos parecen más importantes. Finalmente, en el último
\cref{sec:generalizaciones} explicamos generalizaciones. Únicamente en este capítulo suponemos algunos conocimientos adicionales.

El siguiente diagrama muestra la dependencia de los capítulos.
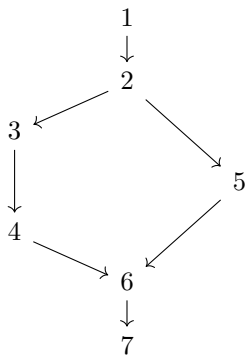
\begin{center}
  \begin{tikzcd}[row sep=1.2ex,column sep=1cm]
    & \ref{chap:preambulo}\arrow[dd]\\\\
    & \ref{sec:algebra-de-iwasawa}\arrow[dl]\arrow[ddr]\\
    \ref{sec:modulos-iwasawa}\arrow[dd]\\
    && \ref{sec:palf}\arrow[ddl]\\
    \ref{sec:grupos-de-clases}\arrow[dr]\\
    &\ref{sec:mc}\arrow[dd]\\\\
    &\ref{sec:generalizaciones}
  \end{tikzcd}
\end{center}

\section*{Otros textos}

Aquí queremos mencionar algunos textos importantes que pueden ser útiles para un lector que
quiere aprender de la Teoría de Iwasawa y que influyeron a los autores en su camino de
conocerla. Por supuesto nuestra presentación de esta área no es ni remotamente completa, y
estos otros textos pueden complementarla.

Los libros de Washington \cite{MR1421575} y Lang \cite{MR1029028} sobre campos ciclotómicos
son clásicos que cubren mucho material, entre otros la Teoría de Iwasawa clásica y mucho de
lo que hacemos aquí, aunque la presentación ya no es la más moderna. Una buena visión
conjunta sobre la Teoría de Iwasawa clásica es dada por las notas de Sharifi
\cite{SharifiIT}, pero él asume más conocimientos previos como por ejemplo cohomología de
grupos. Un texto más compendiado que también incluye la Teoría de Iwasawa para curvas
elípticas son las notas de Wüthrich \cite{MR3586809}.  También queremos mencionar el libro
\cite{NSW} de Neukirch, Schmidt y Wingberg que contiene algunos aspectos (más algebraicos)
de la Teoría de Iwasawa. Finalmente, el libro \cite{MR2256969} de Coates y Sujatha explica
la Conjetura Principal de Iwasawa y contiene una demostración completa sin pedir muchos
conocimientos previos.

El texto corto de Kato \cite{MR2334196} de su plática en el ICM de 2006 da un panorama
general y bonito de la Teoría de Iwasawa clásica y también la más moderna, incluyendo
desarrollos recientes. En un estilo inteligible explica también mucho sobre las motivaciones,
con foco sobre todo en la Conjetura Principal y sus generalizaciones. Otro texto digno de
leerse es \cite{MR1846466} de Greenberg, que también incluye mucha información sobre la
historia de la teoría.

Por supuesto deberíamos mencionar algunos textos de Iwasawa mismo; aquí solo listamos los
que nos parecen más importantes. En \cite{MR0124316} demostró su fórmula del
\cref{thm:iwasawa-intro} para el orden de los grupos de clases, uno de los primeros
resultados importantes en su teoría. Luego en \cite{MR0269627} reinterpretó los resultados
de Kubota y Leopoldt sobre las funciones $L$ $p$-ádicas, dando una nueva demostración de su
existencia. Esto le permitió formular su Conjetura Principal en \cite{MR0255510}, después de
ya haber demostrado un caso muy especial en \cite{MR0215811} (véase
\cite[(4.1)]{MR1846466}). Finalmente, el texto \cite{MR0349627} contiene un gran panorama en
general de los aspectos algebraicos de su trabajo junto con nuevos resultados.

Para literatura más avanzada o más especifica remitimos a las referencias que se encuentran
durante el texto, sobre todo en el \cref{sec:generalizaciones} sobre generalizaciones.

\section*{Agradecimientos}

En un inicio este texto consistía de nuestras notas del curso \enquote{Introducción a la Teoría de Iwasawa} que impartimos en la Escuela de Otoño de Teoría de Números, en el marco de las actividades del 51 Congreso de la Sociedad Matemática Mexicana (21–26 octubre 2018) en Villahermosa, Tabasco. Agradecemos profundamente a Carlos Castaño la invitación que nos extendió para participar como ponentes y por haber alentado la redacción de este trabajo. 

También queremos agradecer a los participantes del curso, especialmente a Tim Gendron, cuyos valiosos comentarios mejoraron la redacción del texto.

Un agradecimiento profundo a nuestros colegas del Instituto de Matemáticas de Heidelberg: Benjamin Kupferer, Pavel Sechin y Oliver Thomas, por su disponibilidad para discutir de matemáticas con nosotros y especialmente a Katharina Hübner quien también leyó una versión preliminar. Las innumerables discusiones que tuvimos con ellos son invaluables para la redacción de este trabajo. 

Agradecemos a Otmar Venjakob quien incondicionalmente apoyó este proyecto, con ánimos de
atraer a más estudiantes hispanohablantes a la Teoría de Iwasawa.

De gran ayuda han sido las anotaciones y comentarios de Cornelius Greither, además de
comentar la parte matemática nos ayudó a mejorar la presentación y el estilo.

Una parte esencial de los \cref{sec:algebra-de-iwasawa,sec:modulos-iwasawa} está basada en un curso dictado por Jean-Fran\c{c}ois Jaulent en la Universidad de Burdeos en 2014, a quien agradecemos sus detalladas explicaciones.

\section*{Notación}

El símbolo $p$ sin mayor explicación siempre denota un primo, y si no decimos otra cosa
asumimos que es impar (la mayoría de lo que haremos funciona también si $p=2$, pero así la
presentación es ligeramente mas fácil).

Para evitar confusión con los números naturales, escribiremos $\Ncero$ para los números
enteros $\ge0$ y $\Nuno$ para aquellos $\ge1$ y no usaremos el símbolo $\mathbb N$.

Fijemos cerraduras algebraicas $\Qbar$ de $\Q$, $\overline\Q_\ell$ de $\Q_\ell$ para cada
primo $\ell$ y $\C$ de $\R$ y encajes de $\Qbar$ en $\overline\Q_\ell$ y en $\C$.

Escribimos $\GQ$ para el grupo absoluto de Galois de $\Q$, es decir $\GQ=\Gal(\Qbar/\Q)$, y
de manera similar para los otros campos. Entonces nuestros encajes definen inclusiones de grupos
\[ \G{\Q_\ell}\hookrightarrow\GQ, \quad \G\R\hookrightarrow\GQ. \]
En particular, tiene sentido hablar de la conjugación compleja en objetos con una acción de $\GQ$.

\tableofcontents

\mainmatter

\chapter{Preámbulo}
\label{chap:preambulo}

En este capítulo coleccionamos algunos resultados de varias partes de las matemáticas que
necesitaremos a lo largo del texto y damos referencias para ellos.
El lector es libre de saltar este capítulo y solo consultarlo por necesidad.

\section{Estructuras algebraicas profinitas}
\label{sec:profinito}

Las estructuras algebraicas profinitas son generalizaciones de las estructuras algebraicas finitas
que en muchos aspectos se comportan similarmente y juegan un papel importante en la teoría de
Iwasawa y en la teoría de números en general. Una referencia muy útil que explica muchos
aspectos del álgebra profinita es el libro \cite{MR2599132}, al cual remitimos para más
detalles. 

\begin{defi}
  Un \define{grupo profinito} es un grupo topológico que, como grupo topológico, es isomorfo
  a un límite inverso de grupos finitos. De la misma manera definimos \emph{anillos profinitos}\index[def]{anillo!profinito}, \emph{módulos profinitos}\index[def]{módulo@módulo!profinito} y
  \emph{álgebras profinitas}\index[def]{algebra profinita@álgebra profinita} sobre anillos profinitos.

  Un grupo profinito se llama \define[grupo pro-p]{pro-$p$} si es (isomorfo a un) límite de grupos
  finitos cuyos órdenes son potencias de $p$, y de manera similar para anillos, módulos, etcétera. 
\end{defi}

Ejemplos de grupos profinitos incluyen por supuesto todos los grupos finitos o los enteros
$p$-ádicos $\Zp=\lim_{r\in\Nuno}\Z/p^r\Z$ (que incluso son un anillo profinito).

Los grupos profinitos aparecen naturalmente en la teoría de Galois para extensiones
infinitas. Resumimos estos resultados a continuación.

Fijamos una extensión de Galois de campos $L/K$ con grupo de Galois $G$. Sea
$\mathcal{Z}$ el conjunto de campos intermedios entre $K$ y $L$ y $\mathcal{S}$ el conjunto
de subgrupos de $G$.
Si $L/K$ es finito entonces el teorema principal de la teoría de Galois dice que los mapeos
\begin{align*}
  F\colon \mathcal{S}\rightarrow\mathcal{Z},&\quad H\mapsto L^H := \{ x\in L : \forall \sigma\in H\colon \sigma(x)=x \} \\
  U\colon \mathcal{Z}\rightarrow\mathcal{S},&\quad M\mapsto \Gal(L/M)
\end{align*}
son biyecciones inversas la una a la otra.
Si $L/K$ es infinito, es fácil ver que todavía tenemos $F\circ U = \id_{\mathcal{Z}}$.
Sin embargo, la composición $U\circ F$ en general no es la identidad, como muestra el
ejemplo siguiente. Sean $p$ y $\ell$ primos diferentes con $\ell\neq2$. Consideramos la
extension de Galois de $K=\F_p$ que es dada por
\[ L := \bigcup_{i=0}^\infty K_{i}, \] donde para $i\in\Nuno$, \[ K_i := \F_{p^{\ell^i}} \] es la única extension de $K$ de grado $\ell^i$. Sea $G=\Gal(L/K)$ y $H$ el subgrupo generado
por el Frobenius $\Frob$. Entonces $L^G=L^H=K$, pero se puede construir un elemento $\sigma$ de $G$
que no está en $H$.  Para esto sea $e_i = 1+\ell+\ell^2+\dotsc+\ell^{i-1}$ para
$i\in\Nuno$. Cada $x\in L$ está en algún $K_i$ y definimos
\[ \sigma(x) = \Frob^{e_i}(x). \] Como $e_{i+1} \equiv e_i\ (\mod \ell^i)$ tenemos
$\Frob^{e_{i+1}}|_{K_i}=\Frob^{e_i}|_{K_i}$, así que esto define un elemento $\sigma\in
G$. Si $\sigma\in H$ entonces existiría $n\in\Ncero$ con $\sigma=\Frob^n$. Esto
significaría 
\[ n\equiv e_i\ (\mod \ell^i) \] para cada $i\ge1$. Multiplicando ambos lados por $(1-\ell)$
obtenemos $n(1-\ell)\equiv 1\ (\mod \ell^i)$ para cada $i\in\Nuno$, es decir $n(1-\ell)=1$,
que no puede ser porque $\ell>2$.

Aun así, este ejemplo nos da una idea como se podría reparar la situación. Aunque
$\sigma\notin H$, los $\Frob^{e_i}\in H$ \enquote{aproximan} $\sigma$ en el sentido que
coinciden con $\sigma$ en subextensiones $K_i$ mas y mas grandes (pero siempre finitas). Esto
sugiere definir la siguiente topología.

\begin{defi}
  Sea $L/K$ una extensión de Galois con grupo $G$. La \define{topología de Krull} es
  definida con decir que los subgrupos \[ \Gal(L/M)\text{, $M$ una subextensión finita} \]
  sean una base de vecindades del elemento neutro.
\end{defi}

Es decir, dos elementos $\sigma,\tau\in G$ están \enquote{cercanos} si y solo si existe una
subextensión finita $M$ con $\sigma^{-1}\tau\in\Gal(L/M)$ para la cual $\Gal(L/M)$ sea
\enquote{pequeño}, es decir si y solo si $\sigma$ y $\tau$ coinciden en un subcampo \enquote{grande}.

Se verifica entonces fácilmente que $G$ es un grupo topológico, es decir la multiplicación y
la inversión son operaciones continuas. Además, para una extensión finita obtenemos la topología discreta.
Con esta topología tenemos \[ U(F(H)) = \overline{H} \] para cada subgrupo $H$ de $G$. Esto
implica la siguiente generalización del teorema principal de la teoría de Galois.

\begin{thm}[Teorema principal de la teoría de Galois infinita]\label{thm:hauptsatz-unendlgal}
  Sea $L/K$ una extensión de Galois con grupo $G$, $\mathcal{Z}$ el conjunto de campos
  intermedios entre $K$ y $L$, y $\mathcal{S}$ el conjunto de los subgrupos \emph{cerrados} de
  $G$.
  \begin{enumerate}
  \item Los mapeos
    \begin{align*}
      F\colon \mathcal{S}\rightarrow\mathcal{Z},&\quad H\mapsto L^H \\
      U\colon \mathcal{Z}\rightarrow\mathcal{S},&\quad M\mapsto \Gal(L/M)
    \end{align*}
    son biyecciones inversas la una a la otra que invierten inclusiones.
  \item  Un campo intermedio $M$ es normal sobre $K$ si y solo si el subgrupo $\Gal(L/M)$ es normal
    en $G$, en este caso la restricción a $M$ da un isomorfismo
    \[ {\Gal(L/K)}/{\Gal(L/M)}\isomarrow\Gal(M/K). \]
  \item  Un campo intermedio $M$ es finito sobre $K$ si y solo si el subgrupo $\Gal(L/M)$ es
    abierto en $G$.
  \item El mapeo canónico
    \[ G\isomarrow\varprojlim_{\substack{M\in\mathcal Z\\\text{finito}}}\Gal(M/K)=\varprojlim_{\substack{H\in\mathcal
          S\\\text{abierto}}}G/H \]
    es un isomorfismo.
  \end{enumerate}
\end{thm}
\begin{proof}
  \cite[Thm.\ 2.11.3]{MR2599132}
\end{proof}

En particular, este teorema muestra que cada grupo de Galois es un grupo
profinito.\footnote{También es verdad que cada grupo profinito es un grupo de Galois, véase
  \cite[Thm.\ 2.11.5]{MR2599132}.}

\begin{defi}
  Sea $G$ un grupo profinito. Entonces un elemento $g\in G$ se llama \define{generador
    topológico} si el subgrupo generado por $g$ es denso en $G$.
\end{defi}

Por ejemplo, $1\in\Zp$ es un generador topológico.

\begin{remark}\label{rem:locally-profinite}
  La propiedad de un grupo topológico de ser profinito puede ser caracterizado de manera
  puramente topológica: Un grupo topológico es profinito si y solo si es compacto, Hausdorff
  y totalmente disconexo. Lo mismo es verdad para un anillo topológico; aquí, esto
  incluso es equivalente al anillo siendo sólo compacto y Hausdorff. Véase \cite[Thm.\ 2.1.3,
  Prop.\ 5.1.2]{MR2599132} para estos hechos. Generalizando esto, un grupo topológico se
  llama \define[grupo topológico localmente profinito]{localmente profinito} si es Hausdorff, totalmente disconexo y localmente
  compacto; análogamente se define anillos y álgebras localmente profinitos.
\end{remark}

Muchas construcciones del álgebra abstracta tienen análogos profinitos. Por ejemplo, $\Z$
tiene la propiedad de que para cualquier grupo $G$ y cada elemento $g\in G$ hay un único
morfismo de grupos $f\colon\Z\rightarrow G$ tal que $f(1)=g$. El grupo profinito $\Z_p$
tiene una propiedad similar: Para cualquier grupo pro-$p$ $G$ y cada elemento $g\in G$ hay
un único morfismo de grupos topológicos $f\colon\Z_p\rightarrow G$ tal que $f(1)=g$. Se dice
que $\Zp$ es un \define{grupo abeliano pro-$p$ libre} de rango $1$. Por esta razón, como
cada grupo abeliano es un módulo sobre $\Z$ de manera única, cada grupo abeliano pro-$p$ es
de manera única un módulo sobre $\Z_p$.

Otra construcción que queremos extrapolar al mundo profinito es la del álgebra de grupos. Si
$R$ es un anillo y $G$ es un grupo, tenemos el álgebra de grupos $R[G]$ con la propiedad
siguiente: para cada $R$-álgebra $S$ y cada morfismo de grupos $G\rightarrow S^\times$ hay
una única extensión a un morfismo de $R$-álgebras $R[G]\rightarrow S$.

\begin{defi}
  Sea $R$ un anillo profinito y $G$ un grupo profinito. Definimos el \define{anillo de
    grupos profinito} como \[ R\llbracket G\rrbracket :=\varprojlim_{N\underset{\circ}{\trianglelefteq} G } R[G/N], \] el límite tomado
  sobre todos los subgrupos normales abiertos de $G$. 
\end{defi}

El grupo $G$ puede ser visto canónicamente como subconjunto de $R\llbracket G\rrbracket$ y
el álgebra de grupos ordinaria $R[G]$ es una $R$-subálgebra que es densa, véase \cite[Lem.\ 5.3.5]{MR2599132}.

\begin{prop}\label{prop:propiedad-universal-anillo-de-grupos}
  Sea $R$ un anillo profinito, $G$ un grupo profinito y $S$ una $R$-álgebra
  profinita. Entonces cada morfismo de grupos topológicos $G\rightarrow S^\times$ se extiende
  de manera única a un morfismo de $R$-álgebras topológicas \[ R\llbracket G\rrbracket
    \rightarrow S. \]
\end{prop}
\begin{proof}
  Por la propiedad universal del álgebra de grupos ordinaria $R[G]$ y porque este anillo es
  denso en $R\llbracket G\rrbracket$, es claro que la extensión es única si existe.
  Escribimos $S=\varprojlim_iS_i$ con $R$-álgebras finitas $S_{i}$. Notemos que entonces
  $S^\times=\varprojlim_i S_i^\times$, en particular $S^\times$ es un grupo
  topológico (¡Esto no es necesariamente cierto para un anillo topológico arbitrario porque la
  inversión no tiene porque ser continua!). Para cada $i$ consideramos la composición
  $G\rightarrow S^\times\rightarrow S_i^\times$. Por continuidad se factoriza a través del
  grupo finito $G/N$ para un subgrupo normal abierto $N$. La propiedad del álgebra de grupos
  $R[G/N]$ nos da un morfismo de $R$-álgebras $R[G/N]\rightarrow S_i$ que claramente es
  continuo. Si componemos esto con la proyección de $R\llbracket G\rrbracket$ obtenemos un
  morfismo $R\llbracket G\rrbracket\rightarrow S_i$. La familia de estos morfismos para cada
  $i$ es obviamente compatible y la propiedad universal del límite luego da el morfismo
  $R\llbracket G\rrbracket\rightarrow S$ que deseamos.
\end{proof}

Si $R$ es un anillo, $G$ es un grupo y $M$ es un $R$-módulo con una acción $R$-lineal de
$G$, entonces $M$ es de manera única un módulo sobre $R[G]$. Algo similar es verdad en el
álgebra profinita, aunque hay que tener cuidado con algunos detalles topológicos.

\begin{prop}\label{prop:modulos-sobre-rg}
  Sea $R$ un anillo profinito, $G$ un grupo profinito y $M$ un $R$-módulo profinito con una
  acción continua y $R$-lineal de $G$. Supongamos que podemos escribir $M$ como
  $M=\lim_iM/U_i$, los $U_i$ siendo submódulos abiertos de $M$ estables bajo la acción de $G$.

  Entonces $M$ es de manera canónica un módulo sobre
  $R\llbracket G\rrbracket$.
\end{prop}
\begin{proof}
  Escribimos $M_i=M/U_i$, que es un módulo finito. Como $U_i$ es estable bajo la acción de
  $G$ obtenemos una acción de $G$ en $M_i$, es decir un morfismo
  $G\rightarrow\Aut_R(M_i)$. Porque estos morfismos son compatibles si cambiamos $i$,
  obtenemos un morfismo continuo
  \begin{equation*}
    G\rightarrow\varprojlim_i\Aut_R(M_i)
  \end{equation*}
  de grupos profinitos. Usando que $\varprojlim_i\Aut_R(M_i)=(\varprojlim_i\End_R(M_i))^\times$, la
  propiedad universal del álgebra profinita de grupos induce un morfismo de $R$-álgebras
  \begin{equation*}
    R\llbracket G\rrbracket \rightarrow\varprojlim_i\End_R(M_i)\rightarrow\End_R(M)
  \end{equation*}
  que da $M$ una estructura canónica como $R\llbracket G\rrbracket$-módulo.
\end{proof}

La hipótesis en la proposición anterior es cierta por ejemplo si cada submódulo abierto de
$M$ es estable bajo la acción de $G$, o si $M$ es dado como límite $M=\lim M_i$ con
$R\llbracket G\rrbracket$-módulos finitos.

\begin{remark}\label{rem:hom-de-representaciones}
  En general, sea $R$ un anillo conmutativo y $G$ un grupo. Como mencionamos, un
  $R[G]$-módulo es lo mismo que un $R$-módulo con una acción $R$-lineal de $G$ -- esto
  se llama \emph{representación $R$-lineal de $G$}\index[def]{representación}\index[def]{representación de grupo|see{representación}}. Si $V$ y
  $W$ son dos tal representaciones, hacemos los morfismos de $R$-módulos $\Hom_R(V,W)$ entre
  ellos un $R[G]$-módulo al definir la acción de $G$ como
  \[ (gf)(v)=g(f(g^{-1}v)) \quad\text{para }g\in G,\ f\in\Hom_R(V,W),\ v\in V \]
  (el lector debería verificar que esto induce una acción por la izquierda de $G$ en
  $\Hom_R(V,W)$).
  
  De manera similar, si $R$ es un anillo conmutativo profinito, $G$ es un grupo profinito y $V$,
  $W$ son $R\llbracket G\rrbracket$-módulos, definimos una acción de $G$ en
  $\Hom_R(V,W)$ con la misma fórmula. Si $V$ y $W$ cumplen la hipótesis de la
  \cref{prop:modulos-sobre-rg} entonces es fácil ver que $\Hom_R(V,W)$ la cumple también,
  así que en esta situación $\Hom_R(V,W)$ es un $R\llbracket G\rrbracket$-módulo.
\end{remark}

\ejercicios

\begin{ejer}\label{ejer:gf-finito}
  Demuestre que el grupo de Galois absoluto de un campo finito es canónicamente isomorfo a
  $\widehat\Z$, que es definido por \[ \widehat\Z:=\varprojlim_{n\in\Nuno}\Z/nZ \]
  donde el conjunto $\Nuno$ sobre el que tomamos el límite es ordenado por divisibilidad y los
  mapeos entre los $\Z/n\Z$ son las proyecciones canónicas. Bajo este isomorfismo,
  $1\in\widehat\Z$ corresponde al morfismo Frobenius. Use el \cref{thm:hauptsatz-unendlgal}
  para esto.
\end{ejer}

\begin{ejer}\label{ejer:morfismos-pro-p-pro-ell}
  Sean $p$ y $\ell$ primos diferentes. Demuestre que cada morfismo de un grupo pro-$p$ a un
  grupo pro-$\ell$ es trivial.
\end{ejer}

\begin{ejer}
  Formule la afirmación de la \cref{prop:propiedad-universal-anillo-de-grupos} usando
  funtores adjuntos.
\end{ejer}

\begin{ejer}
  Demuestre que para cada grupo pro-$p$ $G$ y cada elemento $g\in G$ existe un único
  morfismo de grupos profinitos $f\colon\Zp\rightarrow G$ con $f(1)=g$. Demuestre también
  la afirmación análoga con grupos profinitos arbitrarios y $\widehat\Z$ en lugar de
  $\Zp$. Formule estas afirmaciones usando funtores representables o adjuntos; vea también
  \cite[§3.3]{MR2599132}.
\end{ejer}

\begin{ejer}\label{ejer:locally-profinite}
  Demuestre que la propiedad universal del álgebra de grupos profinita de la
  \cref{prop:propiedad-universal-anillo-de-grupos} se generaliza de la manera siguiente. Si
  $R$ es un anillo profinito, $G$ un grupo profinito y $S$ es una $R$-álgebra localmente
  profinita (véase la \cref{rem:locally-profinite}) entonces cada morfismo continuo
  $G\rightarrow S^\times$ se extiende de manera única a un morfismo de $R$-álgebras
  topológicas \[ R\llbracket G\rrbracket \rightarrow S. \]
\end{ejer}

\section{Compendio de la teoría de números algebraica}

Aquí resumimos los resultados más básicos de la teoría de números algebraica que vamos usar
en el resto del texto. El lector debería estar familiarizado con estas definiciones.

Los objetos de estudio central de la teoría de números algebraica son los siguientes.

\begin{defi}
  \begin{enumerate}
  \item   Un \define{campo de números} es una extensión finita $K$ de $\Q$. Su \define{anillo de
      enteros} son los elementos $\O_K$ que son enteros sobre $\Z$.
  \item Si $K/\Q$ es cualquier extensión algebraica, no necesariamente finita, todavía
    podemos definir su \emph{anillo de enteros} $\O_K$ como el conjunto de los elementos
    enteros sobre $\Z$.
  \end{enumerate}
\end{defi}

Se estudia estos campos sobre todo vía los ideales de su anillo de enteros, como explicamos
a continuación.

\begin{defi}\label{defi:div-etc}
  \begin{enumerate}
  \item Un \define{ideal fraccional} es un $\O_K$-submódulo no trivial de $K$ finitamente
    generado. En particular, un ideal de $\O_K$ es un ideal fraccional. Denotamos
    $\Div(\O_K)$ el conjunto de ideales fraccionales de $\O_K$.
  \item Para un ideal fraccional $I$ sea $I^{-1}=\{x\in K \mid xI\subseteq \O_K\}$. Esto
    también es un ideal fraccional que se llama el \emph{ideal inverso}\esindex[def]{ideal fraccional!  inverso} de $I$.
  \item\label{defi:prod-ideales} El producto $IJ$ de dos ideales fraccionales $I$, $J$ es el
    $\O$-submodulo de $K$ generado por todos los elementos $ij$ con $i\in I$, $j\in J$.
  \item Un \emph{ideal fraccional principal}\esindex[def]{ideal fraccional!  principal} es un ideal fraccional de la forma
    $\alpha\O_K$ con $\alpha\in K^\times$, que denotamos también $(\alpha)$. Denotamos
    $\Prin(\O_K)$ el subconjunto de $\Div(\O_K)$ de los ideales fraccionales principales.
  \end{enumerate}
\end{defi}

En general, los elementos de $\O_K$ no tienen una factorización única en primos, como es el
caso en el anillo $\Z$. Pero es verdad es que los ideales fraccionales tienen dicha
factorización.

\begin{thm}
  Los ideales fraccionales $\Div(\O_K)$ con la multiplicación y inversión como en la
  \cref{defi:div-etc} constituyen un grupo abeliano con elemento neutro $\O_K$. Este grupo
  es libre generado por los ideales primos no ceros de $\O_K$. Es decir, cada ideal
  fraccional $I$ se escribe de forma única (salvo al orden) como
  \[ I=\mathfrak p_1^{e_1}\dotsm\mathfrak p_r^{e_r} \]
  con ideales primos diferentes $\mathfrak p_i$ y $e_i\in\Z$.
\end{thm}
\begin{proof}
  \cite[Chap.\ I, (3.3)]{MR1697859}
\end{proof}

Esto sugiere que es mejor trabajar con ideales fraccionales en lugar de elementos de
$K^\times$ -- de hecho, esto es como se originó la palabra \enquote{ideal}, porque son como
\enquote{números idealizados} de $K$. La discrepancia entre los elementos de $K^\times$ y
los ideales fraccionales es medida por dos grupos:

\begin{defi}
  El \define{grupo de clases} de $K$ es el cociente \[ \Cl(K)=\Div(\O_K)/\Prin(\O_K). \]
\end{defi}

Tenemos una sucesión exacta
\[ 1 \rightarrow \O_K^\times\rightarrow K^\times\rightarrow \Div(\O_K) \rightarrow \Cl(K)
  \rightarrow 1. \] Es decir, el grupo $\O_K^\times$ mide qué tan lejos está la asociación
$\alpha\mapsto(\alpha)$, de $K^\times$ a ideales fraccionales, de ser única. Por su lado, el
grupo $\Cl(K)$ mide \enquote{cuántos más ideales que ideales principales hay} o \enquote{qué
  tan lejos está la factorización en primos de elementos de $\O_K$ de ser única}. De hecho,
$\Cl(K)$ es trivial si y solo si $\O_K$ es un dominio de ideales principales. Por eso es muy
importante estudiar los grupos $\O_K^\times$ y $\Cl(K)$. El primero es descrito por el
teorema siguiente.

\begin{thm}[Teorema de las unidades de Dirichlet]
  Sea $r$ la cantidad de encajes $K\rightarrow\R$ y $c$ la cantidad de parejas conjugadas de
  morfismos $K\rightarrow\C$ con imagen no contenida en $\R$. Entonces (de manera no
  canónica)
  \[ \O_K^\times\isom\mu(K)\times\Z^{r+c-1}, \]
  donde $\mu(K)$ son las raíces de la unidad en $K$.
\end{thm}
\begin{proof}
  \cite[Chap.\ I, (7.4)]{MR1697859}
\end{proof}

Sobre el grupo de clases sólo sabemos lo siguiente.
\begin{thm}
  El grupo $\Cl(K)$ es un grupo abeliano finito.
\end{thm}
\begin{proof}
  \cite[Chap.\ I, (6.3)]{MR1697859}
\end{proof}
En general es muy difícil describir su estructura. La Teoría de Iwasawa permite obtener
resultados sobre el grupo de clases en algunas situaciones, como vamos a explicar en las
secciones más adelante.

\begin{defi}
  El orden del grupo de clases, $\#\Cl(K)$, se llama el \define{número de clases} de $K$ y
  se denota $h_K$.
\end{defi}

\begin{defi}
  Sea $K$ un campo de números. Una \define{plaza} de $K$ es
  \begin{arabiclist}
  \item\label{plaza:primo} un ideal primo no nulo de $\O_K$, o
  \item\label{plaza:real} un encaje $K\hookrightarrow \R$, o
  \item\label{plaza:compleja} una pareja de encajes $K\hookrightarrow\C$ conjugadas con
    imagen no contenida en $\R$.
  \end{arabiclist}
  Las plazas de tipo \ref{plaza:primo} se llaman \emph{plazas no arquimedianas}\index[def]{plaza@plaza!no arquimediana} y las
  plazas de tipo \ref{plaza:real} o \ref{plaza:compleja} se llaman \emph{plazas
    arquimedianas}\index[def]{plaza@plaza!arquimediana}; más específicamente, las de tipo \ref{plaza:real} se llaman
  \emph{plazas reales}\index[def]{plaza@plaza!real} y las de tipo \ref{plaza:compleja} \emph{plazas complejas}\index[def]{plaza@plaza!compleja}.

  Equivalentemente y más uniformemente, se puede definir una plaza de un campo de números
  $K$ como una clase de equivalencia de un valor absoluto en $K$, es decir, de una función
  $|\cdot|\colon K\rightarrow\R_{\ge0}$ que es multiplicativa, cumple la desigualdad
  triangular y que toma el valor $0$ sólo en $0$ (dos valores absolutos son equivalentes si
  definen la misma topología en $K$). Los estudiamos en el \cref{ejer:valores-abs}.  Se
  puede verificar que cada valor absoluto en $K$ viene o de la valuación en un ideal primo
  no nulo (en que caso cumple la desigualdad triangular ultramétrica y se llama no
  arquimediana) o de un encaje en $\R$ o $\C$ (en que caso se llama arquimediana). Para
  $K=\Q$ lo demostramos en el \cref{ejer:ostrowski}.

  Véase \cite[§II.3]{MR1697859} para más detalles.
\end{defi}

Sea $L/K$ una extensión finita de campos de números. Para cada ideal primo $\mathfrak p$ de
$\O_K$ podemos descomponer el ideal $\mathfrak p\O_L$ generado por $\mathfrak p$ en
$\O_L$ como
\[ \mathfrak p\O_L=\mathfrak P_1^{e_1}\dotsm\mathfrak P_g^{e_g} \]
con ideales primos $\mathfrak P_i$ de $\O_L$ y $e_i\in\Nuno$. En esta situación decimos que
cada uno de los $\mathfrak P_i$ está \emph{arriba}\index[def]{plaza@plaza!arriba de} de $\mathfrak p$ y $\mathfrak p$ está
\emph{debajo}\index[def]{plaza@plaza!debajo de} de cada uno de los $\mathfrak P_i$.

Para plazas arquimedianas $\sigma$ de $L$ y $\tau$ de $K$ decimos que $\sigma$ está
\emph{arriba}\index[def]{plaza@plaza!arriba de} de $\tau$ y $\tau$ está \emph{debajo}\index[def]{plaza@plaza!debajo de} de $\sigma$ si $\sigma|_K=\tau$.

\begin{defi}
  \begin{enumerate}
  \item   Los exponentes $e_i$ se llaman \define{grado de ramificación}. Para cada $i$ el grado de
    la extensión $f_i:=[\O_L/\mathfrak P_i:\O_K/\mathfrak p]$ se llama \define{grado de
      inercia}. El primo $\mathfrak P_i$ se llama \emph{ramificado}\index[def]{ideal primo!ramificado} si $e_i>1$ y se llama
    \emph{totalmente ramificado}\index[def]{ideal primo!totalmente ramificado} si además $f_i=1$. Decimos que el primo $\mathfrak p$ es
    \emph{ramificado}\index[def]{ideal primo!ramificado} si uno de los $\mathfrak P_i$ lo es y \emph{totalmente ramificado}\index[def]{ideal primo!totalmente ramificado}
    si todos los $\mathfrak P_i$ lo son. La extensión $L/K$ se llama \emph{ramificada}\index[def]{extensión!ramificada} si
    existe un ideal primo $\mathfrak p\subseteq\O_K$ no trivial que es
    ramificado.
  \item Una plaza arquimediana de $K$ se llama \emph{ramificada}\index[def]{plaza@plaza!ramificada} si es una plaza real que
    está debajo de una plaza compleja de $L$.
  \end{enumerate}
\end{defi}

\begin{ex}\label{ex:qi}
  Consideremos la extension $\Q(\complexi)/\Q$. Su anillo de enteros es $\Z[\complexi]\cong\Z[X]/(X^2+1)$,
  así que para un primo $p$ tenemos
  \[ \Z[\complexi]/(p)\cong\Fp[X]/(X^2+1)\cong
    \begin{cases}
      \Fp\times\Fp&\text{si } p\equiv1\ (\mod 4),\\
      \F_{p^2}&\text{si } p\equiv3\ (\mod 4),\\
      \Fp[X]/(X+1)^2&\text{si } p=2.
    \end{cases}
  \]
  De ahí podemos ver que los primos tal que $p\equiv3$ $(\mod 4)$ siguen siendo primos en
  $\Z[\complexi]$, los primos tal que $p\equiv1$ $(\mod 4)$ se descomponen en un producto de dos
  primos diferentes de $\Z[\complexi]$ y el primo $2$ se escribe como $2=\mathfrak p^2$ con un
  primo $\mathfrak p$ de $\Z[\complexi]$, que entonces está ramificado y es la única plaza no
  arquimediana que está ramificada. Además, $\Q(\complexi)$ obviamente tiene una única plaza
  arquimediana, que es compleja y por lo tanto ramificada, porque extiende la única plaza
  arquimediana de $\Q$, que es real.
\end{ex}

Siempre tenemos $\sum_{i=1}^ge_if_i=[L:K]$ (véase \cite[Chap.\ I, (8.2)]{MR1697859}).
Si la extensión $L/K$ es Galois, los $e_i$ y $f_i$ son iguales para cada $i$ (véase
\cite[Chap.\ I, (9.1) y p.\ 55]{MR1697859}), y los llamamos
simplemente $e$ y $f$. Entonces $[L:K]=efg$.

\begin{defi}\label{defi:ramificacion-inercia}
  Sea $L/K$ una extensión de Galois de campos de números y \[ \mathfrak p\O_L=\mathfrak
    P_1^{e}\dotsm\mathfrak P_g^{e} \] la descomposición de un ideal primo de $\O_K$ como
  arriba. El grupo $\Gal(L/K)$ actúa en los $\mathfrak P_i$ transitivamente \cite[Chap.\ I,
  (9.1)]{MR1697859}.
  \begin{enumerate}
  \item Para cada $i$ el estabilizador de $\mathfrak P_i$ se llama \define{grupo de
      descomposición} de $\mathfrak P_i$ y se denota $G_{\mathfrak P_i/\mathfrak p}$.
  \item El núcleo de la aplicación natural sobreyectiva \cite[Chap.\ I, (9.4)]{MR1697859}
    \[ G_{\mathfrak P_i}\twoheadrightarrow\Gal(\O_L/\mathfrak P_i/\O_K/\mathfrak p) \]
    se llama el \define{grupo de inercia} de $\mathfrak P_i$ y se denota $I_{\mathfrak
      P_i/\mathfrak p}$.
  \end{enumerate}
\end{defi}

\begin{prop}
  En la situación de la \cref{defi:ramificacion-inercia} tenemos
  \[ \#G_{\mathfrak P_i/\mathfrak p} = ef,\quad\#I_{\mathfrak P_i/\mathfrak p}=e. \] En
  particular, $\mathfrak p$ es no ramificado si y solo si $I_{\mathfrak P_i/\mathfrak p}$ es
  trivial y es totalmente ramificado si y solo si
  $I_{\mathfrak P_i/\mathfrak p} =G_{\mathfrak P_i/\mathfrak p}$ (para algún, o
  equivalentemente todo $i$).
\end{prop}
\begin{proof}
  \cite[Chap.\ I, (9.6)]{MR1697859}
\end{proof}

El resultado anterior nos permite extender estas definiciones a extensiones infinitas. Aquí
seguimos \cite[Appendix, §2]{MR1421575}. 

\begin{defi}\label{defi:ram-infinito}
  Sea $L/K$ una extensión de Galois, posiblemente infinita, donde $K$ es una extensión
  algebraica no necesariamente finita de $\Q$. Sea $\mathfrak P$ un ideal primo de $\O_L$ y
  $\mathfrak p=\mathfrak P\cap\O_K$, que es un ideal primo de $\O_K$. En esta situación
  decimos que $\mathfrak P$ está \emph{arriba}\index[def]{plaza@plaza!arriba de} de $\mathfrak p$ y $\mathfrak p$ está
  \emph{debajo}\index[def]{plaza@plaza!debajo de} de $\mathfrak P$.
  \begin{enumerate}
  \item El \define{grupo de descomposición} de $\mathfrak P$ es
    \[ G_{\mathfrak P/\mathfrak p}:=\{\sigma\in\Gal(L/K) \mid \sigma(\mathfrak P)=\mathfrak P \}. \]
  \item El \define{grupo de inercia} de $\mathfrak P$ es
    \[ I_{\mathfrak P/\mathfrak p}:=\{\sigma\in G_{\mathfrak P/\mathfrak p} \mid
      \forall\,x\in\O_K\colon \sigma(x)\equiv x \mod\mathfrak P \}. \]
  \item Decimos que $\mathfrak P$ es \emph{no ramificado}\index[def]{ideal!no ramificado} si $I_{\mathfrak P/\mathfrak p}$
    es trivial y es \emph{totalmente ramificado}\index[def]{ideal!totalmente ramificado} si
    $I_{\mathfrak P/\mathfrak p} =G_{\mathfrak P/\mathfrak p}$.
  \item Decimos que un primo $\mathfrak p$ de $\O_K$ es \emph{ramificado}\index[def]{ideal primo!ramificado} o
    \emph{totalmente ramificado}\index[def]{ideal primo!totalmente ramificado} si existe un primo $\mathfrak P$ de $\O_L$ arriba de
    $\mathfrak p$ que los es (equivalentemente, todos los primo $\mathfrak P$ de $\O_L$
    arriba de $\mathfrak p$ lo son).
  \end{enumerate}
\end{defi}

Terminamos la sección con la importante definición de la norma de ideales.

\begin{defi}\label{defi:norma-de-ideales}
  Sea $L/K$ una extensión finita de campos de números de grado $n$. El morfismo de grupos
  abelianos
  \[ \mathrm N\colon\Div(\O_L)\rightarrow\Div(\O_K),\quad\mathfrak P\mapsto\mathfrak
    p^{f_{\mathfrak P/\mathfrak p}} \] donde $\mathfrak P$ es un ideal primo en $\O_L$ y
  $\mathfrak p=\mathfrak P\cap\O_K$ se llama \define{norma relativa de ideales}. Tiene la
  propiedad de que para cada $I\in\Div(\O_K)$ tenemos $\mathrm N(I\O_L)=I^n$. Envía ideales
  principales a ideales principales y por eso induce un morfismo
  \[ \mathrm N\colon\Cl(L)\rightarrow\Cl(K) \] que también llamamos \define{norma relativa
    de ideales}.
\end{defi}

\ejercicios

\begin{ejer}
  Sea $K$ un campo de números e $I$ un ideal fraccional. Demuestre que $I^{-1}$ es un ideal
  fraccional y que $II^{-1}=\O_K$.
\end{ejer}

\begin{ejer}\label{ejer:p-nmid-h}
  Sea $K$ un campo de números, $h=\#\Cl(K)$ su número de clases y $p$ un primo. Demuestre
  que $p\nmid h$ es equivalente a lo siguiente: Para cada ideal fraccional $I\neq0$ de $K$,
  si $I^p$ es un ideal principal entonces $I$ es un ideal principal.
\end{ejer}

\begin{ejer}\label{ejer:valores-abs}
  Sea $K$ un campo de números y $|\cdot|\colon K\rightarrow\R_{\ge0}$ un valor absoluto, es decir
  para $a,b\in K$ tenemos que $|ab|=|a||b|$, $|a|=0$ $\iff$ $a=0$, y
  $|a+b|\le|a|+|b|$. Definimos $d(a,b)=|a-b|$ para $a,b\in K$.
  \begin{enumerate}
  \item Verifique que esto induce una
    topología en $K$ tal que $K$ es un campo topológico, i.\,e.\ la adición y la
    multiplicación son continuas.
  \item Llamamos dos valores absolutos $|\cdot|_1$, $|\cdot|_2$ equivalentes si inducen la
    misma topología en $K$. Demuestre que esto pasa si y solo si existe $s\in\R_{>0}$ tal
    que $|\cdot|_1^s=|\cdot|_2$.
  \end{enumerate}
\end{ejer}

\begin{ejer}\label{ejer:ostrowski}
  Sea $|\cdot|\colon\Q\rightarrow\R_{\ge0}$ un valor absoluto no trivial (es decir existe $a\in\Q^\times$
  tal que $|a|\neq1$). Le llamamos
  arquimediano si existe $n\in\Nuno$ tal que $|n|>1$.
  \begin{enumerate}
  \item Demuestre que $|\cdot|$ es no arquimediano si y solo si $|2|\le1$. Use una serie
    $2$-ádica para esto.
  \item Sea $|\cdot|$ arquimediano. Use la descomposición en primos en $\Z$ e inducción para
    demostrar que el valor absoluto es equivalente al valor absoluto clásico.
  \item Sea ahora $|\cdot|$ no arquimediano. Demuestre que hay un primo $p$ tal que $|p|<1$
    y que este primo es único. Concluya que $|\cdot|$ es equivalente al valor $p$-ádico.
  \end{enumerate}
\end{ejer}

\begin{ejer}
  En la situación del \cref{ex:qi} con el campo $\Q(\complexi)$, ¿Cuáles son los grupos de
  descomposición e inercia en $\Gal(\Q(\complexi)/\Q)$ para cada primo de $\Q(\complexi)$?
\end{ejer}

\section{Campos ciclotómicos}

En esta sección introducimos los ejemplos de campos de números que serán los más importantes
en este texto y alistamos algunas de sus propiedades.

\begin{defi}\label{defi:campos-ciclotomicos}
  Si $m\in\Nuno$ escribimos $\Q(\mu_m)$ para el campo generado sobre $\Q$ por las raíces de la
  unidad de orden $m$. Llamamos este campo \define[campo ciclotómico]{el $m$-ésimo campo ciclotómico}. El campo de números $\Q(\mu_m)$ es una extensión de Galois de $\Q$ cuyo grupo de Galois es isomorfo canónicamente a $(\Z/m\Z)^\times$ vía
  \begin{equation}
    \label{eqn:isom-gcic}
    (\Z/m\Z)^\times\isomarrow \Gal(\Q(\mu_m)/\Q), \quad a\mapsto\sigma_a
  \end{equation}
  con $\sigma_a$ actuando en una raíz de la unidad $\zeta$ como
  $\sigma_a(\zeta)=\zeta^a$. Muchas veces tomamos este isomorfismo como una identificación.

  En esta situación escribimos $\mu_m$ para el subgrupo de $\Q(\mu_m)^\times$ de las raíces
  $m$-ésimas de la unidad.
\end{defi}

\begin{prop}\label{prop:enteros-cic}
  El anillo de enteros de $\Q(\mu_m)$ es $\Z[\xi]$, donde $\xi$ es una raíz $m$-ésima
  primitiva de la unidad.
\end{prop}
\begin{proof}
  \cite[Chap.\ 1, (10.2)]{MR1697859}  
\end{proof}

Notemos que para cada campo ciclotómico $\Q(\mu_m)$ la conjugación compleja es un automorfismo bien
definido, es decir independiente del encaje en $\C$. Bajo el isomorfismo
\eqref{eqn:isom-gcic} corresponde a $-1\in(\Z/m\Z)^\times$.

En todos los siguientes resultados suponemos que $p$ es primo.

\begin{prop}\label{prop:o-es-mu-o-mas}
  Sea $\O=\O_{\Q(\mu_p)}$. Entonces cada elemento $u\in\O^\times$ se puede escribir
  como $u=\xi v$ con $\xi\in\mu_p$ y $v\in{(\O^\times)}^+$, donde
  ${(\O^\times)}^+\subseteq\O^\times$ son los elementos fijos por la conjugación compleja.
\end{prop}
\begin{proof}
  \cite[Prop.\ 1.5]{MR1421575}
\end{proof}

Los siguientes resultados son esenciales para entender la ramificación en las extensiones
ciclotómicas.

\begin{lem}\label{lem:cicunit}
  Sea $\xi$ una $p^{r}$-raíz primitiva de la unidad y sea $(p^{r},k)=1$, entonces
  $\frac{\xi^{k}-1}{\xi-1}$ es una unidad en $\Z[\mu_{p^{r}}]$.
\end{lem}
\begin{proof}
  Dejemos esto como ejercicio; véase el \cref{ejer:cicunit}, donde demostramos una afirmación
  más general.
\end{proof}

\begin{lem}\label{lem:totram}  Sea $\xi$ una $p^{r}$-raíz primitiva de la unidad. Entonces $(1-\xi)$ es un ideal primo en $\Q(\mu_{p^{r}})$ y $(1-\xi)^{(p-1)p^{r-1}}=(p)$, es
  decir $(p)$ es totalmente ramificado en $\Q(\mu_{p^{r}})$. En particular, $(1-\xi)$ es el
  único ideal de $\Q(\mu_{p^r})$ arriba de $(p)$.
\end{lem}
\begin{proof}
El conjunto de las $p^{r}$-raíces de la unidad satisfacen la ecuación $X^{p^{r}}-1=0$. Si son primitivas no pueden tener orden menor a $p^{r}$ por lo que tenemos
$$\dfrac{X^{p^{r}}-1}{X^{p^{r-1}}-1}=X^{(p-1)p^{r-1}}+X^{(p-2)p^{r-1}}+\cdots +1 = \prod_{(k,p^{r})=1}(X-\xi^{r}),$$
evaluando las expresiones en $1$ y tomando los ideales principales generados por los elementos, tenemos 
$$(p)=\prod_{(k,p^{r})=1}(1-\xi^{k}).$$ 
Por el \cref{lem:cicunit}, hay una igualdad de ideales $(1-\xi)=(1-\xi^{k})$ para $k$ tal que $(k,p^{r})=1$. Es decir $(p)=(1-\xi)^{(p-1)p^{r-1}}$, como el grado de la extensión $[\Q(\mu_{p^{r}}):\Q]$ es $(p-1)p^{r-1}$ tenemos que $(1-\xi)$ debe de ser primo y por lo tanto $p$ es totalmente ramificado. 
\end{proof}

\begin{prop}\label{prop:ramcycfin}
  Un primo $\ell$ es ramificado en $\Q(\mu_m)$ si y solo si $\ell\mid m$, salvo si $\ell=2$ y $(4,m)=2$.
\end{prop}
\begin{proof} Si $\ell\mid m$ entonces $\Q(\mu_{\ell})\subset \Q(\mu_{m})$. Como $\ell$
  ramifica en $\Q(\mu_{\ell})$ (\cref{lem:totram}) también lo hace en $\Q(\mu_m)$. Ahora, si
  $\ell$ no divide a $m=\prod p_i^{r_{i}}$, entonces $\ell$ no ramifica en cada
  $\Q(\mu_{p_i^{r_{i}}})$. Para ver esto hay que usar el discriminante de
  $\Q(\mu_{p_i^{r_{i}}})$, que es un elemento de $\Z$ que se puede definir para cualquier
  campo de números y tiene la propiedad de que un primo ramifica si y sólo si divide a este
  número \cite[Chap.\ III, Thm.\ 2.6]{MR1697859}; según \cite[Prop. 2.1]{MR1421575}
  el discriminante de $\Q(\mu_{p_i^{r_{i}}})$ es una potencia de $p_i$. Por lo tanto,
  tampoco ramifica en el compuesto $\Q(\mu_{m})$. Véase \cite[Chap.\ I,
  Cor. 10.4]{MR1697859} para una demostración diferente.
\end{proof}

También vamos a estudiar campos ciclotómicos infinitos.

\begin{defi}
    Si $p$ es un primo y $N\in\Nuno$ no es divisible por $p$, entonces escribimos
  $\Q(\mu_{Np^\infty})$ para la extensión infinita de $\Q$ generada por todas la raíces de
  la unidad de orden $Np^r$ para cada $r\in\Ncero$. Denotamos
  $\mu_{Np^\infty}\subseteq\Q(\mu_{Np^\infty})^\times$ el subgrupo de estas raíces de la unidad.

  Por la teoría de Galois infinita, su grupo de Galois es isomorfo a
  \begin{multline*}
    \Gal(\Q(\mu_{Np^\infty})/\Q)\isom\varprojlim_{r\in\Ncero}\Gal(\Q(\mu_{Np^r})/\Q)\isom\varprojlim_{r\in\Ncero}(\Z/Np^r\Z)^\times
    \\
    \isom\varprojlim_{r\in\Ncero}\left((\Z/N\Z)^\times\times(\Z/p^r\Z)^\times\right)\isom(\Z/N\Z)^\times\times\Z_p^\times.
  \end{multline*}
  En particular, si $N=1$, tenemos un isomorfismo
  \begin{equation}
    \label{eqn:kappa}
    \Gal(\Q(\mu_{p^\infty})/\Q)\isomarrow\Z_p^\times.
  \end{equation}
\end{defi}

El siguiente resultado, es sólo una reformulación de la \cref{prop:ramcycfin} para campos ciclotómicos infinitos. 

\begin{prop}\label{prop:ramcycinfin}
  Un primo $\ell$ es ramificado en $\Q(\mu_{Np^\infty})$ si y solo si $\ell=p$ o $\ell\mid
  N$, salvo si $\ell=2$ y $(4,Np)=2$.
\end{prop}

\ejercicios

\begin{ejer}\label{ejer:fermat}
  En este ejercicio vamos a esbozar la demostración de Kummer de un caso especial del Último
  Teorema de Fermat, siguiendo \cite[chap.\ 1]{MR1421575}. El Último Teorema de Fermat dice
  que si $n\ge3$ y $a,b,c\in\Z$ son enteros tal que \[ a^n+b^n=c^n \] entonces $abc=0$.
  \begin{enumerate}
  \item Demuestre que para demostrar esta afirmación es suficiente considerar el caso en que
    $n$ es primo y $a,b,c$ son coprimos.
  \end{enumerate}
  Sea $p>2$ un primo. Supongamos que $a,b,c\in\Z$ son coprimos y tal que \[ a^p+b^p=c^p. \]
  Sea $K=\Q(\mu_p)$ y $\O=\O_K=\Z[\xi]$, donde $\xi$ es una raíz primitiva $p$-ésima de la unidad.
  \begin{enumerate}[resume]
  \item Demuestre que en $\O$ tenemos
    \begin{equation}
      \label{eqn:fermat-zerlegung}
      c^p=\prod_{i=0}^{p-1}(a+\xi^ib).
    \end{equation}
  \end{enumerate}
  A partir de ahora supongamos que $p\nmid abc$.
  \begin{enumerate}[resume]
  \item\label{ejer:fermat:coprimo} Demuestre que para $i\neq j$, $i,j\in\{0,\dotsc,p-1\}$,
    los factores $(a+\xi^ib)$ y $(a+\xi^jb)$ son coprimos en $\O$. Para esto, demuestre que
    un divisor común dividiría también a $(1-\xi)$, que es primo en $\O$, y concluya que
    entonces $p\mid c$.
  \item Veamos la ecuación \eqref{eqn:fermat-zerlegung} como una igualdad en el grupo de
    ideales fraccionales $\Div(\O)$. Concluya de lo anterior que para cada
    $i\in\{0,\dotsc,p-1\}$ el ideal principal $(a+\xi^ib)$ es una potencia $p$-ésima de otro
    ideal $I_i\subseteq\O$.
  \end{enumerate}
  A partir de ahora supongamos que $p$ no divide al número de clases de $K$.
  \begin{enumerate}[resume]
  \item Demuestre que los ideales $I_i$ para $i\in\{0,\dotsc,p-1\}$ son principales (lo
    demostramos en el \cref{ejer:p-nmid-h}). Concluya que existe $\alpha\in\O$ y
    $u\in\O^\times$ tal que $a+\xi b=u\alpha^p$.
  \item Trate el caso $p=3$ separadamente considerando la ecuación módulo $9$.
  \end{enumerate}
  A partir de ahora supongamos que $p\ge5$.
  \begin{enumerate}[resume]
  \item Escribimos $u=\xi^mv$ con $v\in{(\O^\times)}^+$ y $m\in\Z$ usando la
    \cref{prop:o-es-mu-o-mas}. Demuestre que \[ \xi^{-m}(a+\xi b)\equiv v k \quad (\mod
      p) \] con un $k\in\Z$ tal que $\alpha\equiv k$ ($\mod (1-\xi)$). Usando la conjugación
    compleja concluya que \[ \xi^{-m}a + \xi^{1-m} b- \xi^m a - \xi^{m-1} b \equiv 0\quad
      (\mod p). \]
  \item Considere los casos $m=0$ y $m=1$ separadamente y deduzca de lo anterior que $p\mid
    b$ o $p\mid a$ en estos casos (que lleva a una contradicción).
  \item En el caso $m>1$ deduzca de lo anterior que $m=\frac{p+1}2$ (sin pérdida de
    generalidad). Concluya que $p\mid 3a$, que otra vez lleva a una contradicción.
  \end{enumerate}
\end{ejer}

\begin{ejer}\label{ejer:cicunit}
  Aquí demostramos una versión más general del \cref{lem:cicunit} que será útil en la
  \cref{sec:coleman}. Sea $p\neq2$ un primo.

  Sea $r\in\Nuno$, $K=\Q(\mu_{p^r})$ y $\xi\in K^\times$ una raíz primitiva $p^r$-ésima de la
  unidad. Para $a,b\in\Z\setminus\{0\}$ coprimos y primos a $p$ definimos
  \[ c_r(a,b)=\frac{\xi^{-a/2}-\xi^{a/2}}{\xi^{-b/2}-\xi^{b/2}}\in K \] (notemos que
  $2\in(\Z/p^r\Z)^\times$, así que $\xi^{1/2}\in K$).

  Demuestre que $c_r(a,b)\in\O_K^\times$. Para eso es útil escribir
  \[ c_r(a,b)=\xi^{a/2-b/2}\frac{\xi^a-1}{\xi^b-1} \]
  y expresar la fracción usando un $t\in\Z$ tal que $a\equiv bt \mod p^r$.
\end{ejer}

\section{Teoría de Kummer y un poco de teoría de campos de clases}

La teoría de campos de clases es una teoría poderosa que permite describir las extensiones
abelianas de un campo local o global en gran generalidad. La teoría de Kummer describe una clase particular de extensiones
abelianas de cualquier campo. Aquí solo vamos a necesitar unos
casos especiales y resumimos lo que necesitamos de ellos.
Empecemos con los resultados más importantes de la teoría de Kummer.

\begin{defi}
  Sea $d\in\Nuno$ y $F$ un campo cuya característica no divida a $d$ y contenga las raíces
  $d$-ésimas de la unidad. Una extensión abeliana $L/F$ tal que el grupo de Galois
  $\Gal(L/F)$ tiene exponente $d$ se llama \emph{extensión de Kummer}\index[def]{extensión!de Kummer} de exponente $d$.
\end{defi}

Para cada subconjunto $\Delta\subseteq F^\times$ la extensión $F(\sqrt[d]\Delta)$ siempre es
una extensión de Kummer. Al revés, se cumple que también cada extensión de Kummer es de esta
forma: 

\begin{thm}[Kummer]\label{thm:kummer}
  Sea $L/F$ una extensión de Kummer de exponente $d\in\Nuno$ y $\Delta={(L^\times)}^d\cap
  F^\times$. Entonces:
  \begin{enumerate}
  \item\label{thm:kummer:kummer} Tenemos $L=F(\sqrt[d]\Delta)$.
  \item\label{thm:kummer:apareamiento} La asociación
    \[ \left<\cdot,\cdot\right>\colon\Gal(L/F)\times \left(\sqrt[d]\Delta\Big/(\sqrt[d]\Delta\cap
      F^\times)\right) \rightarrow
      \mu_d\subseteq F^\times,\quad\left<\sigma,a\right>=\frac{\sigma(a)}{a} \]
    es un apareamiento perfecto de grupos abelianos que se llama el \define{apareamiento de
      Kummer}.
  \end{enumerate}
\end{thm}
\begin{proof}
  Esto es demostrado en \cite[Chap.\ 4, (3.3), (3.6)]{MR1697859}; allí el apareamiento es
  escrito
  \begin{equation*}
    \left<\cdot,\cdot\right>\colon\Gal(L/F)\times \Delta/{(F^\times)}^d\rightarrow
      \mu_d\subseteq F^\times,\quad\left<\sigma,a\right>=\frac{\sigma(\sqrt[d]a)}{\sqrt[d]a} 
  \end{equation*}
  pero con el lema de la serpiente es fácil ver que esto es equivalente a lo de arriba.
\end{proof}

\begin{prop}\label{prop:kummer-apareamiento-equivariante}
  Sean $K\subseteq F\subseteq L$ campos tal que todas las extensiones sean Galois y $L/F$ es
  una extensión de Kummer de exponente $d$, y como antes sea $\Delta={(L^\times)}^d\cap F^\times$. El grupo
  $\Gal(L/K)$ actúa por conjugación en su subgrupo normal $\Gal(L/F)$ y actúa también en
  $\sqrt[d]\Delta\subseteq L$ y en $\mu_d\subseteq F$.

  Entonces el apareamiento del \cref{thm:kummer}~\ref{thm:kummer:apareamiento} es
  equivariante en el sentido
  \[ \left<g\sigma g^{-1},ga\right>=g\left<\sigma,a\right>\quad \text{para cada
    }g\in\Gal(L/K),\ \sigma\in\Gal(L/F),\ a\in\sqrt[d]\Delta. \]
\end{prop}
\begin{proof}
  Esto es una consecuencia directa de la definición del apareamiento que dejamos como ejercicio.
\end{proof}

\begin{remark}\label{rem:kummer-apareamiento-equivariante}
  Usamos la notación de la \cref{prop:kummer-apareamiento-equivariante}. En esta situación
  el apareamiento perfecto del \cref{thm:kummer}~\ref{thm:kummer:apareamiento} define
  isomorfismos de grupos abelianos
  \begin{align*}
     \Gal(L/F) &\isomarrow\Hom_\Z\left(\sqrt[d]\Delta\Big/(\sqrt[d]\Delta\cap
      F^\times), \mu_d\right), &\sigma\mapsto\left<\sigma,\cdot\right>,\\
     \sqrt[d]\Delta\Big/(\sqrt[d]\Delta\cap
      F^\times) &\isomarrow\Hom_\Z(\Gal(L/F), \mu_d), &a\mapsto\left<\cdot,a\right>.
  \end{align*}
  En cada de $\Gal(L/F)$, $\sqrt[d]\Delta/(\sqrt[d]\Delta\cap F^\times)$ y $\mu_d$
  tenemos acciones de $\Gal(L/K)$. Además en la \cref{rem:hom-de-representaciones} definimos
  una acción de $\Gal(L/K)$ en los $\Hom_\Z(-,-)$ que aparecen en el lado derecho. Entonces
  la afirmación del \cref{prop:kummer-apareamiento-equivariante} significa que los
  isomorfismos de arriba no solo son isomorfismos de grupos abelianos sino de
  $\Z[\Gal(L/K)]$-módulos.
\end{remark}

\begin{prop}\label{prop:kummer-ramificacion}
  Sea $F$ un campo de números y $L/F$ una extensión de Kummer de exponente $d\in\Nuno$. Si
  $\mathfrak p$ es un ideal primo de $F$ tal que existe un $a\in\Delta$ con
  $\mathfrak p\mid a$, pero $\mathfrak p^d\nmid a$, entonces $\mathfrak p$ es ramificado en
  la extension $L/F$.
\end{prop}
\begin{proof}
  Véase \cite[Prop.\ 1.83 (2)]{MR1474965}; note que el \enquote{$\mathfrak p\nmid a$} allá
  es una errata.
\end{proof}

Continuamos con unos resultados de la teoría de campos de clases.

\begin{defi}
  Sea $K$ un campo de números. La extensión máxima abeliana no ramificada de $K$ se llama el
  \define{campo de clases de Hilbert} de $K$. Esta extensión siempre existe y es finita.
\end{defi}

La teoría de campos de clases permite describir su grupo de Galois, que va a explicar el nombre.
Sea $H$ el campo de clases de Hilbert de un campo de números $K$, que claramente es
Galois. Si $\mathfrak p$ es un ideal primo de $K$ y
$\mathfrak p=\mathfrak P_1\dotsm\mathfrak P_g$ su descomposición en $H$ (sin exponentes
porque es no ramificada), entonces para cada $\mathfrak P_i$ tenemos un único elemento de
Frobenius en $\Gal(H/K)$ que genera el grupo de Galois de la extensión residual
$\Gal((\O_H/\mathfrak P_i)/(\O_K/\mathfrak p))$, y como $\Gal(H/K)$ actúa transitivamente en
los $\mathfrak P_i$ y es abeliano, este elemento sólo depende de $\mathfrak p$, y lo
denotamos $\operatorname{Frob}_{\mathfrak p}\in\Gal(H/K)$. Véase \cite[§I.9]{MR1697859} para
más detalles.

\begin{thm}\label{thm:campo-de-hilbert}
  La asociaci\'on \[ \mathfrak p \mapsto\operatorname{Frob}_{\mathfrak p} \] induce un
  isomorfismo \[ \left(\frac{H/K}\cdot\right)\colon\Cl(K)\isomarrow\Gal(H/K). \]
\end{thm}
\begin{proof}
  Véase \cite[Prop.\ VI.6.9 y §VI.7]{MR1697859}; notemos que lo que nosotros llamamos el
  campo de clases de Hilbert se llama el \enquote{campo de clases de Hilbert pequeño} para
  Neukirch porque el define la ramificación de las plazas arquimedianas de una manera no
  tan estándar.
\end{proof}

\ejercicios

\begin{ejer}\label{ejer:artin-symbol-eq}
  Sea $K$ un campo de números que es Galois sobre $\Q$ y $H$ su campo de clases de
  Hilbert. Entonces $\Gal(K/\Q)$ actúa en los ideales de $K$.
  Demuestre que para cada ideal primo $\mathfrak p$ de $K$ y cada
  $\sigma\in\Gal(K/\Q)$ tenemos
  \[ \left(\frac{H/K}{\sigma(\mathfrak p)}\right) = \tilde\sigma\left(\frac{H/K}{\mathfrak
        p}\right)\tilde\sigma^{-1} \]
  donde $\tilde\sigma$ es un levantamiento de $\sigma$ a $\Gal(H/\Q)$.
\end{ejer}

\begin{ejer}
  Verifique la equivariancia del apareamiento de Kummer, es decir demuestre la
  \cref{prop:kummer-apareamiento-equivariante}.
\end{ejer}

\section{Números $p$-ádicos y caracteres}

Para nosotros, un \emph{carácter}\index[def]{carácter} será un morfismo de grupos de la forma
$G\rightarrow R^\times$ con $G$ un grupo y $R$ un anillo. Si $G$ y $R$ además son espacios
topológicos, \emph{siempre asumiremos que los caracteres entre ellos son continuos}. En la
mayoría de los casos $G$ será un grupo profinito y $R$ será un anillo profinito o un campo
topológico.

Hay unos caracteres de importancia especial, que tienen que ver con los números $p$-ádicos
y los campos ciclotómicos. En esta sección introducimos y estudiamos esos caracteres.

Asumimos que el lector conoce los números $p$-ádicos $\Zp$ y $\Qp$: $\Qp$ es la
completación de $\Q$ por el valor absoluto $p$-ádico y $\Zp$ es su anillo de enteros
(equivalentemente, los elementos con valor absoluto $\le1$), o alternativamente $\Zp$ es el
límite inverso de los grupos finitos $\Z/p^r\Z$ ($r\in\Nuno$) y $\Qp$ es su campo de
cocientes.
Para más detalles sobre los números $p$-ádicos remitimos a \cite[§II.1--2]{MR1697859}.
Resumamos unos hechos fundamentales sobre las extensiones de $\Qp$.

\begin{prop}
  Sea $K/\Qp$ una extensión finita. Entonces el valor absoluto $p$-ádico se extiende de
  manera única a $K$ y $K$ es completo por la topología definida por este valor absoluto.

  Sea $\O$ la cerradura integral de $\Zp$ en $K$. Entonces $\O=\{x\in K : \abs{x}\le 1\}$, y
  eso es un anillo local de valuación discreta y completo por la topología definida por las
  potencias de su ideal máximo (que es principal) $\pi \O=\{x\in K : \abs{x}<1\}$. El campo
  residual $k=\O/\pi \O$ es una extensión finita de $\Fp$. En particular, $\O$ es un anillo
  profinito.
\end{prop}
\begin{proof}
  \cite[Chap.\ II, (4.8), (3.8), (3.9), (5.2)]{MR1697859}
\end{proof}

\begin{lem}[\importante{Lema de representación}]\label{lem:lemaderep}
  Sea $\O$ un anillo local de valuación discreta completo con campo residual $k=\O/\pi\O$
  finito. Sea $R$ un sistema de representantes de $k$ en
  $\O$. Entonces todo elemento $x\in\O$ se escribe de manera única como
  $$x=\sum_{i=0}^\infty a_{i}\pi^{i},$$
  con $a_{i}\in R$. 
\end{lem}
\begin{proof} Sea $x\in \O$, entonces tomando el mapeo canónico de $\O$ a su campo de residuos tenemos que $x\equiv a_{0}\mod \pi$ para algún $a_{0}\in R$. Por lo tanto $x-a_{0}= x_{1}\pi$ con $x_{1}\in \O$, aplicamos ahora el mapeo canónico a $x_{1}$, i.\,e. $x_{1}\equiv a_{1}\mod \pi$ para algún $a_{1}\in R$. Entonces tenemos
\begin{eqnarray*}
x&=&a_{0}+x_{1}\pi  \	\text{ con }x_{1}\in \O \\
 &=&a_{0}+a_{1}\pi +x_{2}\pi^{2} \text{ con }x_{2}\in \O \\
 &\vdots & \\
 &=&\sum_{i=0}^{\infty} a_{i}\pi^{i}.
\end{eqnarray*}
\end{proof}

\begin{ex} Con $\O=\Zp$, $\pi=(p)$ y $k=\F_{p}$ tenemos que $R=\{0,1,2,\ldots,p-1\}$ es un
  sistema de representantes de $\Fp$ en $\Zp$ con el que podemos escribir cada $x\in\Zp$ como
$$x=\sum_{i\geq 0}a_{i}p^{i} \	\text{ con }a_{i}\in\{0,1,\ldots,p-1\}.$$
\end{ex}

\begin{defi}\label{defi:unidades-principales}
  Sea $K/\Qp$ una extensión finita con anillo de enteros $\O$ e ideal máximo
  $\pi\O$. Entonces
  \[ 1+\pi\O\subseteq\O^\times \]
  es un subgrupo cuyos elementos se llaman \define{unidades principales}.
\end{defi}

Estudiamos con más detalle la estructura del grupo $\Z_p^\times$.
 
\begin{lem}\label{lem:z-p-decomposicion}
  Existe un isomorfismo topológico canónico $\Z_p^\times\isom\F_p^\times\times(1+p\Z_p)$.
\end{lem}
\begin{proof}
  Sólo demostramos esto en el caso $p\neq2$, el caso $p=2$ siendo similar.
  Sea $r\in\Nuno$, la composición
  \begin{equation*} 1 + p\Z_p \hookrightarrow \Z_p^\times \twoheadrightarrow (\Z/p^r\Z)^\times \end{equation*}
  induce una sucesión exacta
  \begin{equation*} 1 \rightarrow \frac{1+p\Z_p}{1+p^r\Z_p} \rightarrow (\Z/p^r\Z)^\times \rightarrow (\Z/p\Z)^\times \rightarrow 1. \end{equation*}
  Esta sucesión se escinde:
  \begin{equation*} (\Z/p\Z)^\times \rightarrow (\Z/p^r\Z)^\times, \quad a \mapsto
    a^{p^{r-1}} \end{equation*} es una sección del mapeo de la derecha, porque $a^p\equiv a$
  ($\mod p$). Más precisamente, este mapeo toma una clase en $(\Z/p\Z)^\times$, la levanta a
  un entero $a\in\Z$, envía esto a $a^{p^{r-1}}$ y luego a la clase módulo $p^r$. Por
  supuesto hay que asegurarse que esto no depende del levantamiento. Para esto, si
  remplazamos $a$ con $a+bp$, este es enviado a
  \begin{equation*}
    \sum_{i=0}^r\binom{p^{r-1}}ia^i(bp)^{p^{r-1}-i}
  \end{equation*}
  y se puede verificar que $\binom{p^{r-1}}ip^{p^{r-1}-i}$ siempre es divisible por $p$ si
  $0\le i<p^{r-1}$. Omitimos los detalles aquí.

  La afirmación resulta al tomar el límite.
\end{proof}

\begin{prop}\label{prop:isom-z-p-log}
  Sea $p\neq2$. Las series de la \emph{función exponencial}\index[def]{función $p$-ádica!exponencial} y del \emph{logaritmo}\index[def]{función $p$-ádica!logaritmo}
  \[ \exp(x)=\sum_{n=1}^\infty \frac{x^n}{n!}, \quad \text{resp.}\quad
    \log(1+y)=\sum_{n=1}^\infty (-1)^{n+1}\frac{y^n}n \] convergen para $x,y\in p\Zp$,
  respectivamente, y dan isomorfismos de grupos topológicos
  \[\begin{tikzcd}
    p\Zp \arrow[r, "\exp", shift left] & 1+p\Zp \arrow[l, "\log", shift left]
  \end{tikzcd}\]
  inversos el uno al otro.

  Es decir, como grupo topológico $1+p\Z_p\isom\Z_p$ canónicamente.
  En particular, para cada $s\in\Z_p$ y $u\in 1+p\Z_p$, el elemento $u^s\in1+p\Zp$ está bien
  definido.
\end{prop}
\begin{proof}
  La primera afirmación es demostrada en \cite[Chap.\ II, (5.4) y (5.5)]{MR1697859}.  Porque
  $p\Zp\isom\Zp$ como grupo topológico, la segunda afirmación resulta; escribimos
  $\Phi\colon\Zp\isomarrow1+p\Zp$ para el isomorfismo. Finalmente, para $s\in\Z_p$ y
  $u\in 1+p\Z_p$ definimos $u^s$ como $\Phi(s\Phi^{-1}(u))$.
\end{proof}

Más generalmente, tenemos la siguiente descripción del grupo de unidades de $\O$, donde $\O$
es el anillo de enteros de una extensión finita de $\Qp$.

\begin{prop}\label{prop:unidades-principales}
  Con $\O$ como arriba con ideal máximo $\pi\O$ y campo residual $k$, tenemos
  \[ \O^\times\isom k^\times\times(1+\pi\O). \] Además, $1+\pi\O$ es (no canónicamente)
  isomorfo al producto de un grupo $p$ finito cíclico y $\Z_p^d$ con $d=[K:\Qp]$; en
  particular es un grupo pro-$p$.
\end{prop}
\begin{proof}
  \cite[Chap.\ II, (5.3), (5.7) (i)]{MR1697859}
\end{proof}

Terminamos la sección definiendo y estudiando algunos caracteres importantes.

\begin{defi}\label{defi:teichmueller}
  \begin{enumerate}
  \item   Por el \cref{lem:z-p-decomposicion} tenemos una sección
    \[ \omega\colon\F_p^\times\hookrightarrow\Z_p^\times \] a la proyección
    $\Z_p^\times\cong\F_p^\times\times(1+p\Zp)\rightarrow\F_p^\times$. Esta sección $\omega$
    se llama \emph{carácter de Teichmüller}\index[def]{carácter!de Teichmüller}.
  \item  Al isomorfismo
    \[ \kappa\colon\Gal(\Q(\mu_{p^\infty})/\Q)\isomarrow\Z_p^\times \]
    de \eqref{eqn:kappa} se le llama \emph{carácter ciclotómico}\index[def]{carácter!ciclotómico}.  
  \item Usando el isomorfismo canónico $(\Z/p\Z)^\times\isom\Gal(\Q(\mu_p)/\Q)$ podemos ver al
  carácter de Teichmüller como \[
    \omega\colon\Gal(\Q(\mu_p)/\Q)\rightarrow\Z_p^\times \]
  o también como carácter de $\Gal(\Q(\mu_{p^\infty})/\Q)$.
  Definimos otro carácter\footnote{En la literatura también es común escribir este carácter
    como $\left<\cdot\right>$ en lugar de $\kappa_0$, es decir
    $\kappa_0(x)=\left<x\right>$.}
  \[ \kappa_0=\omega^{-1}\kappa\colon\Gal(\Q(\mu_{p^\infty})/\Q)\rightarrow\Z_p^\times. \]
  Entonces $\kappa_0$ es la composición del carácter ciclotómico con la proyección
  \[ \Z_p^\times\twoheadrightarrow 1+p\Z_p\subseteq\Z_p^\times \] usando el isomorfismo del
  \cref{lem:z-p-decomposicion}.
  \end{enumerate}
\end{defi}

\begin{remark}\label{defi:zpuno}
  Definimos
  \[ \Zp(1):=\varprojlim_{r\in\Nuno}\mu_{p^r}, \] donde $\mu_{p^r}$ son las raíces
  $p^r$-ésimas de la unidad en $\Qbar$ y las aplicaciones
  $\mu_{p^{r+1}}\rightarrow\mu_{p^r}$ están dadas por $\xi\mapsto\xi^p$. Entonces $\Zp(1)$ es un
  $\Zp$-módulo compacto que es isomorfo a $\Zp$ (de manera no canónica porque $\mu_{p^r}$ es isomorfo no canónicamente a $\Z/p^r\Z$). Además $\Zp(1)$ tiene una acción de $\GQ$. Si vemos
  el carácter ciclotómico como carácter de $\GQ$ vía
  $\GQ\twoheadrightarrow\Gal(\Q(\mu_{p^\infty})/\Q)\labeledarrow{\kappa}\Z_p^\times$
  entonces la acción de $\GQ$ en $\Zp(1)$ está dada por
  \[ gz=\kappa(g)z \quad\text{para cada }z\in\Zp(1),\ g\in\GQ. \]
\end{remark}

  Sea $z\in\Z_p^\times$ y lo escribimos como $z=(f,u)\in\F_p^\times\times(1+p\Z_p)$ usando
  el isomorfismo del \cref{lem:z-p-decomposicion}. Entonces en $\Z_p^\times$
  \[ \omega(f)=\lim_{j\to\infty} z^{p^j}. \]
  En particular, si $\overline a\in\F_p^\times$ y $a\in\Z$ es un levantamiento, entonces
  \[ \omega(\overline a)=\lim_{j\to\infty}a^{p^j}. \]

\begin{lem}\label{lem:caracteres-de-g}
  Cada carácter $\chi\colon\Gal(\Q(\mu_{p^\infty})/\Q)\rightarrow\Z_p^\times$ es de la forma
  $\chi=\omega^a\kappa_0^s$ con únicos $a\in\{1,\dotsc,p-1\}$ y $s\in\Z_p$. Aquí la notación
  $\kappa_0^s$ tiene sentido gracias a la \cref{prop:isom-z-p-log}.
\end{lem}
\begin{proof}
  Identificamos $\Gal(\Q(\mu_{p^\infty})/\Q)$ con $\Z_p^\times$ usando el carácter ciclotómico.
  Por la descomposición del \cref{lem:z-p-decomposicion} es claro que cada carácter
  $\chi\colon\Z_p^\times\rightarrow\Z_p^\times$ se descompone en $\chi=\psi\times\eta$ con
  \begin{equation*}
    \psi\colon\F_p^\times\rightarrow\Z_p^\times, \quad \eta\colon1+p\Z_p\rightarrow\Z_p^\times.
  \end{equation*}
  Porque $1+p\Z_p\isom\Z_p$ no tiene torsión, la imagen de $\psi$ está contenida en
  $\F_p^\times$ y por eso $\psi$ es una potencia del carácter de Teichmüller (véase
  \cref{ejer:caracteres-fp}). Por otro lado, si componemos $\eta$ con la proyección
  $\Z_p^\times\rightarrow\F_p^\times$ obtenemos un morfismo de un grupo pro-$p$ a un grupo
  de orden primo a $p$, que necesariamente es trivial según el
  \cref{ejer:morfismos-pro-p-pro-ell}, y vemos que la imagen de $\eta$ está contenida en
  $1+p\Z_p$. Usando la \cref{prop:isom-z-p-log}, podemos estudiar los homomorfismos
  continuos $\Z_p\rightarrow\Z_p$. Pero como explicamos en la \cref{sec:profinito}, estos
  homomorfismos son únicamente determinados por la imagen de $1\in\Z_p$, y esta imagen puede
  ser cualquier elemento de $\Z_p$.
\end{proof}

\ejercicios

\begin{ejer}\label{ejer:teichmueller-limes}
  Demuestre que \[ \omega(f)=\lim_{j\to\infty} z^{p^j} \] para cada $z\in\Z_p^\times$ que
  escribimos como $z=(f,u)\in\F_p^\times\times(1+p\Z_p)$ usando el isomorfismo del
  \cref{lem:z-p-decomposicion}.
\end{ejer}

\begin{ejer}\label{ejer:caracteres-fp}
  Demuestre que cada carácter $\F_p^\times\rightarrow\Qbar_p^\times$ es una potencia del
  carácter de Teichmüller $\omega$.
\end{ejer}

\chapter{El álgebra de Iwasawa}
\label{sec:algebra-de-iwasawa}

En el caso en que $\Gamma$ es un pro-$p$ grupo cíclico isomorfo a $\Zp$, el
\emph{álgebra de Iwasawa}\index[def]{algebra de Iwasawa@álgebra de Iwasawa} $\LL=\Zp\llbracket\Gamma\rrbracket$ con coeficientes en
$\Zp$ se presenta como una trinidad matemática: Por definición es una álgebra de grupos
profinita, por lo tanto tiene una estructura algebraica y topológica; una vez escogido un
generador $\gamma\in\Gamma$ esta se puede identificar al álgebra de series formales sobre
$\Zp$; además los elementos de $\LL$ pueden ser vistos como medidas en $\Gamma$ con
valores en $\Zp$.

Esta trinidad explica su enorme importancia: Es un objeto versátil y por eso aparece en
varios contextos. Primero, en su talle de álgebra de grupos profinita, muchos objetos son
naturalmente módulos sobre $\LL$ -- en cuanto tenemos acción razonable de $\Gamma$ en un
grupo pro-$p$ abeliano, ya tenemos un $\LL$-módulo (véase la
\cref{prop:modulos-sobre-rg}). En las aplicaciones, $\Gamma$ normalmente es un grupo de
Galois, que actúa naturalmente en una multitud de objetos. Si entonces identificamos $\LL$
con el anillo de series de potencias formales sobre $\Zp$, la maleabilidad de este anillo
permite desarrollar una teoría útil de estructura de sus módulos que permite definir
invariantes importantes (\cref{sec:asp} y \cref{sec:modulos-iwasawa}). Finalmente, el punto
de vista de medidas facilita una conexión de dichos módulos a objetos de origen analítico,
como funciones $L$ $p$-ádicas (\cref{sec:medidas} y \cref{sec:palf}).

\section{El anillo de series de potencias}
\label{sec:asp}

En esta sección empezaremos estudiando las propiedades de las series formales
$\O\llbracket T\rrbracket$ en una variable sobre un anillo local de valuación discreta
$\O$. El ejemplo más importante para nosotros es aquel en que $\O$ es el anillo de enteros de
una extensión finita de $\Qp$, por ejemplo $\O=\Zp$. No obstante, algunos de los resultados
de esta sección pueden ser generalizados, por ejemplo al caso en que $\O$ es un anillo local
conmutativo noetheriano (ver \cite[Cap. V. §3]{NSW}). En particular demostramos que si
$\Gamma$ es un pro-$p$ grupo libre de rango $1$ y $\O$ es un anillo de valuación discreta,
entonces $\O\llbracket T\rrbracket$ es isomorfo de manera no canónica al anillo de grupo
profinito $\O\llbracket \Gamma\rrbracket$.

Sea $\O$ un anillo local de valuación discreta, completo por la topología definida por las potencias de su ideal máximo $\pi \O$ de cuerpo residual $k=\O/\pi \O$ finito, de manera que $\O$ es compacto.
Denotaremos $\LL$ el álgebra $$\LL=\O\llbracket T\rrbracket$$ de las series formales en una variable con coeficientes en $\O$.

\begin{prop} $\LL$ es un anillo local de ideal máximo $\M=\pi\LL+T\LL=(\pi,T)$.
\end{prop}
\begin{proof} Un elemento $f=\sum_{i=0}f_{i}T^{i}\in \LL$ diferente de cero es invertible sí y sólo si su coeficiente constante $\lambda_{0}$ es invertible en $\Zp$. De hecho si $f\neq 0$ es invertible entonces existe $f^{-1}$ tal que $ff^{-1}=1$, en particular $f_{0}f_{0}^{-1}=1\in\Zp$. Del otro lado si $f_{0}$ es invertible entonces podemos escribir $f=f_{0}(1+gT)$ para un $g\in\LL$. Entonces el elemento $f^{-1}=f_{0}^{-1}\sum_{i=0}^{\infty}g^{i}T^{i}$ es un inverso de $f$. 

Entonces tenemos $\LL^{\times}=\LL\setminus\M$.
\end{proof}

\begin{remark} El ideal $\M$ no es principal como era el caso del anillo $\O$.
\end{remark}

Equipamos $\LL$ con la topología $\M$-ádica, tomando los $(\M^{n})_{n\in\Nuno}$ como sistema
fundamental de vecindades alrededor del $0$. Esto hace de $\LL$ un anillo local completo, con
campo de residuos finito $k=\LL/\M=\O/\pi \O$. De hecho $\LL$ es un espacio topológico
compacto y Hausdorff.

\begin{defi}
  \begin{enumerate}
  \item   Sea $f=\sum_{n=0}^\infty a_nT^n\in\LL\setminus\pi\LL$ y $\bar f\in k\llbracket T\rrbracket$ su reducción. El \define{grado
    de Weierstraß} de $f$ es la valuación en $T$ de $\bar f$, es decir el mínimo $n\in\Ncero$
  tal que $a_n\notin\pi\O$.
\item Sea $f\in\O[T]$. Entonces $f$ se llama \emph{distinguido}\index[def]{polinomio!distinguido} si es mónico y su grado es
  igual a su \importante{grado de Weierstraß}, es decir el coeficiente superior es $1$ y
  todos los demás están en $\pi\O$. Algunos textos llaman estos \emph{polinomio de
    Weierstraß}\index[def]{polinomio!de
    Weierstraß|see{polinomio distinguido}}.
  \end{enumerate}
\end{defi} 

\begin{lem}[\importante{Lema de división}]\label{lema:division} Sea $f$ un elemento de
  $\LL\setminus\pi\LL$ y $\nu$ su grado de Weierstraß. Entonces todo elemento de $g\in\LL$ se escribe de manera única como $$g=f\lambda + r,\	\	\	\text{ con }\	\	\lambda\in \LL \text{ y } r\in \O_{\nu-1}[T],$$ donde $\O_{\nu-1}[T]$ es el $\O$-módulo de los polinomios de grado a lo más $\nu-1$. Equivalentemente
$$\LL=f\LL\oplus \O_{\nu-1}[T].$$

\end{lem}
\begin{proof}
  Sea $\nu$ el grado de Weierstraß de $f$, entonces escribimos
$$f=T^{\nu}\mu + \pi R,\	\	\text{	con }\	\	R\in \O_{\nu-1}[T], $$
para un $\mu\in\LL^{\times}$.

\textit{Existencia:} Sea $g\in \LL$, entonces podemos escribir
\begin{eqnarray*}
g&=&\sum_{i=0}^{\nu-1} a_{i}T^{i} + T^{\nu}\sum_{i=\nu}^{\infty} a_{i}T^{i-\nu} \\
&=& T^{\nu}g' + r_{0} \	\	\	\text{ con } g'\in \LL \text{ y } r_{0}\in \O_{\nu-1}[T].
\end{eqnarray*}
Definimos $a_{0}=\mu^{-1}g'$, entonces 
\begin{eqnarray*}
g-a_{0}f&=& T^{\nu}g' + r_{0} - \mu^{-1}g'(T^{\nu}\mu + \pi R) \\
&=& T^{\nu}g' + r_{0} - g'T^{\nu} - \pi R\mu^{-1} g' \\
&=&r_{0}-\pi\underbrace{ R\mu^{-1} g'}_{\in\LL}  \\
&=&r_{0}-\pi g_{1}.
\end{eqnarray*}
Sea $g_{1}=T^{\nu}g'_{1}+r_{1}$ con $g'_{1}\in \LL$ y $r_{1}\in \O_{\nu-1}[T]$, y definimos $a_{1}=\mu^{-1}g'_{1}$. Tenemos
\begin{eqnarray*}
g_{1}-a_{1}f&=& T^{\nu}g'_{1}+r_{1} - \mu^{-1}g'_{1}(T^{\nu}\mu + \pi R) \\
&=& T^{\nu}g'_{1} + r_{1} - g'_{1}T^{\nu} - \pi R\mu^{-1} g'_{1} \\
&=&r_{1}-\pi R\mu^{-1} g'_{1}  \\
&=&r_{1}-\pi g_{2};
\end{eqnarray*}
entonces para $n\geq 2$ (siguiendo el mismo razonamiento) sea $g_{n}=T^{\nu}g'_{n}+r_{n}$ con $g'_{n}\in \LL$, $r_{n}\in \O_{\nu-1}[T]$ y definimos $a_{n}=\mu^{-1}g'_{1}$. Con  $g=g_{0}$, tenemos
$$g_{n}-a_{n}f=r_{n}-\pi g_{n+1},$$
para todo $n\geq 0$, por lo tanto 
$$g-\left(\sum_{i=0}^{n}(-1)^{i}a_{i}\pi^{i}\right)f=\sum_{i=0}^{n}(-1)^{i}r_{i}\pi^{i} + (-1)^{n+1}\pi^{n+1}g_{n+1}\	\	\	\text{ para } n\geq 0.$$
Tomando límites tenemos que $(-1)^{n+1}\pi^{n+1}g_{n+1} \rightarrow 0$ por lo tanto $g-\lambda f = r$ con $\lambda \in \LL$ y $r\in \O_{\nu-1}[T]$.

\textit{Unicidad:} Sea $\lambda f  = r \in f \LL \cap \O_{\nu-1}[T]$. Módulo $\pi$ tenemos $\bar{\lambda}\bar{f}=\bar{\lambda}\bar{\mu}T^{\nu}=\bar{r}$, entonces $\bar{\lambda}\bar{f}=\bar{r}=0$, es decir $\pi\mid \lambda$ y $\pi\mid r$, entonces
$$\dfrac{\lambda}{\pi}f=\dfrac{r}{\pi},$$
volviendo a iterar el proceso vemos que $\pi^{n}\mid \lambda$ y $\pi^{n}\mid r$ para todo
$n\geq 0$, por lo tanto $\lambda = r = 0$.
\end{proof}

\begin{thm}[\importante{Teorema de preparación de Weierstraß}]\label{thm:weierstrass} Todo elemento $f\in
  \LL\setminus\pi\LL$ se escribe de manera única como
$$f=\mu(T^{\nu}+\pi Q),$$
con $Q\in \O_{\nu-1}[T]$. Es decir, como producto de un invertible $\mu\in\LL^{\times}$ y de un polinomio distinguido $P=T^{\nu}+\pi Q$.
\end{thm}
\begin{proof}
Sea $f=T^{\nu}\mu+\pi R$ como en el lema de división.  Aplicamos el lema de división a $T^{\nu}$, es decir 
$$T^{\nu}=\lambda f + r, \	\	\text{ con }r\in \O_{\nu-1}[T].$$ 
Entonces módulo $\pi$ tenemos
$$T^{\nu}=\bar{\lambda}\bar{\mu} T^{\nu}+ \bar{r},$$ 
por lo tanto $\bar{r}=0$ y $\bar{\lambda}\bar{\mu}=1$. En particular $r=\pi Q$ para algún $Q\in \O_{\nu-1}[T]$ y $\lambda$ es invertible. Concluimos que $f=\lambda^{-1}(T^{\nu}-r)$.
\end{proof}

\begin{cor} Los polinomios distinguidos e irreducibles $P\in \O[T]$ son también irreducibles en $\LL$.
\end{cor}
\begin{proof} Sea $P\in \O[T]$ distinguido e irreducible. Supongamos que $P=f_{1}f_{2}$ con $f_{i}\in\LL$, luego por el Teorema de preparación de Weierstraß tenemos $P=\mu_{1}P_{1}\mu_{2}P_{2}$ con $\mu_{i}\in\LL^{\times}$ y $P_{i}$ polinomios distinguidos en $\O[T]$, como la manera de expresar $P$ es única esto implica que $\mu_{1}\mu_{2}=1$ y $P=P_{1}P_{2}$.
\end{proof}

\begin{remark} Los polinomios de Eisenstein son irreducibles en $\LL$.
\end{remark}

\begin{cor} El álgebra $\LL$ es un dominio de factorización única cuyos elementos irreducibles son 
\begin{itemize}
\item el uniformizante $\pi$ de $\O$;
\item los polinomios distinguidos e irreducibles en $\O[T]$.
\end{itemize}
\end{cor}
\begin{proof} Sea $f\in\LL$, sea $\pi^{k}$ la mayor potencia de $\pi$ que divide $f$. Entonces $f=\pi^{k}g$ donde $g\in\LL\setminus\pi\LL$, el resultado sigue aplicando a $g$ el teorema de preparación de Weierstraß. 
\end{proof}

\begin{lem}\label{lem:fgprinci} Sean $f$ y $g$ elementos no nulos en $\LL$ tal que $d$ es su máximo común divisor. Entonces el ideal $f\LL+g\LL$ está contenido en el ideal principal $d\LL$ con índice finito.
\end{lem}
\begin{proof} La primera aseveración es directa. Para la segunda afirmación podemos suponer
  que $f$ y $g$ son coprimos, además por el Teorema de preparación de Weierstraß
  (\cref{thm:weierstrass}) podemos suponer que $f$ y $g$ son polinomios y al menos uno de ellos es distinguido, digamos $f=T^{\nu}+\pi Q$ para algún $Q\in\O_{\nu-1}[T]$. 

Consideremos el resultante $\operatorname{Res}(f,g)$ de $f$ y $g$, es no nulo pues $f$ y $g$ son coprimos, además siendo un elemento de $\O$ lo podemos escribir $\operatorname{Res}(f,g)=\mu\pi^{\alpha}$ con $\mu\in\O^{\times}$ y $\alpha\in\Ncero$. El resultante $\operatorname{Res}(f,g)$ está en el ideal generado por $f$ y $g$, por lo tanto, también $\pi^{\alpha}\in f \LL + g \LL $ porque $\mu$ es invertible. Entonces tenemos la siguiente congruencia
$$T^{\nu\alpha}\equiv \pi^{\alpha}Q^{\alpha} \equiv 0\quad \mod\ ( f\LL + g\LL).$$
Por lo que $f\LL + g\LL$ contiene $T^{\nu\alpha}$ y $\pi^{\alpha}$, lo cual implica $(\LL:f\LL + g\LL) \leq  (\LL:T^{\nu\alpha}\LL + \pi^{\alpha}\LL)$, siendo la expresión del lado derecho una potencia de $p$ divisible por $\nu\alpha$.
\end{proof}

\begin{cor}\label{cor:fg-indice-finito}
  Sean $f,g\in\LL\setminus\{0\}$ coprimos (es decir, su máximo común divisor es
  $1$). Entonces el ideal $f\LL+g\LL$ tiene índice finito en $\LL$.
\end{cor}

\begin{prop}\label{prop:contenfin} Todo ideal (no nulo) $\mathfrak{U}$ del álgebra $\LL$ está contenido en un ideal principal mínimo $a\LL$. Entonces decimos que $a$ es un \define{pseudo-generador} del ideal $\mathfrak{U}$ y tenemos $(a\LL:\mathfrak{U})$ finito. 
\end{prop}
\begin{proof} Sabemos que $\LL$ es un anillo noetheriano, entonces un módulo noetheriano sobre sí mismo. Por lo tanto sea $\mathfrak{U}=\sum_{i=1}^{m}f_{i}\LL$. Denotemos $D$ el máximo común divisor de los $f_{i}$. Tenemos 
\begin{eqnarray*}
\mathfrak{U}\subset a\LL  \	\	\	&\iff & \	\	\	a\mid f_{i},\	\forall\,i=1,\ldots,m 
\end{eqnarray*}
por lo que el mínimo ideal principal que contiene a $\mathfrak{U}$ es $D\LL $. Por último, $(D\LL:\mathfrak{U})$ es finito por el \cref{lem:fgprinci}.
\end{proof}

\begin{remark} Si $\mathfrak{U}$ no es principal, entonces tenemos $a\notin \mathfrak{U}$. Por ejemplo $\mathfrak{U}=\M$, entonces $\mathfrak{U}\subset (1)=\LL$ porque $\M$ es máximo. 
\end{remark}

Sea $p=\operatorname{car}(\O/\pi \O)$ y supongamos que $\Gamma$ es un $p$-grupo profinito
libre de rango $1$, es decir existe un isomorfismo no canónico de $\Gamma$ al grupo aditivo
de $\Zp$.

\begin{defi}\label{defi:omega-r}
  Para $r\geq 0$ consideremos los polinomios distinguidos
  \begin{equation}
    \label{eqn:omega-r}
    \w_{r}=(T+1)^{p^{r}}-1.
  \end{equation}
  Definimos el $p^r$-ésimo \emph{polinomio ciclotómico}\index[def]{polinomio!ciclotómico} como
  \begin{equation}\label{def:polcyclo}
    \Phi_{r}=\frac{\omega_{r}}{\omega_{r-1}} 
  \end{equation}
  para $r\geq 1$ y $\Phi_{0}=\omega_{0}$. 
\end{defi} 

\begin{thm}\label{thm:EquivPowPro} Supongamos que $\gamma$ es un generador topológico de $\Gamma$. Entonces la aplicación 
$$\LL=\O\llbracket T\rrbracket\isomarrow \O\llbracket
\Gamma\rrbracket,\quad T\mapsto\gamma-1$$
es un isomorfismo de $\O$-álgebras topológicas. En particular, $\LL$ es una $\O$-álgebra profinita.
\end{thm}
\begin{proof}
  Para $r\geq 0$ usaremos los polinomios $\w_r$ y denotamos $\Gamma_{r}$ el único subgrupo de $\Gamma$ tal que $\Gamma/\Gamma_{r}\isom \Zp/p^{r}\Zp$. Por el lema de división (\cref{lema:division}) tenemos 
$$\LL/\w_{r}\LL\isom \O[T]/\omega_{r}\O[T]\	\	\	 \text{ para todo }r\geq 0,$$
además tenemos los isomorfismos de $\O$-álgebras para todo $r\geq0$
\begin{eqnarray*}
\O[T]/\omega_{r}\O[T] &\stackrel{\varphi_{r}}{\longrightarrow} & \O[\Gamma/\Gamma_{r}] \\
T\mod \w_{r} &\mapsto & \gamma -1 \mod \Gamma_{r}.
\end{eqnarray*}
Los inversos de estos morfismos están dados por 
$$\gamma \mod \Gamma_{r} \mapsto T+1 \mod \w_{r}.$$
Como en la \cref{defi:omega-r}, tenemos  $\omega_{r+1}=\omega_{r}\Phi_{r+1}$ por lo tanto las
proyecciones \penalty-10000 $\LL/\w_{r+1} \rightarrow \LL/\w_{r}$ son compatibles con los isomorfismos $\varphi_{r}$, es decir el diagrama
\begin{center}
\begin{tikzcd}
\LL/\w_{r+1} \arrow[r,"\varphi_{r+1}"] \arrow[d,twoheadrightarrow] & \O[\Gamma/\Gamma_{r+1}] \arrow [d,twoheadrightarrow]\\ 
\LL/\w_{r}   \arrow[r,"\varphi_{r}"]  & \O[\Gamma/\Gamma_{r}]
\end{tikzcd}
\end{center}
es conmutativo y tomando límites inversos en ambos lados tenemos
$$\varprojlim _{r} \LL/\w_{r} \isomarrow \varprojlim_{r} \O[\Gamma/\Gamma_{r}] = \O\llbracket \Gamma\rrbracket.$$
Falta demostrar entonces que el límite del lado izquierdo es isomorfo al álgebra de series formales $\LL$. 

Sea $\psi\colon\LL\rightarrow \varprojlim \LL/\w_{r}$, veamos que
\begin{eqnarray*}
\ker \psi &=& \{ f\in \LL \mid  f \in \w_{r}\LL \	\forall r \geq 0 \} \\
&\subseteq& \bigcap_{r\geq 0} \w_{r}\LL\\
&=& 0,
\end{eqnarray*}
la última igualdad se deduce del hecho que 
\begin{equation}\label{eq:idealmax}
\M^{p^{r}}=(\pi,T)^{p^{r}}\supset (p,T)^{p^{r}} \supset \w_{r}\LL, 
\end{equation}
y los $(\M^{r})_{r\in\Nuno}$ forman una base de vecindarios alrededor de $0$ de $\LL$ con la topología $\M$-ádica. 

Finalmente, sea $f=(f_{r})_{r\geq 0}$ un elemento de $\varprojlim_{r}\LL/\omega_{r}$, entonces
$$f_{r}\equiv f_{t} \mod \w_{r+1}$$ 
para todo $0\leq t\leq r$. Es decir $f_{r}-f_{t}\in\w_{r}\LL$, en particular $f_{r}-f_{t}\in \M^{r+1}$ por \eqref{eq:idealmax} para todo $r\geq t \geq 0$. Por lo tanto $f\in \LL$ ya que $\LL$ es compacto por la topología $\M$-ádica. 
\end{proof}

Los polinomios $\Phi_r$ y $\omega_r$ de la \cref{defi:omega-r} que aparecieron en la
demostración anterior también serán importantes en otras situaciones más adelante, por eso
demostramos aquí algunas de sus propiedades.

\begin{lem}\label{lem:phi-s-anula-a-modulo-finito}
  Sea $M$ un $\LL$-módulo finito. Entonces existen $s\gg0$ y $t\ge s$ tales que
  $\Phi_s\dotsm\Phi_t$ anula a $M$.
\end{lem}
\begin{proof}
  Sea $I\subseteq\LL$ el anulador de $M$. Entonces $I$ es el núcleo de la aplicación
  continua $\LL\rightarrow\End_{\LL}(M)$, que demuestra que $I$ es un ideal abierto (porque
  $\End_{\LL}(M)$ es finito). Por eso una potencia de $T$ debe estar en $I$. Hacemos
  $s\in\Nuno$ tan grande tal que $T^{p^s}\in I$. Entonces
  \begin{equation*}
    (1+T)^{p^s}=1+\sum_{i=1}^{p^s-1}\binom{p^s}iT^{p^s-1}+T^{p^s}\equiv1+pg \quad(\mod I)
  \end{equation*}
  con $g\in\LL$. Esto implica que
  \begin{align*}
    \Phi_s&=\frac{(1+T)^{p^s}-1}{(1+T)^{p^{s-1}}-1} \\
          &= (1+T)^{(p-1)p^{s-1}}+(1+T)^{(p-2)p^{s-1}}+\dotsm+(1+T)^{p^{s-1}}+1 \\
          &\equiv (1+p g)^{p-1}+(1+p g)^{p-2}+\dotsm+(1+p g)+1 \quad (\mod I)\\
    &\equiv  ph\quad (\mod I)
  \end{align*}
  con algún $h\in\LL$. Esto también es cierto para los $\Phi_{s'}$ con $s'\ge s$. Por la
  misma razón que con $T$, una potencia de $p$ debe estar en $I$. Concluimos que existe
  $t\ge s$ tal que $\Phi_s\dotsm\Phi_t\in I$.
\end{proof}

\begin{lem}\label{lem:congruencia-phi-s}
  Sea $P\in\LL$ un polinomio distinguido de grado $d\in\Ncero$.
  Entonces para $s$ tal que $p^{s-1}\ge d$, $\Phi_s$ es divisible por $\pi$ módulo
  $P$, es decir existe $g\in\LL$ tal que \[ \Phi_s\equiv \pi g\quad (\mod P). \]
  En el caso $\O=\Zp$ (tal que $\pi=p$) tenemos que $g\in\LL^\times$.
\end{lem}
\begin{proof}
  Si $s$ es tal que $p^{s-1}\ge d$ entonces tenemos $(1+T)^{p^{s-1}} \equiv 1+\pi h$ módulo
  $P$, con $h\in\O[T]$, porque $P$ es distinguido. La afirmación resulta de un cálculo
  análogo a aquel que hicimos en la demostración del
  \cref{lem:phi-s-anula-a-modulo-finito}. Lo dejamos como ejercicio.
\end{proof}

\begin{defi}
  Para $G$ un grupo profinito y $R$ un anillo profinito llamamos \emph{álgebra de Iwasawa}\index[def]{algebra de Iwasawa@álgebra de Iwasawa} de $G$ con coeficientes en $R$ al anillo de grupos profinito
$$\LL(G):=R\llbracket G\rrbracket=\varprojlim_{U\unlhd G} R[G/U]$$
(aquí normalmente $R=\O$ es el anillo de enteros de una extensión finita de $\Qp$, por
ejemplo $R=\Zp$, y debería ser claro del contexto).
\end{defi}

En el caso en que $\Gamma$ es un $p$-grupo profinito libre de rango $1$ y $\O$ un anillo de valuación discreta, el \cref{thm:EquivPowPro} nos dice que el álgebra de Iwasawa $\LL(\Gamma)$ de $\Gamma$ es isomorfa a las series formales $\O\llbracket T\rrbracket$ en una variable con coeficientes en $R$. En este caso particular, denotamos simplemente $\LL:=\LL(\Gamma)$ si no hay confusión.

\ejercicios

\begin{ejer}\label{ejer:omega-r-dist-irr}
  Demuestre que los polinomios $\Phi_{r}\in\LL$ de la \cref{defi:omega-r} son elementos
  distinguidos e irreducibles en $\LL$.
\end{ejer} 

\begin{ejer}\label{ejer:weierstrass}
  Cada serie de potencias $f\in\O\llbracket T\rrbracket$ define una función
  \[ \pi\O\rightarrow\O, \quad x\mapsto f(x) \]
  (verifique que esto está bien definido). Supongamos que esta función tiene una infinitud de
  ceros. Use el Teorema de preparación de Weierstraß para demostrar que $f=0$.
\end{ejer}

\begin{ejer}\label{ejer:maxideal}Demuestre las inclusiones \eqref{eq:idealmax} del \cref{thm:EquivPowPro}.
\end{ejer}
 
\begin{ejer}Sea $G$ un $p$-grupo abeliano profinito libre de rango $n$ y $\O$ un anillo de
  valuación discreta, demuestre que el álgebra de Iwasawa $\LL(G)$ es isomorfa al álgebra de
  series formales $\O\llbracket T_{1},\ldots,T_{m}\rrbracket$ en $m$ variables con
  coeficientes en $\O$.
\end{ejer}

\begin{ejer}
  Complemente los detalles en la demostración del \cref{lem:congruencia-phi-s}.
\end{ejer}

 \section{Medidas}
 \label{sec:medidas}

En esta sección explicamos cómo los elementos del álgebra de Iwasawa $\LL(G)$ pueden ser
vistos como medidas en $G$, para cualquier grupo profinito $G$. Luego lo estudiamos en más
detalle en el caso especial $G=\Gamma$ con $\Gamma\isom\Z_p$. Este punto de vista no será
tan importante en este texto, pero es muy común en la literatura y no debería faltar aquí.

Fijamos un anillo $\O$ como anteriormente, que sea un anillo local de valuación discreta,
completo por la topología definida por las potencias de su ideal máximo. El
ejemplo más importante es aquel en que $\O$ es el anillo de enteros de una extensión finita de
$\Qp$. Esta teoría se puede desarrollar en más generalidad, pero lo siguiente será suficiente para
nosotros.

\begin{defi}\label{defi:cgo}
  Sea $G$ un grupo profinito. Escribimos $\mathrm C(G,\O)$ como el
  $\O$-módulo de funciones continuas $G\rightarrow \O$ y $\mathrm C^\infty(G,\O)$ como el
  submódulo de funciones localmente constantes. Definimos una norma en $\mathrm C(G,\O)$ como
  \[ \abs{f}=\sup_{g\in G}\abs{f(g)} \quad (f\in\mathrm C(G,\O)). \]
\end{defi}

El submódulo $\mathrm C^\infty(G,\O)$ puede ser caracterizado de diferentes maneras según el \cref{ejer:propiedades-c-infty}.

\begin{prop}\label{prop:c-infty-denso}
  $\mathrm C(G,\O)$ es completo con respecto a la norma introducida en la
  \cref{defi:cgo} y $\mathrm C^\infty(G,\O)$ es denso.
\end{prop}
\begin{proof}
  Demostramos sólo la densidad y dejamos la completitud como ejercicio, porque los
  argumentos son muy similares.
  Para ver la densidad de $\mathrm C^\infty(G,\O)$ sean $\varepsilon>0$ y
  $f\in\mathrm C(G,\O)$ dados. Para cada $w$ en la imagen de $f$ sea $B(w,\varepsilon)\subseteq\O$
  la bola de radio $\varepsilon$ alrededor de $w$. Como la topología en $\O$ está
  definida por una ultramétrica, $B(w,\varepsilon)$ es abierta y cerrada. La preimagen
  $U_w=f^{-1}(B(w,\varepsilon))\subseteq G$ entonces también es abierta y cerrada. Los $U_w$
  (para todos los $w$) obviamente cubren $G$, y porque $G$ es compacto podemos elegir una
  cantidad finita $w_1,\dotsc,w_n$ tal que los $U_{w_i}$ para $i=1,\dotsc,n$ cubren
  $G$. Pongamos para cada tal $i$ \begin{equation*} A_i = U_{w_i} \setminus
    \bigcup_{j=i+1}^n U_{w_j}.\end{equation*} Entonces los $A_i$ todavía son abiertos y
  cerrados y además son disjuntos. Definimos una función \begin{equation*} f'=\sum_{i=1}^n
    w_i \mathbb{1}_{A_i} \end{equation*} (donde $\mathbb{1}$ denota la función indicador).
  Entonces
  \begin{equation*} \abs{f-f'} = \sup_{x\in X} \abs{f(x)-f'(x)} = \max_{i=1,\dotsc,n}
    \sup_{x\in A_i} \abs{f(x)-w_i} \le \varepsilon \end{equation*}
  y la densidad de $\mathrm C^\infty(G,\O)$ resulta.
\end{proof}

\begin{defi}
  Definimos $\mathrm D(G,\O)=\Hom_\O(\mathrm C(G,\O),\O)$, donde \enquote{$\Hom_\O$} son
  homomorfismos \emph{continuos} de $\O$-módulos. Los elementos de $\mathrm D(G,\O)$ se
  llaman \define{medidas} en $G$ con valores en $\O$.

  Si $f\in\mathrm C(G,\O)$ y $\mu\in\mathrm D(G,\O)$ a veces escribimos
  \[ \int_G f\integrald\mu \] en lugar de $\mu(f)$.
\end{defi}

Una consecuencia importante de la \cref{prop:c-infty-denso} es que de hecho \linebreak
$\mathrm D(G,\O)=\Hom_\O(\mathrm C^\infty(G,\O),\O)$ (\cref{ejer:c-infty-denso}).

\begin{prop}\label{prop:medidas}
  \begin{enumerate}
  \item\label{prop:isomorfismo-medidas} Existe un isomorfismo canónico de $\O$-módulos
    topológicos \[ \LL(G)\isom\mathrm D(G,\O). \]
  \item\label{prop:medidas-propiedad-universal} Sea $\chi\colon G\rightarrow\O^\times$ un
    carácter y $\mu\in\mathrm D(G,\O)$, y escribimos $\mu$ también para la imagen de $\mu$
    en $\LL(G)$ bajo el isomorfismo de \ref{prop:isomorfismo-medidas}. Entonces la imagen de
    $\mu$ bajo el morfismo canónico $\LL(G)\rightarrow\O$ inducido por $\chi$ es
    \[ \int_G\chi\integrald\mu. \]
  \item\label{prop:convolucion} El isomorfismo de \ref{prop:isomorfismo-medidas} se convierte un
    isomorfismo de $\O$-álgebras si definimos la multiplicación en $\mathrm D(G,\O)$ como la
    convolución
  \[ \int_Gf\integrald(\mu_1*\mu_2)=\int_G\left(\int_Gf(gh)\integrald\mu_1(g)\right)\mathrm
    d\mu_2(h)\text{ para }f\in\mathrm C(G,\O),\ \mu_1,\ \mu_2\in\mathrm D(G,\O). \]
  \end{enumerate}
\end{prop}
\begin{proof}
  De los ejercicios \labelcref{ejer:propiedades-c-infty} y \labelcref{ejer:c-infty-denso} se
  ve fácilmente que
  \begin{equation*}
    \mathrm D(G,\O) = \varprojlim_{U\unlhd G}\mathrm D(G/U,\O).
  \end{equation*}
  Pero como $G/U$ es un grupo finito discreto, $\mathrm C(G/U,\O)$ es un $\O$-módulo libre
  de rango finito con una base canónica dada por los elementos de $G/U$, y lo mismo es
  cierto para su módulo dual. Esto muestra que canónicamente
  \begin{equation*}
    \mathrm D(G/U,\O)\isom\mathrm C(G/U,\O)\isom\O[G/U]
  \end{equation*}
  como $\O$-módulos. La afirmación \ref{prop:isomorfismo-medidas} resulta pues de la
  definición de $\LL(G)$.

  Por la construcción del isomorfismo que acabamos de explicar es obvio que un elemento
  $g\in G$ corresponde a la medida de Dirac $\delta_g$ que envía $f\mapsto f(g)$ para
  $f\in\mathrm C(G,\O)$. Por eso el morfismo $\mathrm D(G,\O)\rightarrow\O$,
  $\mu\mapsto\mu(\chi)$ coincide con $\chi$ si lo restringimos a
  $G\subseteq\LL(G)\isom\mathrm D(G,\O)$. Esto demuestra que la composición
  $\LL(G)\rightarrow\mathrm D(G,\O)\rightarrow\O$ debe ser el morfismo dado por la
  propiedad universal porque ese es único. De ahí resulta \ref{prop:medidas-propiedad-universal}.

  Si $g_1,g_2\in G$ y $\delta_{g_1},\delta_{g_2}$ son las medidas de Dirac correspondientes,
  entonces de la definición de la convolución resulta fácilmente que
  $\delta_{g_1}*\delta_{g_2}=\delta_{g_1g_2}$. Por lo tanto
  \begin{equation*}
    G\rightarrow \mathrm D(G,\O)^\times,\quad g\mapsto \delta_g
  \end{equation*}
  es un morfismo de grupos topológicos que induce un morfismo de $\O$-álgebras
  $\LL(G)\rightarrow\mathrm D(G,\O)$. Pero de la construcción es claro que este coincide con
  el isomorfismo de \ref{prop:isomorfismo-medidas}, tenemos pues \ref{prop:convolucion}.
\end{proof}

A partir de ahora identificamos $\LL(G)$ con las medidas $\mathrm D(G,\O)$ y entonces vemos
elementos de $\LL(G)$ como funciones en $\mathrm C(G,\O)$.

\begin{defi}
  Escribimos $\mathcal Q(G)$ como el \define{anillo de cocientes} de $\LL(G)$, es decir la
  localización en que invertimos todos los elementos que no sean divisores de cero. Además
  sea $I(G)\subset\LL(G)$ el \emph{ideal de aumentación}\index[def]{ideal!de aumentación}, es decir el núcleo del morfismo
  $\LL(G)\rightarrow\O$ que envía todos los elementos de $G$ a $1$.

  Llamamos un elemento $\mu$ de $\mathcal Q(G)$ una \define{pseudo-medida} si $(g-1)\mu\in\LL(G)$
  para cada $g\in G$.
\end{defi}

Notemos que para cada pseudo-medida $\mu$, $I(G)\mu$ es un ideal en $\LL(G)$.

El morfismo $\chi\colon\LL(G)\rightarrow\O$ inducido por un carácter
$\chi\colon G\rightarrow\O^\times$ se extiende a un elemento $\frac a b\in\mathcal Q(G)$ si y
sólo si $b\notin\ker\chi$. Si $\mu=\frac{a}{b}$ es una pseudo-medida entonces
$b\notin\ker\chi$ siempre que $\chi$ no sea trivial: porque $(g-1)\mu\in\LL(G)$
para cada $g\in G$ significa que para cada $g$ existe un $c_g\in\LL(G)$ tal que $g-1=c_gb$,
y si escogemos un $g\in G$ con $\chi(g)\neq1$ entonces
$\chi(b)\chi(c_g)=\chi(g-1)=\chi(g)-1\neq0$ y pues $\chi(b)\neq0$. Esto demuestra que la
siguiente definición tiene sentido.

\begin{defi}
  Si $\mu\in\mathcal Q(G)$ es una pseudo-medida y $\chi\colon G\rightarrow\O^\times$ es un
  carácter no trivial, entonces ponemos
  \[ \int_G\chi\integrald\mu := \frac{1}{\chi(g)-1} \int_G \chi \integrald((g-1)\mu) \] para
  $g\in G$ tal que $\chi(g)\neq1$.
\end{defi}
Es claro que la expresión que acabamos de definir no depende de $g$.

A partir de ahora remplazamos $G$ por el grupo $\Zp$. Si $x\in\Zp$ entonces escribimos
$[x]$ para el elemento del grupo $\Zp\subseteq\LL(\Zp)^\times$ para distinguirlo del elemento
$x\in\Zp\subseteq\LL(\Zp)$ en el anillo de coeficientes. Tenemos entonces el generador
topológico canónico $[1]$.

En este caso describimos con más detalle los tres puntos de vista
del álgebra de Iwasawa. Tenemos un diagrama conmutativo de isomorfismos continuos
\begin{equation}
  \label{eqn:trinidad}
  \begin{tikzcd}
    \LL(\Zp) \arrow[rd, "\text{del \cref{thm:EquivPowPro}}", swap, leftrightarrow] \arrow[rr, leftrightarrow, "\text{de la \cref{prop:medidas}}"] & & \mathrm D(\Zp,\O)
    \arrow[ld, shift left, "\mathcal M"] \\
    & \O\llbracket T\rrbracket \arrow[ru, shift left, "\Upsilon"]
  \end{tikzcd}
\end{equation}
Más adelante vamos a describir explícitamente los mapeos $\mathcal M$ y $\Upsilon$. Primero
deduzcamos el siguiente resultado.

\begin{thm}[Mahler]\label{thm:mahler}
  Cada función $f\in\mathrm C(\Z_p,\O)$ se escribe de manera única como
  \begin{equation}
    \label{eqn:mahler}
    f(x)=\sum_{n=0}^\infty a_n\binom x n \quad(x\in\Zp)
  \end{equation}
  con coeficientes $a_n\in\O$ tal que $\displaystyle\lim_{n\to\infty}a_n=0$. Aquí $\displaystyle\binom x 0 = 1$ y 
  \begin{equation}
    \label{eqn:defi-binom}
    \binom x n =\frac{x(x-1)\dotsm(x-n+1)}{n!}\quad(n\in\Ncero,\ x\in\Zp).
  \end{equation}
  Por otro lado, cada serie de la forma \eqref{eqn:mahler} converge y define un elemento de
  $\mathrm C(\Z_p,\O)$.
\end{thm}
\begin{proof}
  Es fácil ver que cada serie de la forma \eqref{eqn:mahler} converge y define un elemento de
  $\mathrm C(\Z_p,\O)$; omitimos los detalles aquí. Para ver la unicidad utilizamos el mapeo
  \begin{equation*}
    \Delta\colon\mathrm C(\Z_p,\O)\rightarrow\mathrm C(\Z_p,\O),\quad \Delta
    f(x)=f(x+1)-f(x) \;(x\in\Zp).
  \end{equation*}
  Porque para $n\in\Ncero$ y $x\in\Zp$ tenemos (usando el \cref{ejer:relaciones-binom})
  \begin{equation*}
    \Delta\binom x n = \binom{x+1}n-\binom x n = \binom x {n-1}
  \end{equation*}
  resulta que si $f\in\mathrm C(\Z_p,\O)$ es de la forma \eqref{eqn:mahler} entonces
  $a_n=\Delta^nf(0)$ para cada $n\in\Ncero$. Esto demuestra la unicidad de los coeficientes.

  Para demostrar que cada $f\in\mathrm C(\Zp,\O)$ se escribe en esta forma, fijemos
  $x\in\Zp$. El elemento $[x]\in\LL(G)$ corresponde a la medida de Dirac $\delta_x\in\mathrm
  D(\Zp,\O)$, y también a la serie de potencias
  \begin{equation*}
    (1+T)^x=\sum_{n=0}^\infty\binom x n T^n\in\O\llbracket T\rrbracket.
  \end{equation*}
  Usamos la medida $\Upsilon(T^n)\in\mathrm D(\Zp,\O)$ que corresponde a $T^n$ para
  cada $n$. Porque todos los isomorfismos en \eqref{eqn:trinidad} son continuos, obtenemos
  \begin{equation*}
    f(x) = \delta_x(f) = \sum_{n=0}^\infty\binom x n \Upsilon(T^n)(f)
  \end{equation*}
  y obtenemos la afirmación poniendo $a_n=\Upsilon(T^n)(f)$. Es claro que $a_n\to0$ para
  $n\to\infty$ porque la continuidad de $\Upsilon$ implica que $\sum_na_n=\Upsilon(\sum_nT^n)(f)$.
\end{proof}

\begin{cor}\label{thm:trinidad}
  Los mapeos $\mathcal M$ y $\Upsilon$ en el diagrama \eqref{eqn:trinidad} están dados por
  \[ \mathcal M(\mu)=\sum_{n=0}^\infty c_n(\mu)T^n\text{ con }c_n(\mu)=\int_{\Zp}\binom x
    n\integrald\mu(x)=\mu\left(\binom\cdot n\right)\text{ para }\mu\in\mathrm D(\Zp,\O),\
    n\in\Ncero, \]
  \[ \Upsilon(g)(f)=\sum_{n=0}^\infty a_nb_n \text{ para }f=\sum_{n=0}^\infty a_n\binom\cdot
    n\in\mathrm C(\Zp,\O), \ g=\sum_{n=0}^\infty b_nT^n\in\O\llbracket T\rrbracket, \]
  donde usamos el \cref{thm:mahler} de Mahler para escribir $f$ en esta forma.
\end{cor}
\begin{proof}
  La continuidad de $\Upsilon$ implica que
  \begin{equation*}
    \Upsilon(g)(f)=\sum_{n=0}^\infty b_n\Upsilon(T^n)(f)
  \end{equation*}
  para $g=\sum_nb_nT^n\in\O\llbracket T\rrbracket$ y $\Upsilon(T^n)(f)=a_n$ según la
  demostración del \cref{thm:mahler}. Esto demuestra la fórmula para $\Upsilon$.

  La afirmación para $\mathcal M$ resulta si verificamos que la fórmula de arriba define un
  mapeo inverso a $\Upsilon$. Esto es un cálculo corto que dejamos como ejercicio.
\end{proof}

\ejercicios

\begin{ejer}\label{ejer:propiedades-c-infty}
  Demuestre que para una función $f\in\mathrm C(G,\O)$ los siguientes enunciados son equivalentes:
  \begin{enumerate}
  \item $f\in\mathrm C^\infty(G,\O)$;
  \item existe un subgrupo abierto $H\subseteq G$ tal que $f$ se factoriza como una función
    $G/H\rightarrow \O$;
  \item la imagen de $f$ es finita.
  \end{enumerate}
\end{ejer}

\begin{ejer}\label{ejer:c-infty-denso}
  \begin{enumerate}
  \item Demuestre que $\mathrm C(G,\O)$ es completo con respecto a la norma introducida en
    la \cref{defi:cgo} (la primera afirmación de la \cref{prop:c-infty-denso}).
  \item Deduzca de la \cref{prop:c-infty-denso} que
    $\mathrm D(G,\O)=\Hom_\O(\mathrm C^\infty(G,\O),\O)$.
  \end{enumerate}
\end{ejer}

\begin{ejer}
  Demuestre que los morfismos $\mathcal M$ y $\Upsilon$ (si las definimos por las fórmulas
  en el \cref{thm:trinidad}) son inversos el uno al otro.
\end{ejer}

\begin{ejer}\label{ejer:upsilon-cero}
  Demuestre que para $g\in\O\llbracket T\rrbracket$ tenemos
  \[ \int_{\Zp}\integrald\Upsilon(g)=g(0). \]
\end{ejer}

\begin{ejer}\label{ejer:relaciones-binom}
  Demuestre las siguientes relaciones para los binomios:
  \begin{enumerate}
  \item  Para cada  $n\in\Ncero$ y $x\in\Zp$, \[ \binom{x+1}n-\binom x n = \binom x {n-1}. \]
  \item Para cada $j,k,x\in\Ncero$ con $k\le x$
    \[ \binom{j+k}k\binom x {j+k} = \binom x k \binom{x-k}j. \]
  \item Para $x,k\in\Ncero$ con $k\le x$
    \[ \sum_{j=0}^{x-k}(-1)^j\binom{x-k}j=
      \begin{cases}
        1,&k=x,\\
        0,&0\le k<x.
      \end{cases}
    \]
    Use el teorema del binomio para esto.
  \end{enumerate}
\end{ejer}

\begin{ejer}
  En este ejercicio damos una demostración directa y elemental del \cref{thm:mahler} de
  Mahler (sin usar los isomorfismos en \eqref{eqn:trinidad}), siguiendo
  \cite[§3.1]{MR1216135} (sin la parte de la unicidad, que ya demostramos de una manera
  elemental en el \cref{thm:mahler}).

  Sea $f\in\mathrm C(\Z_p,\O)$. Para $n\in\Ncero$ definimos
  \begin{equation*}
    a_n = \sum_{k=0}^n(-1)^k\binom n k f(n-k) = \sum_{k=0}^n(-1)^{n-k}\binom n k f(k).
  \end{equation*}
  \begin{enumerate}
  \item   Use relaciones del \cref{ejer:relaciones-binom} para verificar que
    \[  \sum_{n=0}^\infty a_n\binom x n = f(x) \]
    para $x\in\Nuno$. 
  \item Deduzca el teorema de Mahler.
  \end{enumerate}
\end{ejer}
  
\section{Idempotentes y el álgebra de Iwasawa para grupos más grandes}
\label{sec:mas-grandes}

Hasta ahora estudiamos el álgebra de Iwasawa de la forma $\O\llbracket\Gamma\rrbracket$ con
un grupo profinito $\Gamma$ isomorfo a $\Zp$, pero al final queremos usar el anillo análogo
para grupos más grandes de la forma $\Delta\times\Gamma$ con $\Delta$ un grupo
finito. En esta sección estudiamos estas álgebras en las que los \emph{idempotentes
  ortogonales}\index[def]{idempotente!ortogonal} son una herramienta útil. Además resumimos sus propiedades, que son bien
conocidas. En el \cref{ejer:idempotentes} damos algunas indicaciones de la demostración.

\begin{lem}\label{lem:idempotentes}
  Sea $A$ un dominio entero y $\Delta$ un grupo abeliano finito tal que
  $\#\Delta\in A^\times$ y $A$ contiene las raíces de la unidad de orden $\#\Delta$. Para
  cada carácter $\chi\colon \Delta\rightarrow A^\times$ definimos
  \[ e_\chi=\frac{1}{\#\Delta}\sum_{\delta\in \Delta}\chi(\delta)^{-1} \delta\in
    A[\Delta]. \] Este elemento se llama \emph{idempotente asociado a $\chi$}\index[def]{idempotente!asociado a $\chi$}.

  Entonces tenemos lo siguiente:
  \begin{enumerate}
  \item\label{lem:idempotentes:idempotente} $e_\chi^2=e_\chi$ para cada $\chi$;
  \item\label{lem:idempotentes:eigen} $e_\chi \delta=\chi(\delta)e_\chi$ para cada $\chi$ y
    $\delta\in \Delta$;
  \item\label{lem:idempotentes:ortogonal} $e_\chi e_\psi=0$ si $\chi\neq\psi$;
  \item\label{lem:idempotentes:suma} $\sum_\chi e_\chi=1$, la suma tomada sobre todos los
    $\chi$ posibles.
  \end{enumerate}
  Ahora sea $M$ un $A[\Delta]$-módulo. Para cada $\chi\colon \Delta\rightarrow A^\times$ escribimos \[
    M[\chi]=\{m\in M \mid \forall \delta\in \Delta\colon \delta m=\chi(\delta)m\} \]
  para el espacio propio de $\chi$ en $M$.
  Entonces
  \begin{enumerate}[resume]
  \item\label{lem:idempotentes:eigen-m} $M[\chi]=e_\chi M$;
  \item\label{lem:idempotentes:m-descomp} $M=\bigoplus_\chi M[\chi]$, la suma tomada de
    nuevo sobre todos los $\chi\colon\Delta\rightarrow A^\times$ posibles.
  \end{enumerate}
\end{lem}

Ahora sea $G$ un grupo profinito de la forma $G=\Delta\times \Gamma$ con un
grupo finito $\Delta$ y un grupo profinito $\Gamma$. Además sea $A$ un anillo conmutativo
topológico tal que $\#\Delta\in A^\times$ y $A$ contenga las raíces de la unidad de orden
$\#\Delta$. Escribimos
\[ A\llbracket G\rrbracket := \varprojlim_{N}A[G/N], \] el límite tomado sobre
todos los subgrupos normales abiertos de $G$.\footnote{Es decir, si $A$ es un
  anillo profinito, esto es el anillo de grupos profinito que definimos antes, pero ahora
  permitimos anillos más generales como coeficientes. En las aplicaciones el anillo $A$ será
  un anillo profinito o una extensión finita de $\Q$ o $\Qp$.}  Entonces
$A\llbracket G\rrbracket$ es un $A[\Delta]$-módulo y podemos aplicar el
\cref{lem:idempotentes}. Se verifica fácilmente (usando
\cref{lem:idempotentes}~\ref{lem:idempotentes:eigen}) que para cada $\chi$ de hecho
$e_\chi A\llbracket G\rrbracket$ es un $A$-álgebra topológica con multiplicación dada por 
$(e_\chi g) (e_\chi h)=e_\chi (g h)$ para $g,h\in G$.

\begin{lem}\label{lem:descomposicion-lambda}
  Para cada $\chi\colon\Delta\rightarrow A^\times$ existe un isomorfismo canónico de
  $A$-álgebras topologicas
  \[ E_\chi\colon e_\chi A\llbracket G\rrbracket \isomarrow A\llbracket
    \Gamma\rrbracket. \]
  Es decir, \[ A\llbracket
    G\rrbracket\isom\bigoplus_\chi A\llbracket \Gamma\rrbracket, \]
  la suma tomada sobre todos los $\chi\colon\Delta\rightarrow A^\times$.
\end{lem}
\begin{proof}
  Primero observamos que $A\llbracket G\rrbracket=\varprojlim_NA[\Delta\times \Gamma/N]$, el
  límite tomado esta vez sobre los subgrupos normales abiertos de $\Gamma$. Fijemos un tal
  subgrupo $N$. Entonces es suficiente construir un isomorfismo
  $e_\chi A[\Delta\times \Gamma/N]\isom A[\Gamma/N]$, la afirmación resulta de esto al tomar
  el límite.
 
  Para cada $g=(\delta,\gamma)\in \Delta\times \Gamma/N$ tenemos
  $e_\chi g=\chi(\delta)\gamma e_\chi$ (véase \cref{lem:idempotentes}). El mapeo
  \begin{equation*} \quad e_\chi g \mapsto \chi(\delta)\gamma \end{equation*} se extiende
  $A$-linealmente a un morfismo de $A$-módulos $e_\chi A[\Delta\times \Gamma/N] \rightarrow
  A[\Gamma/N]$. Por otra parte el mapeo \begin{equation*} A[\Gamma/N] \rightarrow e_\chi
    A[\Delta\times \Gamma/N], \quad \gamma\mapsto e_\chi(1,\gamma) \quad (\gamma\in
    \Gamma/N)\end{equation*} se extiende a un morfismo en la otra dirección. Un cálculo
  fácil muestra que estos morfismos son inversos el uno al otro, además es obvio de sus
  definiciones que los mapeos son multiplicativos.
\end{proof}

La aplicación más importante de esto será la situación en que $A=\O$ es el anillo de enteros
en una extension finita de $\Qp$ y $G$ es isomorfo a
$\Z_p^\times\isom\F_p^\times\times(1+p\Zp)$, con $\Delta$ correspondiendo a $\F_p^\times$ y
$G$ a $1+p\Zp$, que es entonces isomorfo a $\Zp$ y que por eso denotamos $\Gamma$. Entonces
$\O$ contiene las raíces $(p-1)$-ésimas de la unidad y $p-1\in\O^\times$. Fijemos un
isomorfismo $G\isom\Z_p^\times$. Entonces los caracteres de $\Delta$ son las
potencias del carácter de Teichmüller $\omega$.

\begin{cor}\label{cor:descomposicion-lambda}
  Para cada $i\in\{1,\dotsc,p-1\}$ existe un isomorfismo
  canónico de $\O$-álgebras profinitas
  \[ E_{\omega^i}\colon e_{\omega^i} \O\llbracket G\rrbracket \isomarrow\O\llbracket
    \Gamma\rrbracket. \]
  Es decir, $\O\llbracket
  G\rrbracket\isom\displaystyle\bigoplus_{i=1}^{p-1}\O\llbracket\Gamma\rrbracket$.
\end{cor}

En particular, todos los $e_{\omega^i}\O\llbracket G\rrbracket$ son isomorfos a
$\O\llbracket T\rrbracket$ (no canónicamente, después de elegir un generador topológico de
$\Gamma$, véase \cref{thm:EquivPowPro}) y podemos aplicar los resultados de nuestro estudio
de este anillo.

\ejercicios

\begin{ejer}\label{ejer:delta-hat}
  Sea $\Delta$ un grupo abeliano finito y $A$ un anillo conmutativo que contiene las raíces
  $\#\Delta$-ésimas de la unidad. Denotamos $\widehat\Delta$ el grupo de caracteres
  $\Delta\rightarrow A^\times$ (con multiplicación puntual). Demuestre que
  $\widehat\Delta\isom\Delta$ no canónicamente; en particular, existen $\#\Delta$ caracteres
  $\Delta\rightarrow A^\times$. Use el teorema de estructura de grupos abelianos finitos
  para esto.
\end{ejer}

\begin{ejer}\label{ejer:schur-orthogonalidad}
  Sea $\Delta$ un grupo abeliano finito y $A$ un dominio entero que contiene las raíces
  $\#\Delta$-ésimas de la unidad.
  \begin{enumerate}
  \item Sea $\delta\in\Delta$ fijo. Demuestre que
    \[ \sum_{\chi}\chi(\delta)=
      \begin{cases}
        \#\Delta&\text{si }\delta=1,\\
        0&\text{si }\delta\neq1
      \end{cases} \]
    (la suma corriendo por todos los caracteres $\chi\colon\Delta\rightarrow A^\times$).
  \item Sea $\chi\colon\Delta\rightarrow A^\times$ fijo. Demuestre que
    \[ \sum_{\delta}\chi(\delta)=
      \begin{cases}
        \#\Delta&\text{si }\chi=1,\\
        0&\text{si }\chi\neq1
      \end{cases} \]
    (la suma corriendo por todos los $\delta\in\Delta$).
  \end{enumerate}
  Puede ser útil usar el \cref{ejer:delta-hat} para esto.
\end{ejer}

\begin{ejer}\label{ejer:idempotentes}
  Sea $\Delta$ un grupo abeliano finito y $A$ un dominio entero que contiene las raíces
  $\#\Delta$-ésimas de la unidad.  Demuestre las propiedades de los idempotentes del
  \cref{lem:idempotentes}.

  Las propiedades \ref{lem:idempotentes:idempotente} y \ref{lem:idempotentes:eigen} se pueden
  verificar con cálculos directos usando que la multiplicación con un elemento permuta los
  elementos del grupo $\Delta$, que permite transformar las sumas. La relación
  \ref{lem:idempotentes:ortogonal} resulta de \ref{lem:idempotentes:eigen} considerando el
  elemento $e_\chi e_\psi$. La relación \ref{lem:idempotentes:suma} resulta del
  \cref{ejer:schur-orthogonalidad}.  Las afirmaciones \ref{lem:idempotentes:eigen-m} y
  \ref{lem:idempotentes:m-descomp} son consecuencias directas de las propiedades anteriores.
\end{ejer}

\begin{ejer}\label{ejer:diagonal}
  Consideramos la situación del \cref{lem:descomposicion-lambda}. El morfismo \[
    \Gamma\rightarrow G=\Delta\times\Gamma, \quad \gamma\mapsto (1,\gamma) \]
  induce un morfismo de $A$-álgebras $A\llbracket\Gamma\rrbracket\rightarrow
  A\llbracket G\rrbracket$. Componiendo este con el isomorfismo del
  \cref{lem:descomposicion-lambda} obtenemos un morfismo de $A$-álgebras
  \[ A\llbracket\Gamma\rrbracket\rightarrow\bigoplus_\chi A\llbracket\Gamma\rrbracket. \]
  Demuestre que este morfismo es la aplicación diagonal $x \mapsto (x,\dotsc,x)$.
\end{ejer}

\chapter{Módulos noetherianos sobre el álgebra de Iwasawa}
\label{sec:modulos-iwasawa}

Después de haber conocido el álgebra de Iwasawa vamos a estudiar módulos sobre ella. Esto es
interesante porque aparecen de manera natural en muchas situaciones (véase por ejemplo la
\cref{prop:modulos-sobre-rg}) y porque admiten una teoría de estructura muy similar a la de
módulos sobre anillos de ideales principales.

Sea $R$ un anillo conmutativo.
Recordemos que un $R$-módulo $M$ es llamado \emph{noetheriano}\index[def]{módulo@módulo!noetheriano} si todos sus submódulos son finitamente generados. Equivalentemente, si satisface la condición de la cadena ascendente en el orden parcial formado por sus submódulos e inclusiones. Equivalentemente, que para un conjunto $S$ no vacío de submódulos de $M$ existe un submódulo máximo, es decir un submódulo $M_{0}$ de $M$ tal que para cualquier elemento $N$ de $S$ que contenga $M_{0}$ tenemos $N=M_{0}$.

Un anillo conmutativo se llama \emph{noetheriano}\index[def]{anillo!noetheriano} si es noetheriano
como módulo sobre si mismo. Observemos que un módulo sobre un anillo noetheriano es
noetheriano si y solo si es finitamente generado. Como vamos a estudiar sobre todo módulos
sobre el álgebra de Iwasawa (que es ¡Un anillo noetheriano!) podemos usar las expresiones
\enquote{noetheriano} y \enquote{finitamente generado} indiferentemente en este caso.

Además los módulos noetherianos se comportan bien en sucesiones exactas cortas, ya que si tenemos una sucesión exacta de $R$-módulos
$$0 \rightarrow M' \rightarrow M \rightarrow M'' \rightarrow 0,$$
entonces $M$ es noetheriano si y solamente si $M'$ y $M''$ son noetherianos. 

Para un elemento $x$ en un $R$-módulo $M$ definimos el \define{anulador} de $x$
como $$\Ann(x)=\{r \in\LL\mid r \cdot x =0\}.$$ Un elemento $x\neq0$ es dicho de
\define[elemento de torsión]{torsión} si su anulador no es trivial. Un módulo $M$ se llama de \emph{torsión}\index[def]{módulo@módulo!de torsión} si todos
sus elementos son elementos de torsión y se llama \emph{libre de torsión}\index[def]{módulo@módulo!libre de torsión} si ninguno de sus
elementos es elemento de torsión.

Observemos que si $M$ es un $R$-módulo noetheriano, entonces existen $x_{1},\ldots,x_{d}\in M$ tales que $M=\sum_{i=1}^{d}R x_{i}$, por lo tanto tenemos que $$\Ann(M) := \bigcap_{x\in M} \Ann(x) = \bigcap_{i=1}^{d} \Ann(x_{i}).$$

\section{Pseudo-isomorfismos}

\begin{defi}\label{def:Lmodule} Decimos que $M$ es un \define[módulo sobre $\LL$]{$\LL$-módulo}, si es un grupo topológico abeliano y Hausdorff, además de un $\LL$-módulo tal que la acción de $\LL$ es continua. 
\end{defi}

\begin{prop} Decimos que un $\LL$-módulo $M$ noetheriano es \define[módulo pseudo-nulo]{pseudo-nulo} si satisface las siguientes condiciones equivalentes
\begin{enumerate}
\item Para todo ideal primo $\p$ de altura $\leq 1$,  el localizado $M_{\p}$ es trivial.
\item El único ideal primo que contiene al anulador $\Ann(M)$ es $\M$.
\item $M$ es finito. 
\end{enumerate}
\end{prop}
\begin{proof}
  \begin{description}
  \item[(a) $\Rightarrow$ (b):] Si $x\in M$ y denotamos $x_{\p}$ su imagen en $M_{\p}$, entonces $x_{\p}=0$ si existe un $s\in\LL\setminus\p$ tal que $sx=0$, es decir $s\in\Ann(M)\nsubseteq\p$. Por lo que claramente el único ideal que contiene a $\Ann(M)$ tiene que ser el ideal máximo. 
  \item[(b) $\Rightarrow$ (a):] Si $\Ann(M)\nsubseteq\p$ entonces existe $s\notin \p$ tal que $sM=0$, es decir $M_{\p}=0$.
  \item[(b) $\Rightarrow$ (c):] Si el único ideal que contiene $\Ann(M)$ es $\M$ tenemos que
    el radical $\Rad(\Ann(M))=\M$ (véase el \cref{ejer:radical}). Luego como $\LL$ es
    noetheriano $\Ann(M)\supset \M^{k}$, para algún $k$ suficientemente grande (ver el
    \cref{ejer:radical}). Finalmente, si escribimos $M=x_{1}\LL + \ldots + x_{d}\LL$ y
    tomamos en cuenta que $\Ann(M)\subset \Ann(x)$ para todo $x\in M$, obtenemos que
    $|M|\leq (\LL:\M^{k})^{d}$.
    \item[(c) $\Rightarrow$ (b):] Si $M$ es finito, entonces 
$$\M M \supset \M^{2} M \supset \M^{3} M \supset \ldots $$
es una filtración estricta de módulos finitos que eventualmente es $0$. Es decir $\Ann(M)\supset \M^{k}$ par algún $k$, por lo que ningún ideal de altura $1$ contiene al anulador. 

  \end{description}
\end{proof}

\begin{remark} 
  \begin{enumerate}
  \item Un $\LL$-módulo pseudo-nulo es de torsión ya que $M\otimes_{\LL}K=M_{(0)} =0 $.
  \item El orden de un $\LL$-módulo pseudo-nulo siempre es una potencia de $p$ (porque es un $\Zp$-módulo).
  \end{enumerate}
\end{remark}

\begin{defi} \begin{enumerate}
\item Un morfismo $\phi$ de $\LL$-módulos sera llamado \emph{pseudo-inyectivo}\index[def]{morfismo!pseudo-inyectivo} (respectivamente \emph{pseudo-sobreyectivo}\index[def]{morfismo!pseudo-sobreyectivo}) luego de que el núcleo $\ker(\phi)$ (respectivamente su conúcleo $\coker(\phi)$) es pseudo-nulo. 
\item Un morfismo que es pseudo-inyectivo y pseudo-sobreyectivo, decimos que es un \define{pseudo-isomorfismo}.
\item Si $\varphi:M\rightarrow N$ es un pseudo-isomorfismo escribimos $M\sim N$.
\end{enumerate}
\end{defi}

\begin{remark} Tenemos que $\varphi:M\rightarrow N$ es un pseudo-isomorfismo si equivalentemente $\varphi_{\p}:M_{\p}\rightarrow N_{\p}$ es un isomorfismo para todos los ideales $\p$ de altura menor o igual que $1$.
\end{remark}

\ejercicios

\begin{ejer} La relación de pseudo-isomorfismo generalmente no es simétrica, demuestre que
  $\M\sim\LL$ pero $\LL\not\sim \M$ (no obstante, tenga en cuenta el \cref{ejer:relacion-simetrica}).
\end{ejer}

\begin{ejer}\label{ejer:pseudotransitiva}
  La relación de pseudo-isomorfismo es transitiva: demuestre que la composición de dos
  pseudo-isomorfismos es un pseudo-isomorfismo.
\end{ejer}

\begin{ejer}\label{ejer:fgpseudonul} Sea $M$ un $\LL$-módulo noetheriano. Demuestre que $M$ es pseudo-nulo si y solamente si existen $f$ y $g$ elementos coprimos de $\LL$ tal que $fM=gM=0$.
\end{ejer}

\begin{ejer}\label{ejer:pseudo-isomorfo-cociente-finito}
    Sea $\varphi\colon M\rightarrow N$ un pseudo-isomorfismo de $\LL$-módulos y
    $a=\max\{|\ker(\varphi)|,|\coker(\varphi)|\}$. Sean $\alpha\colon M\rightarrow M$ y
    $\beta\colon N\rightarrow N$ morfismos de $\LL$-módulos tales que el siguiente diagrama conmuta
    \begin{equation*}
    \begin{tikzcd}
      M \arrow[r, "\alpha"] \arrow[d, "\varphi"]
      & M \arrow[d, "\varphi"] \\
      N \arrow[r, "\beta"]
      & N.
    \end{tikzcd}
  \end{equation*}  
  Tenemos un diagrama
      \begin{equation*}
    \begin{tikzcd}
      0\arrow[r] & \alpha(M) \arrow[r] \arrow[d, "\varphi'"]
      & M \arrow[d, "\varphi"] \arrow[r]
      & M/\alpha(M) \arrow[d, "\varphi''"] \arrow[r] & 0 \\
      0\arrow[r] & \beta(N) \arrow[r]
      & N \arrow[r]
      & N/\beta(N) \arrow[r] & 0
    \end{tikzcd}
  \end{equation*}  
  con lineas exactas.
  \begin{enumerate}
  \item\label{ejer:pseudo-isomorfo-cociente-finito:a} Demuestre que
    $$|\ker(\varphi')|,\ |\coker(\varphi')|,\ |\coker(\varphi'')|\ \leq a\ \text{ and }\
    |\ker(\varphi'')|\ \leq a^{2}.$$
  \item\label{ejer:pseudo-isomorfo-cociente-finito:b} Deduzca que si
    $\varphi\colon M\rightarrow N$ es un pseudo-isomorfismo y $\nu\in\LL$ entonces
    $\varphi_\nu\colon M/\nu M\rightarrow N/\nu N$ también es un pseudo-isomorfismo y
    $|\ker\varphi_\nu|\le a^2$, $|\coker\varphi_\nu|\le a$.
  \item\label{ejer:pseudo-isomorfo-cociente-finito:c} En la misma situación, deduzca que
  \[ M/\nu M\text{ es finito} \iff N/\nu N\text{ es finito.} \]
  \end{enumerate}
\end{ejer}

\begin{ejer}\label{ejer:soportealpha}
Recordemos que el \define[soporte]{soporte de un módulo} $M$ está definido como el conjunto de los ideales primos tales que la localización es no trivial, es decir
$$\Sop(M)=\{\p\subset R \text{ primo }\,|\, M_{\p}\neq 0 \}.$$

Sea $M$ un $\LL$-módulo noetheriano y sea $\alpha\in \LL$ un elemento diferente de cero tal que $\Sop(\LL/\alpha \LL)$ y $\{\p\in\Sop(M)\mid\p\text{ es de altura }1\}$ son disjuntos. Entonces la multiplicación por $\alpha$ en $M$ es un pseudo-isomorfismo.
\end{ejer}

\begin{ejer}\label{ejer:radical}
Recordemos que el \define{radical de un ideal} $I$ en un anillo conmutativo $R$ está definido como el conjunto de los elementos $r\in R$ que al elevarlos a alguna potencia pertenecen a $I$, es decir
$$\Rad(I)=\{r \in R \,|\, r^{n}\in I \}.$$

\begin{enumerate}
\item Demuestre que $$\Rad(I)=\bigcap_{\substack{\p \in \Spec(R) \\ I \subset \p}} \p.$$
\item En particular, si $M$ es un módulo finitamente generado sobre un anillo noetheriano $R$. Entonces
$$\Rad(\Ann(M))=\bigcap_{\p \in \Sop(M)} \p.$$
\item Sea ahora $R$ un anillo local con ideal máximo $\M$. Con las hipótesis anteriores,
  demuestre que si $\Rad(\Ann(M))=\M$ entonces $\Ann(M)\supset \M^{k}$, para algún $k$
  suficientemente grande.
\end{enumerate}

\end{ejer}

\section{$\LL$-módulos noetherianos de torsión}

Recordemos que si $F$ es un submódulo finito de un $\LL$-módulo $M$, entonces $F$ es de torsión. 

\begin{prop}\label{prop:sinfin} Sea $M$ un $R$-módulo noetheriano (donde $R$ es un anillo conmutativo). $M$ contiene un submódulo máximo finito  $F$, y $M/F$ no contiene submódulos finitos no triviales. 
\end{prop}
\begin{proof}
El conjunto de los submódulos finitos de $M$ contiene un elemento máximo $F$ por la propiedad noetheriana, de lo contrario si $G$ es otro módulo que no está contenido en $F$, entonces $F+G$ es finito y contiene $F$.
\end{proof}

El anulador de $M$ es un ideal. En el caso especial $R=\LL$, la siguiente proposición revela su naturaleza.

\begin{prop}\label{prop:Annppal} Si $M$ es un $\LL$-módulo sin submódulos finitos no triviales, su anulador $\Ann(M)$ es principal.
\end{prop}
\begin{proof}
  Sea $M=\LL x_1+\dotsm+\LL x_d$.
Supongamos que $\Ann(M)$ no es
principal, entonces por la \cref{prop:contenfin}, $\Ann(M)$ está contenido en un ideal
principal $a\LL$ con índice finito y $a M$ es no trivial.
Existe un morfismo sobreyectivo $(a\LL/\Ann(M))^d\twoheadrightarrow aM$, por lo tanto
$$|a M|=(a M:\Ann(M) M)\leq (a \LL:\Ann(M))^{d},$$
lo cual contradice nuestra hipótesis en la carencia de submódulos finitos no triviales.  
\end{proof}

\begin{lem}\label{lem:pairwiseann} Sea $M$ un $\LL$-módulo noetheriano y de torsión, sin submódulos finitos no triviales. Sea $f=f_{1}f_{2}$ una factorización de su anulador con $f_{1}$ y $f_{2}$ coprimos. Entonces tenemos
$$M\sim f_{1}M\oplus f_{2}M\	\	\	\text{ y }\	\	\	f_{1}M\oplus f_{2}M\sim M$$
con $\Ann(f_{1}M)=f_{2}\LL$ y $\Ann(f_{2}M)=f_{1}\LL $.
\end{lem}
\begin{proof} Es claro que la suma $f_{1}M+f_{2}M$  es directa, ya que si $x\in f_{1}M \cap f_{2}M$, por la \cref{prop:contenfin} tenemos que $\LL/\Ann(x)\isom \LL x$ es finito. Como $M$ no contiene submódulos finitos no triviales concluimos que $x=0$. Por lo tanto tenemos la siguiente sucesión exacta 
$$ 0\rightarrow f_{1}M\oplus f_{2}M \rightarrow M \rightarrow M/f_{1}M\oplus f_{2}M.$$

Veamos que el mapeo $x\mapsto f_{1}x + f_{2}x$ es inyectivo. El núcleo de la aplicación es $\{x\in M \mid f_{1}x = -f_{2}x \}$, por lo que el núcleo es anulado por $(f_{1}+f_{2})$ y $f$. Entonces si $x$ es un elemento del núcleo tenemos
$$ (f_{1}+f_{2})\LL + f\LL  \subset \Ann(x) \subset \LL $$
pero $\LL/\Ann(x)\isom x\LL$ es finito por la \cref{prop:contenfin} ya que 
$(f_{1}+f_{2})$ y $f$ son coprimos. Por lo tanto $x=0$ y tenemos la siguiente sucesión exacta
$$ 0\rightarrow M \rightarrow f_{1}M\oplus f_{2}M \rightarrow (f_{1}M\oplus f_{2}M)/(f_{1}+f_{2})M . $$

Por último, demostramos que los conúcleos son finitos. Observemos que $M/f_{1}M\oplus f_{2}M$ es anulado por $f_{1}$ y por $f_{2}$. Es decir, al localizar en $\p$, existe siempre un elemento $f_{i}$ tal que $f_{i}M \in f_{1}M\oplus f_{2}M$. Por lo que el localizado es nulo para todo $\p$ de altura $\geq 1$. 

Por último, tenemos que
$$(f_{1}+f_{2})(f_{1}M\oplus f_{2}M)\subset (f_{1}+f_{2})M \subset f_{1}M\oplus f_{2}M,$$
además es fácil ver que 
$$f_{1}M\oplus f_{2}M/(f_{1}+f_{2})(f_{1}M\oplus f_{2}M)$$ 
es anulado por $f_{1}$ y $f_{2}$. Por lo tanto es pseudo-nulo.
\end{proof}

\begin{lem}\label{lem:pseudotorsii} Sean $M$ y $N$ $\LL$-módulos noetherianos y de torsión, ambos sin submódulos finitos no triviales. Entonces $M\sim N$ si y solamente si $N\sim M$.
\end{lem} 
\begin{proof} Sea $M\stackrel{\varphi}{\longrightarrow} N$ un pseudo-isomorfismo. Sea $c\in\LL$ un elemento coprimo a $f$, el generador de $\Ann(M)$, tal que $c$ anula $\ker(\varphi)$ (\cref{ejer:fgpseudonul}). La restricción de $\varphi$ a $cM$ es inyectiva, y $cM$ es de índice finito en $M$. 

La imagen $\varphi(cM)$ es de índice finito en $N$ (\cref{ejer:pseudo-isomorfo-cociente-finito}). Es decir, podemos elegir $d$ coprimo con $g$, el generador de $\Ann(N)$, tal que $dN\subset\varphi(cM)$. La multiplicación por $d$ en $N$ es pseudo-inyectiva. Como la composición de pseudo-isomorfismos es un pseudo-isomorfismo (\cref{ejer:pseudotransitiva}), tenemos un pseudo-isomorfismo
$$N\rightarrow dN \stackrel{\varphi^{-1}}{\longrightarrow} cM \hookrightarrow M,$$
donde $\varphi^{-1}(dN)\subset cM$ y como $\varphi(dM)\subseteq dN$, tenemos que $dM\subseteq \varphi^{-1}(dN)$ y deducimos que $\coker(\varphi^{-1})$ es pseudo-nulo.
\end{proof}

El lema precedente es de hecho cierto con más generalidad y es una consecuencia del teorema de estructura. La relación de pseudo-isomorfismo es una relación de equivalencia entre los $\LL$-módulos noetherianos de torsión (ver el \cref{ejer:relacion-simetrica}). 

\begin{prop}\label{prop:anuladorprimariosumapei} Sea $M$ un $\LL$-módulo noetheriano y de torsión sin submódulo finito y con anulador de la forma $\p^e$ con $\p$ un ideal primo principal de $\LL$. Entonces $M$ es pseudo-isomorfo a una suma directa
  $$M\sim \bigoplus_{i=1}^{k} \LL/\p^{e_{i}}.$$
\end{prop}
\begin{proof} Primero supongamos que $\p=P$ con $P$ un polinomio distinguido e irreducible. El cociente $\LL/P^{e}$ es un $\O$-módulo libre de dimensión $e\deg(P)$. Por lo tanto, como $M$ es finitamente generado como $\LL/P^{e}$-módulo, también lo es como $\O$-módulo. El submódulo de $\O$-torsión $T_{\O}(M)$ de $M$, por el teorema de estructura de módulos finitamente generados sobre anillos principales, es anulado por una potencia $\pi^{m}$ del uniformizante de $\O$. Por lo que $T_{\O}(M)$ es un $\LL$-módulo anulado por $\pi^{m}$ y por $P^{e}$, entonces pseudo-nulo. Por hipótesis, $M$ no contiene submódulos finitos no triviales, de donde concluimos que $M$ es libre de $\O$-torsión y por lo tanto libre. Sea $d=\rg_{\O}M$, vamos a proceder por inducción para demostrar el resultado.

Si $d=0$ no hay nada que demostrar. Supongamos que la hipótesis es cierta para $M$ de rango menor que $d$. 

Sea $x\in M$ de anulador mínimo $P^{e}$ y consideremos la sucesión exacta corta
$$0\longrightarrow \LL x \longrightarrow M \longrightarrow M/\LL x \longrightarrow 0.$$
Por hipótesis de inducción $M/\LL x$ es pseudo-isomorfo a una suma directa $M'=\bigoplus_{i=1}^{m} \LL/\p^{e_{i}}$. Tomemos un $c\notin \p$, tal que la imagen de $M/\LL x$ contiene $cM'$. Escribimos
$$cM'=\bigoplus_{i=1}^{m}\LL cx_{i}$$
con $x_{i}\in M'$. Tomamos levantamientos $y_{i}$ en $M$ de los $cx_{i}$. Dado que $P\nmid x$, podemos escribir
$$P^{e_{i}}y_{i}=fP^{e'_{i}}x\in\LL x$$
con $P\nmid f$ y $e'_{i}\geq e_{i}$. Sea $z_{i}=y_{i}-fP^{e'_{i}-e_{i}}x$ en $M$. Claramente las imágenes $\bar{z}_{i}$ y $\bar{y}_{i}$ de $z_{i}$ y de $y_{i}$ en $M/\LL x$ coinciden, además $P^{e_{i}}z_{i}=0$ en $M$, por lo tanto $\LL z_{i} \isom \LL/\p^{e_{i}}$. La suma $\sum_{i=1}^{m}\LL z_{i}$ es directa, isomorfa a 
$$\bigoplus_{i=1}^{m}\LL/\p^{e_{i}}.$$
Definamos por último 
$$M''=\left( \bigoplus_{i=1}^{m} \LL z_{i} \right) + \LL x,$$
la suma es directa y es un submódulo de $M$. El cociente $M''/\LL x=\bigoplus_{i=1}^{m}\LL \bar{z}_{i}$ es isomorfo a $M'$. Como $M/\LL x$ y $M'$ son pseudo-isomorfos, tenemos que el morfismo inducido
$$(M/\LL x )/(M''/\LL x) \longrightarrow M'/(M''/\LL x)=0$$
tiene núcleo y conúcleo finitos (ver el \cref{ejer:pseudo-isomorfo-cociente-finito}). De donde tenemos que $M/M''$ es finito y por el \cref{lem:pseudotorsii} tenemos que $M\sim M''$.

Ahora supongamos que $\p=\pi$. El cociente $M/\pi M$ es un $k\llbracket T \rrbracket$ módulo noetheriano, por lo tanto es una suma directa de un módulo finito y un módulo libre de dimensión $d$. Vamos a demostrar el enunciado por inducción sobre la dimensión del módulo libre.

Supongamos que $d=0$, entonces $M/\pi M$ es finito, anulado por una potencia de $T$ digamos $T^{m}$. Tenemos $T^{m}M\subset \pi M$, es decir $T^{em}M=0$, ya que $\pi^{e}M=0$. Es decir, $M$ es anulado por $T^{em}$ y $\pi^{e}$, por lo tanto pseudo-nulo. 

Supongamos que la hipótesis se verifica para dimensiones menores que $d$ y sea entonces $M/\pi M$ una suma directa de un módulo finito y un $k\llbracket T \rrbracket$-módulo libre de dimensión $d$. Luego procedemos con la inducción exactamente como en el caso anterior, tomando un elemento $x\in M$ de anulador mínimo $\pi^{e}$.
\end{proof}

\ejercicios

\begin{ejer}\label{ejer:pseudonofinito} Demuestre que todo $\LL$-módulo $M$ noetheriano es pseudo-isomorfo a $M/F$ donde $F$ es su submódulo máximo finito. 
\end{ejer}

\section{Teorema de estructura}

Para un $\LL$-módulo noetheriano $M$ escribimos
\begin{enumerate}
\item $T(M)$ como su submódulo de torsión. 
\item $F(M)=M/T(M)$ su máximo cociente libre de torsión.
\end{enumerate}

\begin{prop}\label{prop:Mcontained} Sea $M$ un $\LL$-módulo noetheriano libre de
  torsión. Entonces $M$ es pseudo-isomorfo a $\LL^{d}$ para un único $d\geq 0$.
\end{prop}
\begin{proof}
  Sea $\Phi$ el campo de cocientes de $\LL$.
Vamos a proceder por inducción sobre la dimensión de $V=M\otimes_{\LL}\Phi$, que también llamaremos rango de $M$. Si $d=1$, el resultado es una consecuencia de la \cref{prop:contenfin}. Supongamos que el resultado se cumple para módulos de rango menor que $d$. 

Sea $\pi_{1}:V\rightarrow V_{1}$ una proyección de $V$ a un $\Phi$-espacio vectorial $V_1$ de dimensión $1$. Definimos $M_{1}=\pi_{1}(M)$. Sea $e_{1}\in V_{1}$ una base tal que $M_{1}\subset \LL e_{1}$, con índice finito. Entonces existe $n\in\Nuno$ tal que $ne_{1}\in M_{1}$. Sea $m\in M$ una preimagen de $ne_{1}$ bajo $\pi_{1}$. Definimos
\begin{eqnarray*}
j_{1}:V_{1} & \hookrightarrow & V \\
e_{1} & \longmapsto & \dfrac{1}{n} m.
\end{eqnarray*}
Tenemos $\pi_{1}\circ j_{1} =  \id|_{V_{1}}$. Además para todo elemento $m_{1}\in M_{1}$ tenemos que $nj_{1}(m_{1})\in M$. 

Sea $V_{2}=\ker(V\stackrel{\pi_{1}}{\longrightarrow} V_{1})\isom V/j_{1}(V_{1})$. Entonces tenemos los siguientes morfismos
\begin{equation*}
\begin{tikzcd}
      V_{1} \arrow[r,bend left,"j_{1}"] & \arrow[l,bend left,"\pi_{1}"] V \arrow[r,bend left,"\pi_{2}"] & V_{2}, \arrow[l,bend left,"j_{2}"]
\end{tikzcd}
\end{equation*}
con $V\isom V_{1}\oplus V_{2}$. Definamos $M_{2}=\pi_{2}(M)$. 

Consideremos el morfismo
$$(\pi_{1},\pi_{2}):M\hookrightarrow M_{1} \oplus M_{2} \hookrightarrow V. $$
Queremos demostrar que $M$ está contenido con índice finito en $M_{1}\oplus M_{2}$, para esto queremos demostrar que para todo elemento $m_{1}+m_{2}\in M_{1}\oplus M_{2}$, existe un $n$ tal que $n(m_{1}+m_{2})$ está en la imagen de $M$. 

Sea $m'\in M$ una preimagen de $m_{2}$ bajo $\pi_{2}|_{M}$, es decir $\pi_{2}(m')=m_{2}$. Como $m_{1}-\pi_{1}(m')\in M_{1}$, tenemos que $m'' = nj_{1}(m_{1}-\pi_{1}(m'))$ es un elemento de $M$. De donde obtenemos 
$$\pi_{1}(m'')=n(m_{1}-\pi_{1}(m'))\	\	\	\text{ y }\	\	\	\pi_{2}(m'')=0.$$

Sea $m=nm'+m''$. Entonces
\begin{eqnarray*}
\pi_{1}(m)+\pi_{2}(m) &=& n \pi_{1}(m') + \pi_{1}(m'') + n \pi_{2}(m') + \pi_{2}(m'')  \\
&=& n (\pi_{1}(m') + (m_{1}+\pi_{1}(m')) + m_{2}) \\
&=& n (m_{1}+m_{2}). 
\end{eqnarray*}
Es decir, $M$ tiene índice finito en $M_{1}\oplus M_{2}$. Por inducción $M_{1}$ y $M_{2}$ están contenidos en sendos módulos libres con índice finito y por lo tanto su suma directa también. 
\end{proof}

La siguiente proposición, así como los resultados que le siguen, son análogos a los teoremas de estructura de grupos abelianos finitamente generados y más generalmente de módulos sobre anillos principales (e.g. \cite[I.\S 8 y III.\S 7]{MR1878556}).

\begin{prop}\label{prop:sumadirtorlib} Sea $M$ un $\LL$-módulo noetheriano y $T(M)$ su submódulo de torsión. Entonces $M$ es pseudo-isomorfo a la suma directa
$$M\sim T(M) \oplus  L,$$
de $T(M)$ y de un $\LL$-módulo libre $L$ de rango finito. 
\end{prop}
\begin{proof}[según Prop. 5.1.7 \cite{NSW}]
Por la \cref{prop:Mcontained} basta demostrar que $M\sim T(M) \oplus F(M)$. 

Sea $\{\p_{1},\ldots,\p_{h}\}$ el conjunto de ideales primos de altura $1$ de $\LL$ para los cuales $T(M)_{\p_{i}}\neq 0$, es decir $\{\p_{1},\ldots,\p_{h}\}\subset \Sop(T(M))$. Si el conjunto es vacío, i.e. $h=0$, entonces $T(M)$ es pseudo-nulo. Tenemos la composición de pseudo-isomofimos
$$M \rightarrow F(M) \hookrightarrow T(M)\oplus F(M),$$
de donde obtenemos el resultado. 

Ahora supongamos que $h\geq 1$. Entonces $S^{-1}\LL$ es un dominio de ideales principales
(ver el \cref{ejer:PIDSmLL}). Además $T(S^{-1}M)=S^{-1}T(M)$, lo cual junto con el teorema
de estructura de módulos sobre dominios de ideales principales que
$S^{-1}M=S^{-1}T(M)\oplus F(S^{-1}M)$. Es decir, la proyección de $S^{-1}M$ en $S^{-1}T(M)$
admite una sección.

El anillo $\LL$ es noetheriano, por lo tanto $M$ es un $\LL$-módulo finitamente presentable. Tenemos la siguiente identidad \cite[Prop. 2.10]{Eisenbud95CA}
$$\Hom_{S^{-1}\LL}(S^{-1}M,S^{-1}T(M)) = S^{-1}\Hom_{\LL}(M,T(M)).$$
La proyección $\pi:S^{-1}M\rightarrow S^{-1}T(M)$ la podemos escribir como $\pi=\dfrac{f_{0}}{s_{0}}$ para un $s_{0}$ en $S$ y $f_{0}:M\rightarrow T(M)$. Entonces $\dfrac{f_{0}}{s_{0}}\vert_{S^{-1}T(M)}=\id_{S^{-1}T(M)}$. 

Ahora sea $x\in T$, entonces $\dfrac{f_{0}(x)}{s_{0}}=\dfrac{x}{1}$. Por la definición de localización en $S$ tenemos que existe $s_{1}\in S$ tal que $s_{1}(f_{0}(x)-s_{0}x)=0$, equivalentemente $s_{1}f_{0}(x)=s_{0}s_{1}x$. Haciendo $f_{1}=s_{1}f_{0}$ tenemos que 
$$f_{1}\vert_{T(M)}=s_{0}s_{1}\id_{T(M)}.$$

Definiendo 
$$ f=(f_{1},g):M  \longrightarrow  T(M)\oplus F(M) $$
donde $g$ es el mapeo canónico de $M$ a $F(M)$, tenemos el diagrama conmutativo
\begin{equation*}
    \begin{tikzcd}
      0 \arrow[r] & T(M) \arrow[d, "f_{1}\vert_{T(M)}"] \arrow[r] & M \arrow[d, "f"]\arrow[r]   & F(M) \arrow[r]  \arrow[d,equal] & 0 \\
      0 \arrow[r] & T(M) \arrow[r] & M \arrow[r] &F(M) \arrow[r] & 0 \\
     \end{tikzcd}
  \end{equation*}
De donde obtenemos que $\ker(f_{1}\vert_{T(M)})=\ker(f)$ y $\coker(f_{1}\vert_{T(M)})=\coker(f)$. Finalmente, es fácil ver que la multiplicación por $s_{0}s_{1}$ en $T(M)$ es un pseudo-isomorfismo. 
\end{proof}

\begin{defi}\label{def:modelemental}
  Decimos que un $\LL$-módulo $E$ es \define[módulo elemental]{elemental} si es de la forma
$$E=\LL^{r} \oplus \left( \bigoplus_{i=1}^{k}\LL/\p_{i}^{e_i}\LL \right),$$
donde los $\p_{i}$ ideales primos (posiblemente repetidos) de $\LL$ de altura $1$ y
$r,e_i\ge0$ son enteros.
\end{defi}

El siguiente teorema es uno de los pilares en teoría de Iwasawa. A pesar de que $\LL$ no sea un anillo principal, es muy interesante que tengamos un teorema de estructura salvo pseudo-isomorfismos. Como veremos más adelante, este resultado y unas técnicas de descenso nos permitirán demostrar el teorema de Iwasawa (\cref{thm:iwasawa-ThmdIwa}) sobre grupos de clases.
  
\begin{thm}[\importante{Teorema de Estructura}]\label{thm:estructura}
  Todo $\LL$-módulo noetheriano es pseudo-isomorfo a un único $\LL$-\define{módulo elemental} $E$
$$M\sim E = \LL^{\rho} \oplus \left( \bigoplus_{i=1}^{m}\LL/\pi^{\mu_{i}}\LL\right)\oplus \left( \bigoplus_{j=1}^{l}\LL/P_{j}\LL\right)$$
donde los $P_{j}$ son polinomios distinguidos ordenados por divisibilidad. 
\end{thm}
\begin{proof}
Sea $M$ un $\LL$-módulo noetheriano. $M$ es pseudo-isomorfo a la suma directa de un módulo de torsión $T(M)$ sin submódulo finito y de un $\LL$-módulo libre $L$ de rango finito (\cref{prop:sumadirtorlib,ejer:pseudonofinito}). Teniendo anulador principal (\cref{prop:Annppal}), el submódulo $T(M)$ es pseudo-isomorfo a una suma directa de submódulos $T_{\p_{i}}(M)$ con anuladores primarios $\p_{i}^{e_{i}}$ (\cref{lem:pairwiseann}). A su vez, los submódulos $T_{\p_{i}}(M)$ son pseudo-isomorfos a una suma directa de la forma $\bigoplus_{j=1}^{k_{i}} \LL/\p_{i}^{e_{i,j}}$ con $e_{i,1}\geq e_{i,2} \geq \ldots \geq e_{i,k_{i}} >0$ (\cref{prop:anuladorprimariosumapei}). 

Finalmente, sea $P_{j}=\prod_{\p_{i}\neq\pi}\p_{i}^{e_{i,j}}$, con la convención $e_{i,j}=0$ si $j>k_{i}$. Definiendo $l$ como el máximo de los $k_{i}$, obtenemos que $P_{l}\vert P_{l-1}\vert \cdots |P_{1}$. El resultado sigue aplicando el \cref{lem:pseudotorsii}. 
\end{proof}

A continuación definimos los \define{invariantes $\mu$ y $\lambda$} asociados al $\LL$-submódulo de torsión de un módulo noetheriano $M$.

\begin{defi}\label{defi:invariantes}
  Usamos la situación y la notación del \cref{thm:estructura}.
  \begin{enumerate}
\item Escribimos $\mu=\sum_{i=1}^{m}\mu_{i}$ y $P=\prod_{j=1}^{l}P_{j}$ y definimos el \emph{polinomio característico}\index[def]{polinomio!característico} del submódulo de $\LL$-torsión de $M$ como $P$. 
\item Definimos $\lambda=\deg P =\sum_{j=1}^{l}\deg P_{j}$.
\end{enumerate}
\end{defi}

El nombre \enquote{polinomio característico} es justificado por el
\cref{ejer:polchar}. 

\ejercicios

\begin{ejer}\label{ejer:PIDSmLL} Sea $\{\p_{1},\ldots,\p_{h}\}\neq \emptyset$ un conjunto de ideales primos de altura $1$ de $\LL$. Demuestre que:
\begin{itemize}
\item[(a)] $S_{i}=\LL\setminus \p_{i}$ es un conjunto multiplicativo. Esto es, $1\in S_{i}$ y para todo $x,y \in S_{i}$, entonces $xy\in S_{i}$.
\item[(b)] la intersección $S=\bigcap S_{i}$ también es un conjunto multiplicativo igual a $\LL\setminus\bigcup \p_{i}$.
\item[(c)] $S^{-1}\LL$ es un anillo semi-local, es decir tiene un número finito de ideales maximales (\textit{Consejo}: utilice el lema de evitación de ideales primos).
\item[(d)] $S^{-1}\LL$ es un anillo de Dedekind, es decir es integralmente cerrado, noetheriano y sus ideales primos son maximales. 
\item[(e)] $S^{-1}\LL$ es un dominio de ideales principales.
\end{itemize}
\end{ejer}

\begin{ejer}\label{ejer:polchar}
  Escribimos $L$ como el campo de fracciones de $\O$. Sea $M$ un $\LL$-módulo de
  torsión. Definimos $V=L\tensor_\O M$, que es un espacio vectorial de dimensión finita
  sobre $L$. La aplicación $V\rightarrow V$ definida como multiplicación por $T$ es
  $L$-lineal. Demuestre que el polinomio característico de esta aplicación es justamente el
  polinomio característico de $M$.
\end{ejer}

\begin{ejer}\label{ejer:relacion-simetrica}
  Demuestre que la relación de pseudo-isomorfismo es simétrica en los módulos de torsión
  (así que es una relación de equivalencia). Es decir, demuestre que si $M$ y $N$ son
  $\LL$-módulos noetherianos de torsión entonces \[ M\sim N \iff N \sim M. \] Use el teorema
  de estructura y el \cref{lem:pairwiseann} para esto.
\end{ejer}

\section{Resultados adicionales sobre $\LL$-módulos}
\label{sec:estructura-consecuencias}

Recordemos que $\mathfrak M=(\pi,T)$ denota el ideal máximo de $\LL$, donde $\pi\in\O$ es un uniformizante.

\begin{lem}\label{lem:XcompHausdMadica} Sea $X$ un $\LL$-módulo compacto. Entonces
$$\bigcap_{r\in\Nuno}\mathfrak{M}^{r}X=0.$$
\end{lem}
\begin{proof} Sea $U$ un vecindario de $0$ en $X$, recordemos que un $\LL$-módulo compacto
  tiene un sistema fundamental de vecindarios de cero formado de submódulos abiertos
  \cite[Cor. 5.2.5]{NSW}. Sea $z\in X$, como $\LL$ es completo las sucesiones
  $(\pi^{r})_{r\in\Nuno}$ y $(T^{r})_{r\in\Nuno}$ convergen a $0$ en $\LL$. Por lo tanto existe
  $s_{0}$ tal que $\pi^{s}z$, $T^{s}z$ son elementos de $U$ para $s\geq s_{0}$.

Sea $r\geq 2s_{0}$, entonces para $i+j=r$
$$\pi^{i}T^{j}=\left\lbrace \begin{array}{ll}
(\pi^{i-s_{0}}T^{j}) \pi^{s_{0}}, & \text{ si } i \geq s_{0}; \\
(\pi^{i}T^{j-s_{0}})T^{s_{0}}, & \text{ si } j \geq s_{0}. 
\end{array}\right.$$
Es decir $\pi^{i}T^{j}z\in U$ para todo $i+j=r$.

Las aplicaciones $x\mapsto \pi^{s_{0}}x$ y $x\mapsto T^{s_{0}}x$ son continuas, las preimágenes $V$ y $W$ de $U$, respectivamente, son abiertas. La intersección $U_{z}:=V\cap W$ es un conjunto abierto no vacío. Tenemos entonces que $\pi^{i}T^{j}U_{z}\subset U$ para $i+j=r$, además como $U$ es un $\LL$-submódulo para cualquier combinación lineal $\alpha=\sum_{i+j=r}\lambda_{i,j}\pi^{i}T^{j}$ con $\lambda_{i,j}\in\LL$ y $u\in U_{z}$ tenemos $\alpha u \in U$. 

Hemos demostrado que para todo $z\in X$ existe un vecindario $U_{z}$ y un $r$ suficientemente grande tal que $\mathfrak{M}^{r}U_{z}\subseteq U$. Como $X$ es compacto, existe un número finito $U_{z_{1}},\ldots,U_{z_{k}}$ que cubren $X$. Sea $r=\max\{r_{1},\ldots,r_{k}\}$ tal que $\mathfrak{M}^{r_{i}}U_{z_{i}}\subseteq U$, entonces 
$$\mathfrak{M}^{r}X\subset U,$$ 
lo que demuestra el resultado al tomar la intersección de todos los vecindarios $U$ de $0$.
\end{proof}

\begin{lem}[\importante{Lema de Nakayama Topológico}]\label{lem:nakayama} Sea $X$ un $\LL$-módulo compacto. $X$ es noetheriano si y solamente si $X/\mathfrak{M}X$ es finito. 
\end{lem}
\begin{proof} 
Si $X$ es noetheriano la implicación es clara. Supongamos entonces que $X/\mathfrak{M}X$ es finito. En particular, existen $x_{1},x_{2},\ldots,x_{d}\in X$ que generan $X/\mathfrak{M}X$ como grupo abeliano. Sea 
$$Y=x_{1}\LL+\cdots+x_{d}\LL\subseteq X.$$ 
El $\LL$-módulo $Y$ es cerrado por ser la imagen de $\LL^{d}\rightarrow Y$. La proyección $X\rightarrow X/Y$ es continua, por lo tanto $X/Y$ es compacto con la topología $\mathfrak{M}$-ádica inducida. Tenemos $Y+\mathfrak{M}X=X$, es decir $\mathfrak{M}^{r}(X/Y)=X/Y$ para todo $r\geq 0$. Por el \cref{lem:XcompHausdMadica}  tenemos $X/Y=0$, es decir $Y=X$ y por lo tanto $x_{1},\ldots,x_{d}$ generan $X$.
\end{proof}

El siguiente corolario se deduce inmediatamente del lema anterior.

\begin{cor}\label{cor:cor-de-nakayama} Sea $X$ un $\LL$-módulo compacto, entonces
$$X/\mathfrak{M}X = 0 \	\	\	\iff \	\	\	X=0.$$
\end{cor}

En el siguiente resultado aparecen de nuevo los polinomios $\w_r$ de la \cref{defi:omega-r}.

\begin{prop}\label{prop:Xnoettorsiirgfinsiiacot} Sea $X$ un $\LL$-módulo noetheriano. Entonces
\begin{eqnarray*}
X\text{ es de torsión } &\iff & \rg_{\O} X <\infty \\
&\iff & \rg_{\O} X/\w_{r} X \text{ está acotado para todo }	r\geq 0.
\end{eqnarray*}
\end{prop}
\begin{proof} Primero notemos que el enunciado es invariante bajo pseudo-isomorfismos. Sea
  $f\colon X\rightarrow E$ un pseudo-isomorfismo de $X$ a un módulo elemental $E$ (ver la
  \cref{def:modelemental} y el \cref{thm:estructura}). Tenemos que
  $X \otimes_{\LL}\Phi \isom E \otimes_{\LL} \Phi$, donde $\Phi=\Frac(\LL)$ es el campo de
  cocientes de $\LL$. Es decir $X$ es de torsión
  si y solamente si $E$ es de torsión. Igualmente del
  \cref{ejer:pseudo-isomorfo-cociente-finito} deducimos que
  $\rg_{\O} X <\infty \Leftrightarrow \rg_{\O} E <\infty$. Igualmente, $\rg_{\O} X/\w_{r} X$
  está acotado para todo $r\geq 0$ si y solamente si $\rg_{\O} E/\w_{r} E$ está acotado para
  todo $r\geq 0$. Por lo tanto podemos remplazar $X$ por $E$ y bastará analizar el enunciado
  para los tipos de factores de un módulo elemental, es decir $\LL$, $\LL/\pi^{e}$ o
  $\LL/P^{e}$ donde $P$ es un polinomio distinguido.

  Veamos que $\LL\otimes_{\O}L$ es un $L$-espacio vectorial de dimensión infinita (donde $L$
  es el campo de cocientes de $\O$). Además $\LL/\pi^{e}$ es isomorfo a
  $\O/\pi^{e}\llbracket T \rrbracket $ y $\LL/P^{e}$ es isomorfo a $\O[T]/P^{e}$. Cuando
  tensamos con $L$ los diferentes tipos de factores obtenemos $L$-espacio vectoriales. El
  primero infinito, el segundo trivial y el tercero finito. Es decir, $E$ es de torsión si y
  solamente si $\rg_{\O} E <\infty$.

  Aplicando el funtor exacto $-\otimes_{\O}L$ a
  $$0\longrightarrow \w_{r}E \longrightarrow E \rightarrow E/\w_{r}E \longrightarrow 0, $$ 
  deducimos que $\rg_{\O} E <\infty$ implica que $ \rg_{\O} E/\w_{r} E $ está acotado para
  todo $r$.

  Por otro lado, tomemos $E$ no de torsión, entonces $\rho\geq 1$ como en el
  \cref{thm:estructura}. Si $\LL$ es un factor de $E$ entonces por el lema de división
  (\cref{lema:division}, usando la notación de allá)
  $$\LL/\w_{r}\isom {\O}_{p^{r}-1}[T],$$ 
  es decir $\LL/\w_{r}$ es isomorfo al $\O$-módulo de los polinomios de grado a lo más
  $p^{r}-1$. Es claro que el rango de ${\O}_{p^{r}-1}[T]$ sobre $\O$ es proporcional a
  $r$, y por lo tanto no acotado.
\end{proof}

\section{Ideales característicos}
\label{sec:ideales-caracteristicos}

La teoría de estructura que desarrollamos anteriormente permite definir un invariante
importante de $\LL$-módulos noetherianos: el ideal característico. Este invariante aparecerá en la Conjetura Principal de Iwasawa.

Sea $\O$ el anillo de enteros en una extensión finita de $\Qp$, $\Gamma$ un grupo profinito
isomorfo a $\Zp$ y escribimos $\LL=\O\llbracket\Gamma\rrbracket$ para su álgebra
de Iwasawa. Después de fijar un generador topológico $\gamma$ de $\Gamma$, el
\cref{thm:EquivPowPro} nos permite identificar $\LL$ con el anillo de series de
potencias $\O\llbracket T\rrbracket$.

Si $M$ es un $\LL$-módulo noetheriano, el teorema de estructura \ref{thm:estructura}
dice que $M$ es pseudo-isomorfo a un único $\LL$-módulo elemental de la forma
$$M\sim E = \LL^{\rho} \oplus \left( \bigoplus_{i=1}^{m}\LL/\pi^{\mu_{i}}\LL\right)\oplus \left( \bigoplus_{j=1}^{l}\LL/P_{j}\LL\right)$$
y en la \cref{defi:invariantes} definimos el invariante $\mu\in\Ncero$ y el polinomio
característico $P=\prod_{j=1}^{l}P_{j}$ de $M$.  El elemento $\pi^\mu P\in\LL$ no está bien
definido porque $\pi$ solo está determinado módulo una unidad. Pero esto significa que la
siguiente definición tiene sentido (aunque hay que verificar que no depende del $\gamma$ que
elegimos).

\begin{defi}
  Para cada $\LL$-módulo noetheriano $M$ definimos su \emph{ideal característico}\index[def]{ideal!característico}
  como
  \[ \charideal_{\LL}(M):=(\pi^{\mu}P) \]
  con $\mu$ y $P$ como arriba.
\end{defi}

A continuación, demostramos algunos resultados auxiliares que serán útiles más tarde.

\begin{lem}\label{lem:m-cociente-finito-coprimo}
  Sea $M$ un $\LL$-módulo noetheriano de torsión y sea $(F)=\charideal_\LL(M)$ su
  ideal característico. Además sea $\nu\in\LL$, entonces $M/\nu M$ es finito si y solo si
  $\nu$ y $F$ son coprimos.
\end{lem}
\begin{proof}
  Fijemos un pseudo-isomorfo
  \begin{equation*}
    M\rightarrow E := \left( \bigoplus_{i=1}^{m}\LL/(\pi^{\mu_{i}})\right)\oplus \left( \bigoplus_{j=1}^{l}\LL/(P_{j})\right)
  \end{equation*}
  a un $\LL$-módulo elemental $E$.
  Entonces $F=\pi^\mu\prod_j P_j$ con $\mu=\sum_i\mu_i$.

  Primero asumamos que $\nu$ y $F$ son coprimos. Entonces $\nu$ es coprimo a $\pi$ y a todos
  los $P_j$. Usando el \cref{ejer:pseudo-isomorfo-cociente-finito} obtenemos un
  pseudo-isomorfo
  \begin{equation*}
    M/\nu M\rightarrow \left( \bigoplus_{i=1}^{m}\LL/(\pi^{\mu_{i}},\nu)\right)\oplus \left( \bigoplus_{j=1}^{l}\LL/(P_{j},\nu)\right)
  \end{equation*}
  y cada uno de los sumandos a la derecha es finito según el \cref{cor:fg-indice-finito}. Por
  eso $M/\nu M$ también tiene que ser finito.

  Por otro lado supongamos que $\nu$ y $F$ no son coprimos, es decir existe $d\in\LL$ que no
  es una unidad y divide a $\nu$ y a $F$. Sabemos que $\LL$ es un dominio de factorización
  única, que $\pi$ es coprimo a cada uno de los $P_j$ y además que los $P_j$ son ordenados
  por divisibilidad. Por eso podemos asumir que $d$ divide a $\pi$ o a $P_1$ (y a $\nu$, por
  supuesto). En el primer caso podemos asumir que $d=\pi$ y en el segundo que $d=P$ es un
  polinomio distinguido no trivial. Entonces tenemos sobreyecciones canónicas
  \begin{equation*}
    E/\nu E\rightarrow E/d E\rightarrow \LL/\pi\LL\quad\text o\quad E/\nu E\rightarrow E/d
    E\rightarrow \LL/ P\LL.
  \end{equation*}
  Como $\LL/\pi\LL$ y $\LL/ P\LL$ son infinitos, tenemos que $E/\nu E$ y $M/\nu M$ también son infinitos.
\end{proof}

\begin{lem}\label{lem:elemental-rk-cero}
  Sea $E$ un $\LL$-módulo elemental y $\nu\in\LL$. Escribimos $E[\nu]$ como el núcleo de la
  multiplicación con $\nu$ en $E$. Si $\rk_\O E[\nu]=0$ y $E/\nu E$ es finito entonces
  $E[\nu]=0$.
\end{lem}
\begin{proof}
  Sin pérdida de generalidad podemos asumir que $E$ es de la forma (i) $\LL$, (ii)
  $\LL/\pi^\mu\LL$ con $\mu\in\Nuno$ ó (iii) $\LL/P\LL$ con $P$ un polinomio distinguido.
  
  En el primer caso $E=\LL$ tenemos $E[\nu]=0$ porque $\LL$ no tiene divisores de cero. En
  el caso $E=\LL/\pi^\mu\LL$, si $E[\nu]\neq0$ entonces la clase de $\nu$ tiene que ser un
  divisor de cero en $\LL/\pi^\mu\LL$, es decir es un divisor de $\pi^\mu$ en $\LL$. Pero en
  este caso $E/\nu E$ es infinito.

  Sea entonces $E=\LL/P\LL$. Si $E[\nu]\neq0$ entonces
  existe un $h\in\LL$ con $\nu h=P$, y sin pérdida de generalidad $h,\nu\in\O[T]$ y
  $\deg\nu,\deg h>0$. Pero entonces $h\O[T]/P\O[T]$ es un submódulo de $E[\nu]$ cuyo rango
  es mayor que $0$, contradiciendo a $\rk_\O E[\nu]=0$.
\end{proof}

\begin{cor}\label{cor:m-nu-finito}
  Si $M$ es un $\LL$-módulo noetheriano de torsión y $\nu\in\LL$ es coprimo a
  $\charideal_\LL M$ entonces $M[\nu]$ es finito.
\end{cor}
\begin{proof}
  Usando el teorema de estructura es suficiente demostrar esto para $\LL$-módulos
  elementales, por eso remplazamos $M$ con un módulo elemental $E$ como en la demostración
  de la \cref{prop:Xnoettorsiirgfinsiiacot}. Según el \cref{lem:m-cociente-finito-coprimo}
  $E/\nu E$ es finito, pues por el \cref{lem:elemental-rk-cero} es suficiente demostrar que
  $\rk_\O E[\nu]=0$. Esto resulta de la sucesión exacta
  \begin{equation*}
    0\rightarrow E[\nu]\rightarrow E\labeledarrow{\nu} E\rightarrow E/\nu E\rightarrow 0,
  \end{equation*}
  usando que el $\O$-rango de $E$ es finito porque $E$ es de torsión
  (\cref{prop:Xnoettorsiirgfinsiiacot}).
\end{proof}

Recordemos otra vez que el teorema de estructura puede ser visto como un análogo del teorema
de estructura para grupos abelianos finitamente generados. En esta situación análoga, el
\enquote{ideal característico} de un grupo abeliano finitamente generado sería el ideal en
$\Z$ generado por el orden de la parte finita del grupo. Por eso, heurísticamente podemos
imaginar el ideal característico como análogo del \enquote{orden de un grupo} y de hecho
resultará un invariante muy importante de un módulo.

Porque en las aplicaciones usaremos álgebras de Iwasawa para grupos más grandes como en la
\cref{sec:mas-grandes}, generalicemos la definición del ideal característico a esa
situación. Esto es puramente formal y nada profundo, aunque un poco técnico.

Sea $\Delta$ un grupo finito cuyo orden $\#\Delta$ es invertible en $\O$ y
tal que $\O$ contiene las raíces $\#\Delta$-ésimas de la unidad. El ejemplo más importante
será aquel en que $\Delta=\Gal(\Q(\mu_p)/\Q)\isom\F_p^\times$. Sea $G=\Delta\times\Gamma$ con
$\Gamma$ como antes. El álgebra de
Iwasawa $\LL(G)=\O\llbracket G\rrbracket$ para el grupo $G$ entonces se descompone como
\[ \LL(G)=\bigoplus_{\chi}e_\chi\LL(G) \]
según el \cref{lem:descomposicion-lambda},
con cada $e_\chi\LL(G)$ isomorfo a $\LL=\LL(\Gamma)$, donde $\chi$ recorre los caracteres
$\Delta\rightarrow\O^\times$. 

Fijemos un subconjunto $I\subseteq\Hom(\Delta,\O^\times)$ de caracteres y escribimos
\[ \LL^I(G):=\bigoplus_{\chi\in I}e_\chi\LL(G). \]
Para cada $\LL^I(G)$-módulo $M$ y $\chi\in I$, $e_\chi M$ es un $e_\chi\LL(G)$-módulo, y $M$
tiene las propiedades \enquote{noetheriano} o \enquote{de torsión} como $\LL^I(G)$-módulo si
y solo si cada $e_\chi M$ para $\chi\in I$ las tiene como $e_\chi\LL(G)$-módulo.

Desde ahora asumamos que $M$ es un $\LL^I(G)$-módulo noetheriano y de torsión. Entonces para
cada $i\in I$ el teorema de estructura nos da una sucesión exacta
\[ 0\rightarrow K_i \rightarrow \bigoplus_{j=1}^s\LL(\Gamma)/f_{ij}\rightarrow e_iM\rightarrow
  C_i \rightarrow 0 \]
con $K_i$ y $C_i$ finitos, los $f_{ij}$ siendo de la forma $\pi^{\mu_j}$ o un polinomio
distinguido, y donde $s$ sin pérdida de generalidad \emph{no} depende de $i$ (eso lo podemos
lograr definiendo los $f_{ij}$ superfluos como $1$). Si definimos para $j=1,\dotsc,s$
\[ g_j=(f_{ij})_{i\in I}\in\LL^I(G) \]
entonces obtenemos una sucesión exacta
\[ 0\rightarrow K \rightarrow \bigoplus_{j=1}^s\LL^I(G)/g_j\rightarrow M\rightarrow
  C \rightarrow0 \]
con $K$ y $C$ finitos.

\begin{defi}\label{defi:ideal-caracteristico-i}
  \begin{enumerate}
  \item Para cada $\LL^I(G)$-módulo noetheriano y de torsión $M$, su \emph{ideal
      característico} es \[ \charideal_{\LL^I(G)}(M):=(g_1\dotsm g_s)\subseteq\LL^I(G). \]
  \item Llamamos un $\LL^I(G)$-módulo $E$ \emph{elemental}\index[def]{módulo elemental@módulo elemental} si $e_iE$ es elemental para
    cada $i\in I$, así que el módulo $\bigoplus_j\LL^I(G)/g_j$ de arriba es un tal módulo.
  \end{enumerate}
\end{defi}

\ejercicios

\begin{ejer}
  Verifique que el ideal característico de un $\LL(\Gamma)$-módulo noetheriano no depende
  del isomorfismo $\LL(\Gamma)\isom\O\llbracket T\rrbracket$ inducido por la elección del
  generador topológico $\gamma$ de $\Gamma$.
\end{ejer}

\begin{ejer}
  Sea $M$ un $\LL^I(G)$-módulo. Demuestre que $M$ es
  \begin{enumerate}
  \item noetheriano;
  \item de torsión
  \end{enumerate}
  como $\LL^I(G)$-módulo si y solo si cada $e_\chi M$ para $\chi\in I$ lo es como
  $e_\chi\LL(G)$-módulo.
\end{ejer}

\begin{ejer}
  Demuestre que el ideal característico es multiplicativo en sucesiones exactas, es decir
  para cada sucesión exacta
  \[ 0\rightarrow N\rightarrow M\rightarrow Q \rightarrow 0 \]
  de $\LL^I(G)$-módulos tenemos
  \[ \charideal_{\LL^I(G)}(M)=\charideal_{\LL^I(G)}(N)\charideal_{\LL^I(G)}(Q). \]
\end{ejer}

\section{Adjuntos de Iwasawa}
\label{sec:adjuntos}

Como en las secciones anteriores sea $\LL=\O\llbracket T\rrbracket$ con $\O$ el anillo de
enteros en una extensión finita $L$ de $\Qp$.

En esta sección seguimos \cite[§3.5]{SharifiIT} para explicar los \emph{adjuntos de
  Iwasawa}\esindex[def]{adjunto de Iwasawa!}, que es una manera de asociar un módulo \enquote{adjunto} $\alpha(M)$ a un
$\LL$-módulo $M$ noetheriano de torsión que de hecho es pseudo-isomorfo a $M$,
introducidos por Iwasawa en \cite{MR0349627}. En este texto vamos a usar esta teoría
exclusivamente para demostrar la equivalencia de dos versiones de la Conjetura Principal de
Iwasawa en la \cref{sec:proemio-kummer-adjuntos}. El lector puede saltarse esta
sección si lo desea y volver a consultarla si se interesa en la demostración de dicha equivalencia.

A continuación, como preparación citamos una versión de la teoría de dualidad de Pontryagin, que vamos a necesitar.

\begin{defi}\label{defi:pontryagin}
  Sea $\mathcal A$ la categoría de
  $\O$-módulos topológicos Hausdorff localmente compactos cuya topología está definida por una
  base de vecindades de $0$ que consiste en $\O$-submódulos.
  Además sea $\mathcal A_{\mathrm{comp}}$ la subcategoría llena de $\mathcal A$ de módulos
  compactos y $\mathcal A_{\mathrm{disc}}$ la de módulos discretos.

  Para cada módulo $M\in\mathcal A$ sea
  \[ M^\vee := \Hom_\O(M,L/\O) \] con la topología compacto-abierta (y donde \enquote{$\Hom_\O$}
  significa homomorfismos de $\O$-módulos topológicos).
  Este $M^\vee$ se llama el \define{dual de Pontryagin} de $M$.
\end{defi}

En las aplicaciones usaremos
$\LL$-módulos noetherianos, que son elementos de la categoría $\mathcal A$.

\begin{thm}[Pontryagin]\label{thm:pontryagin}
  Para $M\in\mathcal A$, $M^\vee$ es nuevamente un elemento de $\mathcal A$.
  La asociación
  \[ (\cdot)^\vee\colon\mathcal A\rightarrow\mathcal A,\quad M\mapsto M^\vee \]
  es un funtor contravariante aditivo que es una
  antiequivalencia de $\mathcal A$ a si mismo, y es un cuasi-inverso a si mismo. Es decir,
  para cada módulo $M\in\mathcal A$ hay un isomorfismo
  $M\rightarrow M^{\vee\vee}$ natural en $M$. Este isomorfismo está dado por el mapeo canónico
  \[ M\rightarrow M^{\vee\vee},\quad m\mapsto(f\mapsto f(m)). \]

  Además, $M^\vee$ es discreto si y solo si $M$ es compacto, y viceversa. Es decir, el
  funtor $(\cdot)^\vee$ induce antiequivalencias entre $\mathcal A_{\mathrm{comp}}$ y $\mathcal A_{\mathrm{disc}}$.
\end{thm}
\begin{proof}
  \cite[Prop.\ 5.4]{MR3629658}
\end{proof}

\begin{remark}\label{rem:pontryagin-exacto}
  ¡El funtor $(\cdot)^\vee\colon\mathcal A\rightarrow\mathcal A$ en general \emph{no}
  conserva sucesiones exactas, aunque es una equivalencia de categorías!  Recuerde que un
  funtor en general se llama \define[funtor exacto]{exacto} si conserva límites y colímites
  finitos (entre categorías que los admiten). Para funtores entre categorías abelianas esto
  es equivalente a pedir que conserve sucesiones exactas. Obviamente cada equivalencia de
  categorías (como $(\cdot)^\vee$) es un funtor exacto, pero la categoría $\mathcal A$ no
  es abeliana.

  Sin embargo, si en una sucesión exacta
  \[ 0\rightarrow A\rightarrow B\rightarrow C\rightarrow 0 \] todos los módulos $A$, $B$,
  $C$ son compactos o todos son discretos, entonces la sucesión dual
  \[ 0\rightarrow C^\vee\rightarrow B^\vee\rightarrow A^\vee\rightarrow0 \] es también
  exacta. Esto es cierto porque las categorías $\mathcal A_{\mathrm{comp}}$ y
  $\mathcal A_{\mathrm{disc}}$ son abelianas; también puede ser visto como consecuencia de
  \cite[Rem.\ 5.5]{MR3629658}.\footnote{El teorema de categorías de Baire implica que
    morfismos entre módulos compactos siempre son estrictos (con una demostración similar a
    la del teorema de la función abierta).}
\end{remark}

\begin{defi}
  Usamos los polinomios \[ \omega_r=(T+1)^{p^r}-1 \in\LL\] y $\Phi_r\in\LL$ como en la
  \cref{defi:omega-r}, para cada $r\in\Nuno$. Recuerde que $\omega_r=\Phi_r\omega_{r-1}$ y los
  polinomios $\Phi_r$ son irreducibles (\cref{ejer:omega-r-dist-irr}). Ponemos para cada
  $r\ge m$
  \begin{equation}
    \label{eqn:nu-r-m}
    \nu_{r,m}:=\frac{\omega_r}{\omega_m} = \Phi_r\dotsm\Phi_{m+1}.
  \end{equation}
\end{defi}

\begin{lem}
  Para cada $\LL$-módulo noetheriano $M$ existe un $m\in\Nuno$ tal que para cada
  $r\ge m$ los $\nu_{r,m}$ son coprimos a $\charideal_\LL(M)$.
\end{lem}
\begin{proof}
  Esto es una consecuencia del teorema de estructura y \eqref{eqn:nu-r-m}, porque $\LL$ es
  un dominio de factorización única y los $\Phi_r$ son irreducibles.
\end{proof}

\begin{lem}\label{lem:alpha-no-depende-de-m}
  Sea $M$ un $\LL$-módulo $M$ noetheriano de torsión y sea $m\in\Nuno$ suficientemente
  grande tal que para $r\ge m$ los $\nu_{r,m}$ son coprimos a
  $\charideal_\LL(M)$. Consideramos el módulo
  \[ A_m := \varinjlim_{r\ge m} M/\nu_{r,m} M \] donde el mapeo
  $M/\nu_{r,m}M\rightarrow M/\nu_{r+1,m}M$ con respeto al
  cual tomamos el límite es $x\mapsto\frac{\omega_{r+1}}{\omega_r}x$.
  Entonces $A_m$ no depende de $m$.
\end{lem}
\begin{proof}
  Estudiemos el diagrama
  \begin{equation*}
    \begin{tikzcd}
      0 \arrow[r] & N_r \arrow[d, "\Phi_{r+1}"] \arrow[r] & \frac{M}{\Phi_{m+1}\dotsm\Phi_rM} \arrow[d, "\Phi_{r+1}"] \arrow[r, "\Phi_{m}"] &
      \frac{M}{\Phi_{m}\dotsm\Phi_rM} \arrow[d, "\Phi_{r+1}"] \arrow[r] &
      \frac{M}{\Phi_mM } \arrow[d, "\Phi_{r+1}"] \arrow[r] & 0 \\
      0 \arrow[r] & N_{r+1} \arrow[d, "\Phi_{r+2}"] \arrow[r] & \frac{M}{\Phi_{m+1}\dotsm\Phi_{r+1}M} \arrow[d, "\Phi_{r+2}"] \arrow[r, "\Phi_{m}"] &
      \frac{M}{\Phi_{m}\dotsm\Phi_{r+1}M} \arrow[d, "\Phi_{r+2}"] \arrow[r] &
      \frac{M}{\Phi_mM } \arrow[d, "\Phi_{r+2}"] \arrow[r] & 0 \\
      0 \arrow[r] & N_{r+2} \arrow[d, "\Phi_{r+3}"] \arrow[r] & \frac{M}{\Phi_{m+1}\dotsm\Phi_{r+2}M} \arrow[d, "\Phi_{r+3}"] \arrow[r, "\Phi_{m}"] &
      \frac{M}{\Phi_{m}\dotsm\Phi_{r+2}M} \arrow[d, "\Phi_{r+3}"] \arrow[r] &
      \frac{M}{\Phi_mM } \arrow[d, "\Phi_{r+3}"] \arrow[r] & 0 \\
      & \vdots & \vdots & \vdots & \vdots
    \end{tikzcd}
  \end{equation*}
  Aquí los mapeos con flechas etiquetadas son multiplicación por este elemento y los mapeos
  con flechas no etiquetadas son inclusiones o proyecciones canónicas. Los
  $N_r,N_{r+1},\dotsc$ en la primera columna están definidos como los núcleos de la
  multiplicación por $\Phi_m$, de manera que las sucesiones horizontales son exactas.

  El colímite de la segunda columna es $A_m$ y el colímite de la tercera columna es
  $A_{m-1}$ (supongamos que tan pronto como $r\ge m-1$ los $\nu_{r,m}$ son coprimos a
  $\charideal_\LL(M)$). Tomar colímites de sistemas de módulos es un funtor exacto, por eso
  es suficiente demostrar que los colímites de las columnas exteriores son cero.

  Estudiamos la primera columna. Sea $i\in\Ncero$ y $x\in M$ un representante de un elemento
  de $N_{r+i}$. Entonces existe $d\in M$ tal que $\Phi_mx=\Phi_m\dotsc\Phi_{r+i}d$, es decir
  $x-\Phi_{m+1}\dotsm\Phi_{r+i}d\in M[\Phi_m]$ (donde $M[\Phi_m]$ denota el núcleo de la
  multiplicación por $\Phi_m$ en $M$). Esto demuestra que la aplicación canónica
  \begin{equation*}
    M[\Phi_m]\rightarrow N_{r+i}\subseteq M/\Phi_{m+1}\dotsm\Phi_{r+i}M
  \end{equation*}
  es sobreyectiva para cada $i$.

  Por el \cref{cor:m-nu-finito} $M[\Phi_m]$ es finito. Notemos que también todos los módulos
  que aparecen en el diagrama son finitos (esto sigue del
  \cref{lem:m-cociente-finito-coprimo}), en particular $M/\Phi_mM$.  Entonces para que los
  colímites de las columnas exteriores sean cero, es suficiente ver que cada elemento de un
  $\LL$-módulo finito está anulado si lo multiplicamos con $\Phi_s\Phi_{s+1}\dotsc\Phi_t$
  para $t\ge s\gg r$ suficientemente grande. Esto lo hemos demostrado en el
  \cref{lem:phi-s-anula-a-modulo-finito}.
\end{proof}

\begin{defi}\label{defi:mvee-alpha}
  Sea $M$ un $\LL$-módulo $M$ noetheriano de torsión.
  \begin{enumerate}
  \item\label{defi:ll-mod-dual} Hacemos $M^\vee=\Hom_{\O}(M,L/\O)$ un $\LL$-módulo con la
    acción
    \[ \lambda f(m)=f(\lambda m) \quad\text{para }\lambda\in\LL,\ f\in M^\vee,\ m\in M. \]
  \item\label{defi:alpha} Definimos
    \[ \alpha(M) = \varprojlim_{r\ge m} (M/\nu_{r,m} M)^\vee = (\varinjlim_{r\ge m}
       M/\nu_{r,m} M)^\vee \]
     con $m$ como en el \cref{lem:alpha-no-depende-de-m}. Entonces $\alpha(M)$ es un $\LL$-módulo que se
     llama el \define{adjunto de Iwasawa} de $M$.
  \end{enumerate} 
\end{defi}

En la \cref{defi:mvee-alpha}~\ref{defi:alpha} la igualdad entre el límite y el colímite (que
más bien es un isomorfismo canónico) es cierta porque el funtor $(\cdot)^\vee$ es exacto (\cref{rem:pontryagin-exacto}).

De manera más abstracta se puede definir el adjunto de Iwasawa de un $\LL$-módulo noetheriano de
torsión como $\alpha(M):=\operatorname{Ext}_\LL^1(M,\LL)$, que puede ser útil porque se
extiende directamente a situaciones más generales. Véase \cite[Def.\ (5.5.5), Prop.\
(5.5.6)]{NSW} para una demostración de que esto da lo mismo que la
\cref{defi:mvee-alpha}~\ref{defi:alpha}.

Estudiamos algunas propiedades del funtor $\alpha$.

\begin{prop}\label{prop:alpha}
  La asociación $\alpha$ define un funtor contravariante aditivo de la categoría de
  $\LL$-módulos noetherianos de torsión a si misma, enviando sucesiones exactas a
  la derecha a sucesiones exactas a la izquierda.
\end{prop}
\begin{proof}
  Si $M\rightarrow N$ es un morfismo de $\LL$-módulos noetherianos de torsión
  entonces podemos escoger $m$ en la definición de $\alpha$ suficientemente grande para
  ambos, de manera que obtenemos morfismos
  \begin{equation*}
    M/\nu_{r,m}M\rightarrow N/\nu_{r,m}N\quad\text{para
      cada }r\ge m
  \end{equation*}
  y entonces, aplicando $(\cdot)^\vee$ y tomando el límite, un morfismo $\alpha(N)\rightarrow\alpha(M)$.
  Así $\alpha$ define un funtor contravariante. Como
  \begin{equation*}
    M/\nu_{r,m}M=M\tensor_\LL\LL/\nu_{r,m}\LL
  \end{equation*}
  y el producto tensorial y la dualidad $(\cdot)^\vee$ son aditivos, $\alpha$ también es
  aditivo. Además, $\alpha$ envía sucesiones exactas a la derecha en sucesiones exactas a la
  izquierda porque el producto tensorial es exacto a la derecha, la dualidad $(\cdot)^\vee$
  en este caso envía sucesiones exactas a sucesiones exactas (\cref{rem:pontryagin-exacto})
  y el funtor tomando el límite también es exacto porque tomamos límites de módulos finitos,
  para cuales la condición de Mittag-Leffler siempre es cierta. Falta ver que $\alpha(M)$
  nuevamente es noetheriano y de torsión. Esto va a resultar del
  \cref{thm:iwasawa-adj} abajo.
\end{proof}

\begin{prop}\label{prop:alpha-pseudo-isos}
  El funtor $\alpha$ envía pseudo-isomorfismos a pseudo-isomorfismos.
\end{prop}
\begin{proof}
  Sea $M\rightarrow N$ un pseudo-isomorfismo. Como mencionamos en la demostración de la
  \cref{prop:alpha}, tomar el límite inverso y duales es exacto en módulos finitos. Por eso,
  $\alpha(N)\rightarrow\alpha(N)$ es un pseudo-isomorfismo si y solo si los órdenes de los
  núcleos y conúcleos de
  \begin{equation*}
    M/\nu_{r,m}M\rightarrow N/\nu_{r,m}N
  \end{equation*}
  son acotados para $r\to\infty$ (donde fijamos $m$ suficientemente grande). Pero esto
  resulta del \cref{ejer:pseudo-isomorfo-cociente-finito}~\ref{ejer:pseudo-isomorfo-cociente-finito:b}.
\end{proof}

La siguiente propiedad del funtor $\alpha$ es la razón por qué es importante para nosotros.

\begin{thm}\label{thm:iwasawa-adj}
  Sea $M$ un $\LL$-módulo noetheriano de torsión. Entonces existe un pseudo-isomorfo
  \[ \alpha(M) \sim M. \]
\end{thm}
\begin{proof}
  Sea $E$ un $\LL$-módulo elemental y $E\rightarrow M$ un pseudo-isomorfo. Aplicando
  $\alpha$ a este pseudo-isomorfismo y usando la \cref{prop:alpha-pseudo-isos} obtenemos
  un pseudo-isomorfismo de $\LL$-módulos $\alpha(M)\rightarrow\alpha(E)$. Por eso es
  suficiente demostrar la afirmación en el caso de un módulo elemental. Porque además
  $\alpha$ es aditivo, es suficiente demostrar esto en los casos en que $M=\LL/\pi^\mu$ con
  $\mu\in\Nuno$ y $M=\LL/P$ con un polinomio distinguido $P\in\LL$.
  
  Empecemos con el caso $M=\LL/\pi^\mu$ (donde fijamos un uniformizante $\pi$ de
  $\O$). Ponemos $\gamma=1+T$ y $m=0$, entonces para cada $r\ge0=m$ tenemos
  $\omega_r=\gamma^{p^r}-1$ y $1,\gamma,\gamma^2,\dotsc,\gamma^{p^r-1}$ es un $\O$-base de
  $M/\nu_{r,0}M$. Para cada $f\in M/\nu_{r,0}M$ escribimos
  \begin{equation*}
    f=\sum_{i=0}^{p^r-1}a_i\gamma^i,\quad a_i\in\O
  \end{equation*}
  y definimos
  \begin{equation*}
    \psi_r\colon M/\nu_{r,0}M\rightarrow (M/\nu_{r,0}M)^\vee, \quad \psi_r(f)(\gamma^i)=\frac{a_i}{\pi^\mu}.
  \end{equation*}
  Entonces se puede verificar directamente que $\psi_r$ es biyectivo, $\LL$-lineal y
  que el diagrama
  \begin{equation*}
    \begin{tikzcd}
      M/\nu_{r+1,0}M \arrow[d] \arrow[r, "\psi_{r+1}"] &
      (M/\nu_{r+1,0}M)^\vee \arrow[d] \\
      M/\nu_{r,0}M \arrow[r, "\psi_r"] &  (M/\nu_{r,0}M)^\vee
    \end{tikzcd}
  \end{equation*}
  es conmutativo (aquí el mapeo a la izquierda es la proyección canónica y el mapeo a la
  derecha es dual a la multiplicación con $\Phi_{r+1}$ en
  $M/\nu_{r,0}M\rightarrow M/\nu_{r+1,0}M$). Tomando el
  límite obtenemos $\alpha(M)\isom\varprojlim_rM/\nu_{r,0}M$. Entonces la afirmación resulta porque $M=\LL/\pi^\mu$ y $\bigcap_{r\geq 0} \w_{r}\LL= 0$ (véase la demostración
  del \cref{thm:EquivPowPro}).

  Ahora sea $M=\LL/(P)$ y sea $d\in\Nuno$ el grado de $P$.
  Entonces el \cref{lem:congruencia-phi-s} dice que $\Phi_r$ para $r\gg0$ actúa en $M$ como
  multiplicación con una potencia de $\pi$. Esto implica que
  \begin{multline*}
    \alpha(M)\isom\varprojlim_{s\in\Nuno} {(M/\pi^sM)}^\vee=\varprojlim_{s\in\Nuno} \Hom_\O(M/\pi^sM, L/\O) \\
    \isom\varprojlim_{s\in\Nuno} \Hom_\O(M/\pi^sM,\O/\pi^s)\isom\Hom_\O(M,\O)
  \end{multline*}
  como $\LL$-módulos, donde $\LL$ actúa en $\Hom_\O(M,\O)$ como
  $(\lambda\phi)(m)=\phi(\lambda m)$ para $\lambda\in\LL$, $\phi\in\Hom_\O(M,\O)$ y $m\in M$.  
  El lema de división
  (\cref{lema:division}) nos permite escribir cada $g\in\LL$ como $g=qP+r$ con
  $q,r\in\LL$ únicos y $r$ un polinomio de grado $<d$. Definimos un mapeo
  \begin{equation*}
    \varepsilon\colon\LL\rightarrow\O
  \end{equation*}
  que envía $g$ al coeficiente de $T^{d-1}$ en este $r$. Usando esto definimos
  \begin{equation*}
    \theta\colon M=\LL/(P)\rightarrow M^\vee=\Hom_\O(M,\O),\quad
    g\mapsto\left[h\mapsto\varepsilon(gh)\right]
  \end{equation*}
  Se verifica que esto es un morfismo de $\LL$-módulos bien definido. Si $r\in\O[T]$ es un
  polinomio de grado $k<d$ no nulo, que es representante de un elemento no nulo de $M$,
  entonces $\theta(r)(T^{d-1-k})=\varepsilon(T^{d-1-k}r)\neq0$, por eso $\theta$ es
  inyectivo. Porque los $\O$-rangos de $M$ y $\Hom_\O(M,\O)$ son iguales, el conúcleo de
  $\theta$ tiene que ser finito, así que $\theta$ es un pseudo-isomorfismo.\footnote{De
    hecho, $\theta$ es incluso sobreyectivo, es decir un isomorfismo, pero no
    lo necesitaremos y omitimos la demostración.}
\end{proof}

Concluimos esta sección generalizando los resultados al caso de módulos sobre el álgebra
de Iwasawa para grupos más grandes, como en la \cref{sec:mas-grandes}. Sólo discutimos esto
en el caso que vamos a necesitar en la aplicación, que es aquel en que $G=\Delta\times\Gamma$
con $\Delta\isom\F_p^\times$ y $\Gamma\isom\Zp$. Fijemos estos isomorfismos, así que
obtenemos un isomorfismo $\O\llbracket\Gamma\rrbracket\isom\O\llbracket T\rrbracket$. En el
resultados siguientes, sólo nos interesamos en la existencia de un pseudo-isomorfismo y no nos
importa si es canónico o no -- de hecho, el pseudo-isomorfismo depende de las elecciones
que hicimos, pero esto nos da igual.

Gracias al \cref{cor:descomposicion-lambda}, podemos identificar $\O\llbracket G\rrbracket$
con $(\O\llbracket T\rrbracket)^{p-1}$, así que un módulo sobre $\O\llbracket G\rrbracket$
es lo mismo que $p-1$ módulos sobre $\O\llbracket T\rrbracket$.

\begin{defi}\label{defi:alpha-g}
  Sea $M$ un $\O\llbracket G\rrbracket$-módulo noetheriano de torsión.
  Para $r\in\Nuno$ sea $w_r\in\O\llbracket G\rrbracket$ el elemento que corresponde a
  $(\omega_r,\dotsc,\omega_r)\in(\O\llbracket T\rrbracket)^{p-1}$. Entonces definimos
  \[ \alpha(M) = \varprojlim_{r\ge m} (M/\frac{w_r}{w_m} M)^\vee = (\varinjlim_{r\ge m}
     M/\frac{w_r}{w_m} M)^\vee \]
  con $m$ suficientemente grande, de manera análoga a la
  \cref{defi:mvee-alpha}~\ref{defi:alpha}. Equivalentemente, obtenemos
  $\alpha(M)$ via aplicar $\alpha$ a cada uno de los $p-1$ módulos sobre $\O\llbracket
  T\rrbracket$, es decir
  \[ \alpha(M) = \bigoplus_{i=1}^{p-1}\alpha(e_{\omega^i}M). \]
\end{defi}

\begin{cor}\label{cor:iwasawa-adj}
  Sea $M$ un $\O\llbracket G\rrbracket$-módulo noetheriano de torsión. Entonces
  existe un pseudo-isomorfo \[ \alpha(M) \sim M. \]
\end{cor}
\begin{proof}
  Esto resulta directamente del \cref{thm:iwasawa-adj}, aplicándolo en cada componente
  separadamente.
\end{proof}

\ejercicios

\begin{ejer}
  Demuestre que las diferentes descripciones de $\alpha(M)$ en la \cref{defi:alpha-g} son
  equivalentes. Use el \cref{ejer:diagonal} para esto.
\end{ejer}

\begin{ejer}\label{ejer:alpha-finito-cero}
  Demuestre que si $M$ es un $\LL$-módulo finito entonces $\alpha(M)=0$. Use el
  \cref{lem:congruencia-phi-s} para esto.
\end{ejer}

\chapter{Grupos de clases y el teorema de Iwasawa}
\label{sec:grupos-de-clases}

La teoría que desarrollamos en los \cref{sec:algebra-de-iwasawa,sec:modulos-iwasawa} todavía
no tiene nada que ver con aritmética. En este capítulo empezamos a conectarla con objetos
aritméticos, más precisamente con grupos de clases de campos de números. El hecho de que uno puede
estudiar los grupos de clases con estos métodos fue una de las ideas más importantes de
Iwasawa, y a continuación vamos a ilustrar que tan prolífico es este enfoque. Vamos a
obtener resultados sobre grupos de clases como el famoso \cref{thm:iwasawa-intro} de
Iwasawa y muchos más. Aparte de ser interesantes por sí mismos, algunos resultados también
preparan el lado algebraico de la Conjetura Principal, que estudiaremos en el \cref{sec:mc}.

\section{Grupos de Clases en Extensiones Ciclotómicas}

Comenzamos esta sección con algunos teoremas sobre grupos de clases que luego nos serán útiles. 
Sea $K$ un campo de números y $\O_K$ su anillo de enteros.

\begin{thm}\label{thm:iwahKhL} Supongamos que $L/K$ es una extensión de campos de números
  que no contiene ninguna extensión no trivial $F$ no ramificada de $K$. Entonces el número
  de clases $h_{K}$ divide a $h_{L}$. Además, la aplicación norma (\cref{defi:norma-de-ideales}) de $\Cl_{L}$ a $\Cl_{K}$ es sobreyectiva.
\end{thm}
\begin{proof}
Sea $H$ la máxima extensión no ramificada de $K$. Por la teoría de campos de clases
$\Gal(H/K)\isom \Cl_{K}$ (\cref{thm:campo-de-hilbert}). Además por hipótesis $H\cap
L=K$. Entonces $LH/L$ es una extensión abeliana no ramificada de índice $[LH:L]=[H:K]$, que
además está contenida en la máxima extensión no ramificada $M$ de $L$. El resultado sigue
del siguiente diagrama (v\'ease \cite[p. 400]{MR1421575})
\begin{equation*} 
  \begin{tikzcd}
    \Cl_{L} \arrow[r,"\sim"] \arrow[d, "\mathrm{Norm}"] & \Gal(M/L) \arrow[d, "\mathrm{Res}"] \\
    \Cl_{K} \arrow[r, "\sim"] & \Gal(H/K).
  \end{tikzcd}
\end{equation*}
\end{proof}

\begin{lem}\label{lem:G1normal} Sean $G$ un $p$-grupo e $I$ un subgrupo propio de $G$. Existe un subgrupo normal $G_{1}$ de $G$ de índice $p$ que contiene a $I$.  
\end{lem}
\begin{proof}
Vamos a proceder por inducción sobre $|G|$. Supongamos que para todo $p$-grupo $H$ con $|H|<|G|$ y un subgrupo propio $I'<H$, existe un subgrupo normal $G'_{1}<H$ de índice $p$ con $I'\subseteq G'_{1}$. 

Para $a$ un elemento en el centro $Z(G)$ de orden $p$ sean $$H=G/\langle a \rangle\	\	\text{ e }\	\	I'=I/\langle a \rangle \cap I.$$
Por la hipótesis de inducción existe un subgrupo $G'_{1}$ de $H$ de índice $p$ tal que $I'\subseteq G'_{1}$. 

Sea $\pi$ la proyección de $G$ en $H$. Sea $G_{1}=\langle \pi^{-1}(G'_{1}),a\rangle$. $G_{1}$ es un subgrupo normal de $G$ y además contiene a $I$. Por último, es fácil ver que $$G/G_{1}\rightarrow H/G'_{1}$$
es un isomorfismo de grupos de orden $p$.
\end{proof}

\begin{thm}\label{thm:iwapLpK}
Sea $L/K$ una extensión de Galois finita de campos de números tal que $\Gal(L/K)$ es un
$p$-grupo. Supongamos que existe a lo más un primo ramificado en $L$. Si $p\mid h_{L}$ entonces $p\mid h_{K}$. 
\end{thm}
\begin{proof}
Supongamos que $p\mid h_{L}$. Sean $M$ y $H$ las máximas $p$-extensiones no ramificadas de $L$ y $K$ respectivamente. Por la teoría de campos de clases tenemos $\Cl_{L}(p)=\Gal(M/L)$ y $\Cl_{K}(p)=\Gal(H/K)$. La extensión $M/K$ es de Galois pues $M/L$ es máxima, denotamos $G=\Gal(M/K)$.

Supongamos que $L/K$ no es ramificada. Vamos a proceder por inducción. Como $G$ es un $p$-grupo admite una secuencia de subgrupos 
$$\{e\}=G_{0}\subset G_{1} \subset G_{2} \subset \ldots G_{n} = G$$
tal que $G_{i}$ es normal en $G$ y $G_{i+1}/G_{i}$ es cíclico de orden $p$ \cite[Cor. 6.6 I.\S 6]{MR1878556}. Entonces basta probar que $p\mid h_{L_{1}}$ para $L/L_{1}$, donde $L_{1}$ es la subextensión de $L$ fijada por $G_{1}$. El resultado sigue del hecho que $L/L_{1}$ es abeliana no ramificada de índice $p$. Por lo tanto $p\mid h_{K}$.

Ahora supongamos que $\p$ en $K$ ramifica. Sea $\mathfrak{P}$ el primo de $L$ arriba de $\p$. Sea $I_{\mathfrak{P}}$ el subgrupo de inercia de $\mathfrak{P}$ en $G$. Como $M/L$ no es ramificada, entonces $L\cap H =K$, además 
$$|I_{\mathfrak{P}}|\leq [L:K] < |G|.$$
Por el \cref{lem:G1normal} tenemos que existe un subgrupo normal $G_{1}$ de $G$ de índice
$p$, tal que $I\leq G_{1}$. Los grupos de inercia de otros primos de $M$ son conjugados de
$I_{\mathfrak{P}}$ y por lo tanto están contenidos en $G_{1}$. Pero $G_{1}$ fija una
extensión de Galois de $K$ de índice $p$, por lo tanto abeliana, es decir $p \mid h_{K}$.
\end{proof}

El teorema precedente y el  \cref{thm:iwahKhL} aplicados a extensiones ciclotómicas nos dan el siguiente resultado.

\begin{cor} \label{cor:preg} Sea $r\in\Nuno$, entonces 
$$p\mid h_{\Q(\mu_{p})} \iff	p\mid h_{\Q(\mu_{p^{r}})}.$$
\end{cor}
\begin{proof} Recordemos que el único primo que ramifica en $\Q(\mu_{p^{r}})$ es $p$, además es totalmente ramificado (\cref{prop:ramcycfin}). Entonces ($\Leftarrow$) es claro por el \cref{thm:iwapLpK} y ($\Rightarrow$) se sigue del \cref{thm:iwahKhL}. \end{proof}

Terminamos esta sección con un resultado clásico de Kummer, que será utilizado en la
demostración de que la Conjetura Principal de Iwasawa implica el criterio de Kummer del
\cref{thm:kummer-crit}.

\begin{thm}[Kummer]\label{thm:kummer-hp-plus-minus}
  Sea $K=\Q(\mu_p)$ y $C$ la $p$-parte del grupo de clases de $K$. Escribimos $C^\pm$ para
  los subgrupos donde la conjugación compleja actúa por $\pm1$, respectivamente. Entonces
  \[ C^-=0 \iff C=0. \] En otras palabras, $p\mid h_p \iff p\mid h_p^-$, donde
  $h_p=h_{\Q(\mu_{p})}$, $h^+_p=h_{\Q(\mu_{p})^+}$ y $h_p^-=h_p/h_p^+$, con $\Q(\mu_p)^+$
  siendo el subcampo de $\Q(\mu_p)$ fijado por la conjugación compleja.
\end{thm}
\begin{proof}
  Seguimos \cite[Thm.\ 13.2.1]{MR1029028} o \cite[§10.2, esp.\ Thm.\ 10.11]{MR1421575}. Sea
  $L/K$ la extensión máxima abeliana no ramificada de exponente $p$, que es una extensión de
  Kummer, y sea $G:=\Gal(L/K)$. Entonces $G\isom C[p]$ por la teoría de campos de clases
  (\cref{thm:campo-de-hilbert}), donde $C[p]$ es la $p$-torsión de $C$, que es un espacio
  vectorial sobre $\Fp$. Según el \cref{ejer:artin-symbol-eq}, este isomorfismo es
  equivariante por la acción de la conjugación compleja $\mathbf c\in\Gal(K/\Q)$, donde
  esta actúa en $G$ por conjugación con un levantamiento $\tilde{\mathbf c}\in\Gal(L/\Q)$ de
  $\mathbf c$ que fijamos para el resto de la demostración.  Vamos a demostrar
  \begin{equation*}
    \dim_{\Fp} C[p]^+\le\dim_{\Fp} C[p]^-,
  \end{equation*}
  lo cual implica la afirmación.

  Usamos la teoría de Kummer descrita en el \cref{thm:kummer}.  Sea
  $\Delta:={(L^\times)}^p\cap K^\times$ y $V=\sqrt[p]\Delta/(\sqrt[p]\Delta\cap K^\times)$, de manera que $L=K(\sqrt[p]\Delta)$ y tenemos un apareamiento perfecto
  \begin{equation*}
    G\times V\rightarrow\mu_p
  \end{equation*}
  de espacios vectoriales sobre $\Fp$. El grupo $\Gal(L/\Q)$ actúa naturalmente en $V$, y
  esto nos da una acción de $\tilde{\mathbf c}\in\Gal(L/\Q)$ en $V$; escribimos $V^\pm$ para
  los partes donde $\tilde{\mathbf c}$ actúa por $\pm1$. Según la
  \cref{prop:kummer-apareamiento-equivariante} el apareamiento tiene la propiedad
  \begin{equation*}
    \left<\tilde{\mathbf c}\sigma\tilde{\mathbf c}^{-1},\tilde{\mathbf c} v\right>=\mathbf c\left<\sigma,v\right>
    \quad \text{para cada }\sigma\in G, v\in V.
  \end{equation*}
  Como $p\neq2$, los únicos elementos en $\mu_p$ fijados por $\mathbf c$ son $\pm1$, y
  como $G$ y $V$ son espacios vectoriales sobre $\Fp$, la propiedad anterior implica que
  la restricción del apareamiento a $G^+\times V^+$ es trivial (aquí usamos el
  \cref{ejer:morfismos-pro-p-pro-ell}). Ya que $V=V^+\oplus V^-$, el apareamiento pues se
  restringe a un apareamiento perfecto
  \begin{equation*}
    G^+\times V^-\rightarrow\mu_p
  \end{equation*}
  y por eso $\dim_{\Fp} C[p]^+=\dim_{\Fp}G^+=\dim_{\Fp}V^-$.
  
  Sea $b\in\sqrt[p]\Delta\subseteq L^\times$. Entonces $K(b)/K$ no es ramificado, y según la
  \cref{prop:kummer-ramificacion} el ideal fraccional generado por
  $b^p\in K^\times$ entonces tiene que ser de la forma $(b^p)=\mathfrak b^p$ con un ideal
  fraccional $\mathfrak b$ de $K$. Esto nos permite definir un homomorfiso
  \begin{equation*}
    \varphi\colon V\rightarrow C[p],\quad b\mapsto\mathfrak b
  \end{equation*}
  que es obviamente equivariante por la conjugación compleja e induce pues
  \begin{equation*}
    \varphi^-\colon V^-\rightarrow C[p]^-.
  \end{equation*}
  Vamos a demostrar que $\varphi^-$ es inyectivo, lo cual termina la demostración.

  Fijamos $b\in\sqrt[p]\Delta$ tal que su clase en $V$ está en el núcleo de
  $\varphi^-$. Entonces podemos escribir $b^p=a^pu$ con $a\in K^\times$ y
  $u\in\O_K^\times$. El elemento $u$ no está únicamente determinado, pero su clase en
  $\O_K^\times/{(\O_K^\times)}^p$ sí lo es, y está en
  ${\big(\O_K^\times/{(\O_K^\times)}^p\big)}^-$.  Según la \cref{prop:o-es-mu-o-mas} podemos
  escribir $u\in\O_K$ como $u=\zeta v$ con $\zeta\in\mu_p$ y $v\in{(\O_K^\times)}^+$.
  Entonces $\overline u=\zeta^{-1}v$ y también $\overline u=u^{-1}c^p=\zeta^{-1}v^{-1}c^p$
  para algún $c\in{(\O_K^\times)}^+$, y obtenemos $v^2=c^p$. Pero según el teorema de las
  unidades de Dirichlet, ${(\O_K^\times)}^+$ es $\{\pm1\}$ por un grupo abeliano libre, lo que implica que $v\in{(\O_K^\times)}^p$, digamos $v=w^p$. Concluimos que $b^p=(aw)^p\zeta$.
  El campo $L$ contiene $K(b)=K(\sqrt[p]\zeta)$, pero ya que $L$ no es ramificado,
  $\zeta=1$. Por eso $b=aw\xi$ para algún $\xi\in\mu_p$, es decir $b\in K^\times$, así que su
  clase en $V$ es trivial.
\end{proof}

\ejercicios

\begin{ejer}\label{ejer:hpmenos-divisibilidad}
  Para $r\in\Nuno$ sea $h_{p^r}=h_{\Q(\mu_{p^r})}$. Además sea $\Q(\mu_{p^r})^+$ el subcampo de
  $\Q(\mu_{p^r})$ fijo por la conjugación compleja, $h_{p^r}^+$ su número de clases y
  $h_{p^r}^-=h_{p^r}/h_{p^r}^+$.

  Demuestre que 
  \[ p\mid h_{p}^- \iff p \mid h_{p^r}^-, \]
  usando el \cref{cor:preg} y el \cref{thm:kummer-hp-plus-minus}.
\end{ejer}

\section{$\Zp$-extensiones}

Sea $K$ un campo de números y $p$ un número primo.\footnote{Como siempre, suponemos que $p$ es impar para facilitar la exposición, pero los resultados tienen análogos en el caso $p=2$.} Sea $K_{\infty}$ una extensión de Galois infinita de $K$, decimos que $K_{\infty}$ es una \emph{$\Zp$-extensión}\index[def]{extensiónZp@$\Zp$-extensión} (igualmente llamada extensión $\Zp$) si $\Gal(K_{\infty}/K)$ es isomorfo al grupo aditivo del anillo $\Zp$ de los enteros $p$-ádicos. 

\begin{prop}\label{prop:extcyclo} Todo campo de números $K$ tiene una $\Zp$-extensión.
\end{prop}
\begin{proof} Sabemos que la extensión ciclotómica $\Q(\mu_{p^{r}})/\Q$ obtenida al agregar una raíz primitiva $p^{r}$-ésima de la unidad es abeliana, de grado $\varphi(p^{r})=(p-1)p^{r-1}$ y de grupo de Galois $\Gal(\Q(\mu_{p^{r}})/\Q)$ isomorfo a $(\Z/p^{r}\Z)^{\times}$ para todo $r\geq 0$. Consideremos la extensión infinita
$$\Q(\mu_{p^{\infty}})=\bigcup_{r\geq 0} \Q(\mu_{p^{r}}),$$
su grupo de Galois $\Gal(\Q(\mu_{p^{\infty}})/\Q)$ es isomorfo a
$\varprojlim_{n}(\Z/p^{r}\Z)^{\times}\isom \Zp^{\times}$.  Por el
\cref{lem:z-p-decomposicion} y la \cref{prop:isom-z-p-log} tenemos $\Zp^{\times}\isom(\Z/p\Z)^{\times}\times \Zp$. Sea $\Q^{\mathrm c}$ la subextensión de $\Q(\mu_{p^{\infty}})$ fijada por $(\Z/p\Z)^{\times}$, entonces $\Gal(\Q^{\mathrm c}/\Q)\isom\Zp$.

Sea $K^{\mathrm c}=K\Q^{\mathrm c}$ y sea $F=\Q^{\mathrm c}\cap K$, entonces 
$$\Gal(\Q^{\mathrm c}/F)\isom \Gal(K^{\mathrm c}/K)$$
son isomorfos. Tenemos $\Gal(\Q^{\mathrm c}/F)\isom p^{k}\Z_{p}$ para algún $k\in \Nuno$, pero $p^{k}\Zp\isom \Zp$, entonces $\Gal(K_{\infty}/K)\isom\Zp$.
\end{proof}

La $\Zp$-extensión construida en la \cref{prop:extcyclo} se llama la \emph{$\Zp$-extensión ciclotómica}\index[def]{extensiónZp@$\Zp$-extensión!ciclotómica} de $K$ y la denotamos $K^{\mathrm c}$.

Sea $K_{\infty}/K$ una $\Zp$-extensión. Por la teoría de Galois infinita los subgrupos cerrados de $\Zp$ corresponden a subextensiones de $K_{\infty}$. Es fácil de ver que bajo la topología $p$-ádica los subgrupos cerrados (abiertos) de $\Zp$ son de la forma $p^{r}\Zp$. Por lo que los subgrupos cerrados de $\Zp$ fijan extensiones finitas $K_{r}$ de $K$ tales que $\Gal(K_{r}/K)\isom \Z/p^{r}\Z$. Es decir, una $\Zp$-extensión $K_{\infty}$ es una torre infinita de campos de números
\begin{equation}\label{eq:Zptorre}
K=K_{0}\subset K_{1} \subset K_{2} \subset \cdots \subset K_{\infty}=\bigcup_{r\geq0} K_{r},
\end{equation}
tal que $\Gal(K_{r}/K)\isom \Z/p^{r}\Z$.

A continuación mencionamos algunas propiedades sobre ramificación en las $\Zp$-extensiones. Comenzamos con un teorema que nos será útil más adelante para explicar la conjetura de Leopoldt.

\begin{prop}\label{prop:compononramifi} Sea $L$ una extensión abeliana de un campo de números $K$ tal que $\Gal(L/K)\isom \Zp^{d}$ para algún entero $d\geq 0$. Entonces $L/K$ es no ramificado fuera de $p$. 
\end{prop}
\begin{proof} Sea $\p$ una plaza de $K$ que no divide a $p$. Sea $I_{\p}$ el subgrupo de inercia de $\p$ en $\Gal(K^{\mathrm{ab}}/K)$. Por la teoría de campos de clases la imagen de $I_{\p}$ en $\Gal(L/K)$ es finita. Sin embargo, el grupo $\Zp^{d}$ es libre de torsión, por lo tanto el subgrupo de inercia de $\p$ en $\Gal(L/K)$ es nulo.
\end{proof}

\begin{prop}\label{prop:ramificZpext} Sea $K_{\infty}$ una $\Zp$-extensión de un campo de números $K$. Entonces al menos una plaza $\p$ arriba de $p$ ramifica en $K_{\infty}/K$. 
\end{prop}
\begin{proof} Supongamos que ninguna plaza $\p\mid p$ ramifica. En particular por la \cref{prop:compononramifi} esto implicaría que $K_{\infty}/K$ es no ramificada, pero la máxima extensión no ramificada de un campo de números es una extensión finita. 
\end{proof}

Para un campo de números $K$ denotamos $[K:\Q]$ su dimensión como espacio vectorial sobre $\Q$. Además denotamos $r$ el número de encajes reales $K\hookrightarrow \R$ y $c$ el número de pares de encajes complejos conjugados $K\hookrightarrow \C$. 

Si $X$ es un $\Zp$-módulo compacto, entonces $X/pX$ (resp. $X\otimes_{\Zp}\Qp$) es un espacio vectorial sobre $\F_{p}$ (resp. sobre $\Qp$). 

\begin{defi}\label{def:rg} Llamamos \define[rango sobre $\F_{p}$]{rango de $X$ sobre $\F_{p}$} (resp. \define[rango sobre $\Zp$]{sobre $\Zp$}) a la dimensión de $X/pX$ (resp. a la dimensión de $X\otimes_{\Zp}\Qp$) y lo denotamos $\rg_{\F_{p}}(X)$ (resp. $\rg_{\Zp}(X)$). \end{defi}
 
\begin{thm}\label{thm:Munramleop} Sea $M$ la máxima extensión pro-$p$ abeliana no ramificada
  fuera de $p$ de un campo de números $K$. Entonces $\rg_{\Fp}(\Gal(M/K))$ es finito y 
$$c+1\leq \rg_{\Zp}\Gal(M/K) \leq [K:\Q].$$
\end{thm}
\begin{proof} Para una plaza $\p$ de $K$ sobre $p$, sea $U_{\p}^{(1)}$ el grupo de unidades
  principales del campo local $K_{\p}$ (\cref{defi:unidades-principales}). Consideremos el producto
\[U=\prod_{\p\mid p}U_{\p}^{(1)}\]
y sea $E=\{\varepsilon\in K \,|\,\varepsilon\in U_{\p}^{(1)}\text{ para todo } \p\mid p\}$. Denotamos $\overline{E}$ la cerradura de $E$ en $U$, y denotamos $H'$ la extensión máxima no ramificada de $K$ en $K_{\infty}$, que por la teoría de campos de clases debe ser finita y además
\[\Gal(M/H')\isom U/\overline{E}.\] 
De aquí deducimos que $\rg_{\Fp}(\Gal(M/K))$ es finito, ya que $\rg_{\Fp}(\Gal(M/H'))$ es
finito y $\rg_{\Fp}(U)$ es finito, pues el grupo de unidades principales $U^{(1)}_{\p}$ de
un campo local $K_{\p}$ es el producto de un grupo cíclico finito y de un $\Zp$-módulo de
rango $[K_{\p}:\Qp]$ según la \cref{prop:unidades-principales}. Por esta última observación también tenemos que
\begin{eqnarray*}
\rg_{\Zp}(U)&=&\sum_{\p\mid p} \rg_{\Zp}U_{\p} \\
&=& \sum_{\p\mid p} [K_{\p}:\Qp]\\
&=& [K:\Q].
\end{eqnarray*}
Recordemos que por el teorema de Dirichlet, el anillo de enteros $O_{K}$ de $K $ es un producto de un grupo finito y de un $\Z$-módulo libre de dimensión $r+c-1$ generado por las unidades fundamentales. La imagen de las unidades fundamentales bajo los encajes $K^{\times} \hookrightarrow K_{\p}^{\times}$ está en $U_{\p}^{(1)}$ para toda plaza $\p\mid p$. Por lo tanto, $\rg_{\Zp}\overline{E}\leq r+c-1$. Finalmente como 
\[\rg_{\Zp}\Gal(M/K)=\rg_{\Zp}\Gal(M/H')=\rg_{\Zp}(U)-\rg_{\Zp}(\overline{E}),\]
tenemos $c+1 \leq \rg_{\Zp}(\Gal(M/K))\leq [K:\Q]$.
\end{proof}

\begin{thm} Sea $\widehat{K}$ el compuesto de todas las $\Zp$-extensiones de un campo de números. Entonces $\widehat{K}/K$ es una extensión abeliana y $\Gal(\widehat{K}/K)\isom \Zp^{d}$, donde $d=\rg_{\Zp}\Gal(M/K)$.
\end{thm}
\begin{proof} Por la \cref{prop:compononramifi} la extensión $\widehat{K}$ está contenida en $M$. Por la teoría de Galois, el rango con respecto a $\Zp$ del grupo de Galois $\Gal(M/\widehat{K})$ es nulo.
\end{proof}

\begin{con}[Conjetura de Leopoldt]\label{conj:Leopoldt}\index[def]{conjetura!de Leopoldt}
Sea $d$ el rango sobre $\Zp$ de $\Gal(\widehat{K}/K)$, entonces $d=c+1$, es decir, existen $c+1$ $\Zl$-extensiones independientes sobre $K$.
\end{con}

Existen muchas versiones equivalentes de esta conjetura a través de cuales es relacionada
con varios objetos, una lista se encuentra en \cite[(10.3.6)]{NSW}. Esto enfatiza la
importancia de la conjetura -- por ejemplo tiene que ver con las funciones $L$ $p$-ádicas
que vamos a estudiar en el \cref{sec:palf}; véase \cite[§5.5]{MR1421575} para más sobre esta
conexión. La conjetura de Leopoldt es conocida en el caso especial en que $K$ es una
extensión abeliana de los números racionales $\Q$ o de un campo cuadrático imaginario.

\ejercicios

\begin{ejer} Sea $K^{\mathrm c}$ la $\Zp$-extensión ciclotómica de un campo de números $K$. Demuestre que $K^{\mathrm c}$ es ramificada en todo primo $\p$ sobre $p$. Además demuestre que si $K=\Q$, entonces $K^{\mathrm c}$ es totalmente ramificada en $p$.
\end{ejer}

\begin{ejer} Demuestre que si $K=\Q$ entonces sólo hay una $\Zp$-extensión de $\Q$, i.\,e. la $\Zp$-extensión ciclotómica. 
\end{ejer}
 
\section{Propiedades de los grupos de clases como $\LL$-módulos}
\label{sec:propiedades-x-ll}

Sea $\LL=\Zp\llbracket T \rrbracket$ el álgebra de Iwasawa con coeficientes en
$\Zp$. Recordemos que en este caso $\mathfrak{M}=(p,T)$. En la
\cref{sec:estructura-consecuencias} demostramos algunas propiedades de los $\LL$-módulos
compactos. Estas nos serán útiles aquí pues los $\LL$-módulos con que trabajaremos son
finitamente generados, y estos son compactos por ser la imagen continua de $\LL^{d}$ para
algún $d$.

Sea $K_{\infty}/K$ una $\Zp$-extensión de grupo de Galois $\Gamma=\Gal(K_{\infty}/K)$ y sea $L_{\infty}$ una extensión abeliana de $K_{\infty}$ tal que $L_{\infty}/K$ es una extensión de Galois de grupo $G=\Gal(L_{\infty}/K)$. Si denotamos $X=\Gal(L_{\infty}/K_{\infty})$ tenemos que $\Gamma$ actúa continuamente por conjugación en $X$
\begin{equation}\label{eq:action}
x^{\gamma}=\tilde{\gamma}x\tilde{\gamma}^{-1}\	\text{ para todo }\	x\in X \text{ y }\gamma\in \Gamma,
\end{equation}
donde $\tilde{\gamma}$ es un levantamiento de $\gamma$ a $G$.
Esta construcción de hecho es un patrón general, véase el \cref{ejer:sucesion-exacta-accion}.
En nuestra situación $X$ es un subgrupo cerrado normal de $G$ tal que 
$$\Gamma \isom G/X.$$
Fijando un generador $\gamma\in \Gamma$ podemos extender la acción de $\Gamma$ sobre $X$ al álgebra de Iwasawa $\LL$ (\cref{thm:EquivPowPro}), i.\,e. hacemos de $X$ un $\LL$-módulo compacto. Además, como $\Gamma$ es un $\Zp$-módulo libre, existe una sección de $\Zp$-módulos que describe el producto semidirecto
\begin{equation}\label{eq:prodsemidir}
G= \Gamma \ltimes X.
\end{equation}

Como de costumbre denotamos $\Gamma_{r}$ el único subgrupo cerrado de $\Gamma$ tal que $\Gamma/\Gamma_{r}\isom \Z/p^{r}\Z$. Topológicamente $\Gamma_{r}$ es generado por $\gamma^{p^{r}}$. Sea $\omega_{r}:=\omega_{r}(\gamma)$ (cf.\ \cref{defi:omega-r}), el súbmódulo $\omega_{r}X$ es el mínimo $\Gamma$-submódulo tal que $\Gamma_{r}$ actúa trivialmente en $X/\omega_{r}X$. Por lo tanto $\omega_{r}X$ no depende de la $\LL$-estructura en $X$. 

\begin{lem}\label{lem:conmutatorGr} Sea $G_{r}$ el grupo de Galois
  $\Gal(L_{\infty}/K_{r})$. Para un subgrupo $H$ denotamos $\overline{H}$ la cerradura topológica de $H$. Entonces
$$\overline{[G_{r},G_{r}]}=\omega_{r}X.$$
\end{lem}
\begin{proof} Podemos escribir $G_{r}=\Gamma_{r}\ltimes X$. Sean $a=\alpha x$ y $b=\beta y$ elementos de $G_{r}$ con $\alpha,\beta$ en $\Gamma_{r}$ y $x,y$ en $X$. Es un ejercicio fácil (cf. \cref{ejer:conmutator}) ver que
  $$[a,b]=aba^{-1}b^{-1}=(x^{\alpha})^{1-\beta}(y^{\beta})^{\alpha-1}.$$
  Se cumple que $1-\beta$ y $\alpha-1$ son divisibles por $\omega_r$, por eso $[a,b]\in\omega_rX$.
  
  Por el otro lado, haciendo $\beta=1$ y $\alpha=\gamma^{p^{r}}$ vemos que $y^{\gamma^{p^{r}}-1}\in\overline{[G_{r},G_{r}]}$. Es decir, $\omega_{r}X\subseteq \overline{[G_{r},G_{r}]}$.
\end{proof}

Denotamos $L_{r}$ la máxima subextensión de $L_{\infty}$ abeliana sobre $K_{r}$ (cf. \eqref{eq:Zptorre}). Entonces $L_{\infty}$ es la unión de los $L_{r}$. Del lema precedente deducimos inmediatamente que 
\begin{equation}\label{eq:conmuGal}
\omega_{r}X=\Gal(L_{\infty}/L_{r}) \	\	\	\text{ y } \	\	\	X/\omega_{r}X = \Gal(L_{r}/K_{\infty}).
\end{equation}

Con la notación de la \cref{def:rg}, obtenemos los siguientes resultados. En particular, una consecuencia del \cref{cor:cor-de-nakayama} es el siguiente resultado que nos será útil próximamente. 

\begin{cor}\label{cor:cond1noeth} $X$ es noetheriano si y solamente si $\rg_{\F_{p}}(\Gal(L_{0}/K))$ es finito. 
\end{cor} 
\begin{proof} Sea $Y=X/\w_{0}X$. Entonces $X/\mathfrak{M}X=Y/pY$, por lo que $X/\mathfrak{M}X$ es finito si y solamente si $\rg_{\F_{p}}Y$ es finito. Pero $\rg_{\F_{p}}Y+\rg_{\F_{p}}\Gamma = \rg_{\F_{p}} \Gal(L_{0}/K)$. 
\end{proof}

Vamos a aplicar los resultados de la \cref{sec:estructura-consecuencias} a $\LL$-módulos asociados a las siguientes extensiones y grupos de Galois. 

\begin{defi}\label{defi:iwasawa-modulos} Sea $K_{\infty}/K$ una $\Zp$-extensión. 
  Para $r\in\Ncero$ denotamos
  \begin{itemize}
  \item $K_r=$ el único subcampo de la $\Zp$-extensión $K_\infty/K$ tal que $\Gal(K_r/K)\isom\Z/p^r\Z$,
  \item $H_r=$ la máxima extensión abeliana pro-$p$ de $K_r$ no ramificada,
  \item $M_r=$ la máxima extensión abeliana pro-$p$ de $K_r$ ramificada sólo en $p$,
  \item $C_r=$ la $p$-parte del grupo de clases $\Cl(K_r)$,
  \item $X_r=\Gal(H_r/K_r)$,
  \item $Y_r=\Gal(M_r/K_r)$.
  \end{itemize}
  Escribimos $H_\infty$ para el compuesto de todos los campos $H_r$ para $r\in\Ncero$, y
  análogamente definimos $M_\infty$. Estudiamos los módulos
  \begin{itemize}
  \item $X_\infty := \Gal(H_\infty/K_\infty) \isom \displaystyle\varprojlim_{r\in\Ncero}X_r$,
  \item $Y_\infty := \Gal(M_\infty/K_\infty) \isom \displaystyle\varprojlim_{r\in\Ncero}Y_r$.
  \end{itemize}
  Aquí los límites son tomados con respecto a los mapeos de restricción entre los grupos de
  Galois (véase el \cref{ejer:y-infty-lim} para el último isomorfismo).
\end{defi}

El campo $H_r$ es el \emph{$p$-campo de clases de Hilbert}\esindex[def]{campo de clases de Hilbert!$p$-campo} de $K_r$, es decir la máxima
extensión abeliana pro-$p$ no ramificada, y el \cref{thm:campo-de-hilbert} describe su grupo de
Galois: existe un isomorfismo canónico $X_r\isom C_r$ de grupos abelianos para cada $r\in\Ncero$.
El grupo de Galois $\Gal(K_r/K)$ actúa naturalmente en $C_r$, y de hecho el isomorfismo
$X_r\isom C_r$ es compatible con esta acción (esto sigue del \cref{ejer:artin-symbol-eq}).
Si $\mathrm N_r\colon C_{r+1}\rightarrow C_r$ denota la norma relativa de ideales (\cref{defi:norma-de-ideales}) entonces
se puede comprobar que el diagrama
\begin{equation}
  \label{eqn:diagrama-x-c}
    \begin{tikzcd}
    X_{r+1} \arrow[r,"\sim"] \arrow[d] & C_{r+1} \arrow[d, "\mathrm N_r"] \\
    X_r \arrow[r, "\sim"] & C_r
  \end{tikzcd}
\end{equation}
es conmutativo para $r$ suficientemente grande (para esto hay que usar que $H_r\cap
K_{r+1}=K_r$, que es cierto para $r$ suficientemente grande; véase \cite[p. 277 y
lem. 13.3]{MR1421575}).
Esto muestra que $X_\infty\isom\varprojlim_{r\in\Ncero}C_r$ como $\LL$-módulo, el límite de los $C_r$
tomado con respecto a la norma.

\begin{thm}\label{thm:y-noetheriano}
  El módulo $Y_\infty$ es noetheriano. 
\end{thm}
\begin{proof} Tenemos que $M_{0}$ es la máxima extensión abeliana pro-$p$ de $K$, por el \cref{thm:Munramleop} tenemos que $\rg_{\F_{p}}\Gal(M_{0}/K)$ es finito, de donde deducimos la afirmación del \cref{cor:cond1noeth}. 
\end{proof}

\begin{cor}\label{cor:x-noetheriano-torsion}
  $X_\infty$ es un $\LL$-módulo noetheriano y de torsión. 
\end{cor}
\begin{proof} $X_\infty$ es isomorfo a un cociente de $Y_\infty$, ya que la máxima extensión abeliana pro-$p$ no ramificada $H_{\infty}$ de $K_{\infty}$ está contenida en $M_{\infty}$. Por lo tanto, $X_\infty$ es noetheriano. 

Ahora veamos que es de torsión. Sea $S_{0}$ el conjunto de plazas de $K$ que ramifican en $K_{\infty}$. El conjunto $S_{0}$ es no vacío y finito (\cref{prop:compononramifi} y \cref{prop:ramificZpext}). Sea $I_{\p}\subseteq \Gamma$ el subgrupo de inercia de la plaza $\p\in S_{0}$. Sabemos que $I_{\p}$ es un subgrupo cerrado de $\Gamma$ entonces isomorfo a $\Gamma_{r_{\p}}$ para algún $r_{\p}\geq 0$.

Sea $$r_{0}=\max_{\p\in S_{0}}\{r_{\p}\},$$
entonces la extensión $K_{\infty}/K_{r_{0}}$ es totalmente ramificada en las plazas de $K_{r}$ arriba de $S_{0}$. En particular, el número de plazas $s$ de $K_{r}$ que ramifican en $K_{\infty}$ es el mismo para todo~$r\geq  r_{0}$. 

Ahora usaremos la construcción al inicio de la sección, en el caso
$$L_{\infty}=H_{\infty}\	\	\	\text{ y }\	\	\	X=X_{\infty}=\Gal(H_{\infty}/K_{\infty}).$$
Recordemos que $L_{r}$ es la máxima extensión abeliana de $K_{r}$ contenida en $H_{\infty}$.

Sean $\p_{1},\ldots,\p_{s}$ las plazas en $K_{r}$, con $r\geq r_{0}$, que ramifican en $K_{\infty}$. Sea $I_{i}$ el subgrupo de inercia de $\Gal(L_{r}/K_{r})$ para $i=1,\ldots,s$. Como $L_{n}/K_{\infty}$ es no ramificada y los $\p_{i}$ ramifican totalmente en $K_{\infty}$ obtenemos 
$$I_{i}\isom \Gamma_{r} \isom \Zp, \	\text{ para }\	i=1,\ldots,s.$$
La extensión $H_{r}$ está contenida en $L_{r}$, ya que $H_{r}K_{\infty}$ es abeliana sobre $K_{r}$ y no ramificada sobre $K_{\infty}$.  Además, las únicas plazas de $K_{r}$ que ramifican en $L_{r}$ son precisamente $\p_{1},\ldots,\p_{s}$, entonces 
$$\Gal(L_{r}/H_{r})=I_{1}\cdots I_{s}.$$

Para todo $r\geq r_{0}$ tenemos
\begin{eqnarray*}
\rg_{\Zp}X/\w_{r} X &=& \rg_{\Zp} \Gal(L_{r}/K_{\infty}) \\
&=& \rg_{\Zp} \Gal(L_{r}/K) - 1\\
&=& \rg_{\Zp} \Gal(L_{r}/H_{r}) -1 \\
&\leq & s-1.
\end{eqnarray*}

El resultado sigue observando que $\rg_{\Zp} X/\w_{r}X \leq \rg_{\Zp} X/w_{r+1}X$ y de la \cref{prop:Xnoettorsiirgfinsiiacot}.
\end{proof}

\ejercicios

\begin{ejer}\label{ejer:sucesion-exacta-accion}
  Sea \[ 1\rightarrow A\rightarrow D \rightarrow G\rightarrow 1 \] una sucesión exacta de
  grupos donde $A$ es abeliano. Para $\sigma\in G$ sea $\tilde\sigma\in D$ un levantamiento
  y definimos \[ \sigma\bullet a= \tilde{\sigma}a\tilde{\sigma}^{-1}\text{ para }a\in A. \]
  Demuestre que esto está bien definido y define una acción $\Z$-lineal de $G$ en
  $A$. Verifique que esta acción es continua si todos son grupos topológicos, los morfismos
  en la sucesión son continuos y si el mapeo $D\rightarrow G$ admite una sección continua
  (no necesariamente un morfismo de grupos).
\end{ejer}

\begin{ejer}\label{ejer:conmutator}
  Sea $G=H\ltimes N$, con $H$ y $N$ abelianos. Entonces para todo $g_{1}=h_{1}n_{1}$ y $g_{2}=h_{2}n_{2}$ elementos de $G$ tenemos
  $$[g_{1},g_{2}]=(n_{1}^{h_{1}})^{1-h_{2}}(n_{2}^{h_{2}})^{h_{1}-1}.$$
\end{ejer}

\begin{ejer}\label{ejer:y-infty-lim}
  Sean $K_r$ y $M_r$ como en la \cref{defi:iwasawa-modulos} para cada
  $r\in\Ncero\cup\{\infty\}$ y sea $Y_\infty=\Gal(M_\infty/K_\infty)$.
  \begin{enumerate}
  \item Verifique que $K_\infty\subseteq M_r$ para cada $r\ge0$.
  \item Use el \cref{thm:hauptsatz-unendlgal} para ver que
    \[ Y_\infty=\varprojlim_{r\ge0}\Gal(M_r/K_\infty). \]
  \item Demuestre que la aplicación canónica
    \[ \varprojlim_{r\ge0}\Gal(M_r/K_\infty)\rightarrow\varprojlim_{r\ge0}\Gal(M_r/K_r) \]
    es un isomorfismo.
  \end{enumerate}
\end{ejer}

\section{Teorema de Iwasawa}

En esta sección vamos a demostrar el \cref{thm:iwasawa-intro} de Iwasawa.

Continuamos usando la notación de la \cref{defi:iwasawa-modulos}, en particular sea $K_{\infty}/K$ una $\Zp$-extensión. Además, usaremos continuamente la construcción de la sección pasada en el caso especial
$$L_{\infty}=H_{\infty}\	\	\	\text{ y }\	\	\	X=X_{\infty}=\Gal(H_{\infty}/K_{\infty}).$$
Al inicio de esta sección (en la \cref{prop:Crdescenso} y sus corolarios) hacemos la siguiente hipótesis que es verificada en los casos que trataremos en la exposición de las funciones $L$ $p$-ádicas, es decir en las $\Zp$-extensiones ciclotómicas $\Q^{\mathrm c}$ y $\Q(\mu_{p})^{\mathrm c}$.\\

\textbf{Hipótesis:} Los primos $\p_{1},\ldots,\p_{s}$ de $K$ que ramifican en la $\Zp$-extensión $K_{\infty}$ son totalmente ramificados. \\

Sea $I_{i}$ el subgrupo de inercia de $G=\Gal(H_{\infty}/K)$ correspondiente a la plaza $\p_{i}$, con
$i=1,\ldots,s$. Por la hipótesis anterior, tenemos que $I_{i}\isom \Gamma$ ya que $X_\infty$
es no ramificado. Escogiendo un generador topológico $\gamma$ de $\Gamma$, llamamos
$\sigma_{i}$ la imagen de $\gamma$ en $I_{i}$. El isomorfismo $G=I_{i}X_\infty=X_\infty I_{i}$ (ver \eqref{eq:prodsemidir}) nos permite escribir
$\sigma_{i}=a_{i}\sigma_{1}$, para $a_{i}$ elementos de $X_\infty$, claramente $a_{1}=1$. 

\begin{prop}\label{prop:Crdescenso} Sea $T_{0}$ el $\Zp$-submódulo generado por $\{a_{i}\,|\,2\leq i \leq s\}$ y $\w_{0}X_\infty$. Sea $T_{r}=\frac{\w_{r}}{\w_{0}}T_{0}$, entonces 
\[C_{r}\isom X_\infty/T_{r}\	\	\	\text{ para } r\geq 0.\]
\end{prop}
\begin{proof}
Claramente $H_{0}$ es la máxima extensión abeliana no ramificada de $K=K_{0}$ contenida en $H_{\infty}$. El grupo de Galois $X_{0}=\Gal(H_{0}/K_{0})$ es igual a $X_{0}=G/\Gal(H_{\infty}/H_{0})$ donde  $\Gal(H_{\infty}/H_{0})$ corresponde al subgrupo mínimo de $G$ que contiene $\overline{[G,G]}=\w_{0}X$ \eqref{lem:conmutatorGr} al igual que los subgrupos de inercia $I_{1},\ldots,I_{s}$. Escribiendo $G$ como el producto semidirecto \eqref{eq:prodsemidir} de $I_{1}$ y $X_{\infty}$ obtenemos $$C_{0}\isom X_{\infty}/T_{0}.$$

Ahora, como $K_{\infty}/K_{r}$ es una $\Zp$-extensión, su grupo de Galois es cíclico generado por $\gamma^{p^{r}}$ y los subgrupos de inercia $I_{i}$ por $\sigma_{i}^{p^{r}}$. Además veamos que
\begin{eqnarray*}
\sigma_{i}^{p^{r}}&=&(a_{i}\sigma_{1})^{p^{r}} \\
&=&a_{i}\sigma_{1}a_{i}\sigma^{-1}_{1}\sigma_{1}^{2}a_{i}\sigma^{-2}_{1} \cdots \sigma_{1}^{p^{r}-1}a_{i}\sigma_{1}^{^{1-p^{r}}}\sigma_{1}^{p^{r}}\\
&=&a_{i}^{\w_r/\w_{0}}\sigma_{1}^{p^{r}}.
\end{eqnarray*}
Procediendo como anteriormente tenemos el resultado. 
\end{proof}

\begin{cor}\label{cor:iwasawa-control-x}
  Para la $\Zp$-extensión ciclotómica de $\Q(\mu_p)$ tenemos para $r\in\Ncero$
  \[ C_r\isom X_\infty/\omega_rX_\infty=X_\infty/(\gamma^{p^r}-1)X_\infty \]
  donde $\gamma$ es un generador topológico de $\Gamma$.
\end{cor}
\begin{proof}
  Esto sigue de la proposición anterior porque en este caso $s=1$ según el
  \cref{lem:totram}, así que $Y_0=\omega_0X_\infty=TX_\infty=(\gamma-1)X_\infty$.
\end{proof}

\begin{cor}\label{cor:x-cero-iff-x-menos-cero}
  Tenemos $X_\infty=0\iff X_\infty^-=0$, donde $X_\infty^-$ es la parte donde la conjugación compleja actúa por $-1$.
\end{cor}
\begin{proof}
  Sea $X_\infty^-=0$. Se ve fácilmente que el isomorfismo del \cref{cor:iwasawa-control-x} es
  compatible con la acción de la conjugación compleja, así que
  $\Cl(\Q(\mu_p))^-=C^-_0\isom X_\infty^-/\omega_0X_\infty^-=0$. Entonces el
  \cref{thm:kummer-hp-plus-minus} implica que $C_0=0$, es decir $p\nmid h_{\Q(\mu_{p})}$. El
  \cref{cor:preg} implica que $p\nmid h_{\Q(\mu_{p^{r}})}$ para cada $r\ge0$, es decir
  $X_\infty=0$. La otra implicación es trivial.
\end{proof}

\begin{prop}\label{prop:iwaforeleme}
Sea $E$ un $\LL$-módulo elemental de $\LL$-torsión tal que $\left\vert E/\left(\frac{\w_{r}}{\w_{e}}E\right)\right\vert$ es finito para todo $r\geq e$.  Entonces existe $c$ tal que
\[\left\vert E/\left(\frac{\w_{r}}{\w_{e}}E\right)\right\vert=p^{\mu p^{r}+\lambda r +c,}\	\	\	\text{ para }r\gg 0,\]
donde $\mu$ y $\lambda$ son los invariantes de la \cref{defi:invariantes}.
\end{prop}
\begin{proof} Basta con analizar los distintos factores que aparecen en $E$. Como $E$ es de $\LL$-torsión, entonces hay solamente dos tipos de factores por analizar.

Sea $X=\LL/(p^{m})$, entonces 
$$X/\dfrac{\w_{r}}{\w_{e}}X \isom \LL / (p^{m},\dfrac{\w_{r}}{\w_{e}})$$
como $\dfrac{\w_{r}}{\w_{e}}$ es un polinomio distinguido $P$ de grado $p^{r}-p^{e}$. Por el lema de división \ref{lema:division} tenemos 
$$\LL/(p^{m},\dfrac{\w_{r}}{\w_{e}})\isom (\Z/p^{m}\Z)_{p^{r}-p^{e}-1}[T],$$
es decir cada elemento de la izquierda es representado por un polinomio de grado menor a $p^{r}-p^{e}$ con coeficientes en $\Z/p^{m}\Z$. Es decir
$$\left\vert X/\dfrac{\w_{r}}{\w_{e}}X \right\vert = p^{m(p^{r}-p^{e})}=p^{mp^{r}+c}$$
para $r\gg0$ y $c$ constante.

Ahora sea $X=\LL/(P)$, donde $P$ es un polinomio distinguido de grado $d$. Para todo $r$ tal que $p^{r-1}\geq d$
tenemos que $\dfrac{\w_{r+2}}{\w_{r+1}}$ actúa en $X$ por multiplicación por $p$ y una unidad (ver el \cref{lem:congruencia-phi-s}). Por lo tanto para $r_{0}\geq e$ tal que $p^{r_{0}-1}\geq d$ tenemos
$$\dfrac{\w_{r_{0}+2}}{\w_{e}}X=\dfrac{\w_{r_{0}+2}}{\w_{r_{0}+1}}\left(\dfrac{\w_{r_{0}+1}}{\w_{e}}X\right)=p\dfrac{\w_{r_{0}+1}}{\w_{e}}X,$$
y por inducción obtenemos 
$$\dfrac{\w_{r}}{\w_{e}}X=p^{r-r_{0}-1}\dfrac{\w_{r_{0}+1}}{\w_{e}}X \	\text{ para }\	r>r_{0}.$$ 
Al pasar al cociente obtenemos
\begin{eqnarray*}
\left\vert X/\dfrac{\w_{r}}{\w_{e}}X\right\vert &= & \left\vert X/p^{r-r_{0}-1} \dfrac{\w_{r_{0}+1}}{\w_{e}}X \right\vert \\
&=& \left\vert X/p^{r-r_{0}-1}X \right\vert\left\vert p^{r-r_{0}-1}X/p^{r-r_{0}-1} \dfrac{\w_{r_{0}+1}}{\w_{e}}X \right\vert \\
&=& \left\vert X/p^{r-r_{0}-1}X \right\vert\left\vert X/\dfrac{\w_{r_{0}+1}}{\w_{e}}X \right\vert \\
&=& p^{d(r-r_{0}-1)} \left\vert X/\dfrac{\w_{r_{0}+1}}{\w_{e}}X \right\vert
\end{eqnarray*}
Finalmente como $\left\vert X/\dfrac{\w_{r}}{\w_{e}}X \right\vert$ es finito para todo $r\geq 0$, tenemos  que existe una constante $c$ tal que
$$\left\vert X/\dfrac{\w_{r}}{\w_{e}}X\right\vert=p^{dr+c},$$
para $r$ suficientemente grande. \end{proof}

\begin{thm}[Iwasawa]\label{thm:iwasawa-ThmdIwa}
  Sea $K$ un campo de números y $p$ un primo. Para cada $r\in\Ncero$ sea $K_r/K$ una extensión
  tal que $\Gal(K_r/K)\isom\Z/p^r\Z$, y sea $p^{e_r}$ la máxima potencia de $p$ que divide
  el orden del grupo de clases de $K_r$. Entonces existen constantes $\mu,\lambda\in\Ncero$ y
  $\nu\in\Z$ tal que \[ e_r=\mu p^r+\lambda r + \nu \quad\text{para }r\gg0. \]
\end{thm}
\begin{proof} Sea $K_{\infty}/K$ una $\Zp$-extensión. Existe $r_{0}\geq 0$ tal que todos los primos que ramifican en la extensión $K_{\infty}/K_{r_{0}}$ son totalmente ramificados. Aplicando la \cref{prop:Crdescenso} a la extensión $K_{\infty}/K_{r_{0}}$ vemos que 
$$C_{r}\isom X_{\infty}/\dfrac{\w_{r}}{\w_{r_{0}}}T_{r_{0}} \	\	\text{ para todo }\	\	r\geq r_{0},$$
donde $T_{r_{0}}$ es el módulo generado por $\w_{r_{0}}$ y los respectivos $a_{i}$, que dependen de los subgrupos de inercia de $\Gal(H_{\infty}/K_{r_{0}})$.

Ahora, como $C_{r_{0}}\isom X_{\infty}/T_{r_{0}}$ es finito, tenemos que
\begin{eqnarray*}
\left\vert C_{r} \right\vert &=& \left\vert X_{\infty}/T_{r_{0}} \right\vert \left\vert T_{r_{0}}/\dfrac{\w_{r}}{\w_{r_{0}}} T_{r_{0}} \right\vert \\
&=& p^{c} \left\vert T_{r_{0}}/\dfrac{\w_{r}}{\w_{r_{0}}} T_{r_{0}} \right\vert 
\end{eqnarray*}
para alguna constante $c \geq 0$, además tenemos que 
$$T_{r_{0}}\sim X_{\infty} \sim E=\left( \bigoplus_{i=1}^{m}\LL/\pi^{\mu_{i}}\LL\right)\oplus \left( \bigoplus_{j=1}^{l}\LL/P_{j}\LL\right)$$
con $E$ elemental como en el \cref{thm:estructura} y sin factor libre por el \cref{cor:x-noetheriano-torsion}. Por lo tanto basta determinar el factor $\left\vert T_{r_{0}}/\dfrac{\w_{r}}{\w_{r_{0}}} T_{r_{0}} \right\vert $.
Consideremos el siguiente diagrama para $r\geq r_{0}$
\begin{equation*}
    \begin{tikzcd}
      0\arrow[r] & \dfrac{\w_{r}}{\w_{r_{0}}} T_{r_{0}} \arrow[r] \arrow[d, "\varphi'_{r}"]
      & T_{r_{0}} \arrow[d, "\varphi"] \arrow[r]
      & T_{r_{0}}/\dfrac{\w_{r}}{\w_{r_{0}}} T_{r_{0}} \arrow[d, "\varphi''_{r}"] \arrow[r] & 0 \\
      0\arrow[r] & \dfrac{\w_{r}}{\w_{r_{0}}} E \arrow[r]
      & E \arrow[r]
      & E/\dfrac{\w_{r}}{\w_{r_{0}}} E \arrow[r] & 0
    \end{tikzcd}
\end{equation*} 
Por el \cref{ejer:pseudo-isomorfo-cociente-finito} tenemos que $|\ker\varphi'_{r}$|, $|\coker\varphi'_{r}|$, $|\ker\varphi''_{r}|$, $|\coker\varphi''_{r}|$ son finitos, además son constantes para $r$ suficientemente grande ya que: 
\begin{itemize}
\item Las sucesiones $(|\ker\varphi'_{r}|)_{r\geq r_{0}}$, $(|\coker\varphi'_{r}|)_{r\geq r_{0}}$ y $(|\coker\varphi''_{r}|)_{r\geq r_{0}}$ son monótonas y acotadas (ver el \cref{ejer:mononoque}).  
\item Por el lema de la serpiente tenemos que 
$$|\ker\varphi''_{r}|=|\ker\varphi'_{r}|^{-1}
|\ker\varphi||\coker\varphi'_{r}||\coker\varphi|^{-1}|\coker\varphi''_{r}|,$$
y para $r\gg 0$ los términos de la derecha son constantes por el inciso anterior. 
\end{itemize}
Es decir existe una constante $c'$ tal que
$$\left\vert T_{r_{0}}/\dfrac{\w_{r}}{\w_{r_{0}}} T_{r_{0}} \right\vert = p^{c'} \left\vert E/\dfrac{\w_{r}}{\w_{r_{0}}} E \right\vert$$
para $r$ suficientemente grande. El teorema sigue directamente aplicando la \cref{prop:iwaforeleme}. 
\end{proof}

\ejercicios

\begin{ejer} Sea $K$ un campo de números y $K_{\infty}/K$ una $\Zp$-extensión. Supongamos que una única plaza $\p$ sobre $p$ ramifica en $K_{\infty}$. Entonces
\begin{itemize}
\item[(a)] $C_{r}\isom X_{\infty} / \w_{r} X_{\infty}$. Pista: Use la \cref{prop:Crdescenso}.
\item[(b)] Demuestre que si $C_{r}=0$ entonces $X_{\infty}=0$. Pista: Use el \cref{cor:cor-de-nakayama} del Lema de Nakayama Topológico.
\end{itemize}
\end{ejer}

\begin{ejer} Sea $p$ es un número primo regular, es decir $p$ no divide el grupo de clases de $K=\Q(\zeta_{p})$. Use el ejercicio anterior y el teorema de Iwasawa (\cref{thm:iwasawa-ThmdIwa}) en la extensión $K^{c}/K$ para demostrar que $e_{r}=0$ para todo $r\geq 0$.
\end{ejer}

\begin{ejer}\label{ejer:mononoque}Sea $T$ un $\LL$-módulo de torsión y $\varphi:T \rightarrow E$ un pseudo-isomorfismo donde $E$ es un módulo elemental. Con la notación del \cref{ejer:pseudo-isomorfo-cociente-finito} sea $\alpha=\beta=\dfrac{\w_{r}}{\w_{e}}$ y supongamos que $\dfrac{\w_{r}}{\w_{e}}T$ es finito para todo $r\geq e$.  Suponga que $s\geq r$. 
\begin{itemize}
\item[(a)] Demuestre que $\dfrac{\w_{s}}{\w_{e}}T\subset \dfrac{\w_{r}}{\w_{e}}T$ y $\dfrac{\w_{s}}{\w_{e}}E\subset \dfrac{\w_{r}}{\w_{e}}E$.
\item[(b)] Concluya que $|\ker\varphi'_{s}|\leq|\ker\varphi'_{r}|$ y $|\coker\varphi''_{r}|\leq|\coker\varphi''_{s}|$.
\item[(c)] Demuestre que si $x$ es un levantamiento de un elemento en $\coker\varphi'_{r}$ entonces $\dfrac{\w_{s}}{\w_{r}}x$ es un levantamiento de un elemento en $\coker\varphi'_{s}$ .
\item[(d)] Concluya que $|\coker\varphi'_{s}|\leq |\coker\varphi'_{r}|$.
\end{itemize}
\end{ejer}

\section{Sobre los invariantes $\mu$ y $\lambda$}\label{sec:invariantesmulam}

El \cref{thm:iwasawa-ThmdIwa} describe el crecimiento de la parte $p$ de los grupos de clases asociados a los subcampos finitos (suficientemente grandes) de una $\Zp$-extensión. El exponente de estos grupos, depende de un término exponencial $\mu \cdot p^{r}$ y de un término lineal $\lambda\cdot r + \nu$. Los invariantes $\mu$, $\lambda$ y $\nu$, están relacionados con la estructura de $X_{\infty}$ como $\LL$-módulo. En esta sección haremos un compendio de algunos resultados conocidos sobre los invariantes de Iwasawa. 

La siguiente proposición nos provee de criterios para determinar cuándo el término exponencial no aparece en el exponente. 

\begin{prop} \label{prop:XnoettorsiirgFpfinsiiacotFp} Sea $X$ un $\LL$-módulo noetheriano de torsión. Entonces
\begin{eqnarray*}
\mu(X)=0 &\iff & \rg_{\Fp} X <\infty \\
&\iff & \rg_{\Fp} X/\w_{r} X \text{ está acotado para todo }	r\geq 0.
\end{eqnarray*}
\end{prop}
\begin{proof} El enunciado es invariante bajo pseudo-isomorfismos. Entonces basta demostrar la proposición para un $\LL$-módulo elemental 
$$E=\left( \bigoplus_{i=1}^{m}\LL/(p^{\mu_{i}})\right)\oplus \left( \bigoplus_{j=1}^{l}\LL/(P_{j})\right).$$ 

Es claro que $\mu(E)=0$ si y solamente $\rg_{\Fp}E<\infty$. 

Ahora sea $r$ tal que $\ell^{r}\geq \max_{j=1,...,t}\{\deg P_{j}\}$. Entonces
\begin{eqnarray*}
(E/\w_{r}E)/p(E/\w_{r}E)  & \isom & E/(p,\w_{r})E \\
& \isom & \left( \bigoplus_{i=1}^{m}\LL/(p,\w_{r})\right)\oplus \left( \bigoplus_{j=1}^{l}\LL/(p,\w_{r},P_{j})\right)\\
& \isom & \left( \bigoplus_{i=1}^{m}\LL/(p,T^{\ell^{r}})\right)\oplus \left( \bigoplus_{j=1}^{l}\LL/(p,T^{\deg P_{j}})\right)\\
& \isom & \F_{p}^{m(\ell^{r})+\lambda},
\end{eqnarray*}
es decir $\mu(E)=0$ si y solamente si $\rg_{\F_{p}}E/\w_{r} E$ está acotado para todo $r\geq 0$. 
\end{proof}

Observe que la proposición anterior es de cierta manera análoga a la \cref{prop:Xnoettorsiirgfinsiiacot}.

\begin{defi} Sea $K_{\infty}/K$ una $\Zp$-extensión y $X_{\infty}$ el $\LL$-módulo de torsión asociado (cf. \cref{defi:iwasawa-modulos}). Denotamos 
$$\mu(K_{\infty}/K):=\mu(X_{\infty})\	\	\	y\	\	\	\lambda(K_{\infty}/K):=\lambda(X_{\infty}).$$
\end{defi}

En el caso especial de la $\Zp$-extensión ciclotómica $K^{c}/K$ de un campo de números, Iwasawa formuló la siguiente famosa conjetura.

\begin{conj}[Conjetura de Iwasawa]\label{conj:iwasawa}\index[def]{conjetura!de Iwasawa} Sea $K$ un campo de números. Entonces 
$$\mu(K^{c}/K)=0.$$
\end{conj}

El resultado más amplio conocido hasta el momento es el siguiente teorema de Ferrero y Washington \cite{MR528968}.

\begin{thm}[\importante{Teorema de Ferrero-Washington}]\label{thm:ferrero-washington} Sea $K/\Q$ una extensión abeliana de $\Q$ y $p$ un número primo. Entonces $$\mu(K^{c}/K)=0.$$
\end{thm}

La prueba del teorema de Ferrero-Washington hace uso de la función $L$ $p$-ádica. Al lector interesado en la demostración lo invitamos a consultar el artículo original \cite{MR528968} o \S 7.5 en \cite{MR1421575}.

No obstante, Iwasawa demostró que existen $\Zp$-extensiones $K_{\infty}/K$ con $\mu(K_{\infty}/K)>0$.  

Sea $K=\Q(\sqrt{-d})$ con $\left(\dfrac{-d}{p}\right)\neq 1$, y $K_{\infty}/K$ la $\Zp$-extensión \emph{anti-ciclotómica}\index[def]{extensiónZp@$\Zp$-extensión!anti-ciclotómica} de $K$, es decir la única $\Zp$-extensión de $K$ donde la conjugación compleja actúa por inversión.  Además, consideremos $L$ una extensión de grado $p$ y Galois sobre $\Q$. Denotamos $L_{\infty}=K_{\infty}L$, claramente $L_{\infty}/L$ es una $\Zp$-extensión. 

\begin{thm}[Iwasawa] Con la notación del párrafo precedente. Supongamos que $s$ primos distintos de $p$ son inertes $K/\Q$ y ramifican en $L/K$. Entonces 
$$\mu(L_{\infty}/L) \geq s-1.$$
\end{thm}

\begin{ex}\label{eje:mumayor0} Sea $K=\Q(\zeta_{3})$ y $L=\Q(\zeta_{3},\sqrt[3]{22})$. El discriminante de $\Q(\sqrt[3]{22})$ es $-27\cdot 22^{2}$ y de $\Q(\zeta_{3})$ es $-3$, por la fórmula de discriminantes en torres números (e.g. \cite[III 2.10 Cor.]{MR1697859}) tenemos que $2,3$ y $11$ ramifican en $L/\Q$. Además, $2$ y $11$ no se descomponen en $K$. Así pues el teorema anterior implica que $\mu(L_{\infty}/L)\geq 1$. 
\end{ex}

Recordemos que según la conjetura de Leopoldt (\cref{conj:Leopoldt}) un campo de números totalmente real tiene una sola $\Zp$-extensión, es decir la $\Zp$-extensión ciclotómica. En dado caso, Greenberg formuló la siguiente conjetura sobre la nulidad de los invariantes $\mu$ y $\lambda$.

\begin{conj}[Conjetura de Greenberg]\label{conj:greenberg}\index[def]{conjetura!de Greenberg} Sea $K$ un campo de números totalmente real. Entonces 
$$\mu(K^{c}/K)=0\	\	\	\text{ y }\	\	\	\lambda(K^{c}/K)=0.$$
Es decir, el exponente del grupo de clases de $K_{r}\subset K^{c}$ es acotado para $r\to \infty$.
\end{conj}

Como hemos visto, las conjeturas y resultados sobre los invariantes involucran sea la $\Zp$-extensión ciclotómica, sea el invariante $\mu$. Parece ser que el estudio del invariante $\lambda$ depende del buen comportamiento o conocimiento del invariante $\mu$, como veremos a continuación.

Recordemos que un campo de números $K$ con $c$ pares de encajes complejos conjugados $K\hookrightarrow \C$, tiene al menos $c+1$ $\Zp$-extensiones independientes sobre $K$. Esto implica que si $c\geq 1$, entonces existe un número infinito de $\Zp$-extensiones $K_{\infty}/K$. 

Sea $\Delta(K)$ el conjunto de todas las $\Zp$-extensiones de $K$. Podemos equipar $\Delta(K)$ con la siguiente \define[topología de Greenberg]{topología introducida por Greenberg} en \cite{MR332712}. Sea $K_{\infty}\in\Delta(K)$ y $r\geq 0$ un número natural, definimos
$$\Delta(K_{\infty},r):=\{K'_{\infty}\in\Delta\,|\,[K_{\infty}\cap K'_{\infty}:K]\geq p^{r}\}.$$
Los conjuntos $\Delta(K_{\infty},r)$ forman una base para la topología del límite inverso $\Delta(K)=\varprojlim \Delta_{r}(K)$, donde $\Delta_{r}(K)$ es el espacio discreto de las $\Zp$-extensiones de $K$ que coinciden hasta el nivel $r$ con $K_{\infty}$. 

Decimos que una plaza $\p$ de $K$ es finitamente descompuesta en $K_{\infty}$ si la imagen del subgrupo de descomposición $D_{\p}\subset \Gal(\bar{K}/K)$ en $\Gal(K_{\infty}/K)$ tiene índice finito. El subconjunto $\Delta^{0}(K)$, que consiste de las extensiones $K_{\infty}\in\Delta$ tales que todas las plazas sobre $p$ son finitamente descompuestas, es denso en $\Delta(K)$ \cite[Prop. 3]{MR332712}.

\begin{thm}[Greenberg]\label{thm:Greenbound} Sea $K_{\infty}\in\Delta^{0}(K)$. Entonces 
\begin{itemize}
\item[(i)] El invariante $\mu$ es acotado en una vecindad de $K_{\infty}$.
\item[(ii)] Si $\mu(K_{\infty}/K)=0$, entonces en una vecindad de $K_{\infty}$ los invariantes $\mu$ son nulos y los invariantes $\lambda$ acotados.
\end{itemize}
\end{thm}

Es decir, la filosofía que nos transmitió Greenberg es pensar los invariantes $\mu$ y $\lambda$ como funciones continuas
$$\mu,\lambda : \Delta \longrightarrow \Ncero.$$
En particular, uno se puede preguntar si dichas funciones están acotadas. En los años 80,
Baba\u\i cev \cite{Babauicev80} y Monsky \cite{Monsky81SomeInvariantsZpdExt} respondieron independientemente una de estas preguntas.

\begin{thm}[Baba\u\i cev/Monsky]\label{thm:baba-monky} El invariante $\mu$ es acotado en $\Delta(K)$. 
\end{thm}

Recientemente, trabajos de Kleine \cite{MR3627696} haciendo uso de una topología en $\Delta(K)$ que toma en cuenta la ramificación, demuestran que de hecho el invariante $\mu$ es localmente máximo y dan condiciones suficientes para que el invariante $\lambda$ sea localmente máximo. 

\ejercicios

\begin{ejer} Demuestre que si $K\subseteq K'$ y $K_{\infty}/K$ es una $\Zp$ extensión y $K'_{\infty}=K_{\infty}K'$ entonces $\mu(K_{\infty}/K)\leq\mu(K'_{\infty}/K')$.
\end{ejer}
\begin{ejer} Generalice tanto como pueda el \cref{eje:mumayor0}. 
\end{ejer}
\begin{ejer} Sea $K_{\infty}/K$ una $\Zp$-extensión de un campo $K$ de tipo CM, es decir una extensión cuadrática imaginaria de un campo totalmente de real (e.g. $\Q(\sqrt{-d})$). Si $A$ es un grupo abeliano en el que la conjugación compleja actúa como un automorfismo, denotamos $A^{+}$ y $A^{-}$ las partes en que actúa trivialmente y por $-1$, respectivamente. Demuestre que
$$\mu=\mu^{+}+\mu^{-}.$$ 
\end{ejer}
\begin{ejer} Sea $K$ es un campo de números y $K_{\infty}/K$ una $\Zp$-extensión. Sea $r$ como en el \cref{thm:iwasawa-ThmdIwa} aplicado a $K_{\infty}/K$ y supongamos que $R\geq r$. Si $K'=K_{R}$, entonces $K_{\infty}/K'$ es una $\Zp$-extensión. Demuestre que aplicando el \cref{thm:iwasawa-ThmdIwa} a $K_{\infty}/K'$ tenemos que
$$\mu'=\mu p^{R},\	\	\	\lambda'=\lambda\	\	\	\text{ y }\	\	\	\nu'=\nu+\lambda R.$$
\end{ejer}

\chapter{Funciones $L$ $p$-ádicas}
\label{sec:palf}

En la teoría de las funciones $L$, como la función zeta de Riemann, hay dos fenómenos
importantes que implican la existencia de análogos $p$-ádicos de esas funciones. Primero,
algunos valores especiales de esas funciones, que a priori son números complejos, de hecho
son números algebraicos o incluso racionales (en general después de dividir por un factor
transcendental normalizante). Y segundo, esos valores son en un sentido $p$-ádicamente
continuos, que significa que hay congruencias entre ellos. De la definición no es inmediato
que las funciones $L$ tienen estas propiedades: Algunos valores especiales ni siquiera están
bien definidos después de considerar la continuación analítica de la función. Sin duda estos
fenómenos son inmensamente sorprendentes, ocasionando que la existencia de las funciones $L$
$p$-ádicas en general no sea nada trivial.

Aquí explicamos la teoría de las funciones $L$ $p$-ádicas en el caso más básico: la función zeta
de Riemann, y ligeramente más general, las funciones $L$ de Dirichlet. La idea de construir
un tal análogo $p$-ádico de funciones $L$ es originalmente de Kubota y Leopoldt \cite{MR0163900} y fue
más tarde reinterpretada por Iwasawa, como explicaremos en la \cref{sec:stickelberger}.

\section{Proemio sobre las funciones $L$}
\label{sec:proemio-l}

Aquí coleccionamos unos hechos analíticos sobre las funciones $L$ que nos interesan. Como este
texto pone su énfasis más en la teoría algebraica omitimos algunas de las demostraciones
en esta sección.

\begin{defi}
  Una \define{serie de Dirichlet} es una serie de la forma
  \[ \sum_{n=1}^\infty a_n n^{-s} \]
  con coeficientes $a_n\in\C$.
\end{defi}

El conducto de convergencia de este tipo de series es explicado por el siguiente resultado.

\begin{prop}
  Para cada serie de Dirichlet \[ \sum_{n=1}^\infty a_n n^{-s} \] existe un
  $\sigma_0\in[-\infty,\infty]$ tal que la serie converge localmente uniformemente para los
  $s\in\C$ con $\Re s>\sigma_0$ y diverge para $\Re s<\sigma_0$. Este $\sigma_0$ se llama
  \define{abscisa de convergencia}.

  La función definida por una serie de Dirichlet en el semiplano de convergencia es
  holomorfa.

  Si los coeficientes son multiplicativos en el sentido \[ a_{nm}=a_na_m \text{ para todo
    }m,n\in\Nuno \] entonces la serie tiene un \define{producto de Euler}
  \[ \sum_{n=1}^\infty a_n n^{-s} = \prod_{\ell\text{ primo}}(1-a_\ell\ell^{-s})^{-1} \]
  para los $s\in\C$ con $\Re s>\sigma_0$.
\end{prop}
\begin{proof}
  \cite[§1, Satz 1, §2, Satz 1]{MR0631688}
\end{proof}

En el resultado anterior, el producto significa lo siguiente. Para coeficientes
$c_n \in \C^\times$ decimos que el producto \[ \prod_{n=1}^\infty c_n \] \define[producto
convergente]{converge} si el límite \[ \lim_{N \to \infty} \prod_{n=1}^N c_n \] existe y es
diferente de $0$.

El ejemplo más importante es la serie de Dirichlet con $a_n=1$ para cada $n$.

\begin{defi}
  La \emph{función zeta de Riemann}\index[def]{función!zeta de Riemann} es la función
  \[ \zeta(s)=\sum_{n=1}^\infty n^{-s}=\prod_{\ell\text{ primo}}(1-\ell^{-s})^{-1} \] para
  $s\in\C$.
  Su abscisa de convergencia es $\sigma_0=1$.
\end{defi}

Muchas veces las series de Dirichlet tienen una continuación analítica. Como preparación
a esto demostramos primero la \emph{fórmula de transformación
  de Mellin}\index[def]{formula@fórmula!de transformación
  de Mellin}, que involucra la \emph{función $\Gamma$}\index[def]{función!$\Gamma$}
\begin{equation*}
  \Gamma(s)=\int_0^\infty t^{s-1}\e^{-t}\integrald t \quad (s\in\C,\ \Re(s)>0).
\end{equation*}
Es fácil ver que
\begin{equation}
  \label{eqn:convergencia-integral}
  \int_0^1 t^{s-1}\integrald t \text{ converge }\iff\Re(s)>0 \quad \text{para }s\in\C
\end{equation}
y que $\int_1^\infty g(t)t^{s-1}\integrald t$ converge para cada $s\in\C$ si $g$ es una
función que decrece exponencialmente para $t\to\infty$, por eso la integral de arriba
que define la función $\Gamma$ converge para los $s\in\C$ con parte real positiva.
Las propiedades básicas de la función $\Gamma$ son alistadas en el siguiente resultado.

\begin{thm}\label{thm:gammafkt}
  La función $\Gamma$ se extiende a una función meromorfa en todo de $\C$ con polos simples
  en los enteros $\le0$ y sin ceros. Los residuos están dados por
  \[ \operatorname*{Res}\limits_{s=-n}\Gamma(s)=\frac{(-1)^n}{n!} \quad\text{para
    }n\in\Nuno. \]
  Ademas tenemos $\Gamma(s+1)=s\Gamma(s)$ para cada $s\in\C$.
\end{thm}
\begin{proof}
  \cite[Prop.\ IV.1.2]{MR2513384}
\end{proof}

\begin{lem}[fórmula de transformación de Mellin]
  Sea $L(s)=\sum_{n=1}^\infty a_nn^{-s}$ una serie de Dirichlet con abscisa de convergencia
  $\sigma_0$ y sea
  \[ F(t):=\sum_{n=1}^\infty a_n\e^{-nt} \]
  que converge para $t\ge0$.
  Entonces
  \[ \Gamma(s)L(s)=\int_0^\infty F(t)t^{s-1}\integrald t \]
  para $s\in\C$ con $\Re(s)>\sigma_0$.
\end{lem}
\begin{proof}
  Los coeficientes $a_n$ pueden crecer a lo más polinomialmente, de lo contrario la serie
  de Dirichlet $L(s)$ sería divergente para todo $s\in\C$. Por eso la serie que define
  $F(t)$ converge absolutamente para $t\ge0$.
  
  De la fórmula de arriba que define la función $\Gamma$ se obtiene fácilmente
  \begin{equation*}
    \int_0^\infty t^{s-1}\e^{-nt}\integrald t =\Gamma(s)n^{-s} \quad(n\in\Nuno)
  \end{equation*}
  (con una sustitución $r=nt$). De esto la fórmula resulta directamente, porque $L(s)$
  converge localmente uniformemente.
\end{proof}

El resultado siguiente nos da la continuación analítica y describe los valores en los
enteros negativos. Ya que este resultado es de fondo para nuestra análisis de las funciones
$L$ vamos a esbozar su demostración.

\todo{A partir de aquí hasta el siguiente todo todavía no está listo.}

\begin{prop}\label{continuacion-analitica}
  Sea \[ L(s) = \sum_{n=1}^\infty a_n n^{-s} \] una serie de Dirichlet cuya abscisa de
  convergencia $\sigma_0$ cumple $\sigma_0\le1$. Supongamos que la suma
  \[ F(t)=\sum_{n=1}^\infty a_n\e^{-nt} \]
  define una función holomorfa en $\C\setminus\{0\}$ con posiblemente un polo
  en $t=0$ de orden $\le1$ (o holomorfa en todo de $\C$).
  
  Entonces la función $L$ tiene una continuación meromorfa a todo $\C$ con único polo simple
  en $s=1$ si $F$ tiene un polo en $t=0$ y holomorfa si no. Además,
  \[ L(1-n) = -\frac{b_n}n \] para $n\in\Nuno$.
\end{prop}
\begin{proof}
  Esta demostración es combinada de \cite[§7, Satz 1]{MR0631688} y \cite[Lem.\
  1.1.1]{ColmezTsinghua}; véase allá para más detalles.

  Porque $L(s)$ converge para algún $s\in\C$, la función $F$ decrece rápidamente, es decir,
  para cada $m\in\Nuno$ tenemos que $t^mF(t)\to0$ para $t\to\infty$.
  
  Para $a,b\in\{0,1,\infty\}$ y una función meromorfa $G\colon\C\rightarrow\C$ con único
  polo posiblemente en $s=0$ usaremos la notación
  \begin{equation*}
    I_{a,b}(G,s):=\int_a^bG(t)t^{s-1}\integrald t \quad (s\in\C)
  \end{equation*}
  para los $s\in\C$ tal que este integral converge. El conducto de convergencia es descrito
  en el \cref{ejer:integrales}.

  Según la fórmula de transformación de Mellin tenemos que
  \begin{equation*}
    L(s)=\frac1{\Gamma(s)}(I_{0,1}(F,s)+I_{1,\infty}(F,s))
  \end{equation*}
  para $s\in\C$ con $\Re(s)>1$. Según el \cref{ejer:integrales} el integral
  $I_{1,\infty}(F,s)$ converge para cada $s\in\C$
  y $I_{0,1}(F,s)$ se extiende
  meromorficamente a $\C$ con polos de orden $\le1$ en $\{s\in\Z : s\le1\}$. Porque
  $\Gamma(s)$ tiene polos simples en $\{s\in\Z : s\le0\}$ y no tiene ceros
  (\cref{thm:gammafkt}) obtenemos la afirmación sobre la continuación meromorfa de $L(s)$.

  Con integración por partes obtenemos que para $\Re(s)>1$
  \begin{equation*}
    \frac{1}{\Gamma(s)}I_{0,1}(F,s)
    =\frac{1}{\Gamma(s)}\left({\left[F(t)\frac{t^s}s\right]}_{t=0}^1-\int_0^1F'(t)\frac{t^s}s\integrald
      t\right)
    = \frac1{\Gamma(s+1)}\left(F(1)-I_{0,1}(F',s+1)\right).
  \end{equation*}
  Inductivamente obtenemos para estos $s$
  \begin{equation*}
    \frac1{\Gamma(s)}I_{0,1}(F,s)=
    \sum_{j=1}^N(-1)^{j+1}\frac{1}{\Gamma(s+j)}F^{(j-1)}(1)+(-1)^N\frac1{\Gamma(s+N)}I_{0,1}(F^{(N)},s+N)
  \end{equation*}
  para cada $N\in\Nuno$. El integral $I_{0,1}(F^{(N)},s+N)$ converge para $\Re(s)>N+1$ según
  el \cref{ejer:integrales}. Porque $\Gamma(s)$ tiene polos simples en $\{s\in\Z : s\le0\}$
  y no tiene ceros \cref{thm:gammafkt}) obtenemos la afirmación sobre la continuación
  meromorfa de $L(s)$.
\end{proof}

En el caso de la función zeta, tenemos
\[ f(t) = -t\sum_{n=1}^\infty {(\e^{t})}^{n}=\frac{t}{1-\e^{-t}} \quad (t\le0). \] Esta función
es holomorfa en todo de $\C$, por eso es analítica y tiene una representación
\begin{equation}
  \label{eqn:bernoulli}
   \frac t{1-\e^{-t}} = \sum_{n=0}^\infty B_n\frac{t^n}{n!} \quad (t\in\C)
\end{equation}
con coeficientes $B_n$, que en este caso es la serie de Taylor. Por eso
\[ B_n=f^{(n)}(0)\in\Q \quad\text{para }n\in\Ncero \]
y vemos que los coeficientes $B_n$ son racionales.

\todo{A partir de aquí todo está bien.}

\begin{defi}\label{defi:bernoulli}
  Los coeficientes $B_n\in\Q$ definidos por la ecuación \eqref{eqn:bernoulli} se llaman
  \define{números de Bernoulli}.\footnote{En algunos textos se encuentra la definición con
    $f(t)=\frac t{\e^t-1}$ en lugar de la $f$ de arriba. Esto da los mismos números de
    Bernoulli salvo para $B_1$, en cuyo caso da $B_1=-\frac12$ en lugar de $B_1=\frac12$.}
  Notemos que $B_n=0$ si $n>1$ es impar, porque $f(t)=f(-t)+t$.
\end{defi}

Esto demuestra:
\begin{prop}\label{zeta-continuacion}
  La función zeta de Riemann tiene una continuación meromorfa a todo $\C$ con un único polo en $s=1$
  de orden $1$, y
  \[ \zeta(1-n)=-\frac{B_n}n\in\Q \]
  para $n\ge1$, los $B_n$ siendo los números de Bernoulli.
\end{prop}

\begin{ex}\label{ex:bernoulli}
  Los primeros números de Bernoulli no nulos son
  \begin{multline*}
    B_0=1,\; B_1=\frac12,\; B_2=\frac16,\; B_4=-\frac1{30},\; B_6=\frac1{42},\;
    B_8=-\frac1{30}, \\
    B_{10}=\frac5{66},\; B_{12}=-\frac{691}{2730},\;
    B_{14}=\frac76,\; B_{16}=-\frac{3617}{510}, \;
    \dotsc .
  \end{multline*}  
\end{ex}

Es útil estudiar también series de Dirichlet ligeramente más generales. Para esto fijamos un
encaje $\Qbar\hookrightarrow\C$.

\begin{defi}\label{defi:l-chi}
  Sea $N\in\Nuno$. Un \emph{carácter de Dirichlet}\index[def]{carácter!de Dirichlet} es un homomorfismo de grupos
  \[ \chi\colon(\Z/N\Z)^\times\rightarrow\Qbar^\times. \] Se llama \emph{primitivo}\index[def]{carácter!primitivo} si no
  se factoriza a través de $(\Z/M\Z)^\times$ para algún $M\mid N$, $M\neq N$, y en este caso
  $N$ se llama el \define{conductor} de $\chi$. Los valores de $\chi$ son raíces de la
  unidad y generan un subcampo de $\Qbar$ que llamamos $\Q(\chi)$.

  Si $\chi$ es un carácter de Dirichlet, definimos una función que también llamamos $\chi$
  así \[ \chi\colon\Nuno\rightarrow\Qbar^\times,\quad n\mapsto
    \begin{cases} \chi(n\mod N),& (N,n)=1,\\ 0,&(N,n)>1. \end{cases} \]

  La \emph{función $L$ de Dirichlet asociada a $\chi$}\index[def]{función!$L$ de Dirichlet asociada a $\chi$} es la serie de Dirichlet
  \[ L(\chi,s)=\sum_{n=1}^\infty\chi(n)n^{-s}. \] Aunque esta función depende del encaje
  escogido, no lo incluimos en la notación porque siempre será claro del contexto.

  La función $\chi$ es claramente multiplicativa, así que la función $L(\chi,-)$ tiene un
  producto de Euler
  \[ L(\chi,s)=\prod_{\ell\text{ primo}}(1-\chi(\ell)\ell^{-1})^{-1}. \]
\end{defi}

En este caso la abscisa de convergencia es $\sigma_0=0$ (salvo si $\chi$ es trivial; véase
\cite[p.\ 42]{MR0631688}). Otra vez podemos aplicar la \cref{continuacion-analitica} para
obtener la continuación analítica y formulas para los valores en los enteros negativos,
de manera similar a la \cref{zeta-continuacion}. Omitimos los detalles y simplemente enunciamos
el resultado.

\begin{prop}\label{prop:numeros-de-bernoulli-gen}
  Sea $\chi$ un carácter de Dirichlet de conductor $N>1$. Ponemos
  \[ f_\chi(t) := \sum_{a=1}^N \frac{\chi(a)t\e^{at}}{\e^{Nt}-1}\quad (t\in\C). \] Esta función se
  escribe como una serie
  \[ f_\chi(t) = \sum_{n=0}^\infty B_{n,\chi} \frac{t^n}{n!} \quad(t\in\C) \] con coeficientes
  $B_{n,\chi}\in\Q(\chi)$ que se llaman \emph{números de Bernoulli generalizados por
    $\chi$}\index[def]{números de Bernoulli@números de Bernoulli!generalizados por
    $\chi$}. La función $L(\chi,-)$ tiene una continuación holomorfa a todo $\C$ y
  \[ L(\chi,1-n)=-\frac{B_{n,\chi}}n\in\Q(\chi) \quad \text{para cada }n\ge1. \]
\end{prop}

Terminamos esta sección introduciendo los polinomios de Bernoulli, que serán usados en las siguientes demostraciones.

\begin{defi}\label{defi:polinomios-de-bernoulli}
  Para cada $n\in\Nuno$ definimos un polinomio \[ B_n(X) = \sum_{i=0}^n\binom n i B_iX^{n-i}
    \in\Q[X] \]
  que se llama el \emph{polinomio de Bernoulli}\index[def]{polinomio!de Bernoulli} $n$-ésimo.
\end{defi}

\begin{lem}\label{lem:polinomios-de-bernoulli}
  Definimos una función
  \[ F(t,x)=\frac{t\e^{t(1+x)}}{\e^t-1}\quad(t,x\in\C). \] Entonces para cada $x,t\in\C$
  tenemos
  \[ F(t,x)=\sum_{n=0}^\infty B_n(x)\frac{t^n}{n!}. \]
\end{lem}
\begin{proof}
  Como $F(t,x)=f(t)\e^{tx}$ para $t,x\in\C$, esto resulta directamente de la fórmula del
  producto de Cauchy:
  \begin{equation*}
    F(t,x)=\sum_{n=0}^\infty B_n\frac{t^n}{n!}\sum_{n=0}^\infty
    x^n\frac{t^n}{n!}= \sum_{n=0}^\infty\sum_{i=0}^n\binom n
    i B_ix^{n-i}\frac{t^n}{n!}=  \sum_{n=0}^\infty B_n(x)\frac{t^n}{n!}.
  \end{equation*}
\end{proof}

\ejercicios

\begin{ejer}
  Demuestre que la función zeta de Riemann tiene abscisa de convergencia $\sigma_0=1$ y las
  funciónes $L$ de Dirichlet para caracteres de Dirichlet no triviales tienen abscisa de
  convergencia $\sigma_0=0$.
\end{ejer}

\begin{ejer}
  Demuestre que una serie de Dirichlet cuyos coeficientes son multiplicativos tiene un
  producto de Euler.
\end{ejer}

\begin{ejer}\label{ejer:integrales}
  Sea $G\colon\C\rightarrow\C$ una función meromorfa con un único polo posiblemente en $s=0$
  de orden $m\in\Ncero$.
  Para $a,b\in\{0,1,\infty\}$ introducimos la notación
  \begin{equation*}
    I_{a,b}(G,s):=\int_a^bG(t)t^{s-1}\integrald t \quad (s\in\C)
  \end{equation*}
  para los $s\in\C$ tal que esta integral converge.
  \begin{enumerate}
  \item Demuestre que la integral define una función holomorfa en la región donde converge.
  \item Si $a=1$, $b=\infty$ y $G$ decrece rápidamente en el siguiente sentido
    \[ \forall m\in\Nuno\colon \lim_{t\to\infty} t^m G(t)= 0, \]
    demuestre que la integral converge para cada $s\in\C$.
  \item Si $a=0$, $b=1$ demuestre que la integral converge para $\Re(s)>m$.
  \item Escribamos $G$ como serie de Laurent
    \begin{equation*}
      G(t)=\sum_{n=-m}^\infty c_n t^n
      \quad(t\ge0).
    \end{equation*}
    Entonces para $s\in\C$ con $\Re(s)>m$ y cada $N\in\Nuno$
    \begin{equation*}
      I_{0,1}(G,s) = \int_0^1\sum_{n=-1}^{N-1}c_nt^{n+s-1}\integrald t +
      \int_0^1\sum_{n=N}^{\infty}c_nt^{n+s-1}\integrald t.
    \end{equation*}
    Verifique que el primer sumando es igual a
    \begin{equation*}
      \sum_{n=-m}^N\frac{c_n}{n+s}
    \end{equation*}
    y estime el segundo sumando con
    \begin{equation*}
      \int_0^1\left(\sum_{n=N}^\infty\abs{c_n}\right)t^{N+s-1}\integrald t.
    \end{equation*}
    Verifque que la expresión anterior converge para cada $s\in\C$ con $\Re(s)>-N$.
    Concluya que $I_{0,1}(G,s)$
    se extiende a una función meromorfa en $\C$ con únicos polos posiblemente en
    $\{s\in\Z : s\le m\}$ de orden $\le1$.
  \end{enumerate}
\end{ejer}

\begin{ejer}
  Sea $f(t)=\displaystyle\frac{t}{1-\e^{-t}}$ para $t\in\R$ como en la \cref{defi:bernoulli}.
  \begin{enumerate}
  \item Verifique que $f(t)=-t\displaystyle\sum_{n=1}^\infty {(\e^{t})}^{n}$ para $t\le 0$.
  \item Verifique la relación $f(t)=f(-t)+t$ para $(t\in\R)$.
  \end{enumerate}
\end{ejer}

\begin{ejer}
  Calcule unos de los primeros números de Bernoulli.
\end{ejer}

\begin{ejer}\label{ejer:formula-b-n-chi}
  Sea $\chi$ un carácter de Dirichlet de conductor $N$. Verifique la relación
  \[ f_\chi(t)=\frac1N\sum_{a=1}^N\chi(a)F\left(Nt,\frac aN-1\right) \quad\text{para }t\in\C, \]
  donde $f_\chi$ es la función de la \cref{prop:numeros-de-bernoulli-gen} y $F$ es la
  función del \cref{lem:polinomios-de-bernoulli}. Use esto para deducir la siguiente fórmula
  para los números de Bernoulli generalizados:
  \[ B_{n,\chi}=N^{n-1}\sum_{a=1}^N\chi(a)B_n\left(\frac aN-1\right)\quad\text{para
    }n\in\Nuno. \]
  Concluya que si $\chi$ es no trivial entonces
  \[ B_{1,\chi}=\frac1N\sum_{a=1}^N\chi(a)a. \]
\end{ejer}

\section{Teoría elemental $p$-ádica de valores de la función zeta}

En la \cref{zeta-continuacion} vimos que algunos valores especiales de la
función zeta son racionales, lo que significa que podemos verlos como números
$p$-ádicos. Examinarlos desde este punto de vista lleva a un análogo $p$-ádico: la función zeta
$p$-ádica, que (a priori) es una función de $\Zp$ a $\Qp$ para la cual una fórmula similar a la
de la \cref{zeta-continuacion} es cierta.

Para la Teoría de Iwasawa y la Conjetura Principal que explicaremos en el \cref{sec:mc} es
necesario ver la función zeta $p$-ádica como un elemento del álgebra de Iwasawa $\LL$. En la
siguiente sección explicaremos qué significa esto exactamente y cómo se construye dicho
elemento.  No obstante, los resultados de la siguiente sección quizás parezcan más naturales
después de estudiar la función zeta $p$-ádica desde un punto de vista más elemental. Por eso,
en esta sección damos una construcción elemental de dicha función con el fin de fomentar la
intuición del lector. Cabe notar que la mayoría de los resultados en esta sección no son
necesarios para el resto del texto.

Empecemos estudiando en mas detalle los números de Bernoulli. El objetivo original de
Bernoulli era calcular expresiones como \[ 1^n+2^n+3^n+\dotsm +k^n \] para $k,n\in\Nuno$ y
expresarlas como un polinomio en $k$. Su
resultado es la siguiente \emph{fórmula de Bernoulli}\index[def]{formula@fórmula!de Bernoulli}.

\begin{lem}[Bernoulli]\label{formula-de-bernoulli}
  Si definimos \[ S_n(k) = \sum_{a=1}^k a^n \] para $n,k\in\Nuno$, entonces \[ S_n(k) =
    \frac{1}{n+1}\sum_{i=1}^n\binom{n+1}i B_ik^{n-i+1}. \]
\end{lem}
\begin{proof}
  Usamos los polinomios de Bernoulli y la función $F$ del
  \cref{lem:polinomios-de-bernoulli}.  Como $F(t,x)-F(t,x-1)=t\e^{tx}$ tenemos
  \begin{equation*}
    B_{n+1}(x)-B_{n+1}(x-1) = (n+1)x^n
  \end{equation*}
  para cada $n\in\Ncero$ y $x\in\C$. Pongamos aquí
  $x=1,\dotsc,k$ y sumemos todas las ecuaciones, esto nos da
  \begin{equation*}
    S_n(k)=\frac{1}{n+1}\left(B_{n+1}(k)-B_{n+1}(0)\right).
  \end{equation*}
  La fórmula entonces resulta de la
  definición de los polinomios de Bernoulli.
\end{proof}

En el siguiente resultado vemos los números de Bernoulli como números $p$-ádicos y obtenemos,
entre otras, la importante propiedad de que son $p$-ádicamente acotados.

\begin{prop}[Clausen-von Staudt]\label{clausen-von-staudt}
  Sea $n\in\Nuno$ par y $p\neq2$ un primo.
  \begin{enumerate}
  \item Si $p-1\nmid n$ entonces $B_n\in\Zp$.
  \item Si $p-1\mid n$ entonces $B_n+\frac1p\in\Zp$.
  \end{enumerate}
  En particular, $pB_n\in\Zp$, es decir $\abs{B_n}_p\le p$. Además, si $p-1\mid n$ entonces
  $pB_n\equiv-1$ ($\mod p$).
\end{prop}
\begin{proof}
  Esta demostración es una combinación de \cite[§2.2, B3, p.\ 35]{MR1029028} y \cite[Thm.\
  5.10]{MR1421575}. Usamos otra vez los polinomios de Bernoulli de la
  \cref{defi:polinomios-de-bernoulli} y también la función
  \begin{equation*}
    F(t,x)=\frac{t\e^{t(1+x)}}{\e^t-1} \quad (t,x\in\C)
  \end{equation*}
  del \cref{lem:polinomios-de-bernoulli}.
  
  Empezamos con derivar una relación para los polinomios de Bernoulli. Para cada $m\in\Nuno$
  tenemos
  \begin{equation*}
    \sum_{a=0}^{m-1}{(\e^t)}^a=\frac{\e^{mt}-1}{\e^t-1}. 
  \end{equation*}
  Usando esto, obtenemos para cada $x,t\in\C$
  \begin{align*}
    F(t,x) &= \frac{t\e^{(1+x)t}}{\e^{mt}-1}\sum_{a=0}^{m-1}\e^{at} 
    =  \sum_{a=0}^{m-1}\frac{t\e^{(x+1+a)t}}{\e^{mt}-1} \\ &=\frac1m\sum_{a=0}^{m-1}\frac{mt\e^{\frac{x+1+a}mmt}}{\e^{mt}-1}
                                                                   =
                                                                   \frac1m\sum_{a=0}^{m-1}F\left(mt,\frac{x+1+a}m\right) 
    \\ &= \frac1m\sum_{a=0}^{m-1}\sum_{n=0}^\infty
         B_n\left(\frac{x+1+a}m\right)\frac{(mt)^n}{n!} \\
           &=
             \sum_{n=0}^\infty\left(m^{n-1}\sum_{a=0}^{m-1}B_n\left(\frac{x+1+a}m\right)\right)\frac{t^n}{n!}
  \end{align*}
  y por lo tanto
  \begin{equation*}
    B_n(X)=m^{n-1}\sum_{a=1}^{m}B_n\left(\frac{X+a}m\right)
  \end{equation*}
  para cada $n\in\Nuno$.

  Ahora continuamos con inducción en $n$, es decir sea $n\ge2$ par y asumamos que la
  afirmación es cierta para cada $i<n$. Pongamos $m=p$ y $X=0$ en la relación de arriba, lo que nos da
  \begin{equation*}
    B_n=B_n(0)=p^{n-1}\sum_{a=1}^{p}B_n\left(\frac {a}
      p\right)=p^{n-1}\sum_{a=1}^{p}\sum_{i=0}^n\binom n i B_i{\left(\frac {a}
        p\right)}^{n-i}= \sum_{i=0}^n\sum_{a=1}^{p}\binom n i(pB_i)a^{n-i}p^{i-2}.
  \end{equation*}
  Por inducción sabemos que $pB_i\in\Zp$ para $i<n$, así que en la suma sobre $i$ todos los
  sumandos para $i=2,\dotsc,n-1$ están en $\Zp$. Porque $B_1=\frac12$ y $p\neq2$, el
  sumando para $i=1$ también está en $\Zp$. Por eso (note que $B_0=1$)
  \begin{equation*}
    B_n\equiv \sum_{a=1}^{p}\big(a^{n}p^{-1}+B_np^{n-1}\big) \quad \mod\Zp
  \end{equation*}
  o equivalentemente
  \begin{equation*}
    (1-p^n)B_n-\frac1p\sum_{a=1}^{p-1}a^n\in\Zp.
  \end{equation*}
  Como $1-p^n\in\Z_p^\times$ es suficiente demostrar que
  \begin{equation*}
    \sum_{a=1}^{p-1}a^n \equiv
    \begin{cases}
      0 &\text{si }p-1\nmid n,\\
      -1&\text{si }p-1\mid n\\
    \end{cases}
    \qquad(\mod p).
  \end{equation*}
  El caso en que $p-1\mid n$ es claro gracias al pequeño teorema de Fermat. Sea $u\in\Z$ tal
  que su clase módulo $p$ genera $\F_p^\times$. Entonces la multiplicación por $u$ induce un
  automorfismo de $\F_p^\times$, así que
  \begin{align*}
    (u^n-1)\sum_{a=1}^{p-1}a^n=\sum_{a=1}^{p-1}(ua)^n-\sum_{a=1}^{p-1}a^n\equiv
    0\quad(\mod p).
  \end{align*}
  Si $p-1\nmid n$ entonces $u^n\not\equiv 1$ ($\mod p$), y la afirmación resulta.
\end{proof}

Estos dos resultados se pueden usar para encontrar una representación $p$-ádica de los
números de Bernoulli:

\begin{cor}\label{bernoulli-p-adico}
  Para cada $n\in\Nuno$ tenemos en $\Q_p$
  \[ B_n=\lim_{j\to\infty}\frac1{p^j}S_n(p^j). \]
\end{cor}
\begin{proof}
  Con el \cref{formula-de-bernoulli} tenemos
  \[ \frac1{p^j}S_n(p^j) = \frac{1}{n+1}\sum_{i=1}^n\binom{n+1}i B_ip^{j(n-1)} = B_n +
    p^j\cdot C_j \] con $C_j\in\Z$ que, si bien es cierto que depende de $j$, su valor
  absoluto $p$-ádico es acotado independientemente de $j$, como se ve fácilmente con la
  \cref{clausen-von-staudt}.
\end{proof}

Gracias a esto, podemos demostrar las \define{congruencias de Kummer}, que son el primer paso en la
dirección para obtener una función zeta $p$-ádica.

\begin{prop}[Kummer]\label{congruencias-de-kummer}
  Sea $p\neq2$ primo y $k,m,n\in\Nuno$, y sea $c\in\Z$ no divisible por $p$. Supongamos que
  $m\equiv n\ (\mod p^k(p-1))$. Entonces
  \[ (1-c^m)(1-p^{m-1})\frac{B_m}{m} \equiv (1-c^n)(1-p^{n-1})\frac{B_n}{n} \quad (\mod
    p^{k+1}). \]
  Si $m,n$ no son divisibles por $p-1$ entonces incluso
  \[ (1-p^{m-1})\frac{B_m}{m} \equiv (1-p^{n-1})\frac{B_n}{n} \quad (\mod p^{k+1}). \]
\end{prop}
\begin{proof}[según \cite{MR1682940}]
  Primero notemos que para cada $x\in\Z$ con $(x,p)=1$ tenemos
  \begin{equation} \label{eqn:xmxn}
    x^m \equiv x^n \ (\mod p^{k+1}) \tag{$*$}
  \end{equation}
  porque $p^k(p-1)=\#(\Z/p^{k+1}\Z)^\times$. En particular, $1-c^m\equiv 1-c^n$
  ($\mod p^{k+1}$). Si $m\not\equiv 0$ ($\mod p-1$), entonces siempre existe un $c\in\Z$
  no divisible por $p$ tal que $c^m\not\equiv1$ ($\mod p$). Por eso la segunda congruencia
  resulta de la primera.
  
  Del \cref{bernoulli-p-adico} vemos fácilmente que en $\Q_p$
  \begin{equation*} (1-p^{m-1}) B_m =
    \lim_{j\to\infty}\frac{1}{p^j}\sum_{\substack{a=1\\p\nmid a}}^{p^j}
    a^m
  \end{equation*}
  y por eso, si $j\in\Nuno$ es suficientemente grande, tenemos
  \begin{equation*} \bigg\vert (1-p^{m-1}) B_m - \frac{1}{p^j}
    \sum_{\substack{a=1\\p\nmid a}}^{p^j} a^m \bigg\vert_p \le p^{-(k+1)}.
  \end{equation*}
  Para este $j$ resulta que
  \begin{equation*} (1-c^m)(1-p^{m-1})\frac{B_m}{m} \equiv
    (1-c^m)\frac{1}{mp^j}\sum_{\substack{a=1\\p\nmid a}}^{p^j} a^m \quad (\mod
    p^{k+1}). \end{equation*} Llamamos $A\in\Q$ al número racional del lado derecho. Notemos
  que la congruencia de arriba no es una igualdad en $\Z/p^{k+1}\Z$, porque ambos lados
  pueden contener potencias no triviales de $p$ en su denominador. Pero su diferencia no
  contiene potencias de $p$ en el denominador y el numerador es divisible por $p^{k+1}$.

  Por lo tanto
   \begin{equation*} A = \frac1m \sum_{\substack{a=1\\p\nmid a}}^{p^j}
    \frac{a^m-(ca)^m}{p^j}. \end{equation*} Para cada $a$ ocurriendo en esta suma sea
  $b_a\in\{1,\dotsc,p^j\}$ el único elemento con $b_a\equiv ca\ (\mod p^j)$. Como $(c,p)=1$,
  es decir $c\in(\Z/p^j\Z)^\times$, el mapeo $a\mapsto b_a$ es una permutación del
  conjunto \begin{equation*} \{a\in\{1,\dotsc,p^j\} : p\nmid a \}, \end{equation*} y esto
  nos da
  \begin{equation*} A = \frac1m \sum_{\substack{a=1\\p\nmid a}}^{p^j}
    \frac{b_a^m-(ca)^m}{p^j}. \end{equation*} Ahora pongamos
  $t_a=\frac{b_a-ca}{p^j}\in\Z$. Entonces
  \begin{equation*} b_a^m - (ca)^m = (ca + p^j t)^m - (ca)^m = \sum_{i=1}^m \binom{m}{i}
    (ca)^{m-i} (p^j t)^i = mp^j t(ca)^{m-1} + K\cdot p^{2j} \end{equation*} para un
  $K\in\Z$. Sustituyendo esto en el cálculo anterior obtenemos
  \begin{equation*} A = \sum_{\substack{a=1\\p\nmid a}}^{p^j} t_a(ca)^{m-1} +
    p^j\frac{K(p-1)p^{j-1}}{m}. \end{equation*}
  Si hacemos $j$ tan grande que la fracción de la derecha no contenga una potencia de $p$
  en el denominador y el numerador sea divisible por $p^{k+1}$, entonces usando la ecuación
  \eqref{eqn:xmxn} obtenemos la primera congruencia módulo $p^{k+1}$:
  \begin{equation*} (1-c^m)(1-p^{m-1})\frac{B_m}{m} \equiv A \equiv
    \sum_{\substack{a=1\\p\nmid a}}^{p^j} t_a(ca)^{m-1}
    \overset{\text{\eqref{eqn:xmxn}}}{\equiv} \sum_{\substack{a=1\\p\nmid a}}^{p^j}
    t_a(ca)^{n-1} \equiv (1-c^n)(1-p^{n-1})\frac{B_n}{n}. \end{equation*}
\end{proof}

La observación importante de Kubota y Leopoldt es que las congruencias que acabamos de
demostrar pueden ser interpretadas de la siguiente forma.

\begin{prop}\label{existencia-zeta-p-funcion}
  Existe una única función continua
  \[ \zeta_{p,0}\colon \Z_p\setminus\{1\} \rightarrow \Q_p \]
  tal que \[ \zeta_{p,0}(1-n) = (1-p^{n-1})\zeta(1-n) \]
  para todo $n\in\Nuno$ con $n\equiv 0\ (\mod p-1)$.
\end{prop}
\begin{proof}
  Fijemos un $c\in\Z$ que no es divisible por $p$.
  
  Primero observemos que las congruencias de Kummer significan que la función
  \[ f\colon D:= \{ n\in\Nuno : p-1\mid n \} \rightarrow \Q_p, \quad n \mapsto
    (1-c^n)(1-p^{n-1})\frac{B_n}{n} \] es uniformemente continua para la métrica $p$-ádica:
  En efecto, esta continuidad significa que
  \begin{equation*} \forall \varepsilon>0\, \exists \delta>0\, \forall m,n\in D\colon
    \abs{m-n}_p<\delta \implies \abs{f(m)-f(n)}_p<\varepsilon. \end{equation*} Si
  $m,n\in D$, la congruencia $m\equiv n\ (\mod p-1)$ es cierta automáticamente. Si asumimos
  sin pérdida de generalidad que $\varepsilon$ es de la forma $\varepsilon=p^{-(k+1)}$ con
  $k\in\Nuno$, podemos poner $\delta=p^{-k}$ y la afirmación es equivalente a las congruencias
  de Kummer.

  Segundo, observemos que el conjunto $D$ es denso en $\Z_p$: necesitamos ver que para cada
  $z\in\Z$ y $k\in\Nuno$ existe un $l\in\Z$ con $p^kl-z\in D$ (porque entonces $z$ está en la
  cerradura de $D$). Por eso simplemente tomamos $l$ como un entero
  que satisface \begin{equation*} l \equiv \frac{a-z}{p^k} \quad (\mod p-1) \end{equation*} y
  suficientemente grande tal que $p^kl-z\in\Nuno$.

  El conjunto $D$ es denso también en $\Z_p\setminus\{1\}$, y esto demuestra que la función
  que buscamos es única si existe. Por la densidad de $D$ en $\Z_p$ y la continuidad
  uniforme, la función $f$ de arriba se extiende a una función continua
  $f\colon \Z_p\rightarrow\Q_p$. Entonces escogemos $c\in 1+p\Z$ y definimos $\zeta_p$ como
  \begin{equation*} \zeta_{p,0}\colon \Z_p\setminus\{1\}\rightarrow\Q_p, \quad 1-s \mapsto
    -\frac{f(s)}{1-c^s}, \end{equation*}
  donde $c^s$ con $s\in\Zp$ está bien definido gracias a la \cref{prop:isom-z-p-log}.
  Es claro que esta función tiene la propiedad deseada, y por la unicidad es independiente
  de $c$.
\end{proof}

La función que acabamos de construir se podría llamar \importante{de buena fe}
\enquote{función zeta $p$-ádica de Riemann}. Sin embargo, todavía no es la función que se
puede usar en la Teoría de Iwasawa -- la \enquote{verdadera} construcción la haremos en la
siguiente sección. Para motivar el resultado que demostraremos, continuamos estudiando la
función que hemos construido. Por ejemplo, ¿Cómo se comporta para los $n$ con
$n\not\equiv 0$ ($\mod p-1$)?

Antes de discutir esto, necesitamos un análogo del \cref{bernoulli-p-adico} para los números de
Bernoulli generalizados para un carácter de Dirichlet $\chi$. Para esto fijamos encajes
$\Qbar\hookrightarrow\C$ y $\Qbar\hookrightarrow\Qpbar$, que nos permiten ver números
algebraicos como números complejos o números $p$-ádicos a conveniencia.

\begin{prop}\label{numeros-de-bernoulli-p-adicos}
  Sea $\chi$ un carácter de Dirichlet de conductor $N$.
  Pongamos \[ S_{n,\chi}(k) = \sum_{a=1}^k \chi(a)a^n. \]
  Entonces en $\Qpbar$
  \[ B_{n,\chi} = \lim_{j\to\infty} \frac{1}{Np^j}S_{n,\chi}(Np^j). \]
\end{prop}
\begin{proof}
  La demostración es la misma que la del \cref{bernoulli-p-adico}, usando la fórmula
  \[ S_{n,\chi}(kN) = \frac{1}{n+1} \sum_{i=1}^{n} \binom{n+1}{i} B_{i,\chi} (kN)^{n+1-i} \]
  que generaliza el \cref{formula-de-bernoulli}. Esta última fórmula se puede obtener
  de manera similar a la del \cref{formula-de-bernoulli}, véase \cite[p.\ 11]{MR0360526}.
\end{proof}

Nuestros encajes fijados nos permiten en particular considerar el carácter de Teichmüller
$\omega$ como carácter de Dirichlet y definir su serie $L$ y números de Bernoulli.

\begin{prop}\label{zeta-p-otros-valores}
  Tenemos \[ \zeta_{p,0}(1-n) = -(1-\omega^{-n}(p)p^{n-1})\frac{B_{n,\omega^{-n}}}{n} \] para
  todo $n\in\Nuno$, donde $\omega$ es el carácter de Teichmüller.
\end{prop}
\begin{proof}
  Fijemos $n\in\Nuno$, y sin perdida de generalidad supongamos que $n\not\equiv0$
  ($\mod p-1$), porque en el caso $p-1\mid n$ ya sabemos el resultado. Notemos que el
  conductor de $\omega^{-n}$ es $p$ si $p-1\nmid n$, así que la afirmación en este caso
  simplemente es $\zeta_{p,0}(1-n) = -B_{n,\omega^{-n}}/n$.

  Necesitamos una
  sucesión $(n_k)_{k\in\Nuno}$ de enteros positivos, todos divisibles por $p-1$, que
  converja a $n$ en $\Z_p$: podemos usar
  \begin{equation*} n_k = n(p^k(p-2)+1) = n\big((p^k-1)^2-p^{k+1}(p^{k-1}-1)\big). \end{equation*}

  Por el \cref{bernoulli-p-adico},
  \begin{equation*} 
    - (1-p^{n_k-1})B_{n_k} = - \lim_{j\to\infty}\frac{1}{p^j}\sum_{\substack{a=1\\p\nmid
        a}}^{p^j} a^{n_k}. \end{equation*}
  Ahora queremos hacer $k\to\infty$ e intercambiar este con el límite $j\to\infty$. Para que
  esto sea permitido, tenemos que ver que la convergencia para $j\to\infty$ es uniforme con
  respecto a $k$. Para la claridad de la demostración probaremos esto en el lema siguiente.

  Ahora intercambiamos los límites con respecto a $k$ y $j$ y obtenemos
  \begin{equation*} \lim_{k\to\infty} (1-p^{n_k-1})B_{n_k} = \lim_{j\to\infty}\frac{1}{p^j}\sum_{\substack{a=1\\p\nmid a}}^{p^j} \lim_{k\to\infty} a^{n_k}. \end{equation*}
  Como $n_k=n-np^k+n(p-1)p^k$, tenemos en $\Z_p$ según el \cref{ejer:teichmueller-limes}
  \begin{equation*}\lim_{k\to\infty}a^{n_k} = \lim_{k\to\infty}a^n(a^{p^k})^{-n}(a^{p^k})^{(p-1)n} = a^n\omega(a)^{-n}\omega(a)^{(p-1)n} = \omega^{-n}(a)a^n, \end{equation*}
  porque $\omega(a)$ es una raíz de la unidad $(p-1)$-ésima.
  Usando la \cref{numeros-de-bernoulli-p-adicos}, esto nos da
  \begin{equation*} \lim_{k\to\infty} (1-p^{n_k-1})B_{n_k} = \lim_{j\to\infty}\frac{1}{p^j}\sum_{\substack{a=1\\p\nmid a}}^{p^j} \omega^{-n}(a)a^n = B_{n,\omega^{-n}} \end{equation*}
  y por eso  \begin{equation*} \lim_{k\to\infty} \zeta_{p,0}(1-n_k) = -\frac{B_{n,\omega^{-n}}}n \end{equation*}
  como deseamos.
\end{proof}

Falta probar la convergencia uniforme que nos permitió intercambiar
los límites en la demostración de arriba. Esto no depende de la forma concreta de la sucesión $(n_k)$.

\begin{lem}
  Tenemos
  \[ \forall \varepsilon>0 \,\exists J_\varepsilon\in\Nuno \,\forall j\ge J_\varepsilon \,\forall n\in\Nuno \colon \Big\vert \frac{1}{p^j}\sum_{\substack{a=1\\p\nmid a}}^{p^j} a^{n} - (1-p^{n-1})B_{n}\Big\vert_p \le \varepsilon. \]
\end{lem}
\begin{proof}
  Sean $n,j\in\Nuno$. Si sustituimos en la relación (véase también el
  \cref{ejer:relacion-s-n-chipsi})
  \begin{equation*} \sum_{\substack{a=1\\p\nmid a}}^{p^j} a^{n} = S_{n}(p^j) - p^{n}
    S_{n}(p^{j-1}) \end{equation*} la fórmula de Bernoulli (\cref{formula-de-bernoulli}),
  después de un poco de cálculos obtenemos
  \begin{equation*} \frac1{p^j}\sum_{\substack{a=1\\p\nmid a}}^{p^j} a^{n} = \frac{1}{n+1} \sum_{i=1}^{n} \binom{n+1}{i} B_i p^{j(n-i)}(1 - p^{i-1}). \end{equation*}
  Si $i=n$ en la suma de la derecha, el sumando correspondiente es $(n+1)(1-p^{n-1})B_{n}$, pues
  \begin{align*}
    \Big( \frac{1}{p^j}\sum_{\substack{a=1\\p\nmid a}}^{p^j} a^{n} - (1-p^{n-1})B_{n}\Big) &= \frac{1}{n+1} \sum_{i=1}^{n-1} \binom{n+1}{i} B_i p^{j(n-i)}(1 - p^{i-1}) \\ &= p^j \sum_{i=1}^{n-1} \frac{p^{n-i}}{n+1-i}\binom{n}{i}B_i(1-p^{i-1}).
  \end{align*}
  Tenemos que acotar esta expresión. Claramente $\abs{\binom n i}_p\le 1$ y
  $\abs{1-p^{i-1}}_p=1$. Además, por el teorema de Clausen-von Staudt (\cref{clausen-von-staudt}) tenemos \begin{equation*} \abs{B_n}_p \le p. \end{equation*}

  Demostremos que para cada $i=1,\dotsc,n-1$ tenemos
  \begin{equation*} \abs{\frac{p^{n-i}}{n+1-i}}_p\le 1, \end{equation*}
  lo cual es equivalente a $v_p(n+1-i)\le n-i$. Para cada $k\in\Nuno$, $p^k$ es el menor número
  natural para cual $v_p$ toma el valor $k$. Porque siempre que $p^k\ge k+1$, se tiene
  $v_p(k+1)\le k$. Haciendo $k=n-i$ nos da la desigualdad deseada.
  
  Ahora sea $\varepsilon>0$, que sin pérdida de generalidad es de la forma $\varepsilon=p^{-J}$ con
  un $J\in\Nuno$. Entonces ponemos $J_\varepsilon=J+1$ y entonces tenemos para $j\ge J_\varepsilon$ y
  cada $n\in\Nuno$:
  \begin{equation*} \Big\vert \frac{1}{p^j}\sum_{\substack{a=1\\p\nmid a}}^{p^j} a^{n} - (1-p^{n-1})B_{n}\Big\vert_p \le p^{-j+1} \le p^{-J_\varepsilon+1} = \varepsilon. \end{equation*}
\end{proof}

Ahora casi estamos listos para formular el resultado que demuestra cómo la función zeta $p$-ádica
es vista en la Teoría de Iwasawa.

En la construcción de la función $\zeta_{p,0}$ en la \cref{existencia-zeta-p-funcion} usamos la
primera de las congruencias de Kummer de la \cref{congruencias-de-kummer}. Si usamos la segunda
en lugar de la primera, podemos construir de la misma manera funciones continuas
\[ \zeta_{p,a}\colon\Z_p\rightarrow\Q_p \] para cada $a\in\{1,\dotsc,p-2\}$ con la
propiedad de que \[ \zeta_{p,a}(1-n) = (1-p^{n-1})\zeta(1-n) \] para todo $n\in\Nuno$ con
$n\equiv a\ (\mod p-1)$. El mismo razonamiento que hicimos en la
\cref{zeta-p-otros-valores} (véase el \cref{ejer:zeta-p-a}) muestra que
\begin{equation}
  \label{eqn:interpolacion-preliminaria}
   \zeta_{p,a}(1-n) = -(1-\omega^{-(n+a)}(p)p^{n-1})\frac{B_{n,\omega^{-(n+a)}}}{n}
\end{equation}
para todo
$n\in\Nuno$.  Estas $p-1$ funciones se llaman ramas de la función zeta $p$-ádica, pero esta
terminología no es muy importante porque se pueden juntar las ramas en una sola función.

Sea $G=\Gal(\Q(\mu_{p^\infty})/\Q)$, el cual es isomorfo a $\Z_p^\times$ vía el carácter
ciclotómico $\kappa$. Definimos una función en el grupo de caracteres de $G$.\footnote{Desde el punto de vista moderno, en general las funciones $L$ $p$-ádicas son funciones en grupos de
  caracteres de grupos de Galois. El lector que conozca la tesis de Tate está invitado a
  comparar esto con el punto de vista de ahí, que considera las funciones $L$ complejas como
  funciones en grupos de caracteres de grupos de idèles, que según la teoría de campos de
  clases están relacionados con grupos de Galois.} Del \cref{lem:caracteres-de-g} sabemos
que cada carácter $\chi\colon G\rightarrow\Z_p^\times$ es de la forma
\[ \chi=\omega^a\kappa_0^s \] con únicos $a\in\{1,\dotsc,p-1\}$ y $s\in\Z_p$. Usando esto definimos la
función zeta $p$-ádica de Riemann como
\[
  \zeta_p\colon\Hom(G,\Z_p^\times)\setminus\{\kappa\}\rightarrow\Q_p,\quad\chi=\omega^a\kappa_0^s\mapsto\zeta_{p,a-1}(s). \]

Resumamos los resultados de esta sección en el teorema siguiente.

\begin{thm}\label{thm:zeta-p-adica-prelim}
  La función continua\footnote{Equipamos $\Hom(G,\Z_p^\times)$ con la topología
    compacto-abierta.}
  \[ \zeta_p\colon\Hom(G,\Z_p^\times)\setminus\{\kappa\}\rightarrow\Q_p \] tiene la
  propiedad de que
  \[ \zeta_p(\psi^{-1}\kappa^{1-n}) = (1-\psi(p)p^{n-1})L(\psi,1-n) \] para cada $n\in\Nuno$ y
  cada carácter de Dirichlet $\psi$ de $\Gal(\Q(\mu_p)/\Q)$, y esta propiedad la
  caracteriza únicamente.
\end{thm}
\begin{proof}
  Cada carácter $\psi$ como arriba tiene la forma $\psi=\omega^i$ con un
  $i\in\{1,\dotsc,p-1\}$ si identificamos $\Gal(\Q(\mu_p)/\Q)$ con $(\Z/p\Z)^\times$, y
  entonces
  \[ \zeta_p(\psi^{-1}\kappa^{1-n}) = \zeta_p(\omega^{-i}\kappa^{1-n}) =
    \zeta_p(\omega^{-i+1-n}\kappa_0^{1-n}).  \] La propiedad anunciada resulta de
  \eqref{eqn:interpolacion-preliminaria}.  La unicidad la tenemos porque los caracteres
  $\psi^{-1}\kappa^{1-n}$ como arriba son densos en
  $\Hom(G,\Z_p^\times)\setminus\{\kappa\}$, pero omitimos los detalles de verificar
  esto.
\end{proof}

Aquí ya podemos reconocer el fenómeno que es típico de las funciones $L$ $p$-ádicas: en
general, son funciones en caracteres de grupos de Galois tal que si las evaluamos en un
producto de un carácter de orden finito $\psi$ y una potencia entera del carácter
ciclotómico obtenemos una modificación de un valor de una función $L$ compleja
chanfleada por $\psi$. Observe que la multiplicación por
$1-\psi(p)p^{n-1}$ quita el factor de Euler para el primo $p$ de la función $L(\psi,-)$. La
fórmula en el teorema que describe estos valores se llama la \emph{fórmula de
  interpolación}\index[def]{formula@fórmula!de interpolación}. ¡El lector debería volver a considerar cuán sorprendente es que tal cosa exista!

\ejercicios

\begin{ejer}\label{ejer:zeta-p-a}
  Demuestre la formula \eqref{eqn:interpolacion-preliminaria}. Para esto es útil usar la
  sucesión $(n_k)_{k\in\Nuno}$ con $n_k=n-(n+a)p^k+n(p-1)p^k$.
\end{ejer}

\begin{ejer}
  Demuestre la unicidad de la función $\zeta_p$ en el \cref{thm:zeta-p-adica-prelim}.
\end{ejer}

\section{La construcción de funciones $L$ $p$-ádicas mediante elementos de Stickelberger}
\label{sec:stickelberger}

Como hemos dicho anteriormente, queremos ver las funciones $L$ $p$-ádicas como elementos en
el álgebra de Iwasawa. Explicaremos con más detalle que quiere decir esto.

En esta sección usamos la notación siguiente: Para $m\in\Nuno$ sea $G_m$ el grupo
$\Gal(\Q(\mu_m)/\Q)$, que identificamos con $(\Z/m\Z)^\times$ de la manera usual (véase \cref{defi:campos-ciclotomicos}).
Si escribimos $m=Np^r$ con $N$ no divisible por $p$ entonces $G_m\isom G_N\times G_{p^r}$
canónicamente. Escribimos
$G_{Np^\infty}=\Gal(\Q(\mu_{Np^\infty})/\Q)=\varprojlim_{r\in\Nuno}G_{Np^r}$.
El grupo $G_{p^\infty}$ será el más importante de todos y lo llamamos simplemente
$G$.
Notemos que entonces $G_{Np^\infty}\isom G_N\times G$. Además, por el carácter ciclotómico
tenemos un isomorfismo $G\isom\Z_p^\times$, que se descompone como
$\Z_p^\times\isom\F_p^\times\times(1+p\Zp)$ según \cref{lem:z-p-decomposicion}, y denotemos
$\Delta$ y $\Gamma$ los subgrupos de $G$ que corresponden a $\F_p^\times$ y $1+p\Zp$,
respectivamente. Es decir, $\Delta=G_p=\Gal(\Q(\mu_p)/\Q)$ y $\Gamma=\Gal(\Q_{\mu_{p^\infty}}/\Q(\mu_p))$.

Sea $\LL(G)$ el álgebra de Iwasawa de $G$.  Si $\mu\in\LL(G)$ es un elemento, la propiedad
universal (\cref{prop:propiedad-universal-anillo-de-grupos}) nos da para cada carácter
$\psi\colon G\rightarrow\overline\Q_p^\times$ un morfismo, que llamamos igualmente,
$\psi\colon\LL(G)\rightarrow\overline\Q_p$.\footnote{Para ser más exactos, la propiedad
  universal nos da esto para caracteres con valores en anillos profinitos. Pero cada
  carácter $\psi\colon G\rightarrow\Qbar_p^\times$ tiene valores en un subanillo de
  $\Qbar_p$ compacto, y cada anillo topológico que es Hausdorff y compacto automáticamente
  es profinito. Véase la \cref{rem:locally-profinite} y el \cref{ejer:locally-profinite}.}
Abusando de la notación, escribimos $\mu(\psi):=\psi(\mu)$ como la imagen de $\mu$ por este
morfismo. De esta manera, cada elemento del álgebra de Iwasawa define una función
\[ \mu\colon\Hom(G,\overline\Q_p^\times)\rightarrow\overline\Q_p. \]
Esta función es la misma que la que obtenemos si consideramos $\mu$ como medida en $G$ como
en la \cref{sec:medidas}.

Lo que vamos a demostrar aquí es que la función $\zeta_p$ del \cref{thm:zeta-p-adica-prelim}
esencialmente viene de un tal elemento. Esto lo lograremos con una construcción
completamente nueva mediante elementos de Stickelberger, la cual es independiente de los
resultados de la sección anterior. En resumen, en esta sección damos otra construcción de la
función $\zeta_p$ más apegada a los métodos modernos en la Teoría de Iwasawa.

En esta sección continuamos fijando encajes $\Qbar\hookrightarrow\C$ y
$\Qbar\hookrightarrow\Qpbar$. Aquí seguimos esencialmente \cite[§6]{MR0360526}, aunque
nuestro punto de vista es un poco más moderno.

\begin{defi}
  Sea $m\in\Nuno$, el \define{elemento de Stickelberger de nivel $m$} es \[ \Sigma_m =
    -\frac1m\sum_{\substack{a=1\\(a,m)=1}}^ma\sigma_a^{-1} \in \Q[G_m]. \]
\end{defi}

El interés en este elemento originalmente proviene del \define[Teorema de Stickelberger]{teorema de Stickelberger}
(véase \cite[Thm.\ 6.10]{MR1421575} y el \cref{cor:stickelberger} más tarde), que dice que
el ideal $\Z[G_m]\cap\Sigma_m\Z[G_m]$ de $\Z[G_m]$ anula el grupo de clases de $\Q(\mu_m)$,
pero aquí nos interesamos en ese elemento por razones diferentes.

Ahora fijemos $N$ tal que $p$ no divide a $N$.  Para cada $r$ entonces tenemos el elemento
de Stickelberger $\Sigma_{Np^r}\in\Q[G_{Np^r}]$. Un cálculo fácil que omitimos aquí
demuestra que estos elementos son compatibles con respecto a los mapeos
$\Q[G_{Np^r}]\rightarrow\Q[G_{Np^s}]$ para $r\ge s\ge 1$. Así obtenemos un elemento
\[ \Sigma_{Np^\infty}=(\Sigma_{Np^r})_{r\in\Nuno}\in\Q\llbracket G_{Np^\infty}\rrbracket \] que
llamamos el elemento de Stickelberger de nivel $Np^\infty$. Aquí, por supuesto,
$\Q\llbracket G_{Np^\infty}\rrbracket$ denota el límite de anillos de grupos $\Q[G_{Np^r}]$
aunque $\Q$ no sea un anillo profinito.

Ahora sea $K$ una extension finita de $\Qp$ que contiene las raíces de la unidad de orden
$\varphi(N)=\#G_N$, sea $\O$ su anillo de enteros $K$ y $\LL(G)=\O\llbracket
G\rrbracket$. Entonces $K\llbracket G_{Np^\infty}\rrbracket$ es un $K[G_N]$-módulo. Para
cada carácter $\chi\colon G_N\rightarrow K^\times$ tenemos el idempotente
\[ e_\chi=\frac1{\varphi(N)}\sum_{g\in G_N}\chi(g)^{-1}g\in K[G_N] \] y podemos
descomponer
\[ K\llbracket G_{Np^\infty}\rrbracket = \bigoplus_\chi e_\chi K\llbracket
  G_{Np^\infty}\rrbracket \] (véase \cref{lem:idempotentes}).
Entonces, según el \cref{lem:descomposicion-lambda},
para cada $\chi$ existe un isomorfismo canónico de $K$-álgebras
\[ E_\chi\colon e_\chi K\llbracket G_{Np^\infty}\rrbracket \isomarrow K\llbracket
  G\rrbracket. \]

Veamos $\Sigma_{Np^\infty}$ como elemento de $K\llbracket G_N\rrbracket$. Además fijemos un
encaje $K\hookrightarrow\Qpbar$ y un carácter de Dirichlet $\chi$ de conductor $N$, que
podemos ver como carácter de $G_N$ con valores en $K$.

\begin{defi}\label{defi:mu-chi}
  Definimos $\mu_\chi := E_{\chi^{-1}}(e_{\chi^{-1}}\Sigma_{Np^\infty}) \in K\llbracket
  G\rrbracket$. Este elemento se llama la \define{función $L$ $p$-ádica} de $\chi$.
\end{defi}

En los siguientes teoremas vamos a justificar este nombre.
El \cref{ejer:mu-chi-explicito} da una representación m\'as explicita de este elemento que
usamos a continuación.

Por simplicidad, por ahora asumimos que
$\chi$ no es trivial, es decir $N>1$. El caso del carácter trivial es similar, pero
ligeramente más complicado, y lo explicaremos más tarde.

\begin{lem}\label{lem:mu-chi-entero}
  De hecho $\mu_\chi\in\LL(G)$ si $\chi$ es un carácter no trivial.
\end{lem}
\begin{proof}
  Usamos la representación del \cref{ejer:mu-chi-explicito}. Tenemos
  \begin{align*}
    \sum_{\substack{a=1\\(a,Np)=1}}^{Np^r}a\chi(a)\pi_r(\sigma_a)^{-1}  &=
                                                                               \sum_{\substack{b=1\\(b,p)=1}}^{p^r}\sum_{\substack{c=1\\(c,Np)=1\\c\equiv
    b\mod p^r}}^{Np^r}c\chi(c)\pi_r(\sigma_b)^{-1} \\
    &\equiv \sum_b\left(\sum_c\chi(c)\right)b\pi_r(\sigma_b)^{-1}\quad\mod  p^r
  \end{align*}
  (las últimas sumas yendo sobre los mismos $b$ y $c$ como en la linea anterior).  Pero si
  $c$ pasa por estos elementos, la clase de $c$ módulo $N$ pasa por todos elementos de
  $G_N$. Por eso $\sum_c\chi(c)=0$ porque $\chi$ es un carácter no trivial (como demostramos
  en el \cref{ejer:schur-orthogonalidad}). Esto demuestra que la suma con que empezamos
  arriba es divisible por $p^r$, as\'i $\mu_{\chi,r}\in\O[G_{p^r}]$.
\end{proof}

El elemento $\mu_\chi\in\LL(G)$ que acabamos de construir se comporta como anunciamos al
principio de la sección: Define una función en $\Hom(G,\Qbar_p^\times)$ que es descrita por
una fórmula de interpolación similar a la del \cref{thm:zeta-p-adica-prelim}.

\begin{thm}\label{thm:palf-stickelberger-no-trivial}
  Sea $\psi\colon G\rightarrow\Qbar_p^\times$ un carácter de orden finito y
  $n\in\Nuno$. Entonces
  \[ \mu_\chi(\psi^{-1}\kappa^{1-n})=(1-\chi\psi(p)p^{n-1})L(\chi\psi,1-n). \]
\end{thm}

La demostración de este teorema la haremos más tarde. Primero expliquemos como podemos
extender la construcción al caso en que el carácter $\chi$ sea trivial, es decir $N=1$, puesto que
este caso (que corresponde a la función zeta de Riemann) es el que nos interesa más.

Si $\chi=\mathbf 1$ es trivial, el elemento $\mu_{\mathbf 1}$ (que en este caso simplemente es
igual al elemento de Stickelberger $\Sigma_{p^\infty}\in\Qp\llbracket G\rrbracket$ de nivel
$p^\infty$) se comporta de manera similar. Lo que ya no funciona es la demostración del
\cref{lem:mu-chi-entero}, que es la única ocasión en la que usamos que $\chi$ era no
trivial, y de hecho la afirmación $\mu_{\mathbf 1}\in\LL(G)$ ya no es verdad. Sin embargo
$\mu_{\mathbf 1}$ está en el anillo de cocientes de $\LL(G)$, y podemos calcular explícitamente
su denominador.

\begin{defi}\label{defi:h-n}
  Para cada $N\in\Nuno$ definimos un elemento
  \[ h_N=1-(1+Np)\sigma_{1+Np}^{-1}\in\LL(G). \]
  donde $\sigma_{1+Np}\in G$ es el único
  elemento con $\kappa(\sigma_{1+Np})=1+Np$. Además definimos para cada
  $i\in\{1,\dotsc,p-1\}$   \[ h_N^{(i)}=1-(1+Np)e_{\omega^i}\sigma_{1+Np}^{-1}\in\LL(G) \]
  donde $e_{\omega^i}$ es el idempotente del \cref{lem:idempotentes}.
\end{defi}

\begin{lem}\label{lem:integral-h}
  Tenemos $h^{(1)}_1\mu_{\mathbf 1}\in\LL(G)$.
\end{lem}
\begin{proof}
  Aunque la afirmación de este lema es importantísima, no damos la demostración en detalle
  porque usa argumentos muy similares a los que explicamos en otras demostraciones de esta
  sección (por ejemplo en la del \cref{thm:palf-stickelberger-trivial}). Sin embargo,
  describimos los pasos más importantes en el \cref{ejer:integral-h}.
\end{proof}

\begin{lem}\label{lem:nullstellen-h}
  Sea $N\in\Nuno$ coprimo a $p$, $i\in\{1,\dotsc,p-1\}$ y
  $\phi\colon G\rightarrow\Qbar_p^\times$ un carácter. Escribimos $\phi=\omega^j\phi_0$
  usando la descomposición $G\isom\Z_p^\times\isom\F_p^\times\times(1+p\Zp)$ del
  \cref{lem:z-p-decomposicion}, con $j\in\{1,\dotsc,p-1\}$.  Entonces tenemos
  \[ \phi(h_N^{(i)})=
    \begin{cases}
      1-(1+Np)\phi_0(1+Np)^{-1}, & j=i,\\
      1, & j\neq i.
    \end{cases} \]
  En particular
  \[ \phi(h_N^{(i)})=0 \iff \phi=\omega^i\kappa_0. \]
\end{lem}
\begin{proof}
  De la definición de $e_{\omega^i}$ obtenemos fácilmente
  \begin{equation*}
    \phi(e_{\omega^i}\sigma_{1+Np}^{-1})=\frac1{p-1}\Bigg(\sum_{a=1}^{p-1}\omega^{j-i}(a)\Bigg)\phi_0(1+Np)^{-1}.
  \end{equation*}
  La suma en los paréntesis es igual a $p-1$ si $j=i$ y es cero si $j\neq i$. Obtenemos la
  fórmula para $\phi(h_N)$.

  La última afirmación entonces es cierta porque $1+Np$ es un generador topológico de $1+p\Zp$.
\end{proof}

Ahora podemos enunciar el teorema para el carácter trivial. Los lemas \ref{lem:integral-h} y
\ref{lem:nullstellen-h} muestran que, aunque $\mu_{\mathbf 1}$ no define una función en todo $\Hom(G,\Qbar_p^\times)$, todavía define una función en $\Hom(G,\Qbar_p^\times)\setminus\{\kappa\}$.

\begin{thm}\label{thm:palf-stickelberger-trivial}
  Sea $\psi\colon G\rightarrow\Qbar_p^\times$ un carácter de orden finito y
  $n\in\Nuno$. Entonces
  \[ \mu_{\mathbf 1}(\psi^{-1}\kappa^{1-n})=(1-\psi(p)p^{n-1})L(\psi,1-n). \]
\end{thm}
\begin{proof}[conjunta de los teoremas \labelcref{thm:palf-stickelberger-no-trivial}
  y \labelcref{thm:palf-stickelberger-trivial}]
  Asumamos sin pérdida de generalidad que $\O$ sea suficientemente grande tal que el
  carácter $\psi$ tenga valores en $\O$. Vamos a demostrar que, para cualquier $N$ y $\chi$,
  tenemos
  \begin{equation*}
    h_N\mu_\chi(\psi^{-1}\kappa^{1-n})=h_N(\psi^{-1}\kappa^{1-n})(1-\chi\psi(p)p^{n-1})L(\chi\psi,1-n).
  \end{equation*}
  Del \cref{ejer:integral-h-n-no-cero} sabemos que $h_N(\psi^{-1}\kappa^{1-n})\neq0$, por
  eso esto es suficiente.

  Sea $r\in\Nuno$. Usamos la siguiente notación: para cada $a\in\{1,\dotsc,Np^r\}$ escribimos
  $a(1+Np)=a_1+a_2Np^r$ con $0\le a_1<Np^r$. Entonces para estos $a$ tenemos
  $\chi(a(1+Np))=\chi(a_1)$ y $\pi_r(\sigma_{a(1+Np)})=\pi_r(\sigma_{a_1})$. Sea
  $h_{N,r}\in\O[G_{p^r}]$ la imagen de $h_N$.
  Con esto vemos (escribiendo $\sum_a$ para la
  suma sobre los $a\in\{1,\dotsc,Np^r\}$ coprimos a $Np$):
  \begin{align}
    \label{eqn:calculacion-integralidad-as}
      h_{N,r}\mu_{\chi,r}&=\mu_{\chi,r}+\frac1{Np^r}\sum_a(a_1+a_2Np^r)\chi(a_1)\pi_r(\sigma_{a_1})^{-1}
      \notag \\
      &=
      -\frac1{Np^r}\sum_aa\chi(a)\pi_r(\sigma_a)^{-1}+\frac1{Np^r}\sum_aa_1\chi(a_1)\pi_r(\sigma_{a_1})^{-1}+\sum_aa_2\chi(a_1)\pi_r(\sigma_{a_1})^{-1}
      \notag \\
      &= \sum_aa_2\chi(a_1)\pi_r(\sigma_{a_1})^{-1}\in\O[G_{p^r}]
  \end{align}
  por que si $a$ pasa por los $a\in\{1,\dotsc,Np^r\}$ coprimos a $Np$, $a_1$ pasa por los
  mismos elementos. 
  
  Ahora fijamos $r$ tal que el conductor de $\psi$ divida a $p^r$, lo que significa que podemos
  ver $\psi$ como carácter de $G_r$. Sea $\kappa_r$ el isomorfismo
  $G_r\isomarrow(\Z/p^r\Z)^\times$ y $\psi_r$ la composición de
  $\psi\colon G_r\rightarrow\O^\times$ con la proyección
  $\O^\times\rightarrow(\O/p^r\O)^\times$.  Entonces tenemos diagramas conmutativos
  \begin{equation*}
    \begin{tikzcd}
      G \arrow[r, "\kappa"] \arrow[d, twoheadrightarrow]
      & \Z_p^\times \arrow[d, twoheadrightarrow] \\
      G_r \arrow[r, "\kappa_r"]
      & (\Z/p^r\Z)^\times
    \end{tikzcd}
    \qquad\qquad
    \begin{tikzcd}
      G \arrow[r, "\psi"] \arrow[d, twoheadrightarrow]
      & \O^\times \arrow[d, twoheadrightarrow] \\
      G_r \arrow[r, "\psi_r"]
      & (\O/p^r\O)^\times
    \end{tikzcd}
  \end{equation*}
  y como $(\Z/p^r\Z)^\times\subseteq(\O/p^r\O)^\times$, el carácter
  $\psi^{-1}_r\kappa_r^{1-n}$ induce un homomorfismo
  $\psi_r^{-1}\kappa_r^{1-n}\colon\O[G_r]\rightarrow\O/p^r$.

  Ahora consideremos $\psi_r^{-1}\kappa_r^{1-n}(h_{N,r}\mu_{\chi,r})\in\O/p^r$. Por
  la definición de $\kappa_r$ y $\pi_r$ tenemos
  \begin{equation*}
    \kappa_r^{1-n}(\pi_r(\sigma_a)^{-1})=\kappa_r^{n-1}(\pi_r(\sigma_a))=\kappa_r^{n-1}(\pi_r(\sigma_{a_1})) \equiv a_1^{n-1}\mod p^r
  \end{equation*}
  para $a\in\{1,\dotsc,Np^r\}$ coprimo a $Np$. Además tenemos
  $\psi^{-1}(\pi_r(\sigma_a)^{-1})=\psi(\pi_r(\sigma_a))=\psi(a)=\psi(a_1)$ para los mismos
  $a$. Esto demuestra que
  \begin{equation*}
    \psi_r^{-1}\kappa_r^{1-n}(h_{N,r}\mu_{\chi,r}) \equiv \sum_a\chi\psi(a_1)a_1^{n-1}a_2\mod p^r.
  \end{equation*}

  Ahora usamos un truco con el teorema del binomio. Este nos dice que
  \begin{equation*}
    {((1+Np)a)}^n={(a_1+a_2Np^r)}^n\equiv a_1^n+na_1^{n-1}a_2Np^r \mod p^{2r}
  \end{equation*}
  y pues
  \begin{align*}
    \chi\psi(1+Np)(1+Np)^n\sum_a\chi\psi(a)a^n &= \sum_a\chi\psi(a_1)((1+Np)a)^n \\
    &\equiv\sum_a\chi\psi(a_1)a_1^n+nNp^r\sum_a\chi\psi(a_1)a_1^{n-1}a_2 \mod p^{2r}.
  \end{align*}
  Pero si dos expresiones en $\O$ son congruentes módulo $p^{2r}$ y una de ellas es
  divisible por $p^r$, entonces la otra también es divisible por $p^r$, y si dividimos las
  dos por $p^r$ entonces los resultados serán congruentes módulo $p^r$. Aplicamos esta
  observación aquí y obtenemos
  \begin{multline*}
    \psi_r^{-1}\kappa_r^{1-n}(h_{N,r}\mu_{\chi,r}) \equiv-(1-\chi\psi(1+Np)(1+Np)^n)\frac{1}{nNp^j}\sum_a\chi\psi(a)a^n \\
                                                   =-\frac1nh_N(\psi^{-1}\kappa^{1-n})\frac1{Np^r}\left(S_{n,\chi\psi}(Np^r)-\chi\psi(p)p^nS_{n,\chi\psi}(Np^{r-1})\right) \mod p^r
  \end{multline*}
  con $S_{n,\chi\psi}(Np^r)$ como en la \cref{numeros-de-bernoulli-p-adicos} (aquí usamos
  una fórmula del \cref{ejer:relacion-s-n-chipsi}).
  
  Por la conmutatividad de los diagramas de arriba, esto significa que
  \begin{multline*}
    h_N\mu_\chi(\psi^{-1}\kappa^{1-n})\equiv \\
    -\frac1nh_N(\psi^{-1}\kappa^{1-n})\left(\frac1{Np^r}S_{n,\chi\psi}(Np^r)-\chi\psi(p)p^{n-1}\frac1{Np^{r-1}}S_{n,\chi\psi}(Np^{r-1})\right) \mod p^r
  \end{multline*}
  para cada $r\in\Nuno$ suficientemente grande y la afirmación resulta de la
  \cref{numeros-de-bernoulli-p-adicos}.
\end{proof}

\begin{prop}\label{prop:unicidad-palf}
  Las fórmulas de interpolación de los
  teoremas \labelcref{thm:palf-stickelberger-no-trivial}
  y \labelcref{thm:palf-stickelberger-trivial} caracterizan el elemento $\mu_{\chi}$ de
  manera única.
\end{prop}
\begin{proof}
  Si tenemos dos elementos con esta propiedad de interpolación entonces su diferencia está
  en el núcleo del morfismo $\psi^{-1}\kappa^{1-n}\colon\LL(G)\rightarrow\O$ para cada $n\in\Nuno$
  y $\psi$ de orden finito. Para $i\in\{1,\dotsc,p-1\}$ y $n\in\Nuno$ sea $K_{i,n}\subseteq\LL(G)$
  el núcleo del morfismo $\LL(G)=\O\llbracket G\rrbracket\rightarrow\O$ inducido por
  $\omega^i\kappa_0^{1-n}\colon\Delta\times\Gamma\rightarrow\O^\times$. Vamos a demostrar que
  \begin{equation*} \bigcap_{\substack{1\le i\le p-1\\n\in\Nuno}}K_{i,n} = 0. \end{equation*}
  Para eso usamos el morfismo
  $\O\llbracket G\rrbracket\rightarrow\O\llbracket\Gamma\rrbracket$ inducido por
  \begin{equation*}
    G\isom\Delta\times\Gamma\rightarrow\LL(\Gamma)^\times, \quad
    g=(\delta,\gamma)\mapsto\omega^i(\delta)\gamma
  \end{equation*}
  que llamamos $\omega^i\id_\Gamma$. Es fácil ver
  que entonces el diagrama
  \begin{equation*}
    \begin{tikzcd}
      \O\llbracket G\rrbracket \arrow[r,swap,"\omega^i\id_{\Gamma}"]
      \arrow[rr, bend left, "\omega^i\kappa_0^{1-n}"] & \O\llbracket\Gamma\rrbracket
      \arrow[r, swap, "\kappa_0^{1-n}"] & \O 
    \end{tikzcd}
  \end{equation*}
  es conmutativo. Llamamos
  $\widetilde K_n$ el núcleo del morfismo $\O\llbracket\Gamma\rrbracket\rightarrow\O$
  inducido por $\kappa_0^{1-n}$. Entonces, para cada elemento de la intersección de los
  $K_{i,n}$ de arriba, su imagen bajo $\omega^i\id_\Gamma$ está en $\widetilde K_n$ para
  cada $n\in\Nuno$.  Por eso es suficiente demostrar que
  \begin{equation*} \bigcap_{n\in\Nuno}\widetilde K_n = 0. \end{equation*}

  Ahora usamos el isomorfismo
  \begin{equation*}
    \O\llbracket T\rrbracket \isomarrow\O\llbracket\Gamma\rrbracket,\quad T\mapsto\gamma-1
  \end{equation*}
  del \cref{thm:EquivPowPro} (donde usamos $\gamma=\sigma_{1+p}\in\Gamma$ como generador
  topológico). Un elemento de $\bigcap_n\widetilde K_n$ corresponde a una serie de potencia
  $f\in\O\llbracket T\rrbracket$ tal que
  \begin{equation*}
    f((1+p)^{1-n}-1)=0 \quad\text{para cada }n\in\Nuno.
  \end{equation*}
  Pero el teorema de preparación de Weierstraß (\cref{thm:weierstrass} y
  \cref{ejer:weierstrass}) muestra que una serie de potencia con una infinitud de ceros es
  cero.
\end{proof}

Esto termina la construcción de las funciones $L$ $p$-ádicas mediante elementos de
Stickelberger. Comparemos los resultados con el resultado preliminar del
\cref{thm:zeta-p-adica-prelim}. Ahí construimos una función continua
\[ \Hom(G,\Qbar_p^\times)\setminus\{\kappa\}\rightarrow\Qbar_p^\times \] que es
caracterizada de manera única por una fórmula de interpolación que relaciona los valores de
la función con valores de funciones $L$ complejas. Aquí construimos un elemento
$\mu_{\mathbf 1}$ en el anillo de cocientes de $\LL(G)$ que define la misma función en
$\Hom(G,\Qbar_p^\times)\setminus\{\kappa\}$ y por eso la hace \enquote{más algebraica};
aunado a esto, el elemento otra vez es únicamente caracterizado por esta propiedad. Además,
logramos construir un elemento análogo $\mu_\chi$ también para las funciones $L$ de cualquier
carácter de Dirichlet $\chi$ cuyo conductor no sea divisible por $p$.

\ejercicios

\begin{ejer}\label{ejer:integral-h-n-no-cero}
  \begin{enumerate}
  \item\label{ejer:integral-h-n-no-cero:a} Use las relaciones del \cref{lem:idempotentes}
    para demostrar que
    \[ h_N-1=\sum_{i=1}^{p-1}(h_N^{(i)}-1) \]
    (donde usamos la notación introducida en la \cref{defi:h-n}).
  \item Sea $\phi\colon G\rightarrow\Qbar_p^\times$ que escribimos como
    $\phi=\omega^j\phi_0$ como en el \cref{lem:nullstellen-h}. Deduzca de
    \ref{ejer:integral-h-n-no-cero:a} y este lema que $\phi(h_N)=1-(1+Np)\phi_0(1+Np)^{-1}$,
    independientemente de $j$.
  \item Concluya que $\phi(h_N)=0 \iff \phi_0=\kappa_0$.
  \end{enumerate}
\end{ejer}

\begin{ejer}\label{ejer:carideal-zpuno}
  Demuestre que $h_1\in\LL(G)$ es un generador del núcleo de
  $\kappa\colon\LL(G)\rightarrow\Zp$, así que tenemos un isomorfismo
  \[ \LL(G)/h_1\isom\Zp. \] Verifique que este isomorfismo es compatible con la acción de
  $G$ si dejamos actuar $g\in G$ en $\LL(G)$ por multiplicación con $g$ y en $\Zp$ por
  multiplicación con $\kappa(g)$.
\end{ejer}

\begin{ejer}
  Demuestre que los elementos de Stickelberger $\Sigma_{Np^r}\in\Q[G_{Np^r}]$ son compatibles
  con respecto a los mapeos $\Q[G_{Np^r}]\rightarrow\Q[G_{Np^s}]$ para $r\ge s\ge 1$.
\end{ejer}

\begin{ejer}\label{ejer:mu-chi-explicito}
  Demuestre que $\mu_\chi=(\mu_{\chi,r})\in\varprojlim_rK[G_{p^r}]$ con
  \[
    \mu_{\chi,r}=-\frac1{Np^r}\sum_{\substack{a=1\\(a,Np)=1}}^{Np^r}a\chi(a)\pi_r(\sigma_a)^{-1} \]
  donde $\pi_r\colon G_{Np^r}\rightarrow G_{p^r}$ es la proyección canónica.
\end{ejer}

\begin{ejer}\label{ejer:relacion-s-n-chipsi}
  Sea $S_{n,\chi\psi}(Np^r)$ como en la \cref{numeros-de-bernoulli-p-adicos} (para
  $n,N,r\in\Nuno$ y caracteres $\chi$ y $\psi$ como antes).
  Demuestre que entonces
  \begin{equation*}
    \sum_{\substack{a=1\\(a,Np)=1}}^{Np^r}\chi\psi(a)a^n=S_{n,\chi\psi}(Np^r)-\chi\psi(p)p^nS_{n,\chi\psi}(Np^{r-1}).
  \end{equation*}
\end{ejer}

\begin{ejer}\label{ejer:integral-h}
  En este ejercicio explicamos como demostrar el importante \cref{lem:integral-h}, el cual
  dice que $h^{(1)}_1\mu_{\mathbf 1}\in\LL(G)$.

  Fijemos $r\in\Nuno$.
  A continuación, cuando escribimos $\sum_a$ esto siempre significa una suma sobre los
  $a\in\{1,\dotsc,p^r\}$ coprimos a $p$.
  \begin{enumerate}
  \item\label{ejer:integral-h:a} Haga un cálculo similar a
    \eqref{eqn:calculacion-integralidad-as} para ver que
    \[ h_{1,r}^{(1)}\mu_{\mathbf 1,r}=-\frac1{p^r}\sum_aa\pi_r(\sigma_a)^{-1}+\frac1{p^r}\sum_aa_1
      e_\omega\pi_r(\sigma_{a_1})^{-1} + \text{ algo en }\O[G_{p^r}], \]
    donde escribimos otra vez $a(1+p)=a_1+a_2p^r$ con $0\le a_1<p^r$ para
    $a\in\{1,\dotsc,p^r\}$.
  \end{enumerate}
  Ahora consideremos la expresión
  $A:=\sum_aa(e_\omega-1)\pi_r(\sigma_a)^{-1}\in\O[G_{p^r}]$. Si podemos demostrar que $A$ es
  divisible por $p^r$ entonces la afirmación resulta de \ref{ejer:integral-h:a}.
  \begin{enumerate}[resume]
  \item Use la relación \ref{lem:idempotentes:suma} del \cref{lem:idempotentes} para
    obtener \[ A = \sum_a a\Big(\sum_{i=2}^{p-1} e_{\omega^i}\Big)\pi_r(\sigma_a)^{-1}. \]
  \end{enumerate}
  Tenemos una descomposición
  \[ \O[G_{p^r}] \isom \bigoplus_{i=1}^{p-1}e_{\omega^i}\O[G_{p^r}], \] y en el
  \cref{cor:descomposicion-lambda} vimos que cada uno de los factores
  $e_{\omega^i}\O[G_{p^r}]$ es isomorfo a $\O[\Gamma_r]$ con $\Gamma_r$ siendo tal que
  $G_{p^r}=\Delta\times\Gamma_r$. Escribimos $\pi_r^0(\sigma_a)$ para el imagen de
  $\pi_r(\sigma_a)$ bajo la proyección $G_{p^r}\rightarrow\Gamma_r$.
  \begin{enumerate}[resume]
  \item Verifique que la imagen de $A$ en la componente $e_{\omega^i}\O[G_{p^r}]\isom\O[\Gamma_r]$
    corresponde bajo este isomorfismo al elemento \[
      A_i:=\sum_a\frac{a}{\omega(a)}\omega^{1-i}(a)\pi_r^0(\sigma_a)^{-1}\in\O[\Gamma_r] \]
    si $i\neq1$ y a $0$ si $i=1$.
  \end{enumerate}
  Esto demuestra que es suficiente demostrar que $A_i$ es divisible por $p^r$ para
  $i\in\{2,\dotsc,p-1\}$ . Note que estos $i$ son justamente aquellos para los cuales el
  carácter $\omega^{1-i}$ no es trivial.
  \begin{enumerate}[resume]
  \item Sea $i\in\{2,\dotsc,p-1\}$. Para $b\in\Z$ con $p\nmid b$ consideremos la suma parcial
    \[ \sum_{\substack{a=1\\(a,p)=1\\\pi^0_r(\sigma_a)=\pi_r^0(\sigma_b)}}^{p^r}
      \frac{a}{\omega(a)}\omega^{1-i}(a)\pi_r^0(\sigma_a)^{-1}\in\O[\Gamma_r]. \]
    Verifique que \[ \frac{a}{\omega(a)}\equiv\frac{b}{\omega(b)}\quad \mod p^r \] para
    cada $a$ como en la suma de arriba. Use esto para demostrar que $A_i$ es divisible por
    $p^r$ mediante el argumento que usamos para demostrar el \cref{lem:mu-chi-entero}.
  \end{enumerate}
\end{ejer}

\section{Suplementos y consecuencias de la existencia de las funciones $L$ $p$-ádicas}
\label{sec:consecuencias-palf}

En la \cref{congruencias-de-kummer} demostramos las congruencias de Kummer, que son
congruencias módulo potencias de $p$ entre (esencialmente) valores especiales de la función
zeta de Riemann, y luego explicamos que estas congruencias implican la existencia de una
función continua $p$-ádica que interpola estos valores (\cref{thm:zeta-p-adica-prelim}). En
la otra dirección, la existencia de los elementos del álgebra de Iwasawa que construimos en
la sección anterior implica la existencia de congruencias de este estilo. Este fenómeno pasa en general para
cualquier elemento del álgebra de Iwasawa y pues aplica también a funciones $L$ $p$-ádicas
más generales (vamos a mencionar unos ejemplos en el \cref{sec:generalizaciones}). Por eso
explicamos la derivación de estas congruencias en una situación general y luego lo aplicamos
a los elementos de la sección anterior.

Fijamos $\O$, el anillo de enteros de una extensión finita de $\Qp$, y
$G=\Gal(\Q(\mu_{p^\infty})/\Q)\isom\Delta\times\Gamma$ con $\Delta=\Gal(\Q(\mu_p)/\Q)$ y
$\Gamma\isom\Zp$. Escribimos $\LL(G)=\O\llbracket G\rrbracket$ y
$\LL(\Gamma)=\O\llbracket\Gamma\rrbracket$.

En lo siguiente vamos a estudiar caracteres $\psi\colon G\rightarrow\O^\times$ y los
morfismos $\LL(G)\rightarrow\O$ inducidos por ellos. Se puede ver fácilmente que cada tal
carácter es de la forma $\psi=\omega^i\psi_0$ con $i\in\{1,\dotsc,p-1\}$ y
$\psi_0\colon\Gamma\rightarrow\O^\times$ un carácter (donde $\omega$ denota el carácter de
Teichmüller).

Fijamos $\mu\in\LL(G)$. Por el isomorfismo del \cref{lem:descomposicion-lambda}, $\mu$
corresponde a una colección
\[ (\mu_1,\dotsc,\mu_{p-1})\in\LL(\Gamma)^{p-1}\isom\bigoplus_{i=1}^{p-1}e_i\LL(G)=\LL(G). \]
Para cada $i\in\{1,\dotsc,p-1\}$ el elemento $\mu_i\in\LL(\Gamma)$ puede ser definido
también como la imagen de $\mu$ bajo el morfismo
$\omega^i\id_\Gamma\colon\LL(G)\rightarrow\LL(\Gamma)$ que ya usamos en la demostración de
la \cref{prop:unicidad-palf}, inducido por
\[ G\isom\Delta\times\Gamma\rightarrow\LL(\Gamma)^\times, \quad
  g=(\delta,\gamma)\mapsto\omega^i(\delta)\gamma. \] 
En otras palabras, el diagrama
\begin{equation}\label{eqn:diagrama-evaluacion-g-gamma}
  \begin{tikzcd}
    \LL(G) \arrow[d,"\sim"] \arrow[r,"\omega^i\id_{\Gamma}"]
     & \LL(\Gamma)
    \\
    \displaystyle\bigoplus_{i=1}^{p-1}e_{\omega^i}\LL(G) \arrow[r, twoheadrightarrow] &
    e_{\omega^i}\LL(G) \arrow[u,"E_{\omega^i}","\sim"']
  \end{tikzcd}
\end{equation}
es conmutativo.  Eso implica que para cada carácter $\psi\colon G\rightarrow\O^\times$ que
podemos escribir de la forma
$\psi=\omega^i\psi_0$ con $i\in\{1,\dotsc,p-1\}$ y $\psi_0\colon\Gamma\rightarrow\O^\times$
tenemos
\begin{equation}
  \label{eqn:mu-ramas}
   \mu(\psi)=\mu_i(\psi_0).
\end{equation}
Con estas preparaciones ahora podemos formular y probar dos resultados generales sobre
propiedades de valores especiales de $\mu$.

\begin{prop}\label{prop:congruencias-de-kummer-general}
  Sea $\mu\in\LL(G)$ y $\psi\colon G\rightarrow\O^\times$ un carácter. Entonces tenemos las
  siguientes congruencias entre los valores de $\mu$: Para cada $k\in\Nuno$ existe
  $l\in\Nuno$ tal
  que para cada $m,n\in\Z$ tenemos
  \[ m\equiv n \mod (p-1)p^l \quad\implies\quad \mu(\psi\kappa^m)\equiv\mu(\psi\kappa^n)\mod p^k. \]
\end{prop}
\begin{proof}
  Escribimos $\psi=\omega^j\psi_0$ con $\psi_0\colon\Gamma\rightarrow\O^\times$. 
  Sea $i\in\{1,\dotsc,p-1\}$ fijo. La función
  \begin{equation*}
    \Zp\rightarrow\O,\quad s\mapsto\mu_i(\psi_0\kappa^s)
  \end{equation*}
  es continua, y además uniformemente continua porque $\Zp$ es compacto. Como ya explicamos en
  la demostración de la \cref{existencia-zeta-p-funcion}, la continuidad uniforme
  significa
  \begin{equation*}
    \forall\, k\in\Nuno\;\exists\, l\in\Nuno\; \forall\, m,n\in\Z\colon\quad m\equiv n\mod p^l\implies
    \mu_i(\psi_0\kappa^m)\equiv\mu_i(\psi_0\kappa^n)\mod p^k.
  \end{equation*}
  La afirmación entonces resulta de \eqref{eqn:mu-ramas}.
\end{proof}

\begin{prop}\label{prop:unidad-solo-depende-de-i}
  Sea $\mu\in\LL(G)$
  y $\psi\colon G\rightarrow\O^\times$ un carácter que escribimos como $\psi=\omega^i\psi_0$
  con $i\in\{1,\dotsc,p-1\}$ y $\psi_0\colon\Gamma\rightarrow\O^\times$. Entonces
  \[ \mu(\psi)\in\O^\times \iff \mu_i\in\LL(\Gamma)^\times. \]
  En particular, si
  $m,n\in\Z$ con $m\equiv n$ ($\mod p-1$) entonces \[ \mu(\kappa^m)\in\O^\times \iff
    \mu(\kappa^n)\in\O^\times. \]
\end{prop}
\begin{proof}
  Observemos que el morfismo $\LL(\Gamma)\rightarrow\O$ inducido por $\psi_0$ es local,
  es decir $\psi_0(\mathfrak M)\subseteq\mathfrak m$, donde $\mathfrak M$ y
  $\mathfrak m$ son los ideales máximos de $\LL(\Gamma)$ y $\O$ respectivamente. En
  efecto, si identificamos $\LL(\Gamma)\isom\O\llbracket T\rrbracket$ mediante el
  isomorfismo del \cref{thm:EquivPowPro} usando un generador topológico $\gamma$ de
  $\Gamma$, entonces $\mathfrak M=(\pi,T)$, y solo hay que ver que
  $\psi(T)\in\mathfrak m=\pi\O$ (porque por linealidad claramente
  $\psi(\pi)\in\mathfrak m$). El elemento $T$ corresponde a $\gamma-1$, por eso es
  suficiente demostrar que la imagen de $\Gamma\subseteq\LL(\Gamma)^\times$ está contenida
  en $1+\pi\O\subseteq\O^\times$ (que también demuestra que el argumento no depende del
  generador topológico $\gamma$ que elegimos). Para ver esto consideremos la composición
  \begin{equation*}
    \Gamma\labeledarrow{\psi_0}\O^\times\rightarrow(\O/\pi\O)^\times
  \end{equation*}
  donde el segundo mapeo es la proyección canónica. El grupo a la derecha es finito con
  orden primo a $p$ mientras el grupo a la izquierda es un grupo pro-$p$. Por eso la
  composición debe ser el morfismo trivial (usamos el \cref{ejer:morfismos-pro-p-pro-ell}),
  así que $\psi_0(\Gamma)\subseteq1+\pi\O$.

  Esto último y \eqref{eqn:mu-ramas} implican que
  $\mu(\omega^i\kappa_0^s)=\mu_i(\kappa_0^s)=\kappa_0^s(\mu_i)\in\O^\times$ si y solo si
  $\mu_i\in\LL(\Gamma)^\times$.
\end{proof}

Ahora apliquemos estos resultados a los elementos especiales que construimos en la sección
anterior. El primer resultado es una reminiscencia de las congruencias de Kummer (\cref{congruencias-de-kummer}).

\begin{cor}\label{cor:congruencias-de-kummer-dirichlet}
  Sea $\xi$ un carácter de Dirichlet cualquiera. Entonces para cada $k\in\Nuno$ existe
  $l\in\Nuno$ tal que para cada $m,n\in\Nuno$ tenemos
  \begin{multline*}
     m\equiv n \mod (p-1)p^l \quad\implies \\
    (1-\xi(p)p^{n-1})L(\xi,1-n)\equiv(1-\xi(p)p^{m-1})L(\xi,1-m)\mod p^k. 
  \end{multline*}
\end{cor}
\begin{proof}
  Descomponemos $\xi=\chi\psi$ con $\chi$ de conductor no divisible por $p$ y $\psi$ de
  conductor una potencia de $p$. Si $\chi$ no es trivial, la afirmación resulta de aplicar
  la \cref{prop:congruencias-de-kummer-general} al elemento $\mu=\mu_\chi$ de la sección
  anterior, remplazando $\psi$ por $\psi^{-1}$ y $\kappa^n$ por $\kappa^{1-n}$ y usando la
  fórmula de interpolación del \cref{thm:palf-stickelberger-no-trivial}. Si $\chi$ es
  trivial, tenemos que usar $\mu_{\mathbf 1}$ y la fórmula de interpolación del
  \cref{thm:palf-stickelberger-trivial}, que estrictamente no resulta de la
  \cref{prop:congruencias-de-kummer-general} porque $\mu_{\mathbf 1}\notin\LL(G)$, pero se
  puede ver fácilmente que todavía funciona en esta situación (omitimos los detalles).
\end{proof}

\begin{cor}\label{cor:unidad-depende-solo-de-i}
  Sea $\chi$ un carácter de Dirichlet con valores en $\O$ cuyo conductor no es divisible por
  $p$ (permitimos el carácter trivial). Para cada $m,n\in\Nuno$, si $m\equiv n$ ($\mod p-1$)
  entonces
  \[ L(\chi,1-n)\in\O^\times\iff L(\chi,1-m)\in\O^\times. \]
\end{cor}
\begin{proof}
  Esto resulta de la \cref{prop:unidad-solo-depende-de-i} y las fórmulas de interpolación
  con las mismas calibraciones como en la demostración anterior (usamos que los factores de
  Euler $1-\chi(p)p^{n-1}$ siempre están en $\O^\times$).\footnote{En el caso del carácter
    trivial tenemos que aplicar la \cref{prop:unidad-solo-depende-de-i} a
    $\mu=h^{(1)}_1\mu_{\mathbf 1}$, que nos da la equivalencia afirmada solo en el caso
    $m,n\not\equiv1$ ($\mod p-1)$ porque por lo demás los valores de $h_1^{(1)}$ son
    divisible por $p$ (\cref{lem:nullstellen-h}). Pero aún así la afirmación es cierta
    también si $m,n\equiv1$ ($\mod p-1)$, lo que podemos ver de las congruencias de Kummer
    clásicas (\cref{congruencias-de-kummer}).}
\end{proof}

Ahora calculamos unos valores especiales más de la función zeta $p$-ádica que serán
interesantes más tarde. Como vimos en el \cref{lem:integral-h} de la sección anterior, la
función zeta $p$-ádica es un elemento $\mu_{\mathbf 1}$ del anillo de cocientes de $\LL(G)$
tal que $h_1^{(1)}\mu_{\mathbf 1}\in\LL(G)$. El elemento $h_1^{(1)}\in\LL(G)$ tiene su único
zero en el carácter $\kappa$ (\cref{lem:nullstellen-h}), por eso podemos interpretar esto
como si la función $\mu_{\mathbf 1}$ tuviera un único \enquote{polo} simple en $\kappa$, tal como la
función zeta de Riemann compleja tiene un único polo simple en $s=1$. El valor de
$h_1^{(1)}\mu_{\mathbf 1}$ es como el \enquote{residuo} de $\mu_{\mathbf 1}$ en $\kappa$. En
una aplicación ulterior será importante saber que este \enquote{residuo} nunca es divisible
por $p$, y lo demostramos aquí.

\begin{lem}\label{lem:valor-zeta-omega}
  Para cada $i\in\Z$ tenemos $\mu_{\mathbf 1}(\omega^i)=-B_{1,\omega^{-i}}\in\Zp$.
\end{lem}
\begin{proof}
  Del \cref{ejer:mu-chi-explicito} sabemos que $\mu_{\mathbf 1}=(\mu_{\mathbf 1,r})_r$ con
  \begin{equation*}
    \mu_{\mathbf 1,r}=
    -\frac1{p^r}\sum_{\substack{a=1\\(a,p)=1}}^{p^r}a\pi_r(\sigma_a)^{-1}\in\Qp[G_{p^r}].
  \end{equation*}
  El morfismo $\LL(G)[1/h_1^{(1)}]\rightarrow\Qp$ inducido por $\omega^i$ se factoriza a través
  de $\LL(G)[1/h_1^{(1)}]\rightarrow\Qp[G_p]$, así que $\mu_{\mathbf 1}(\omega^i)$ es la imagen de
  \begin{equation*}
    \mu_{\mathbf 1,1}=
    -\frac1{p}\sum_{a=1}^{p-1}a\pi_1(\sigma_a)^{-1}
  \end{equation*}
  bajo el morfismo $\Qp[G_p]\rightarrow\Qp$ inducido por $\omega^i$. La afirmación entonces
  resulta del \cref{ejer:formula-b-n-chi}.
\end{proof}

\begin{remark}\label{rem:stickelberger-mu-uno}
  Observemos que el elemento $\mu_{\mathbf 1,1}$ simplemente es el elemento de Stickelberger
  $\Sigma_p$. Es decir, la demostración anterior muestra que $\mu_{\mathbf 1}(\omega^i)$ es
  la imagen del elemento de Stickelberger $\Sigma_p\in\Q[G_p]\subseteq\Qp[G_p]$ bajo el
  morfismo $\Qp[G_p]\rightarrow\Qp$ inducido por $\omega^i$ (que es igual a
  $B_{1,\omega^{-i}}$).
\end{remark}

\begin{cor}\label{cor:residuo-zeta-p-adica}
  Tenemos $h_1^{(1)}\mu_{\mathbf 1}(\kappa)\in\Z_p^\times$.
\end{cor}
\begin{proof}
  Según la \cref{prop:unidad-solo-depende-de-i}, la afirmación es equivalente a
  $h_1^{(1)}\mu_{\mathbf 1}(\omega\kappa_0^s)\in\Z_p^\times$, para cualquier $s\in\Zp$.  Lo
  probamos para $s=0$.  Del \cref{lem:nullstellen-h} vemos que $h_1^{(1)}(\omega)=-p$, así
  que gracias al \cref{lem:valor-zeta-omega} sabemos que
  \begin{equation*}
    h_1^{(1)}\mu_{\mathbf 1}(\omega)=pB_{1,\omega^{-1}}
  \end{equation*}
  Según el \cref{ejer:formula-b-n-chi} tenemos
  \begin{equation*}
    pB_{1,\omega^{-1}}=\sum_{a=1}^{p-1}a\omega^{-1}(a)
  \end{equation*}
  y falta que ver que esto es invertible en $\Zp$. Esto resulta porque para cada
  $a\in\{1,\dotsc,p-1\}$ tenemos $a\omega^{-1}(a)\in1+p\Zp$, así que
  \begin{equation*}
    \sum_{a=1}^{p-1}a\omega^{-1}(a)\equiv -1\; (\mod p).
  \end{equation*}
\end{proof}

Terminamos esta sección interpretando las funciones $L$ $p$-ádicas geométricamente,
asumiendo que el lector conoce los conceptos básicos de la geometría algebraica. Esto no
será importante para el resto del texto, así que omitimos algunos detalles.

En lo esencial, todas las funciones $L$ $p$-ádicas son dadas por un elemento $\mu$ del
anillo $\LL(G)$ y por eso define una sección global en el esquema $X=\Spec\LL(G)$; en el
caso de la función zeta $p$-ádica de Riemann es un elemento del anillo de fracciones de
$\LL(G)$ que define un elemento de $\O_X(D(h_1))$. La fórmula de interpolación del
\cref{thm:palf-stickelberger-no-trivial} describe los valores de esta sección en los ideales
primos que son los núcleos de los morfismos $\LL(G)\rightarrow\O$ inducidos por los
caracteres de la forma $\psi^{-1}\kappa^{1-n}$ como en el teorema. La razón detrás de la
unicidad que demostramos en la \cref{prop:unicidad-palf} es que estos ideales primos son
Zariski densos en $X$ (en el fondo esto es lo que demostramos allá). Es decir, las funciones
$L$ $p$-ádicas son elementos de $\O_X(X)$ (o $\O_X(D(h_1))$) cuyos valores en un conjunto
denso de puntos son dados por una fórmula de interpolación.

El anillo $\LL(G)$ tiene una propiedad universal bonita en la categoría de $\O$-álgebras
profinitas (\cref{prop:propiedad-universal-anillo-de-grupos}), pero el esquema $X$ no tiene
la propiedad análoga en la categoría de $\O$-esquemas porque los morfismos entre esquemas
vienen de \emph{todos} los morfismos entre anillos y no solo de los que son continuos. Es
posible rescatar la propiedad universal en el mundo geométrico usando \define{esquemas
  formales} en lugar de esquemas. Más precisamente, el espectro formal
$\operatorname{Spf}\LL(G)$ sí tiene una propiedad universal: representa el funtor que a cada
$\O$-esquema formal (localmente noetheriano) $Y$ asocia los caracteres (continuos!)
$G\rightarrow\O_Y(Y)^\times$. Desde este punto de vista, la función $L$ $p$-ádica es una
sección global en $\operatorname{Spf}\LL(G)$, es decir una función en un \enquote{espacio de
  móduli de caracteres}. Pero ahora la imagen geométrica ya no es tan bonita: como espacio
topológico $\operatorname{Spf}\LL(G)$ es finito y discreto, así que la intuición de que los
valores  en un conjunto denso de puntos son dados por una fórmula de
interpolación ya no funciona.

Se puede combinar las dos intuiciones usando la teoría de \define{espacios ádicos}
introducida por Huber en \cite{MR1306024}. El espectro ádico
$\operatorname{Spa}(\LL(G),\LL(G))$ tiene ambas propiedades deseadas: tiene una propiedad
universal análoga en la categoría de espacios ádicos sobre $\O$ topológicamente de tipo
finito y contiene un conjunto denso de puntos que corresponden a los caracteres
$\LL(G)\rightarrow\O$. Por eso, la interpretación geométrica más elegante de las funciones
$L$ $p$-ádicas es como funciones en el espacio ádico $\operatorname{Spa}(\LL(G),\LL(G))$.

\ejercicios

\begin{ejer}
  En la demostración del \cref{lem:valor-zeta-omega} el cálculo era mucho menos trabajo
  que en la situación similar en la demostración del
  \cref{thm:palf-stickelberger-trivial}. ¿Por qué es esto y por qué no funciona un argumento
  similar para simplificar la demostración del
  \cref{thm:palf-stickelberger-trivial}?
\end{ejer}

\section{La teoría de Coleman y otras maneras de construir la función zeta $p$-ádica}
\label{sec:coleman}

En la \cref{sec:stickelberger} construimos la función zeta $p$-ádica usando los elementos de
Stickelberger. Aunque esto fue un poco laborioso, nos ayudó a comprender la naturaleza de
estos elementos y nos permitió demostrar los suplementos de la sección anterior. Sin
embargo, también hay otros métodos para construir la función zeta $p$-ádica, y por
completitud queremos mencionar algunos aquí. Los resultados de esta sección no serán
usados en el resto del texto, así que el lector es libre de omitirla.

La manera alternativa más importante usa la teoría de las series de potencia de Coleman y
unidades ciclotómicas, que son unidades de forma especial en campos ciclotómicos. Su
importancia proviene de que esta construcción da aún más información sobre la función zeta
$p$-ádica que se puede aprovechar para elaborar una demostración de la Conjetura Principal
(véase \cref{sec:dem-cp}). No vamos a explicar esta teoría aquí en completo, pero queremos
esbozar unas ideas de la construcción.

Antes de empezar con esto, queremos mencionar que se pueden usar los polinomios de Bernoulli
que introducimos en la \cref{defi:polinomios-de-bernoulli} en lugar de los elementos de
Stickelberger para definir medidas en $\Z_p^\times$ más directamente cuyas integrales
están relacionadas con los números de Bernoulli y por lo tanto con la función zeta. Esta construcción
está relacionada con la construcción que usa los elementos de Stickelberger y es explicada en
\cite[§12.1--2]{MR1421575} o \cite[§2.2]{MR1029028}. Por otra parte, en su libro
\cite[Chap.\ 4]{MR1216135}, Hida desarrolla una manera para construir las funciones $L$ de
Dirichlet $p$-ádicas, usando métodos de grupos de (co)homología, que es paralela a la construcción
de funciones $L$ $p$-ádicas para formas modulares, proveyendo una perspectiva
completamente diferente a estas funciones.

En el resto de la sección esbozamos la construcción de la función zeta $p$-ádica mediante
unidades ciclotómicas y la teoría de Coleman, siguiendo \cite{MR2256969}.

\begin{defi}
  Para cada $r\in\Nuno$ sea $U_r=\O_r^\times$ con $\O_r$ siendo el anillo de enteros del campo
  $\Q_p(\mu_{p^r})$. La norma envía $U_s$ a $U_r$ para cada $s\ge r\ge 1$ de manera
  compatible. Definimos \[ U_\infty=\varprojlim_{r\in\Nuno} U_r, \] el límite tomado con respecto a
  la norma.
\end{defi}

El resultado de Coleman, que es demostrado en \cite{MR0560409}, dice lo siguiente. Fijamos
para cada $r\ge1$ una raíz $p^r$-ésima primitiva de la unidad $\xi_r$ tal que
$\xi_{r+1}^p=\xi_r$ para cada $r\ge1$, y ponemos $\pi_r=\xi_r-1$, que es un uniformizante de
$\O_r$.

\begin{thm}[Coleman]
  Para cada $u=(u_r)_{r\in\Nuno}\in U_\infty$ existe un único $f_u\in{\Z_p\llbracket
  X\rrbracket}^\times$ tal que $f_u(\pi_r)=u_r$ para cada $r\in\Nuno$.
\end{thm}

La demostración de este teorema es un poco laboriosa y remitimos a \cite{MR0560409} o
\cite[Thm.\ 2.1.2 y §2.3]{MR2256969} para ella. Las series de potencia que aparecen aquí
juegan un papel diferente que las que aparecieron en el \cref{sec:algebra-de-iwasawa}, por
eso llamamos la variable $X$ en lugar de $T$.

Vamos a aplicar el teorema de Coleman sólo a elementos de $U_\infty$ de una forma especial.

\begin{defi}\label{dfn:unidades-ciclotomicas}
  Para  $a,b\in\Z$ coprimos a $p$ y cada $r\in\Nuno$ usamos los elementos
  \[ c_r(a,b)=\frac{\xi_r^{-a/2}-\xi_r^{a/2}}{\xi_r^{-b/2}-\xi_r^{b/2}} \in U_r \] que
  estudiamos en el \cref{ejer:cicunit}.  Estos elementos se llaman \define{unidades
    ciclotómicas}. Por el \cref{ejer:norm-compatible} tenemos $c_r(a,b)\in U_r$ y los
  elementos son compatibles con respecto a la norma y definen un elemento
  \[ c(a,b)=(c_r(a,b))_{r\in\Nuno} \in U_\infty. \]
\end{defi}

Por el teorema de Coleman entonces existe una única serie de potencias
$f_{c(a,b)}\in{\Zp\llbracket X\rrbracket}^\times$ tal que $f_u(\pi_r)=c_r(a,b)$ para cada
$r\in\Nuno$. No obstante, en este caso particular no es necesario usar toda la fuerza del
teorema: es posible dar la serie $f_{c(a,b)}$ de forma explicita.

\begin{lem}\label{lem:cps-cab}
  Sea $k\in\Z$ coprimo a $p$ y \[ w_k=\frac{(1+X)^{-k/2}-(1+X)^{k/2}}X. \] Entonces
  $w_k\in{\Z_p\llbracket X\rrbracket}^\times$ y $f_{c(a,b)}=\displaystyle\frac{w_a}{w_b}$
  para $a,b\in\Z$ coprimos a $p$. Además, la serie de potencias $f_{c(a,b)}$ tiene
  coeficientes en $\Q\cap\Zp$.
\end{lem}
\begin{proof}
  En $\Z_p\llbracket X\rrbracket$ vale\footnote{La expresión $(1+X)^s$ está bien definida porque $\Z_p\llbracket X\rrbracket\isom\Z_p\llbracket\Gamma\rrbracket$ para cada grupo
    profinito $\Gamma$ isomorfo a $\Z_p$ con $1+X$ correspondiendo a un generador topológico
    de $\Gamma$.}
  \begin{equation*}
    (1+X)^s = \sum_{n=0}^\infty\binom s n X^n
  \end{equation*}
  para cada $s\in\Z_p$.
  Esto se puede ver del teorema del binomio y un argumento de continuidad.
  Con esto tenemos
  \begin{align*}
    w_k=\frac{(1+X)^{-k/2}-(1+X)^{k/2}}X
    &=\frac1X\left(\sum_{n=0}^\infty\left(\binom{-k/2}n-\binom{k/2}n\right)X^n\right) \\
    &=\sum_{n=1}^\infty\left(\binom{-k/2}n-\binom{k/2}n\right)X^{n-1}.
  \end{align*}
  De esto vemos que $w_k(0)=-k\in\Z_p^\times$, así que
  $w_k\in{\Z_p\llbracket X\rrbracket}^\times\cap{\Q\llbracket X\rrbracket}^\times$ (recuerde
  la definición de $\binom\cdot n$ en \eqref{eqn:defi-binom}). El resto es claro.
\end{proof}

A partir de aquí sólo indicamos los próximos pasos para obtener la función zeta
$p$-ádica a partir del elemento $f_u$ encontrado arriba. En la sección anterior vimos
$\Z_p^\times$ como \enquote{$p-1$ copias de $\Zp$}. Aquí lo consideramos mas bien como
subconjunto de $\Zp$. Como $\Z_p^\times$ es abierto y cerrado en $\Z_p$, la función
característica $\mathbb 1_{\Z_p^\times}$ es continua. Por eso podemos restringir medidas de
$\Z_p$ a $\Z_p^\times$ de la manera siguiente.
\begin{defi}
  Para $\mu\in\mathrm D(\Z_p,\Zp)$ definimos una medida $\mu_!\in\mathrm D(\Z_p,\Zp)$ como \[ \mu_!(f)=\mu(\mathbb 1_{\Z_p^\times}f)\quad(f\in\mathrm C(\Z_p,\Zp)). \]
  Además definimos una inclusión
  \[ i\colon\mathrm D(\Z_p^\times,\Zp)\rightarrow\mathrm D(\Zp,\Zp),\quad\int_{\Zp}f\integrald i(\eta)=\int_{\Z_p^\times}f|_{\Z_p^\times}\integrald\eta \quad(\eta\in\mathrm D(\Z_p^\times,\Zp),\ f\in\mathrm C(\Zp,\Zp)). \]
\end{defi}
Se puede verificar entonces que el diagrama
\[ 
  \begin{tikzcd}
    \mathrm D(\Zp,\Zp) \arrow[r,"\mu\mapsto\mu_!"] \arrow[d, "\sim"] & \mathrm D(\Zp,\Zp) \arrow[d, "\sim"] \\
    \Zp\llbracket \Zp\rrbracket \arrow[r] & \Zp\llbracket T\rrbracket
  \end{tikzcd}
\]
es conmutativo, donde el mapeo de abajo está dado por
\[ g\mapsto g-\frac1p\sum_{\xi\in\mu_p}
  g(\xi(1+T)-1)\quad(g\in\Zp\llbracket\Zp\rrbracket). \]
Eso describe la restricción de medidas en términos de series de potencias. También se puede
caracterizar las medidas en $\Zp$ que se pueden obtener de esta manera: Tenemos \[
  i(\mathrm D(\Z_p^\times,\Zp)) = \{\mu\in\mathrm D(\Zp,\Zp) \mid \mu=\mu_!\}. \]
Para estos hechos véase \cite[Lem.\ 3.4.1, 3.4.2]{MR2256969}.

Usando estas técnicas y una operación $\mathcal L$ en las series de potencias se puede obtener
una medida en $\Z_p^\times$ para cada elemento de $U_\infty$. Solo citamos el resultado aquí.

\begin{prop}\label{prop:u-medida}
  Para $f\in{\Z_p\llbracket X\rrbracket}^\times$ ponemos
  \[ \mathcal L(f)=\frac1p\log\left(\frac{f^p(T)}{f((1+T)^p-1)}\right). \]
  Entonces esto define un morfismo $\mathcal L\colon{\Z_p\llbracket X\rrbracket}^\times\rightarrow\Z_p\llbracket T\rrbracket$.\footnote{Notemos que después del mapeo $\mathcal L$ ya llamamos la variable $T$, porque estas series de potencias las identificaremos con medidas, es decir s\'i juegan el mismo papel que las de el \cref{sec:algebra-de-iwasawa}.} Si lo componemos con el morfismo $\Upsilon$ del \cref{thm:trinidad}, entonces la composición \[ \Upsilon\circ\mathcal L\colon{\Z_p\llbracket X\rrbracket}^\times\rightarrow\mathrm D(\Zp,\Zp) \] envía los elementos de la forma $f_u$ para $u\in U_\infty$ a $i(\mathrm D(\Z_p^\times,\Zp))$. De esta manera obtenemos para cada $u\in U_\infty$ una medida $\lambda_u\in\mathrm D(\Z_p^\times,\Zp)$.
\end{prop}
\begin{proof}
  \cite[Lem.\ 2.5.1]{MR2256969}
\end{proof}

Por supuesto las integrales de las medidas obtenidas de esta manera deberían tener algo que
ver con las unidades en $U_\infty$ con las cuales empezamos. Para explicar esta relación usamos
los mapeos
\[ D\colon\Z_p\llbracket X\rrbracket\rightarrow\Z_p\llbracket X\rrbracket,\quad f\mapsto
  (1+X)f', \] y para cada $n\in\Nuno$
\[ \delta_n\colon U_\infty\rightarrow\Zp,\quad u\mapsto\left(D^{n-1}\left(\frac{(1+X)f_u'}{f_u}\right)\right)(0) \]
(donde \enquote{$(0)$} significa evaluación de la serie de potencia en $X=0$).

\begin{lem}\label{lem:delta-n}
  \begin{enumerate}
  \item\label{lem:delta-n-eval} Para cada $u\in U_\infty$ y $n\in\Nuno$ tenemos
    \[ \int_{\Z_p^\times} x^n\integrald\lambda_u(x) = (1-p^{n-1})\delta_n(u). \]
  \item\label{lem:delta-n-unicic} Para $a,b\in\Z$ coprimos a $p$ y cada $n\in\Nuno$ tenemos
    \[ \delta_n(c(a,b))=(b^n-a^n)\zeta(1-n). \]
  \end{enumerate}
\end{lem}
\begin{proof}
  Seguimos \cite[Lem.\ 3.3.5, Prop.\ 3.5.2, Prop.\ 2.6.3]{MR2256969}, a donde remitimos para
  más detalles.

  Primero vamos a describir qué hace el mapeo $D$ a las medidas. Hacemos esto con ayuda del diagrama
  \begin{equation*}
    \begin{tikzcd}
      \Zp\llbracket T\rrbracket \arrow[r, "D"] \arrow[d, shift left, "\Upsilon"] & \Zp\llbracket T\rrbracket
      \arrow[d, "\Upsilon", shift left] \\
      \mathrm D(\Zp,\Zp) \arrow[r] \arrow[u, "\mathcal M", shift left] & \mathrm D(\Zp,\Zp)
      \arrow[u, "\mathcal M", shift left] 
    \end{tikzcd}
  \end{equation*}
  que conmuta si definimos el mapeo abajo como
  \begin{equation*}
    \mu\mapsto\left[f\mapsto\int_{\Zp} xf(x)\integrald\mu(x)\right];
  \end{equation*}
  eso es fácil de verificar y omitimos aquí el cálculo que lo demuestra.
  Inductivamente, esto nos da
  \begin{equation*}
    \int_{\Zp}x^n\integrald\Upsilon(g)(x)=(D^ng(T))(0)\quad\text{para cada }g\in\Zp\llbracket
    T\rrbracket,\ n\in\Ncero
  \end{equation*}
  (el caso $n=0$ es el \cref{ejer:upsilon-cero}).

  Ahora ponemos $g=\mathcal L(f_u)$ aquí. Un cálculo usando las definiciones del mapeo
  $\mathcal L$ y de $D$ nos muestra que
  \begin{equation*}
    D(\mathcal L(f_u)) = h_u-h_u((1+T)^p-1)
  \end{equation*}
  con $h_u=(1+T)\frac{f_u'}{f_u}$, así que
  inductivamente obtenemos
  \begin{equation*}
    \int_{\Zp}x^n\integrald\Upsilon(\mathcal L(f_u))(x)=D^{n-1}(h_u-h_u((1+T)^p-1))(0).
  \end{equation*}
  Además se puede verificar la relación
  \begin{equation*}
    D^{n-1}(h_u((1+T)^p-1))=p^{n-1}D^{n-1}(h_u)((1+T)^p-1).
  \end{equation*}
  Usando esto, la linealidad de $D$ y la definición de $\delta_n$ obtenemos
  \ref{lem:delta-n-eval}.

  Para \ref{lem:delta-n-unicic} usamos que conocemos la forma de la serie de potencias
  involucrada: sea
  \begin{equation*}
    f=\frac{(1+T)^{-a/2}-(1+T)^{a/2}}{(1+T)^{-b/2}-(1+T)^{b/2}}.
  \end{equation*}
  Sabemos del \cref{lem:cps-cab} que $f_{c(a,b)}=f$ y que ademas esta serie de potencias
  tiene coeficientes en $\Q$. Esto nos permite poner números reales o complejos en
  $f$. Vamos a poner $T=\e^z-1$ con $z\in\C$ porque esto convierte al mapeo $D$ en algo simple:
  $Df(\e^z-1)=\frac{\mathrm d}{\mathrm d z}f(\e^z-1)$. Es decir
  \begin{equation*}
    \delta_n(c(a,b))=\left({\left(\frac{\mathrm d}{\mathrm d z}\right)}^{n-1}g(z)\right)_{z=0}
  \end{equation*}
  con
  \begin{equation*}
    g(z)=\frac{\mathrm d}{\mathrm d z}\log f(\e^z-1).
  \end{equation*}
  Calculando explícitamente $g$ obtenemos
  \begin{equation*}
    g(z)=\frac12b\left(\frac1{\e^{-bz}-1}-\frac1{\e^{bz}-1}\right)
    -\frac12a\left(\frac1{\e^{-az}-1}-\frac1{\e^{az}-1}\right)
  \end{equation*}
  Si escribimos
  \begin{equation*}
    \tilde f(t)=\frac{t}{1-\e^{-t}}=\sum_{n=0}^\infty B_n\frac{t^n}{n!}\quad(t\in\C)
  \end{equation*}
  para la función definiendo los números de Bernoulli (para distinguirla de la serie de
  potencias $f$), esto es igual a
  \begin{align*}
    g(z) &= -\frac{1}{2z}(\tilde f(-bz)+\tilde f(bz))+\frac{1}{2z}(\tilde f(-az)+\tilde
           f(az)) \\
    &= \sum_{n=1}^\infty B_n\frac{z^{n-1}}{n!}(a^n-b^n)
  \end{align*}
  porque $B_n=0$ para $n>1$ impar. De esto obtenemos el resultado.
\end{proof}

Estamos listos para definir la función zeta $p$-ádica.

\begin{defi}\label{defi:lambda-trivial}
  Sean $a,b\in\Z$ coprimos a $p$.
  Definimos \[ \lambda_{\mathbf 1} = \frac{\lambda_{c(a,b)}}{\sigma_b-\sigma_a}. \]
\end{defi}

\begin{thm}\label{thm:lambda-trivial}
  El elemento $\lambda_{\mathbf 1}$ es una pseudo-medida, independiente de $a$ y $b$, y para
  cada $n\in\Nuno$ tenemos
  \[ \int_{\Z_p^\times}x^n\integrald\lambda_{\mathbf 1}(x)=(1-p^{n-1})\zeta(1-n). \]
  Esta propiedad lo caracteriza únicamente.
\end{thm}
\begin{proof}
  Seguimos \cite[Prop.\ 4.2.4]{MR2256969}. El hecho de que la propiedad de interpolación
  determina únicamente al elemento $\lambda_{\mathbf 1}$ se demuestra de manera igual a como
  lo hicimos para el elemento $\mu_{\mathbf 1}$ en la \cref{prop:unicidad-palf}. En
  particular, $\lambda_{\mathbf 1}$ no depende de $a$ y $b$ una vez que sabemos la propiedad
  de interpolación (si es una pseudo-medida). Para ver esta fórmula solo hay que observar
  que
  \begin{equation*}
    \int_{\Z_p^\times}x^n\integrald(\sigma_b-\sigma_a)=b^n-a^n \quad\text{para cada }n\in\Nuno,
  \end{equation*}
  así que  la fórmula sigue del \cref{lem:delta-n}~\ref{lem:delta-n-unicic}. Entonce lo único que falta ver es que $\lambda_{\mathbf 1}$ es una pseudo-medida. Para eso ponemos
  $b=1$ y $a=e$ con $e$ un generador de $\F_p^\times$ tal que $e^{p-1}\not\equiv 1$
  ($\mod p^2$). Entonces $\sigma_e$ es un generador topológico de $\Z_p^\times$, que muestra
  que $\sigma_e-1$ es un generador del ideal de aumentación y entonces
  $(g-1)/(\sigma_e-1)\in\LL(G)$ para cada $g\in\Z_p^\times$, así que $\lambda_{\mathbf 1}$ es
  una pseudo-medida.
\end{proof}

El teorema anterior muestra que el elemento $\lambda_{\mathbf 1}$ tiene una propiedad muy
similar a la del elemento $\mu_{\mathbf 1}$ del \cref{thm:palf-stickelberger-trivial},
aunque no es exactamente la misma: una vez se evalúa en $\kappa^n$ y otra vez en
$\kappa^{1-n}$. Vamos a discutir esta discrepancia en la \cref{sec:conjetura-principal}
(véase el \cref{ejer:nu-mu-lambda}).

\ejercicios

\begin{ejer}\label{ejer:norm-compatible}
  Para $r\in\Nuno$ denotamos $F_r=\Qp(\mu_{p^r})$.
  Para $a,b\in\Z$ coprimos a $p$ definimos $c_r(a,b)\in F_r$.
  Por el \cref{ejer:cicunit} sabemos que $c_r(a,b)\in U_r$.
  \begin{enumerate}
  \item Demuestre que para cada polinomio $f\in\Z[X]$ tenemos\
    \[ \mathrm N_{F_r/F_{r-1}} (f(\xi_r-1))=\prod_{\xi\in\mu_p} f(\xi\xi_r-1) \]
    usando la sucesión exacta de grupos de Galois
    \[
      1\rightarrow\Gal(F_r/F_{r-1})\rightarrow\Gal(F_r/\Qp)\rightarrow\Gal(F_{r-1}/\Qp)\rightarrow1. \]
  \item Aplique esto a $f=X^a-1$ y concluya que los elementos $c_r(a,b)$ son compatibles con
    respecto a la norma para $r\in\Nuno$.
  \end{enumerate}
\end{ejer}

\chapter{La Conjetura Principal de Iwasawa}
\label{sec:mc}

La Conjetura Principal de Iwasawa es una de las maneras más profundas de expresar y precisar
el fenómeno de que valores especiales de funciones $L$ complejas tengan algo que ver con
aritmética. Ella utiliza las funciones $L$ $p$-ádicas y las conecta con objetos de origen
aritmético, como grupos de clases o ciertos grupos de Galois. Más precisamente, estos
objetos son de manera canónica módulos noetherianos sobre el álgebra de Iwasawa, así que
tienen un ideal característico. La Conjetura Principal entonces dice que este ideal es
generado por una función $L$ $p$-ádica. De esta manera establece la conexión de la izquierda en
\[
  \begin{minipage}[c]{.2\linewidth}
    \begin{center}
      objetos de origen aritmético  
    \end{center}
  \end{minipage}
  \longleftrightarrow
  \begin{minipage}[c]{.2\linewidth}
    \begin{center}
      funciones $L$ $p$-ádicas
    \end{center}
  \end{minipage}
  \longleftrightarrow
  \begin{minipage}[c]{.2\linewidth}
    \begin{center}
      funciones $L$ complejas
    \end{center}
  \end{minipage}
\]
mientras la conexión de la derecha es establecida por las fórmulas de interpolación como en
los teoremas \labelcref{thm:palf-stickelberger-trivial},
\labelcref{thm:palf-stickelberger-no-trivial} o \labelcref{thm:lambda-trivial}. De esta manera la Conjetura Principal completa la descripción del vínculo maravilloso entre valores especiales de las funciones
$L$ y la aritmética.

La Conjetura Principal fue formulada por Iwasawa en 1969 en su artículo \cite{MR0255510}. En
este texto Iwasawa busca analogías entre campos de funciones, es decir extensiones finitas
de $\F(t)$, donde $\F$ es un campo finito (equivalentemente, campos de funciones de una
curva algebraica sobre un campo finito), y campos de números. Entre estos dos tipos de
campos hay una multitud de similaridades, por eso estos campos son llamados
\enquote{campos globales}. La Conjetura Principal de Iwasawa es un intento de formular un
análogo para campos de números de los resultados de Weil acerca de curvas sobre campos
finitos (que son generalizadas en las famosas Conjeturas de Weil, véase por ejemplo
\cite{MR926276}). Ambos resultados -- el de Weil y la Conjetura Principal de Iwasawa --
conectan polinomios característicos con algún tipo de función zeta, de ahí la analogía, pero
véase \cite{MR0255510} o \cite[§XI.6]{NSW} para una explicación más precisa.

Advertimos que, aunque la llamamos \enquote{conjetura}, de hecho es un teorema, demostrado
por primera vez por Mazur y Wiles en 1984 y otra vez con métodos diferentes por Rubin
alrededor de 1990. Sin embargo, esta conjetura tiene una multitud de generalizaciones a
otras situaciones (unas de los cuales indicaremos en la
\cref{sec:generalisaciones-conjetura-princial}), y la mayoría de estas todavía carecen de una
demostración. Por eso, y por razones históricas, el teorema todavía se llama
\enquote{Conjetura Principal}.

\section{Las formulaciones de la Conjetura Principal}
\label{sec:conjetura-principal}

Los módulos que aparecen en las formulaciones de la conjetura principal fueron introducidos
en la \cref{defi:iwasawa-modulos}. Recordemos las definiciones, que aquí solo necesitamos
un caso especial.

Para cada $r\in\Ncero$ sea
\begin{itemize}
\item $K_r=\Q(\mu_{p^r})$,
\item $H_r=$ la extensión máxima abeliana pro-$p$ de $K_r$ no ramificada,
\item $M_r=$ la extensión máxima abeliana pro-$p$ de $K_r$ ramificada sólo en $p$,
\item $C_r=$ la $p$-parte del grupo de clases de $K_r$,
\item $X_r=\Gal(H_r/K_r)$,
\item $Y_r=\Gal(M_r/K_r)$.
\end{itemize}
Escribimos $K_\infty$ para el compuesto de todos los campos $K_r$ para $r\in\Ncero$, y
análogamente definimos $H_\infty$ y $M_\infty$.\footnote{Estrictamente,
    aquí la notación es incompatible con la \cref{defi:iwasawa-modulos}. La extensión
    $K_\infty/K_1$ es una $\Zp$-extensión ($K_\infty/K_0$ no lo es) y si aplicaramos la
    \cref{defi:iwasawa-modulos} obtendríamos $K_r=\Q(\mu_{p^{r+1}})$. Esto causaría
    otras molestias en la notación, por eso preferimos definir $K_r=\Q(\mu_{p^r})$ --
    de cualquier modo este detalle no será importante.\label{footnote:r-offset}}
Además escribimos
\begin{itemize}
\item $X_\infty = \Gal(H_\infty/K_\infty) = \displaystyle\varprojlim_{r\in\Ncero}X_r$,
\item $Y_\infty = \Gal(M_\infty/K_\infty) = \displaystyle\varprojlim_{r\in\Ncero}Y_r$.
\end{itemize}
Usamos también la notación de la \cref{sec:stickelberger}, es decir escribimos
$G_{p^r}=\Gal(K_r/\Q)$ para $r\in\Ncero$ y 
$G=\Gal(K_\infty/\Q)\isom\Delta\times\Gamma$, y $\LL(G)$ para el álgebra de Iwasawa de $G$
con coeficientes en $\Zp$.

Tenemos sucesiones exactas de grupos de Galois
\begin{align*}
  1 \rightarrow X_r \rightarrow &\Gal(H_r/\Q) \rightarrow G_{p^r} \rightarrow 1, \\
  1 \rightarrow Y_r \rightarrow &\Gal(M_r/\Q) \rightarrow G_{p^r} \rightarrow 1
\end{align*}
y por la construcción general descrita en el \cref{ejer:sucesion-exacta-accion} tenemos
acciones continuas de $G_{p^r}$ en $X_r$ y $Y_r$. Estas acciones se extienden a
una acción de $G$ en $X_\infty$ y $Y_\infty$, y esto hace $X_\infty$ y $Y_\infty$ módulos
profinitos sobre $\LL(G)$.

Queremos usar algunos resultados del \cref{sec:grupos-de-clases}. Notemos que $K_\infty/K_1$
es una $\Zp$-extensión y que $X_\infty$ y $Y_\infty$ son los módulos de la
\cref{defi:iwasawa-modulos} para esta $\Zp$-extensión (pero tenga en cuenta que el $r$ de
aquí y el $r$ de la \cref{sec:propiedades-x-ll} se diferencian por $1$, véase el
\cref{footnote:r-offset} en la \cpageref{footnote:r-offset}). En la
\cref{sec:propiedades-x-ll} consideramos $X_\infty$ como módulo sobre $\LL(\Gamma)$ con
$\Gamma=\Gal(K_\infty/K_1)$ mientras aquí lo consideramos como módulo sobre $\LL(G)$ con
$G=\Gal(K_\infty/\Q)$. El \cref{ejer:diagonal} dice que bajo la identificación
$\LL(G)\isom\LL(\Gamma)^{p-1}$ del \cref{cor:descomposicion-lambda} el morfismo
$\LL(\Gamma)\rightarrow\LL(G)$ inducido por la inclusión del subgrupo $\Gamma\subseteq G$
corresponde al mapeo diagonal $\LL(\Gamma)\rightarrow\LL(\Gamma)^{p-1}$. Por eso propiedades
como \enquote{noetheriano} o \enquote{de torsión} son verdad para $X_\infty$ como
$\LL(G)$-módulo si lo son como $\LL(\Gamma)$-módulo. Además, estas observaciones implican
también la siguiente reformulación del \cref{cor:iwasawa-control-x}.

\begin{cor}\label{cor:iwasawa-control-x-g}
  Tenemos $X_\infty/w_rX_\infty=C_{r+1}$, donde $w_r\in\LL(G)$ es el elemento que
  corresponde a $(\omega_r,\dotsc,\omega_r)\in\LL(\Gamma)^{p-1}$ bajo la identificación
  $\LL(G)\isom\LL(\Gamma)^{p-1}$ del \cref{cor:descomposicion-lambda}.
\end{cor}

Del \cref{thm:y-noetheriano} y el \cref{cor:x-noetheriano-torsion} entonces sabemos que
$X_\infty$ y $Y_\infty$ son noetherianos y que $X_\infty$ incluso es de torsión como
$\LL(G)$-módulos. De hecho $Y_\infty^+$ también es de torsión (aunque $Y_\infty$ no los es),
esto lo vamos a demostrar más tarde en el \cref{cor:y-mas-torsion}.  Además sabemos que
$X_\infty$ es isomorfo al límite inverso de las $p$-partes de los grupos de clases de $K_r$
con respecto a la norma relativa de ideales (\cref{defi:norma-de-ideales}).

Necesitamos también la función zeta $p$-ádica de Riemann.  Escribimos $\mathcal Q$ como el
anillo de cocientes de $\LL(G)$ e $I\subseteq\LL(G)$ como el ideal de aumentación.  Sea
$\mu_{\mathbf 1}\in\mathcal Q$ el elemento de \cref{defi:mu-chi}, $h_1\in\LL(G)$ como en la
\cref{defi:h-n} y $\lambda_{\mathbf 1}\in\mathcal Q$ el elemento de la
\cref{defi:lambda-trivial}. Entonces $(h_1\mu_{\mathbf 1})$ y $I\lambda_{\mathbf 1}$ son
ideales en $\LL(G)$. De hecho tenemos $(h_1\mu_{\mathbf 1})\subseteq\LL(G)^-$ y
$I\lambda_{\mathbf 1}\subseteq\LL(G)^+$. Esto es una consecuencia del hecho de que la
función zeta de Riemann se anula en los enteros negativos pares (véase el
\cref{ejer:palf-parity} o \cite[Cor.\ 4.2.3]{MR2256969}).

La Conjetura Principal finalmente conecta la función de Riemann $p$-ádica con nuestros
módulos:

\begin{thm}[Conjetura Principal]\label{thm:conjetura-principal}
  Tenemos las igualdades de ideales en $\LL(G)^-$ resp.\ $\LL(G)^+$:
  \begin{enumerate}
  \item\label{thm:conjetura-principal:x}
    $\charideal_{\LL(G)^-}(X^-_\infty)=(h_1\mu_{\mathbf 1})$.
  \item\label{thm:conjetura-principal:y}
    $\charideal_{\LL(G)^+}(Y^+_\infty)=I\lambda_{\mathbf 1}$.
  \end{enumerate}
  (Estas dos afirmaciones son equivalentes.)
\end{thm}

La equivalencia de las dos igualdades la demostraremos en la siguiente sección. Mencionamos
por supuesto que hay muchas más posibilidades equivalentes para formular la Conjetura
Principal y solo damos las dos más usuales. Para más formulaciones equivalentes véase
\cite[Appendix, §8]{MR1029028}.

Recordemos que podemos imaginar el ideal característico como un análogo del orden de algún
grupo. Por eso, heurísticamente la conjetura principal dice que \enquote{ordenes de grupos
  de clases tienen algo que ver con valores especiales de la función zeta}. En la
\cref{sec:consecuencias} vamos a demostrar que esto de hecho es verdad en un sentido preciso
(véase la \cref{prop:consecuencia-numero-de-clases}).

\begin{remark}\label{nota:cp-cociente-de-ideales-car}
    Mencionamos otra manera de interpretar la Conjetura Principal (en la version
    \ref{thm:conjetura-principal:x}). Sea $\Zp(1)=\varprojlim_r\mu_{p^r}$ como en la
    \cref{defi:zpuno}. Entonces $\Zp(1)$ es un $\Zp$-módulo compacto con una acción continua
    de $G$, así que es un $\LL(G)$-módulo. El \cref{ejer:carideal-zpuno} dice que su ideal
    característico es generado por $h_1$. Es decir, la Conjetura Principal dice
    que\footnote{Note que aquí no hace diferencia si consideramos $\Zp(1)$ como
      $\LL(G)$-módulo o $\LL(G)^-$-módulo -- aunque sus ideales característicos son
      diferentes, si multiplicamos con el de $X_\infty^-$ son iguales porque la parte en
      $\LL(G)^+$ entonces es $0$.}
  \[ \charideal_{\LL(G)^-}(X^-_\infty)=\charideal_{\LL^-(G)}(\Zp(1))(\mu_{\mathbf 1}), \]
  que nos podemos imaginar como
  \[ \text{\enquote{$\mu_{\mathbf 1}=\displaystyle
        \frac{\charideal_{\LL(G)^-}(X^-_\infty)}{\charideal_{\LL^-(G)}(\Zp(1))}$}}, \] es
  decir la función zeta $p$-ádica de Riemann es el cociente de los generadores de dos
  ideales característicos (módulo unidades en $\LL(G)^-$).
  En esta forma la Conjetura Principal recuerda la fórmula de números de clases
  \cite[Chap.\ VII, (5.11)]{MR1697859} (que no es una coincidencia).
\end{remark}

En la Conjetura Principal no aparecen los módulos $X_\infty$ o $Y_\infty$ sino sólo
sus partes donde la conjugación compleja actúa por $-1$ o $1$, y las afirmaciones son
igualdades de ideales en $\LL(G)^-$ o $\LL(G)^+$, respectivamente. Sabemos que
$(h_1\mu_{\mathbf 1})\subseteq\LL(G)^-$, es decir $(h_1\mu_{\mathbf 1})\cap\LL(G)^+=0$. Por
eso si una afirmación análoga a la Conjetura Principal sería verdad para todo $\LL(G)$, esto
significaría que $X_\infty^+$ debería ser finito. De hecho existe una conjetura que afirma
incluso más:

\begin{conj}[Kummer/Vandiver]\label{conj:vandiver}
  Tenemos $X_\infty^+=0$. Equivalentemente y más concreto, $p$ no divide al número de
  clases del subcampo $\Q(\mu_p)^+$ de $\Q(\mu_p)$ fijado por la conjugación
  compleja.
\end{conj}

La conjetura de Kummer y Vandiver es controvertida entre los expertos: Hay indicaciones y
heurísticas en su favor y también en su contra, así que no es claro si se debería creer en
su veracidad. Hasta hoy parece poco claro como se podría abordar una demostración, aunque
numéricamente la conjetura ha sido confirmada para todos los primos menores que
$1.63\cdot 10^8$.

Si esta conjetura fuera cierta entonces la Conjetura Principal sería una
igualdad de ideales $\charideal_{\LL(G)}(X_\infty)=(h_1\mu_{\mathbf 1})$ en
$\LL(G)$. Pero de hecho la conjetura de Kummer y Vandiver es más fuerte que la Conjetura
Principal porque incluso la implica. Esto sigue de un teorema de Iwasawa que
también es un paso importante en una de las demostraciones de la Conjetura Principal. Véase
\cite[Thm.\ 4.4.1]{MR2256969} para el teorema de Iwasawa y \cite[§4.5, Cor.\
4.5.4]{MR2256969} para una demostración de que la conjetura de Kummer y Vandiver implica la
Conjetura Principal.

\ejercicios

\begin{ejer}\label{ejer:palf-parity}
  Demuestre que $(h_1\mu_{\mathbf 1})\subseteq\LL(G)^-$ y $I\lambda_{\mathbf
    1}\subseteq\LL(G)^+$, usando las fórmulas de interpolación de los
  teoremas \labelcref{thm:palf-stickelberger-trivial} y 
  \labelcref{thm:lambda-trivial} y el hecho de que $B_n=0$ para $n>1$ impar.
\end{ejer}

\begin{ejer}
  Demuestre que la conjetura principal como la formulamos en esta sección es equivalente a
  la version de la introducción (\cref{thm:mc-intro}). Use el \cref{lem:pairwiseann} para
  esto.
\end{ejer}

\begin{ejer}
  Demuestre que la versión \enquote{$X_\infty^+=0$} de la conjetura de Kummer y Vandiver
  (\cref{conj:vandiver}) es equivalente a la versión \enquote{$p$ no divide al número de
    clases de $\Q(\mu_p)^+$}. Use el \cref{cor:preg} para esto.
\end{ejer}

\section{La equivalencia de las formulaciones}
\label{sec:proemio-kummer-adjuntos}

Las dos afirmaciones en la Conjetura Principal (\cref{thm:conjetura-principal}) de hecho son
equivalentes. Decidimos poner las dos porque ambas tienen su importancia: La versión
\ref{thm:conjetura-principal:x} contiene el módulo $X_\infty^-$ y por eso directamente
implica resultados sobre los grupos de clases, que son objetos de gran interés (como
explicaremos en la \cref{sec:consecuencias}), mientras la
versión \ref{thm:conjetura-principal:y} es la que se integra naturalmente en una fila de
generalizaciones (como explicaremos en la \cref{sec:generalisaciones-conjetura-princial}) y por
eso aparece en esta forma muchas veces en la literatura, por ejemplo en
\cite[(2.5)]{MR2276851}.

En esta sección vamos a demostrar la equivalencia de las dos versiones de la Conjetura
Principal. Los resultados de esta sección no serán usados en las secciones subsecuentes, así
que el lector podría considerar saltar esta sección. No obstante, ganaremos más
conocimiento de los módulos que aparecen en la Conjetura Principal.

Nuestra demostración de la equivalencia es esencialmente elemental, usando sólo la teoría de
Kummer y los adjuntos de Iwasawa que explicamos en la \cref{sec:adjuntos}. Sin embargo, desde
un punto de vista más abstracto esta equivalencia puede ser demostrada con resultados de
dualidad en la cohomología de Galois, usando que los adjuntos de Iwasawa son grupos
$\operatorname{Ext}$, como mencionamos en esta sección. Una demostración con estos métodos
se encuentra en \cite[Thm.\ 11.1.8]{NSW}, que afirma lo mismo que la
\cref{prop:x-c-pseudo-iso} y el \cref{prop:equivalencia-x-y} abajo.

Para ver la equivalencia de \ref{thm:conjetura-principal:x} y
\ref{thm:conjetura-principal:y} en el \cref{thm:conjetura-principal} necesitamos una
involución en el álgebra de Iwasawa.

\begin{defi}\label{defi:chanfle-nu}
  Sea $\nu\colon\LL(G)\rightarrow\LL(G)$ el morfismo inducido por
  \[ G\rightarrow\LL(G)^\times, \quad g\mapsto \kappa(g)g^{-1}. \] Si $M$ es un
  $\LL(G)$-módulo, sea $\LL(G)^\nu$ el $\LL(G)$-módulo $M$ con la acción de $\LL(G)$
  chanfleada por $\nu$, es decir \[ M^\nu=\LL(G)\tensor_{\LL(G),\nu} M. \]
\end{defi}

\begin{ex}\label{ex:ll-f-nu}
  Si $f\in\LL(G)$ entonces canónicamente
  \[ {\left(\LL(G)/(f)\right)}^\nu\isom\left(\LL(G)/(\nu(f))\right) \] como $\LL(G)$-módulos.
\end{ex}

El morfismo $\nu$ es una involución, así que $(M^\nu)^\nu=M$ para cada $\LL(G)$-módulo $M$.
De la definición de $\nu$ y las propriedades de los elementos $\mu_{\mathbf 1}$ y
$\lambda_{\mathbf 1}$ de los teoremas \labelcref{thm:palf-stickelberger-trivial},
\labelcref{thm:lambda-trivial} y \labelcref{prop:unicidad-palf} se deduce fácilmente que
$\nu(I\lambda_{\mathbf 1})=(h_1\mu_{\mathbf 1})$ (véase el \cref{ejer:nu-mu-lambda}). Esto
ya explica una parte de la equivalencia anunciada. Para la otra parte, en el
\cref{prop:equivalencia-x-y} abajo vamos a demostrar que existe un pseudo-isomorfismo
${(Y_\infty^+)}^\nu\sim X^-_\infty$.

Para preparar la demostración del \cref{prop:equivalencia-x-y} estudiamos algunos objetos
en más detalle con métodos de la teoría de Kummer, siguiendo \cite[Chap.\ 6, §2]{MR1029028}.
Primero damos una descripción más explicita del campo $M_\infty$. Usamos la
siguiente notación: para $r\in\Ncero$ sea $\mathfrak p_r$ el único ideal sobre $p$ de $K_r$ (que también
es el único ideal en que la extensión $K_r/\Q$ es ramificada, véase el \cref{lem:totram} y
la \cref{prop:ramcycfin}). Este ideal es
principal y escribimos $\pi_r$ como un generador (por ejemplo $\pi_r=\xi_r-1$ con
$\xi_r\in K_r$ una raíz primitiva $p^r$-ésima de la unidad). 

\begin{prop}\label{prop:m-infty-explicito}
  Para $r\in\Nuno$ sea
  \[ D_r=\{\alpha\in K_r^\times : (\alpha)=\mathfrak a^{p^r} \text{ para un ideal fraccional
    }\mathfrak a\text{ primo a }\mathfrak p_r \}. \]
  Entonces \[ M_\infty=\bigcup_{r\in\Nuno} K_\infty(\sqrt[p^r]{D_r}). \]
\end{prop}
\begin{proof}
  Escribimos $L:=\bigcup_r K_\infty(\sqrt[p^r]{D_r})$. Claramente, $M_\infty$ es una
  composición de extensiones cíclicas de Kummer de $K_\infty$ cuyos exponentes son potencias
  de $p$. Según el \cref{thm:kummer}~\ref{thm:kummer:kummer} estas extensiones son de la
  forma $K_\infty(\sqrt[p^m]\alpha)$ con $\alpha\in K_\infty^\times$ y $m\in\Nuno$. Es
  suficiente demostrar que $K_\infty(\sqrt[p^m]\alpha)\subseteq L$ para tales
  $\alpha$ y $m$.

  Fijemos $\alpha$ y $m$ como arriba. Entonces $\alpha\in K_r$ para algún $r\in\Nuno$, y sin
  pérdida de generalidad supongamos que $r\ge m>1$ y $K_r(\sqrt[p^m]\alpha)\subseteq M_r$. De la
  \cref{prop:kummer-ramificacion} sabemos que el ideal principal generado
  por $\alpha$ en el anillo de enteros de $K_r$ debe ser de la forma
  \begin{equation*}
    (\alpha)=\mathfrak a^{p^r}\mathfrak p_r^t
  \end{equation*}
  con $t\in\Z$ y $\mathfrak p_r\nmid\mathfrak a$ (primero con $\mathfrak a^{p^m}$ en lugar
  de $\mathfrak a^{p^r}$, pero lo podemos cambiar por $\mathfrak a^{p^r}$ porque $r\ge m$).
  Por eso, si definimos $\beta:=\alpha\pi_r^{-t}$ entonces $\beta\in D_r$ y pues
  $K_\infty(\sqrt[p^m]\beta)\subseteq K_\infty(\sqrt[p^r]\beta)\subseteq L$ por la definición
  de $L$. Esto implica que $K_\infty(\sqrt[p^m]\alpha)\subseteq L(\sqrt[p^r]{\pi_r})$.

  Para terminar la demostración probemos que de hecho $L(\sqrt[p^r]{\pi_r})=L$. Si $s>r$
  entonces la extension $K_s/K_r$ es de grado $p^{s-r}$ y puramente ramificada, es decir
  $\mathfrak p_r=\mathfrak p_s^{s-r}$ en $\O_{K_s}$, pues $\pi_r=\pi_s^{s-r}u$ con
  $u\in\O_{K_s}^\times$. Por eso, si queremos adjuntar una raíz $p^{s-r}$-ésima de $\pi_r$,
  solo necesitamos adjuntar una raíz de $u$, porque $\pi_s$ ya esta en $K_\infty\subseteq
  L$. Pero claramente $\O_{K_s}^\times\subseteq D_s$, por eso estas raíces ya están en $L$.
\end{proof}

Introducimos otro campo
\[ E_\infty=\bigcup_{r\in\Nuno}E_r,\quad
  \text{con } E_r=K_\infty\left(\sqrt[p^r]{\bigcup_{s\in\Nuno}\O_{K_s}^\times}\right) \text{
    para }r\in\Nuno. \]
Entonces tenemos extensiones y grupos de Galois
\[
  \begin{tikzcd}
    M_\infty \arrow[d, dash] \arrow[dd, dash, bend right=50, "Y_\infty", swap] \\
    E_\infty \arrow[d, dash] \\
    K_\infty \arrow[d, dash, "G", swap] \\
    \Q
  \end{tikzcd}
\]

Vamos a estudiar también el grupo $\Gal(M_\infty/E_\infty)$ porque tiene que ver con el
grupo de clases y nos permitirá relacionar los módulos $X_\infty$ y $Y_\infty$. Notemos que
este grupo es un $\LL(G)$-submódulo compacto de $Y_\infty$.

\begin{defi}
  Definimos \[ C^\infty=\varinjlim_{r\in\Ncero}C_r, \] el colímite tomado con respecto a los mapeos
  naturales entre los grupos de clases, inducidos por las inclusiones de campos
  $K_r\hookrightarrow K_{r+1}$. Esto es un $\LL(G)$-módulo discreto.
\end{defi}

\begin{prop}\label{prop:apareamiento-kummer-m}
  El apareamiento de Kummer induce un apareamiento perfecto
  \[ \Gal(M_\infty/E_\infty)\times C^\infty\rightarrow\mu_{p^{\infty}}\subseteq
    K_\infty^\times \]
  de $\Z_p$-módulos topológicos, equivariante en el sentido
  \[ \left<g\sigma g^{-1},ga\right>=g\left<\sigma,a\right>\quad \text{para cada
    }g\in G,\ \sigma\in\Gal(M_\infty/E_\infty),\ a\in C^\infty. \]
  Se restringe a un apareamiento perfecto
  \begin{equation*}
    {\Gal(M_\infty/E_\infty)}^+\times {(C^\infty)}^-\rightarrow\mu_{p^{\infty}}.
  \end{equation*}
\end{prop}
\begin{proof}
  Para $r\in\Nuno$ definimos $D_r$ como en la \cref{prop:m-infty-explicito}. Usamos el morfismo
  de grupos
  \begin{equation*}
    \sqrt[p^r]{D_r}\rightarrow C_r, \quad \sqrt[p^r]\alpha\mapsto \mathfrak a \text{ para
    }(\alpha)=\mathfrak a^{p^n}
  \end{equation*}
  cuyo núcleo es $\sqrt[p^r]{\O_{K_r}^\times}$, que es claramente equivariante para la
  acción de $G_{p^r}$. Escribimos $B_r$ para su imagen, así que
  tenemos un isomorfismo
  \begin{equation*}
    \sqrt[p^r]{D_r}/\sqrt[p^r]{\O_{K_r}^\times}\isom B_r.
  \end{equation*}
  Se verifica fácilmente que
  $\sqrt[p^r]{\O_{K_r}^\times}=\sqrt[p^r]{D_r}\cap E_\infty^\times$, por eso si combinamos
  el apareamiento de Kummer del \cref{thm:kummer}~\ref{thm:kummer:apareamiento} con el
  isomorfismo anterior obtenemos un apareamiento perfecto de grupos finitos
  \begin{equation*}
    \Gal(E_\infty(\sqrt[p^r]{D_r})/E_\infty)\times B_r\rightarrow\mu_{p^r},
  \end{equation*}
  que es equivariante en el sentido que queremos (para $G_{p^r}$) gracias al
  \cref{prop:kummer-apareamiento-equivariante}.
  Para $s\ge r$ el diagrama 
  \begin{equation*}
    \begin{tikzcd}[column sep=.1em]
      \Gal(E_\infty(\sqrt[p^r]{D_r})/E_\infty) & \times & B_r \arrow[d] \arrow[rrrrrrrrrrrrrrrrrrrrrr] &&&&&&&&&&&&&&&&&&&&&&
      \mu_{p^r} \arrow[d] \\
      \Gal(E_\infty(\sqrt[p^s]{D_s})/E_\infty) \arrow[u] & \times & B_s
      \arrow[rrrrrrrrrrrrrrrrrrrrrr] &&&&&&&&&&&&&&&&&&&&&& \mu_{p^s}
    \end{tikzcd}
  \end{equation*}
  es conmutativo, así que podemos tomar el límite y colímite, respectivamente. Esto junto
  la \cref{prop:m-infty-explicito} nos da un apareamiento perfecto de grupos topológicos
  \begin{equation*}
    \Gal(M_\infty/E_\infty)\times B^\infty\rightarrow\mu_{p^\infty}
  \end{equation*}
  con $B^\infty=\varinjlim_rB_r\subseteq C^\infty$ que es equivariante, y falta que ver que
  $B^\infty=C^\infty$.

  Por la definición de $D_r$, los elementos de $B_r$ son estas clases de ideales
  fraccionales de $K_r$ que tienen un representante $\mathfrak a$ primo a $\mathfrak p_r$
  tal que $\mathfrak a^{p^r}$ es principal. Como $\mathfrak p_r$ ya es principal, podemos
  omitir \enquote{primo a $\mathfrak p_r$} aquí. El grupo $C_r$ es la $p$-parte del grupo de
  clases, es decir sus elementos son las clases de ideales fraccionales de $K_r$ que
  tienen un representante $\mathfrak a$ tal que $\mathfrak a^{p^t}$ es principal para algún
  $t\in\Ncero$. De estas descripciones vemos que los colímites de los $C_r$ y $B_r$ coinciden.

  Falta demostrar que el apareamiento se restringe como anunciado. El
  grupo $\Gal(M_\infty/\Q)$ actúa por conjugación en su subgrupo normal
  $\Gal(M_\infty/E_\infty)$, y esto nos da una acción de la conjugación compleja $\mathbf c$
  en $\Gal(M_\infty/E_\infty)$. Según la \cref{prop:kummer-apareamiento-equivariante}
  el apareamiento tiene la propiedad
  \begin{equation}
    \label{eqn:conj-invariant}
    \left<\mathbf c\sigma,\mathbf c a\right>=\mathbf c\left<\sigma,a\right>
    \quad \text{para cada }\sigma\in\Gal(M_\infty/E_\infty),\ a\in C^\infty.
  \end{equation}
  La afirmación entonces resulta del mismo argumento que usamos en la demostración del
  \cref{thm:kummer-hp-plus-minus}.
\end{proof}

\begin{cor}\label{cor:apareamiento-kummer-m}
  Tenemos isomorfismos canónicos de $\LL(G)$-módulos compactos
  \begin{align*}
    \Gal(M_\infty/E_\infty) &\isomarrow \Hom_{\Zp}(C^\infty,\mu_{p^\infty}),\\
    \Gal(M_\infty/E_\infty)^+ &\isomarrow \Hom_{\Zp}((C^\infty)^-,\mu_{p^\infty})
  \end{align*}
  donde la acción de $\LL(G)$ al lado derecho es la de la
  \cref{rem:hom-de-representaciones}, es decir inducida por
  \[ (gf)(c)=g(f(g^{-1}c)) \quad\text{para }g\in G,\ f\in\Hom_{\Zp}(C^\infty,\mu_{p^\infty}),
   \ c\in C^\infty. \]
\end{cor}
\begin{proof}
  Esto resulta de la \cref{prop:apareamiento-kummer-m} y la
  \cref{rem:kummer-apareamiento-equivariante}.
\end{proof}

Como $\Zp$-módulo, $\mu_{p^\infty}$ es isomorfo a $\Qp/\Zp$ (aunque no canónicamente). Es
decir, en el \cref{cor:apareamiento-kummer-m} podríamos remplazar $\mu_{p^\infty}$ por
$\Qp/\Zp$, que es útil porque los morfismos a $\Qp/\Zp$ son el dual de Pontryagin de la
\cref{defi:pontryagin}. Sin embargo, si hacemos eso tenemos que tener cuidado con la acción
de $G$ porque en el dual de Pontryagin introdujimos una acción un poco diferente en la
\cref{defi:mvee-alpha}~\ref{defi:ll-mod-dual}. En el lema siguiente nos cuidamos de esta
diferencia.

\begin{lem}\label{lem:isom-dual-nu}
  Sea $M$ un $\LL(G)$-módulo tal que $\Hom_{\Zp}(M,\mu_{p^\infty})$ nuevamente es un
  $\LL(G)$-módulo como en la \cref{rem:hom-de-representaciones}. De la
  \cref{defi:chanfle-nu} tenemos entonces el $\LL(G)$-módulo
  $\Hom_{\Zp}(M,\mu_{p^\infty})^\nu$, que como conjunto es lo mismo que
  $\Hom_{\Zp}(M,\mu_{p^\infty})$. Por otro lado, en $\Hom_{\Zp}(M,\Qp/\Zp)$ tenemos la
  estructura como $\LL(G)$-módulo introducida en la
  \cref{defi:mvee-alpha}~\ref{defi:ll-mod-dual}.

  Fijamos un isomorfismo de $\Zp$-módulos $\psi\colon\mu_{p^\infty}\isomarrow\Qp/\Zp$.
  Entonces la asociación
  \[ \Hom_{\Zp}(M,\mu_{p^\infty})^\nu \rightarrow \Hom_{\Zp}(M,\Qp/\Zp), \quad
    f\mapsto\psi\circ f \]
  es un isomorfismo de $\LL(G)$-módulos.
\end{lem}
\begin{proof}
  Es claro que la asociación es biyectiva, $\Zp$-lineal y continua. Hay que verificar que es
  compatible con las acciones de $G$ en ambos lados. Esto es un cálculo fácil, aunque un
  poco espinoso, que dejamos como ejercicio.
\end{proof}

Aplicando este lema a nuestra situación obtenemos lo siguiente.

\begin{cor}\label{cor:isom-dual-nu}
    Existen isomorfismos (no canónicos) de $\LL(G)$-módulos compactos
    \begin{align*}
      \Gal(M_\infty/E_\infty)^\nu &\isom (C^\infty)^\vee, \\
      (\Gal(M_\infty/E_\infty)^+)^\nu &\isom ((C^\infty)^-)^\vee.
    \end{align*}
\end{cor}

Ahora usamos la teoría de los adjuntos de Iwasawa explicada en la \cref{sec:adjuntos} (que
involucra duales de Pontryagin). Aplicándola a la situación de nuestro interés obtenemos el
resultado siguiente.

\begin{prop}\label{prop:x-c-pseudo-iso}
  Existe un isomorfismo canónico de $\LL(G)$-módulos
  \[ \alpha(X_\infty)\isom{(C^\infty)}^\vee. \]
  En particular, existe un pseudo-isomorfismo
  \[ X_\infty\sim {(C^\infty)}^\vee. \]
\end{prop}
\begin{proof}
  Por \cref{defi:alpha-g},
  $\alpha(X_\infty)={(\varinjlim_{r\ge m}X_\infty/\frac{w_r}{w_m}X_\infty)}^\vee$ con
  $m\in\Ncero$ suficientemente grande y $w_r\in\LL(G)$ siendo el elemento que corresponde a
  $(\omega_r,\dotsc,\omega_r)\in\LL(\Gamma)^{p-1}$ bajo la identificación
  $\LL(G)\isom\LL(\Gamma)^{p-1}$ del \cref{cor:descomposicion-lambda}.  Aquí, \enquote{$m$
    suficientemente grande} significa que para $r\ge m$ el polinomio $\omega_r$ es coprimo
  al polinomio característico de $e_iX_\infty$ para $i=1,\dotsc,p-1$, donde $e_i$ es el
  idempotente para el carácter $\omega^i$ (\cref{lem:idempotentes}). El
  \cref{lem:m-cociente-finito-coprimo} dice que esto es equivalente a
  $e_iX_\infty/\omega_re_iX_\infty$ siendo finito para cada $i$, es decir a
  $X_\infty/w_rX_\infty$ siendo finito. Pero el \cref{cor:iwasawa-control-x-g} dice que
  $X_\infty/w_rX_\infty=C_{r+1}$, el cual es finito para cada $r\in\Ncero$. Por eso podemos usar
  $m=0$ y obtenemos $\alpha(X_\infty)=(\varinjlim_rC_r)^\vee$, que es la primera
  afirmación. La segunda resulta de la primera y el \cref{cor:iwasawa-adj}, que dice que un
  módulo es pseudo-isomorfismo a su adjunto.
\end{proof}

Ahora estamos listos para demostrar la equivalencia de las formulaciones.

\begin{thm}\label{prop:equivalencia-x-y}
  Existe un pseudo-isomorfismo ${(Y_\infty^+)}^\nu\sim X^-_\infty$. En particular, las dos
  afirmaciones del \cref{thm:conjetura-principal} son equivalentes.
\end{thm}
\begin{proof}
  El pseudo-isomorfismo de la \cref{prop:x-c-pseudo-iso} es $\LL(G)$-lineal, pues es compatible
  con la acción de la conjugación compleja. Además es fácil ver que
  $((C^\infty)^\vee)^-=((C^\infty)^-)^\vee$. Combinando esto con la
  \cref{prop:x-c-pseudo-iso} obtenemos un pseudo-isomorfismo
  $X_\infty^-\sim(\Gal(M_\infty/E_\infty)^+)^\nu$.
  
  Para lo que falta usamos la sucesión exacta
  \begin{equation*}
    1 \rightarrow \Gal(M_\infty/E_\infty)\rightarrow Y_\infty \rightarrow
    \Gal(E_\infty/K_\infty)\rightarrow 1.
  \end{equation*}
  Vamos a demostrar que ${\Gal(E_\infty/K_\infty)}^+=0$, que por la sucesión exacta implica
  $\Gal(M_\infty/E_\infty)^+=Y_\infty^+$, con lo que la demostración estará completa.

  Para esto usamos otra vez la teoría de Kummer. De manera similar a anteriormente, tomando
  el (co)límite de los apareamientos de Kummer resulta un apareamiento perfecto de grupos
  topológicos
  \begin{equation*}
    {\Gal(E_\infty/K_\infty)}\times\O_{K_\infty}^\times\rightarrow\mu_{p^{\infty}}
  \end{equation*}
  que también tiene la propiedad análoga a \eqref{eqn:conj-invariant} (omitimos los detalles
  aquí). Con el mismo argumento que arriba, usando que $\Gal(E_\infty/K_\infty)$ es un grupo
  pro-$p$, vemos que la restricción a
  ${\Gal(E_\infty/K_\infty)}^+\times{(\O^\times_{K_\infty})}^+$ es trivial. La restricción a
  ${\Gal(E_\infty/K_\infty)}\times{(\O^\times_{K_\infty})}^-$ también es trivial porque
  ${(\O^\times_{K_\infty})}^-=\mu_{p^\infty}\subseteq K_\infty$, en que
  $\Gal(E_\infty/K_\infty)$ actúa trivialmente. Es decir, la restricción del apareamiento a
  \begin{equation*}
    {\Gal(E_\infty/K_\infty)}^+\times\big({(\O^\times_{K_\infty})}^+\oplus{(\O^\times_{K_\infty})}^-\big)
  \end{equation*}
  es trivial. Pero
  ${(\O^\times_{K_\infty})}^+\oplus{(\O^\times_{K_\infty})}^-\subseteq\O^\times_{K_\infty}$
  es un subgrupo de índice $2$, y porque $\Gal(E_\infty/K_\infty)$ es pro-$p$ eso implica
  que el apareamiento es trivial en todo
  ${\Gal(E_\infty/K_\infty)}^+\times\O_{K_\infty}^\times$. Porque es perfecto en
  ${\Gal(E_\infty/K_\infty)}\times\O_{K_\infty}^\times$ obtenemos que
  ${\Gal(E_\infty/K_\infty)}^+=0$, lo que termina la demostración.
\end{proof}

\begin{cor}\label{cor:y-mas-torsion}
  El $\LL(G)$-módulo $Y_\infty^+$ es de torsión.
\end{cor}
\begin{proof}
  Esto resulta del \cref{prop:equivalencia-x-y} porque ya sabemos que $X_\infty^-$ es de
  torsión, y del hecho de que la propiedad de ser de torsión es invariante bajo
  pseudo-isomorfismos y al aplicar $\nu$.
\end{proof}

\ejercicios

\begin{ejer}
  Demuestre que $\nu\colon\LL(G)\rightarrow\LL(G)$ es una involución, es decir $\nu\circ\nu=\id$.
\end{ejer}

\begin{ejer}\label{ejer:nu-mu-lambda}
  Demuestre que $\nu(\mu_{\mathbf 1})=\lambda_{\mathbf 1}$ y que $\nu(h_1)$ es un generador
  de $I$. ¿Qué significa esto para las funciones $L$ $p$-ádicas de un carácter de Dirichlet
  $\chi$ no trivial? Formule una versión de la existencia y la fórmula de interpolación para
  esta función $L$ $p$-ádica en el estilo del \cref{thm:lambda-trivial}.
\end{ejer}

\begin{ejer}
  Demuestre que para cada $\LL(G)$-módulo tenemos ${(M^\nu)}^\pm={(M^\mp)}^\nu$.
\end{ejer}

\begin{ejer}
  Verifique la afirmación del \cref{ex:ll-f-nu}.
\end{ejer}

\begin{ejer}
  Verifique que la asociación en el \cref{lem:isom-dual-nu} es compatible con la acción de $G$.
\end{ejer}

\section{Consecuencias de la Conjetura Principal}
\label{sec:consecuencias}

Aunque la Conjetura Principal conecta la función zeta de Riemann con los módulos de origen
aritmético, no es una afirmación fácil de concebir. Para exponer de una manera más concreta
la importancia de la Conjetura Principal aquí deducimos algunas implicaciones de ella, las
cuales ojalá ilustren su poder. No obstante, hay que reconocer que algunas de estas
implicaciones ya eran conocidas antes de la demostración de la Conjetura Principal. Los
resultados mas concretos que obtendremos conciernen el grupo de clases $C_1$ del campo
$K_1=\Q(\mu_p)$, el cual denotaremos en esta sección simplemente como $C$.

Primero queremos deducir el criterio de Kummer (\cref{thm:kummer-crit}) de la Conjetura
Principal. En este criterio aparece el enunciado de que algunos valores de la función zeta
sean divisibles por $p$. Si el ideal característico de un módulo es generado por una función
$L$ $p$-ádica, como dice la Conjetura Principal, y valores de esta función son unidades
$p$-ádicas, entonces ¿Qué significa esto para el módulo?

Empezamos con un lema que contesta esta pregunta en general, y luego lo aplicamos para
obtener el criterio de Kummer. En el lema sea $\LL(G)$ el álgebra de Iwasawa de $G$ con
coeficientes en $\O$, el anillo de enteros de una extensión finita de $\Qp$.

\begin{lem}\label{lem:wert-einheit-bedeutung}
  Sea $\mu\in\LL(G)$ y $X$ un $\LL(G)$-módulo noetheriano de torsión tal que
  $\charideal_{\LL(G)}(X)=(\mu)$. Sea $\psi\colon G\rightarrow\O^\times$ un carácter de la forma
  $\psi=\omega^i\kappa_0^s$ con $i\in\{1,\dotsc,p-1\}$ y $s\in\Zp$. Sea $e_i$ el idempotente
  asociado a $\omega^i$.
  Entonces  \[ \mu(\psi)\in\O^\times \iff e_iX\text{ es finito}. \]
  Si $X$ no tiene submódulos finitos no triviales entonces
  \[ \mu(\psi)\in\O^\times \iff e_iX=0. \] También tenemos afirmaciones análogas si $X$ es
  un módulo sobre $\LL(G)^\pm$ e $i$ es par o impar, respectivamente.
\end{lem}
\begin{proof}
  Según la \cref{prop:unidad-solo-depende-de-i} y \eqref{eqn:mu-ramas} tenemos la primera
  equivalencia en
  \begin{align*}
    \mu(\psi)=\mu_i(\kappa_0^s)\in\O^\times &\iff\mu_i\in\LL(\Gamma)^\times \\
                                            & \iff \charideal_{\LL(\Gamma)}(e_iX)=\LL(\Gamma)\\
                                            & \iff e_iX\sim 0\\
                                            & \iff e_iX\text{ es finito}
  \end{align*}
  y obtenemos la primera afirmación. El resto de resulta directamente de esto.
\end{proof}

Aplicando este lema a la situación de la Conjetura Principal nos da el criterio de Kummer.
Aquí y también más tarde, utilizaremos el hecho importante de que el módulo $X_\infty^-$ no
tiene submódulos finitos no triviales. Esto es una consecuencia de un resultado de Iwasawa
\cite{MR0124320} cuya demostración usa la teoría de cohomología
de Galois, que no queremos usar en este texto. Por eso solo citamos reste resultado aquí,
una demostración se encuentra en \cite[Prop.\ 4.4.2]{SharifiIT}.

\begin{prop}\label{prop:cp-implica-kummer}
  La Conjetura Principal implica el criterio de Kummer, que dice
  \begin{align*}
    p \mid h_{\Q(\mu_p)} &\iff p \mid \zeta(1-n) \text{ para algún }n\in\Nuno\text{ par} \\
                                          &\iff p \text{ divide uno de } \zeta(-1),\zeta(-3),\dotsc,\zeta(4-p). \\
  \end{align*}
\end{prop}
\begin{proof}
  Primero explicamos la equivalencia de la segunda y tercera afirmación, que es cierta
  independientemente de la Conjetura Principal. Si $p-1\mid n$ entonces $\zeta(1-n)$ nunca
  es divisible por $p$ según la \cref{clausen-von-staudt} (recuerde que decimos que $p$
  divide a un número racional si divide al numerador en una representación simplificada). Si
  $p-1\nmid n$ entonces por el \cref{cor:unidad-depende-solo-de-i}, la pregunta si
  $p\mid\zeta(1-n)$ para $n\in\Nuno$ solo depende del resto de $1-n$ módulo $p-1$. Los números
  $1-n$ con $n\in\Nuno$ par y no divisible por $p-1$ tienen $-1,\dotsc,4-p$ como representantes
  módulo $p-1$. Esto muestra la equivalencia deseada.

  Del \cref{cor:preg} resulta que $p\mid h_{\Q(\mu_p)}$ si y solo si $X_\infty\neq0$, que
  según el \cref{cor:x-cero-iff-x-menos-cero} es equivalente a $X_\infty^-\neq0$.
  Esto es equivalente a $e_iX_\infty=0$ para $i\in{1,3,\dotsc,p-2}$ impar.

  A partir de ahora asumamos la Conjetura Principal. Aplicamos el
  \cref{lem:wert-einheit-bedeutung} con $\O=\Z_p$, $X=X_\infty^-$ y
  $\mu=h_1^{(1)}\mu_{\mathbf 1}$.  Primero observamos que si ponemos $\psi=\kappa$ entonces,
  como $h_1^{(1)}\mu_{\mathbf 1}(\kappa)\in\Z_p^\times$ según el
  \cref{cor:residuo-zeta-p-adica}, el lema nos dice que $e_1X_\infty^-=0$. Por eso
  \begin{equation*}
    X_\infty^-=\bigoplus_{i=3\text{ impar}}^{p-1}e_iX_\infty^-
    =\bigoplus_{n=2\text{ par}}^{p-3}e_{1-n}X_\infty^-.
  \end{equation*}
  Ponemos en el lema $\psi=\kappa^{1-n}$ con $n\in\{2,4,\dotsc,p-3\}$. Entonces
  $n\not\equiv0$ $(\mod p-1)$, así que $e_{1-n}h_1^{(1)}=1$ y pues
  \begin{equation*}
    h_1^{(1)}\mu_{\mathbf 1}(\kappa^{1-n})=\mu_{\mathbf 1}(\kappa^{1-n})=(1-p^{n-1})\zeta(1-n)
  \end{equation*}
  (esto resulta del diagrama \eqref{eqn:diagrama-evaluacion-g-gamma}) según la fórmula de
  interpolación del \cref{thm:palf-stickelberger-trivial}. El lema dice entonces que
  \begin{equation*}
    e_{1-n}X_\infty^-=0 \iff \zeta(1-n)\in\Z_p^\times.
  \end{equation*}
  Como $\{1-n : n=2,4,\dotsc,p-3\}=\{-1,-3,\dotsc,4-p\}$, esto implica que
  \begin{equation*}
    p \text{ divide uno de } \zeta(-1),\zeta(-3),\dotsc,\zeta(4-p) \iff X_\infty^-\neq0
  \end{equation*}
  y la afirmación resulta.
\end{proof}

Por supuesto, la demostración original de Kummer no utiliza la Conjetura Principal. Una
demostración elemental de este criterio se encuentra en \cite[§19.2]{MR1821363}.

Un resultado más fino es el teorema de Herbrand y Ribet. Para esto definimos
$V:=\Fp\tensor_\Z C=C/{(C)}^p$, que es un espacio vectorial de dimension finita sobre
$\Fp$ con una acción $\Fp$-lineal de $\Delta=\Gal(\Q(\mu_p)/\Q)$. Por eso se descompone en
una suma \[ V=\bigoplus_{i=1}^{p-1} V_i \] donde $V_i=e_{\omega^i}V$ es el subespacio en que
$\Delta$ actúa por la potencia $i$-ésima del carácter de Teichmüller. El teorema de Herbrand
y Ribet describe estos sumandos.

\begin{thm}[Herbrand/Ribet]\label{thm:herbrand-ribet}
  Sea $n\in\{2,\dotsc,p-3\}$ par. Entonces
  \[ V_{1-n}\neq0 \iff p\mid\zeta(1-n). \]
\end{thm}

La implicación \enquote{$\Longrightarrow$} aquí fue demostrada por Herbrand en 1932,
mientras la implicación \enquote{$\Longleftarrow$}, que es mucho más difícil, fue demostrada
por Ribet en 1976 \cite{MR0419403}, usando métodos similares a los que más tarde condujeron
a la demostración de la Conjetura Principal por Mazur y Wiles.

La Conjetura Principal implica este teorema e incluso nos da una fórmula para los tamaños de
los grupos $e_iC$, que es un resultado aún más fuerte. Mientras el teorema de Herbrand y
Ribet ya era conocido antes de la demostración de la Conjetura Principal, esta fórmula es un
resultado nuevo.

La técnica para derivar esta fórmula de la Conjetura Principal es esencialmente una
aplicación simple del lema de la serpiente, que funciona en una situación general. Resumamos
esto en el lema siguiente antes de aplicarlo a nuestra situación.

\begin{lem}\label{lem:cp-tamano}
  Sea $X$ un $\LL(G)$-módulo noetheriano de torsión sin submódulos finitos no triviales con
  $\charideal_{\LL(G)}(X)=(\mu)$ y sea $\nu\in\LL(G)$ tal que $X/\nu X$ es finito. Para cada
  $i\in\{1,\dotsc,p-1\}$ sean $\mu_i,\nu_i\in\LL(\Gamma)$ las imagenes de $e_i\mu$ y
  $e_i\nu$, respectivamente, bajo el isomorfismo $e_i\LL(G)\isom\LL(\Gamma)$. Asumamos que
  $\nu_i$ es coprimo a $\pi$ y a cada polinomio distinguido en
  $\LL(\Gamma)\isom\O\llbracket T\rrbracket$.

  Entonces
  $\LL(\Gamma)/(\mu_i,\nu_i)$ también es finito y
  \[ \#\left(e_i(X/\nu X\right))= \#\left(\LL(\Gamma)/(\mu_i,\nu_i)\right). \]

  Tenemos una afirmación análoga para módulos sobre $\LL(G)^\pm$.
\end{lem}
\begin{proof}
  Aquí seguimos \cite[Appendix, Lem.\ 8.6]{MR1029028}.
  
  Sea $X\rightarrow E$ un pseudo-isomorfo con $E$ un $\LL(G)$-módulo elemental, que es
  inyectivo, y sea $Z$ el conúcleo. Escribimos $X[\nu]$ para el núcleo de la multiplicación
  con $\nu$ en $X$, y similar para otros módulos. Entonces por las definiciones y el lema de
  la serpiente tenemos un diagrama conmutativo
  \begin{equation*}
    \begin{tikzcd}
      &
      & 0 \ar{d}
      & 0 \ar{d}
      & 0 \ar{d}
      & \\
      & 0\ar{r}
      & X[\nu] \ar{r} \ar{d}
      & E[\nu] \ar{r} \ar{d}
      & Z[\nu] \ar{d}
      &
      & \\
      & 0 \ar{r}
      & X \ar{r} \ar{dd}[near start]{\nu}
      & E \ar{r} \ar{dd}[near start]{\nu}
      & Z \ar{r} \ar{dd}[near start]{\nu}
      &  0
      & \\
      &
      &
      & ~
      &
      &
      \ar[r, phantom, ""{coordinate, name=Y}]
      &~\\
      ~& \ar[l, phantom, ""{coordinate, name=Z}] 0 \ar{r}
      & X \ar{r} \ar{d}
      & E \ar{r} \ar{d}
      & Z \ar{r} \ar{d}
      & 0
      & \\
      &
      & \ar[from=uuuurr, crossing over, rounded corners,
      to path=
      { -- ([xshift=2ex]\tikztostart.east)
        -| (Y) [near end]\tikztonodes
        -| (Z) [near end]\tikztonodes
        |- ([xshift=-2ex]\tikztotarget.west)
        -- (\tikztotarget)}
      ]  X/\nu X \ar{r} \ar{d}
      & E/\nu E \ar{r} \ar{d}
      & Z/\nu Z \ar{r} \ar{d}
      & 0
      & \\
      &
      & 0 & 0 & 0
    \end{tikzcd}
  \end{equation*}
  con lineas, columnas y serpientes exactas.  De este diagrama obtenemos: Porque $Z[\nu]$ es
  finito, $\rk_\O X[\nu]=\rk_\O E[\nu]$, y porque $X/\nu X$ es finito, $\rk_\O X[\nu]=0$
  (usamos aquí la \cref{prop:Xnoettorsiirgfinsiiacot}), es decir $\rk_\O E[\nu]=0$. Según el
  \cref{lem:elemental-rk-cero} entonces $E[\nu]=0$. Porque $Z$ es finito,
  $\#Z[\nu]=\#(Z/\nu Z)$. El lema de la serpiente entonces implica que
  $\#\left(X/\nu X\right)=\#\left(E/\nu E\right)$. De la misma manera, si aplicamos $e_i$ al
  diagrama entero,\footnote{Aplicar $e_i$ es lo mismo que aplicar
    $-\tensor_{\LL(G)}e_i\LL(G)$, y del isomorfismo $\LL(G)\isom\LL(\Gamma)^{p-1}$
    (\cref{cor:descomposicion-lambda}) es claro que $e_i\LL(G)\isom\LL(\Gamma)$ es plano
    sobre $\LL(G)$.} obtenemos
  \begin{equation*}
    \#\left(e_i(X/\nu X\right))=\#\left(e_i(E/\nu E\right))
  \end{equation*}
  y lo que falta ver es que $\#\left(e_i(E/\nu E\right))=\#(\LL(\Gamma)/(\mu_i,\nu_i))$. Es
  fácil ver que $e_i(E/\nu E)=e_iE/\nu_ie_iE$, y $e_iE$ es un $\LL(\Gamma)$-módulo
  elemental, pues de la forma $\bigoplus_{j=1}^s\LL(\Gamma)/(f_j)$, y $(\prod_{j=1}^sf_j)=(\mu_i)$. Por
  eso es suficiente demostrar que
  \begin{equation*}
    \#\left(\bigoplus_{j=1}^s\LL(\Gamma)/(f_j,\nu_i)\right)=\#\left(\LL(\Gamma)/(\prod_{j=1}^sf_j,\nu_i)\right).
  \end{equation*}
  Eso lo demostramos con inducción en $s$; el caso $s=1$ es trivial. Sea entonces $s>1$, ponemos
  $f=f_1\dotsm f_{s-1}$ y $g=f_s$. Hay una sucesión exacta
  \begin{equation*}
    0\rightarrow(f,\nu_i)/(fg,\nu_i)\rightarrow
    \LL(\Gamma)/(fg,\nu_i)\rightarrow\LL(\Gamma)/(f,\nu_i)\oplus \LL(\Gamma)/(g,\nu_i)\rightarrow
    \LL(\Gamma)/(g,\nu_i) \rightarrow 0
  \end{equation*}
  donde el mapeo en el centro es $r\mapsto(r,0)$. Por eso es suficiente demostrar que
  $\#(f,\nu_i)/(fg,\nu_i)=\#\LL(\Gamma)/(g,\nu_i)$. Pero el mapeo 
  \begin{equation*}
    \LL(\Gamma)/(g,\nu_i)\rightarrow(f,\nu_i)/(fg,\nu_i),\quad x\mapsto fx
  \end{equation*}
  es obviamente sobreyectivo y porque $f$ es coprimo a $\nu_i$ también es inyectivo.
\end{proof}

\begin{prop}\label{prop:consecuencia-numero-de-clases}
  La Conjetura Principal implica que para cada $n\in\{2,\dotsc,p-3\}$ par tenemos
  \[ \#e_{1-n}C = p^{v_p(B_{1,\omega^{n-1}})}. \]
\end{prop}
\begin{proof}
  Asumamos la Conjetura Principal y apliquemos el \cref{lem:cp-tamano} con $X=X_\infty^-$,
  $\mu=h_1^{(1)}\mu_{\mathbf 1}$, $\nu=\sigma_{1+p}-1$ y $i=1-n$
  (que es impar). Entonces por el \cref{cor:iwasawa-control-x} tenemos $e_i(X/\nu X)\isom
  e_iC$. Además, $\nu_i=\sigma_{1+p}-1$ y como $i\neq1$ el elemento $\mu_i$ es la imagen
  de $e_i\mu_{\mathbf 1}$ en $\LL(\Gamma)$. El lema entonces dice
  \begin{equation*}
    \#e_iC=\#(\LL(\Gamma)/(\sigma_{1+p}-1,\mu_i).
  \end{equation*}
  Ahora el morfismo $\LL(\Gamma)\rightarrow\Zp$ inducido por el carácter trivial
  $\kappa_0^0\colon\Gamma\rightarrow\Z_p^\times$ es sobreyectivo con núcleo generado por
  $\sigma_{1+p}-1$ (si identificamos $\LL(\Gamma)$ con $\Zp\llbracket T\rrbracket$ usando el
  generador topológico $\sigma_{1+p}$ entonces el morfismo es inducido por $T\mapsto0$, cuyo
  núcleo es generado por $T$). El isomorfismo $\LL(\Gamma)/(\sigma_{1+p}-1)\isomarrow\Zp$
  que obtenemos envía $\mu_i$ a
  $\mu_{\mathbf 1}(\omega^i)=-B_{1,\omega^{-i}}$ (usamos
  aquí el diagrama \eqref{eqn:diagrama-evaluacion-g-gamma} y el
  \cref{lem:valor-zeta-omega}) y entonces induce un isomorfismo
  \begin{equation*}
    \LL(\Gamma)/(\mu_i\sigma_{1+p}-1)\isomarrow\Zp/B_{1,\omega^{-i}}
  \end{equation*}
  y esto termina la demostración.
\end{proof}

\begin{cor}\label{cor:stickelberger}
  La Conjetura Principal implica el siguiente caso especial del teorema de
  Stickelberger. Sea $\Sigma_p\in\Q[G_p]$ el elemento de Stickelberger y sea $I$ el ideal
  $\Sigma_p\Z[G_p]\cap\Z[G_p]$. Entonces cada elemento de $I^-$ anula a $C^-$.
\end{cor}
\begin{proof}
  Se puede demostrar que $I=\Sigma_pJ$, donde $J$ es el ideal generado por los elementos de
  la forma $\sigma_a-a$ con $a\in\Z$, $p\nmid a$. Omitimos esta demostración aquí, véase
  \cite[§1.2, Lemma 2, p.\ 11]{MR1029028} para ella. Usando esto, la afirmación que queremos
  demostrar es equivalente a: Para cada $i\in\{1,\dotsc,p-1\}$ impar, los elementos del
  ideal $e_i\Sigma_pJ\subseteq\Zp$ anulan a $e_1C$. El ideal $e_iJ$ es generado por los
  $\omega^i(a)-a$ con $a\in\Z$, $p\nmid a$, y se verifica fácilmente que esto es igual a
  $(p)$ si $i=1$ y igual a $(1)$ si $i\neq1$, véase \cite[§1.3, Lemma 1, p.\
  15]{MR1029028}. Entonces el caso de $i=1$ ya es claro. En el caso $i\neq1$ tenemos
  $e_i\Sigma_pJ=(B_{1,\omega^{-i}})$ según la \cref{rem:stickelberger-mu-uno}, y la
  afirmación resulta de la \cref{prop:consecuencia-numero-de-clases}.
\end{proof}

Concluimos con un ejemplo que muestra las consecuencias de la Conjetura Principal en un
caso concreto.

\begin{ex}
  Sea $p=691$, $K=\Q(\mu_{691})$ y $C$ la $p$-parte del grupo de clases de $K$.
  
  Según el \cref{ex:bernoulli} y la \cref{zeta-continuacion} tenemos que
  \[ \zeta(-11)=\frac{691}{32760}=\frac{691}{2^3\cdot3^2\cdot5\cdot7\cdot13}. \] Tenemos
  además $-11\in\{-1,-3,\dotsc,4-p\}$. Por eso el criterio de Kummer implica que $691$
  divide al número de clases de $K$, es decir $C\neq0$. En particular, usando el
  \cref{thm:campo-de-hilbert} vemos que $K$ tiene una extensión cíclica no ramificada de
  grado $691$. La Conjetura Principal y sus consecuencias nos proveen con más comprensión
  sobre la estructura de $C$.

  El \cref{thm:herbrand-ribet} de Herbrand y Ribet dice que $V_{-11}\neq0$ con la notación
  de allá, porque $-11=1-n$ para $n=12$. En particular $e_{-11}C\neq0$, y la
  \cref{prop:consecuencia-numero-de-clases} incluso nos permite calcular su orden.
  Calculemos $v_{691}(B_{1,\omega^{11}})$. Según el \cref{ejer:formula-b-n-chi} tenemos
  \[ B_{1,\omega^{11}}=\frac{1}{691}\sum_{a=1}^{690}\omega^{11}(a)a. \]
  Podemos calcular esto con SAGE usando el siguiente código:
  \begin{center}
\begin{lstlisting}[language=Python]
from sage.rings.padics.padic_generic import ResidueLiftingMap
R = Zp(691, prec = 10, type = 'fixed-mod')
k = R.residue_field()
omega = ResidueLiftingMap._create_(k, R)
v = R.valuation()
s = sum(a*omega(a)^11 for a in [1..690])
v(s/691)
\end{lstlisting}    
  \end{center}
  El resultado es que $v_{691}(B_{1,\omega^{11}})=1$.

  Con esto concluimos que la $691$-parte del grupo de clases de $K=\Q(\mu_{691})$ tiene un
  subgrupo de orden $691$ en que un elemento del grupo de Galois
  $a\in(\Z/691\Z)^\times\isom\Gal(K/\Q)$ actúa como multiplicación con $a^{-11}$, y este
  subgrupo es máximo con esta propiedad.
\end{ex}

Finalmente mencionamos que en \cite[§XI.6]{NSW} se encuentran aún más consecuencias de la
Conjetura Principal que no vamos a describir aquí.

\ejercicios

\begin{ejer}
  Demuestre que la Conjetura Principal implica el teorema de Herbrand y Ribet:
  \begin{enumerate}
  \item Demuéstrelo usando argumentos similares a los de la demostración de la
    \cref{prop:cp-implica-kummer}.
  \item Demuestre que la afirmación de la \cref{prop:consecuencia-numero-de-clases} implica
    el teorema de Herbrand y Ribet, usando el \cref{cor:unidad-depende-solo-de-i}.
  \end{enumerate}
\end{ejer}

\section{Unos comentarios sobre la demostración de la Conjetura Principal}
\label{sec:dem-cp}

En esta sección describimos de manera muy breve los métodos con los que se puede demostrar la
Conjetura Principal, siguiendo principalmente el texto \cite[§2.4--5]{MR2334196}.

Existen esencialmente dos enfoques para demostrar la Conjetura Principal. El enfoque
original de Mazur y Wiles usa formas modulares y es una extensión de métodos usados por
Ribet en su demostración del \cref{thm:herbrand-ribet}. El segundo enfoque, hallado por
Rubin usando ideas de Kolyvagin y Thaine, usa algo que se llama \define{sistemas de
  Euler}. Por su naturaleza, el enfoque con formas modulares demuestra la inclusión
$\charideal_{\LL(G)^-}(X_\infty^-)\subseteq(h_1\mu_{\mathbf 1})$ en la conjetura principal
(o la inclusión análoga en la otra versión) y el método de sistemas de Euler demuestra la
inclusión opuesta. En el caso de la Conjetura Principal original de Iwasawa, de hecho una de
estas inclusiones ya implica la otra (esto es una consecuencia de la fórmula de clases). Sin
embargo, en las generalizaciones de la Conjetura Principal que vamos a explicar en la
\cref{sec:generalisaciones-conjetura-princial} este lujo ya no existe, así que hay que usar
ambos métodos para demostrarlas.

La relación entre las formas modulares y la Conjetura Principal es, muy a manera de esbozo,
la siguiente. Como vimos en la \cref{sec:consecuencias-palf}, la función zeta $p$-ádica
tiene que ver con congruencias módulo $p$ entre valores especiales de la función zeta. Estos
valores también aparecen como el coeficiente constante en la expansión $q$ de series de
Eisenstein: para $k\ge4$ la series de Eisenstein $E_k$ de peso $k$ tiene la expansión de
Fourier \[ E_k = \frac{\zeta(1-k)}2 + \sum_{n=1}^\infty\sigma_{k-1}(n)q^n \] donde
$\sigma_{k-1}(n)$ es la suma de las potencias $(k-1)$-ésimas de los divisores de $n$. Por
eso, si $\zeta(1-k)$ es divisible por $p$ entonces $E_k\mod p$ parece una forma modular
cuspidal, y de hecho existe una forma cuspidal $f$ congruente a $E_k$ módulo $p$. Esta
congruencia implica que las representaciones de Galois de $E_k$ y $f$ están relacionadas
módulo $p$, más precisamente tienen las mismas semisimplificaciones. La representación de
$E_k$ es conocida explícitamente y es reducible, así que la de $f$ módulo $p$ es una extensión de
los componentes irreducibles de la de $E_k$. Finalmente, estas extensiones están relacionadas
con sus grupos de clases por la teoría de campos de clases. Para más detalles, véase
\cite[§2.4]{MR2334196} y el artículo original de Mazur y Wiles \cite{MR0742853}.

Un sistema de Euler es una colección de elementos de grupos de cohomología de Galois de la
forma $c_m\in\mathrm H^1(\Q(\mu_m),T)$ para cada $m\in\Nuno$, donde $T$ es un retículo
estable de una representación $V$ de $\Gal_\Q$ con coeficientes en $\Qp$, tal que los $c_m$
cumplen algunas relaciones de compatibilidad con respecto al mapeo de la correstricción. Esta
definición probablemente parece poco iluminadora, pero la razón por la que uno se interesa en
este tipo de colección es que su sola existencia produce anuladores de grupos de clases (u
objetos más generales) que permiten acotar el exponente o incluso el orden de dichos
grupos. Por otro lado los sistemas de Euler están relacionados con valores especiales de
funciones $L$: existe un mapeo llamado el mapeo de Coleman del límite de los
$\mathrm H^1(\Q(\mu_m),T)$ al álgebra de Iwasawa enviando el sistema de Euler a la función
zeta $p$-ádica. Esto implica relaciones entre los órdenes de los grupos de clases y los
valores especiales. De esta manera el sistema de Euler permite conectar los dos objetos que
aparecen en la Conjetura Principal.

Existe una teoría bastante general de sistemas de Euler. La dificultad principal en este
ámbito es \emph{construir} sistemas de Euler, es decir demostrar que existen y que no son
triviales. Aquí no hay un método general y hasta ahora conocemos sólo unos pocos sistemas de
Euler. En el caso de la Conjetura Principal, el sistema de Euler que se usa es construido
usando las unidades ciclotómicas de la \cref{dfn:unidades-ciclotomicas}.

Textos introductorios a la teoría de sistemas de Euler son \cite{LoefflerES} y
\cite{MR1749177}. Una demostración completa de la Conjetura Principal usando sistemas de
Euler se encuentra en \cite[Appendix]{MR1029028} y en \cite{MR2256969}.

\chapter{Generalizaciones}
\label{sec:generalizaciones}

\section[Aritmética Logarítmica]{Teoría de Iwasawa para aritmética logarítmica}

La aritmética logarítmica se sitúa en el contexto de la teoría $p$-ádica de campos de clases. Esta última es una especialización de la teoría de campos de clases (e.g. \cite[Cap. IV]{MR1697859}) en el caso especial de pro-$p$-extensiones de un campo de números. Los resultados principales en la teoría de campos de clases se leen de la siguiente manera en la teoría $p$-ádica.\\ 

Para un campo $K$ denotamos $K_{\p}$ su completado en la plaza finita $\p$. El límite inverso \hbox{$\Rc_{\p}=\varprojlim K_{\p}^{\times} / K_{\p}^{\times,p^{r}}$}
admite una descomposición de la forma $\Rc_{\p}=\Uc_{\p}\pi_{\p}^{\Zp}$, donde $\pi_{\p}$ es un uniformizante de $K_{\p}$, es decir $v_{\p}(\pi_{\p})=1$ y $\Uc_{\p}$ es un subgrupo de unidades que depende de $\p$
$$\Uc_{\p}=\left\lbrace \begin{array}{lc}
U_{\p}^{(1)} & \text{ si } \p\,|\,p, \\
\mu_{\p} & \p\,\nmid\,p ;
\end{array} \right.$$
donde $U_{\p}^{(1)}$ es el grupo de unidades principales de $K_{\p}$ y $\mu_{\p}$ es subgrupo de raíces de la unidad de orden $p^{r}$ (para algún $r$) contenidas en $K_{\p}$.

\begin{ex}\label{ex:aritp3KQ} Sea $p=3$ y $K=\Q$ entonces 
$$\Rc_{3}=(1+3\Z_{3})\times 3^{\Z_{3}},\	\	\	\Rc_{5}= 5^{\Z_{3}}\	\	\	y \	\	\	\Rc_{7}=\{1,\zeta_{3},\zeta_{3}^{2}\}\times 7^{\Z_{3}}.$$
\end{ex}
 
\begin{thm}[Clases de campos local]\label{thm:cclpadic} La aplicación de reciprocidad induce un isomorfismo de $\Zp$-módulos topológicos 
$$\Rc_{\p} \isom \Gal(K_{\p}^{\ab}/K_{\p}),$$
donde $K_{\p}^{\ab}$ es la máxima pro-$p$-extensión abeliana de $K_{\p}$. Además la imagen de $\Uc_{\p}$ se identifica al subgrupo de inercia $I_{\p}\subset \Gal(K_{\p}^{\ab}/K_{\p})$.
\end{thm} 
 
\begin{ex} Continuando con el \cref{ex:aritp3KQ}
$$\begin{tikzcd}
    \Q_{3}^{c} \arrow[r,-] & \Q_{3}^{\ab} \arrow[d,-] \\
    \Q_{3} \arrow[u,-] & \Q_{3}^{\nr} \arrow[l,-] 
  \end{tikzcd} \	\	\	\	\	\	
  \begin{tikzcd}
    \Q_{5}^{c}=\Q_{5}^{\nr}=\Q_{5}^{\ab} \arrow[d,-] \\
    \Q_{5}
  \end{tikzcd} \	\	\	\	\	\	
  \begin{tikzcd}
   \Q_{7}^{c}=\Q_{7}^{\nr} \arrow[r,-] & \Q_{7}^{\ab}  \\
    \Q_{7} \arrow[u,-] &  
  \end{tikzcd}
$$
\end{ex} 

Sea $\Pl_{K}$ el conjunto plazas de $K$. Denotamos $\Jc_{K}=\prod_{\p\in \Pl_{K}}\Rc_{\p}$ el producto restringido de los $\Rc_{\p}$, es decir $(x_{\p})_{\p}\in\Jc_{K}$ si $x_{\p}\in \Uc_{\p}$ para casi toda $\p$ (e.g. \cite[Def. 1.1.12]{NSW}). El producto tensorial $\Rc_{K}=\Zp\otimes_{\Z}K^{\times}$ se inyecta en $\Jc_{K}$ vía el encaje diagonal canónico \cite[I.1.1.4]{Jaulent86}.
 
\begin{thm}[Clases de campos global]\label{thm:ccgpadic} La aplicación de reciprocidad induce un isomorfismo de $\Zp$-módulos topológicos
$$\Jc_{K}/\Rc_{K} \isom G_{K}=\Gal(K^{\ab}/K),$$
donde $K^{\ab}$ es la máxima pro-$p$-extensión abeliana de $K$. El subgrupo de descomposición $D_{\p}$ de una plaza $\p$ en $K$ corresponde a la imagen de $\Rc_{\p}$ en $G_{K}$ y el subgrupo de inercia $I_{\p}$ a la imagen de $\Uc_{\p}$.
\end{thm}
 
\subsection{Valores absolutos $p$-ádicos}

Una \importante{valuación $p$-ádica} es un epimorfismo de grupos \hbox{$\hat{v}_{\p}:\Rc_{\p}\twoheadrightarrow \Zp$.} Notemos que esta definición de valuación difiere de la definición usual, primero porque toma valores en un grupo profinito y además porque la imagen no es discreta. 
 
De manera similar decimos que un epimorfismo $|\cdot|_{\p}:\Rc_{\p}\twoheadrightarrow 1+p\Zp$ es un \importante{valor absoluto $p$-ádico}. 

El ejemplo siguiente muestra cómo construir una valuación $p$-ádica a partir de la valuación del campo $K_{\p}$.

\begin{ex}\label{ex:valabspadic} Sea $K_{\p}$ la localización de un campo de números $K$ en la plaza $\p$. Sea $v_{\p}:K_{\p}^{\times}\rightarrow \Z$ la valuación ordinaria. Podemos inducir una valuación $p$-ádica de la siguiente forma:
\begin{eqnarray*}
\hat{v}_{\p}:\Rc_{\p} & \rightarrow & \Zp \\
x_{\p}=(x_{i})_{i\in\Nuno} & \longmapsto &\hat{v}_{\p}=(v_{\p}(x_{i}))_{i\in\Nuno}.
\end{eqnarray*}
Las aplicaciones $\hat{v}_{\p}$ son claramente epimorfismos. 

Ahora, es fácil ver que vía la función exponencial $p$-ádica (e.g. \cite[II.5.5]{MR1697859}) las valuaciones definidas arriba dan lugar a valores absolutos $p$-ádicos $|x|_{\p}=\exp(pv_{\p}(x))$.
\end{ex}   
 
Una familia de valores absolutos $p$-ádicos es admisible si el homomorfismo 
\begin{eqnarray*}
\Jc_{K} & \longrightarrow & 1+p\Zp \\
(x_{\p})_{\p\in \Pl_{K}} & \longmapsto & \prod_{\p} |x_{\p}|_{\p}
\end{eqnarray*}
es continuo y el núcleo contiene a $\Rc_{K}$. 

Esta definición no es más que una reinterpretación de la fórmula del producto. 

\begin{ex} Claramente la familia $(v_{\p})_{\p}$ inducida por el \cref{ex:valabspadic} es admisible. 
\end{ex}

\subsection{Teoría de Ramificación}

Recordemos que dada una valuación $p$-ádica $\hat{v}_{\p}$ tenemos una sucesión exacta 
\begin{equation}\label{eq:sucexacvalpad}
1\longrightarrow \widehat{\Uc}_{\p}\longrightarrow \Rc_{\p}\stackrel{\hat{v}_{\p}}{\longrightarrow} \Zp \rightarrow 0, 
\end{equation}
donde $\widehat{\Uc}_{\p}=\ker(\hat{v}_{\p})$. Bajo la correspondencia del \cref{thm:cclpadic} llamamos grupo de inercia asociado a $\hat{v}_{\p}$, a la imagen $\hat{I}_{\p}$ de $\widehat{\Uc}_{\p}$ en $\Gal(K^{\ab}_{\p}/K_{\p})$. Así mismo, denotamos $\widehat{K_{\p}}$ la $\Zp$-extensión de $K_{\p}$ fijada por $\hat{I}_{\p}$, es decir $\widehat{K_{\p}}=(K^{\ab}_{\p})^{\hat{I}_{\p}}$.

$$\begin{tikzcd}[column sep = 2cm]
    \hat{K}_{\p} \arrow[r,-,"\widehat{\Uc}_{\p}\isom\hat{I}_{\p}"] & K_{\p}^{\ab}  \\
    K_{\p} \arrow[u,-] &  
  \end{tikzcd}$$

Sea $L_{\mathfrak{P}}$ una extensión finita de $K_{\p}$ y un primo $\mathfrak{P}$ un primo sobre $\p$. Definimos el índice de ramificación y el grado de inercia de $L_{\mathfrak{P}}/K_{\p}$ como 
$$\hat{e}_{\p}=[L_{\mathfrak{P}}:\widehat{K_{\p}}\cap L_{\mathfrak{P}}]\	\	\	y\	\	\	\hat{f}_{\p}=[\widehat{K_{\p}}\cap L_{\mathfrak{P}}:K_{\p}],$$ 
respectivamente. Decimos que la extensión $L_{\mathfrak{P}}$ es no ramificada con respecto a $\hat{v}_{\p}$ si $\hat{e}_{\p}=1$, de lo contrario decimos que es ramificada con respecto a $\hat{v}_{\p}$. 

Si $L/K$ es una $p$-extensión de campos de números, decimos que $L$ es no ramificada en $\p\in \Pl_{K}$ con respecto a $\hat{v}_{\p}$ si $L_{\mathfrak{P}}/K_{\p}$ es no ramificada para toda plaza $\mathfrak{P}\in \Pl_{L}$ sobre $\p$. Además $L/K$ es no ramificada con respecto a la familia $(\hat{v}_{\p})_{\p}$ si $L$ es no ramificada para toda $\p\in\Pl_{K}$.

Note que para la familia de valuaciones $p$-ádicas definida en el \cref{ex:valabspadic}, el concepto de ramificación clásico coincide con la definición arriba.

\ejercicios

\begin{ejer}\label{ejer:compositum} Sea $L$ el compuesto de todas las $\Zp$-extensiones de un campo local $K_{\p}$. Demuestre que el grupo de Galois $\Gal(L/K_{\p})$ es un $\Zp$-módulo noetheriano libre. 
\end{ejer}

\begin{ejer}\label{ejer:compositumrango} Calcule el rango del módulo del ejercicio anterior cuando $K_{\p}=\Q_{p}$.
\end{ejer}

\begin{ejer}\label{ejer:existenciaZp}
Demuestre que para todo campo local $K_{\p}$ existe una única $\Zp$-extensión ciclotómica. 
\end{ejer}

\begin{ejer}\label{ejer:existenciaUnr} Demuestre que para todo campo local $K_{\p}$ existe una única $\Zp$-extensión no ramificada. 
\end{ejer}

\subsection{El caso logarítmico}

Además del caso clásico, que hemos venido mencionando en los ejemplos, existe un panorama natural en el que podemos considerar las nociones definidas anteriormente. 

El panorama descrito arriba se presenta desde el caso $K=\Q$. Recordemos que en este caso tenemos
$$\Rc_{q}=\left\lbrace \begin{array}{cl}
\mu_{q} q^{\Zp} &  \text{ si } q\neq p,\\
U_{p}^{(1)}p^{\Zp} & \text{ si } p=q.
\end{array}\right.$$
Por lo tanto con la sucesión exacta \eqref{eq:sucexacvalpad} obtenemos que cuando $q\neq p$ existe una única valuación $p$-ádica módulo $\Zp^{\times}$ tal que $v_{q}(q)=1$. Sin embargo, cuando $p=q$ tenemos dos valuaciones $p$-ádicas canónicas. Dado que $U_{p}^{(1)}\isom (1+p\Zp)\isom \Zp$ podemos definir una valuación no clásica como la proyección de $U_{p}^{(1)}$ en $\Zp$ tal que $v_{p}(1+p)=1$. De hecho  
\begin{equation}\label{eq:padicvalQ}
\tilde{v}_{p}(x)=-\dfrac{\Log_{p}x}{p}
\end{equation}
tiene esa propiedad, donde $\Log_{p}$ es el logaritmo de Iwasawa, es decir la extensión del
logaritmo $p$-ádico de la \cref{prop:isom-z-p-log} a $\Qp^{\times}$ con la convención $\Log_{p}(p)=0$. La familia $(\tilde{v}_{q})_{q}$ con $\tilde{v}_{q}=v_{q}$ para $q\neq p$, con $v_{q}$ como en \cref{ex:valabspadic} y con $\tilde{v}_{p}$ definida como en \eqref{eq:padicvalQ}, define una familia de valores absolutos $p$-ádicos admisibles. Además a la ramificación con respecto a esta familia la llamamos \importante{ramificación logarítmica}. 

En todos los casos, el núcleo $\widetilde{\Uc}_{q}$ de $\tilde{v}_{q}$ se identifica al subgrupo de Galois de $\Gal(\Q^{\ab}_{q}/\Q_{q})$ que fija la $\Zp$-extensión ciclotómica $\Q_{q}^{c}$ de $\Q_{q}$. Es decir, la ramificación logarítmica (i.\,e. con respecto a estas valuaciones $p$-ádicas) mide qué tan lejos una extensión $K_{\mathfrak{q}}$ de $\Q_{q}$ está de ser ciclotómica. Notemos que si $q\neq p$ entonces $\Q_{q}^{c}=\Q_{q}^{\nr}$ donde $\Q_{q}^{\nr}$ es la máxima pro-$p$-extensión abeliana no ramificada de $\Q_{q}$. 

La máxima $p$-extensión abeliana de $\Q$ que no es ramificada con respecto a la familia $(\tilde{v}_{q})_{q}$ es la $\Zp$-extensión ciclotómica $\Q^{c}$ de $\Q$. El lector debe notar ahora, que este fenómeno difiere enormemente del caso clásico, donde la máxima $p$-extensión no ramificada de $\Q$, es $\Q$ mismo. 

El fenómeno apenas descrito se reproduce en un campo de números arbitrario. Es decir, haciendo las debidas generalizaciones, se tiene que la máxima $p$-extensión no ramificada con respecto a la familia $(\tilde{v}_{\p})$, donde 
$$\tilde{v}_{\p}(x)= \left\lbrace \begin{array}{cr}
v_{\p}(x) & \text{ si } \p \nmid p \\
-\dfrac{\Log_{p}(\Norm_{\Kp/\Qp}(x))}{\tilde{f}_{p}\Log_{p}(1+p)} & \text{ si } \p\,|\,p;
\end{array}\right.$$
contiene la $\Zp$-extensión ciclotómica de $K$. Por supuesto en el caso clásico esta extensión es finita. La teoría de campos de clases o su versión $p$-ádica, hace corresponder el grupo de Galois de esta extensión sobre $K$ con la $p$-parte del grupo de clases de $K$. 

\begin{defi} Llamamos \emph{extensión localmente ciclotómica}\index[def]{extensión!localmente ciclotómica} a la máxima extensión abeliana logarítmicamente no ramificada $L$, i.\,e. con respecto a la familia $(\tilde{v}_{\p})_{\p\in\Pl_{K}}$.
\end{defi}

Por supuesto, la extensión $L/K$ es de Galois por ser máxima. En particular, la extensión $L/K^{c}$ es de Galois y el subgrupo $\Gal(L/K^{c})$ es de interés. 

\begin{defi} Llamamos \emph{grupo de clases logarítmicas}\index[def]{grupo de clases@grupo de clases!logarítmicas} al grupo de Galois $\Gal(L/K^{c})$ y lo denotamos $\Clog{K}$.
\end{defi}

En el caso $K=\Q$, vimos que $\Clog{\Q}$ es trivial. Análogamente al caso clásico, el grupo de clases logarítmicas juega el papel del grupo de clases en la teoría logarítmica. El grupo de clases logarítmicas es abeliano, ya que la extensión $L$ es abeliana. No obstante, la finitud de $\Clog{K}$ es desconocida en general.

\begin{con}[Gross-Kuz'min] El grupo de Galois $\Gal(L/K)$ es un $\Zp$-módulo de rango $1$ sobre $\Zp$. 
\end{con}

Ya que la extensión localmente ciclotómica $L$ contiene a la $\Zp$-extensión ciclotómica $K^{c}$ de $K$, la conjetura de Gross-Kuz'min implica que el grupo de clases logarítmicas $\Clog{K}$ es un grupo finito. Por lo tanto, en este sentido el grupo $\Clog{K}$ es realmente análogo al grupo de clases. \\

La conjetura de Gross-Kuz'min es válida, usando ciertos argumentos de trascendencia de Brumer, en el caso en que $K/\Q$ es una extensión abeliana. Recientemente, un algoritmo implementado por Belabas y Jaulent \cite{Belabas&Jaulent16} en Pari/GP  permite calcular el grupo de clases logarítmicas para un primo $p$ y un campo de números $K$. En este caso, cuando el programa provee un resultado, entonces verifica la conjetura para el par $(p,K)$. Esto sugiere que quizás sea suficiente contar con un resultado teórico que verifique la conjetura para el grado de $K$ suficientemente grande y para primos suficientemente grandes. Entonces los demás casos podrían ser eventualmente verificados computacionalmente. 

\begin{ex}\label{ex:clogvarios} Calculamos las siguientes tablas con ayuda de las funciones \texttt{bnfinit} y \texttt{bnflog} en Pari/GP \cite{PARI2}. La primera columna corresponde a la ecuación que satisface el elemento primitivo $\alpha$. En la segunda columna se muestra el grupo de clases de $K=\Q(\alpha)$. En las columnas restantes se muestran los grupos de clase logarítmicos de $\Q(\alpha)$ con respecto a los primeros tres números primos regulares y los dos primeros irregulares. La notación $[m_{1},\ldots, m_{r}]$ con $m_{i}\geq 1$, describe el grupo abeliano $\Z/m_{1}\Z\times\cdots\times\Z/m_{r}\Z$. 

\begin{center}
\begin{tabular}{|c|c|c|c|c|c|c|}
\cline{3-7}
\multicolumn{2}{c|}{ }  & $p=2$ & $p=3$  & $p=5$  &  $p=37$  &  $p=59$ \\ \hline
$K=\Q(\alpha)$ & $Cl(K)$ & $\Clog{K}$ & $\Clog{K}$ & $\Clog{K}$ & $\Clog{K}$ & $\Clog{K}$ \\ \hline
$\alpha^{2}+86$ & $[10]$ & $1$ & $[3]$  & $1$ & $1$ & $1$
\\ \hline 
$\alpha^{2}-7726$ & $[3]$ & $1$ & $[3,3]$ & $1$ & $1$ & $1$
\\ \hline
$\alpha^6 + 3\alpha^5 + 6\alpha^4 + 123\alpha^3$  & \multirow{2}{*}{$[6,6]$} & \multirow{2}{*}{$[2,2]$} & \multirow{2}{*}{$[3]$} & \multirow{2}{*}{$[5,5]$} & \multirow{2}{*}{$1$} & \multirow{2}{*}{$1$} 
\\
$+ 180\alpha^2 - 171\alpha + 3249$ &  &  &  &  &  & 
\\ \hline
$\sum_{i=0}^{36}\alpha^{i}$ & $[37]$ & $1$ & $1$ & $1$ & $[37]$ & $1$
\\ \hline
$\sum_{i=0}^{58}\alpha^{i}$ & $[3*59*233]$ & $1$ & $1$ & $1$ & $1$ & $[59]$ 
\\ \hline
\end{tabular}
\end{center}
\end{ex}

En particular desde los primeros dos campos cuadráticos podemos observar que el grupo de clases logarítmicas no es un subgrupo del grupo de clases y la $p$-parte del grupo de clases no es un subgrupo del grupo de clases logarítmicas. Además, podemos ver que tampoco hay una relación entre el tamaño de estos dos. No obstante, en la siguiente sección veremos que se puede estudiar estos grupos a lo largo de $\Zp$-extensiones siguiendo el espíritu de Iwasawa. 

\ejercicios

\begin{ejer} Descargue Pari/GP en su ordenador y calcule el grupo de clases logarítmicas de su campo de números favorito con respecto a los primeros primos.  
\end{ejer}

\subsection{Teoría de Iwasawa}

En el \cref{sec:grupos-de-clases} vimos el extraordinario descubrimiento de Iwasawa sobre la $p$-parte de los grupos de clases de los subcampos $K_{r}$ de una $\Zp$-extensión (\cref{thm:iwasawa-ThmdIwa}). Es muy natural preguntarse si existe una fórmula que describa el orden de los grupos de clases logarítmicas para $r$ suficientemente grande. La respuesta es sí, y en esta sección contamos su historia. \\

Primero, recordemos que para todo campo de números $K$, su $\Zp$-extensión ciclotómica $K^{c}$ es logarítmicamente no ramificada. Es decir, contrario a la \cref{prop:ramificZpext}, en este caso todas las plazas de $K$ son logarítmicamente no ramificadas. Así es, incluso las plazas arriba de $p$. Sin embargo, Jaulent demostró en su tesis de doctorado que el grupo de clases logarítmico del nivel $K_{r}$ tiene una interpretación como un cociente de $\Gal(H'_{\infty}/K^{c})$, donde $H'_{\infty}$ es la máxima extensión pro-$p$ abeliana no ramificada $p$-descompuesta de $K^{c}$. El grupo $X'_{\infty}:=\Gal(H'_{\infty}/K^{c})$ es isomorfo al límite inverso $\varprojlim C'_{r}$ de los $p$-grupos de clases cocientes de $C_{r}$ por el subgrupo generado por las plazas arriba de $p$. Para $r\geq 0$ la máxima extensión abeliana pro-$p$ logarítmicamente no ramificada $L_{r}$ contiene $K^{c}$ y además está contenida en $H'_{\infty}$. De hecho, $L_{r}$ es la subextensión de $H'_{\infty}$ fijada por $\w_{r}X'_{\infty}$. Es decir, el grupo de clases logarítmicas del nivel $K_{r}$ está dado por
\[\Clog{r}:=\Clog{K_{r}}=\Gal(L_{r}/K^{c})\isom X'_{\infty}/\w_{r}X'_{\infty}\] 
para $r$ suficientemente grande.

El módulo $X'_{\infty}$ es un $\LL$-módulo noetheriano y de torsión. Aplicando el teorema de estructura de $\LL$-módulos noetherianos obtenemos invariantes $\mut$, $\lat$. Finalmente, un argumento clásico de descenso como en la demostración del \cref{thm:iwasawa-ThmdIwa}, nos da el resultado para el caso $K^{c}/K$. \\

En el caso de una $\Zp$-extensión no ciclotómica es distinto \cite{Villanueva18}. La manera en que está definida la familia de valuaciones $(\tilde{v}_{\p})_{\p}$, hace que las $\Zp$-extensiones $K_{\infty}$ de un campo de números $K$ sean logarítmicamente no ramificadas fuera de $p$. Asumiendo que la conjetura de Gross-Kuz'min es cierta para $K$, entonces al menos una plaza $\p\in\Pl_{K}$ arriba de $p$ ramifica en $K_{\infty}$. Este hecho es clave, pues a partir de un nivel $K_{r}$ en la $\Zp$-extensión $K_{\infty}$, esta última es independiente de la $\Zp$-extensión ciclotómica de $K_{r}$ y por lo tanto de su máxima extensión logarítmicamente no ramificada. Además, existe un nivel en que todas las plazas que ramifican logarítmicamente en $K_{\infty}/K$ son totalmente logarítmicamente ramificadas. 

El escenario arriba descrito es paralelo al escenario que teníamos en el \cref{sec:grupos-de-clases} cuando demostramos el teorema de Iwasawa. Como ya hemos visto, la máxima extensión abeliana logarítmicamente no ramificada es infinita. No obstante, al ajustar las técnicas de descenso, podemos describir el grupo de clases logarítmicas como un cociente de un $\LL$-módulo noetheriano y de torsión (suponiendo la conjetura de Gross-Kuz'min en todos los niveles de $K_{\infty}$). Es decir los resultados de Jaulent y uno de nosotros demuestran el Teorema de Iwasawa en el caso logarítmico.

\begin{thm}\label{thm:iwasawa-ThmdIwaLOG} Sea $K_{\infty}/K$ una $\Zp$-extensión tal que la conjetura de Gross-Kuz'min es válida. Sea $\Clog{r}$ el grupo de clases logarítmicas del nivel $K_{r}$ y sea $p^{\tilde{e}_{r}}$ su orden. Entonces existen enteros $\mut,\lat\geq 0$ y $\nut$ tal que 
$$\tilde{e}_{r}=\mut \ell^{r} + \lat r + \nut, \	\	\	\text{ para } r \text{ suficientemente grande. }$$
\end{thm}

Este teorema es un pilar de las interacciones de la teoría de Iwasawa y la aritmética logarítmica. 

Es sumamente interesante que pese a las diferencias de los grupos de clases clásico y logarítmico, se puedan estudiar y presenten propiedades semejantes desde el punto de vista de la teoría de Iwasawa. En la siguiente sección vamos a ahondar en las semejanzas y diferencias entre los dos casos.

\subsection{Relaciones entre invariantes clásicos y logarítmicos}

Si $K_{\infty}$ es una $\Zp$-extensión de $K$, denotamos $\mu(K_{\infty}/K)$ y $\lambda(K_{\infty}/K)$ (resp. $\mut(K_{\infty}/K)$ y $\lat(K_{\infty}/K)$) sus invariantes de Iwasawa clásicos (resp. logarítmicos). En el \cref{ex:clogvarios}, dimos evidencia de que cuando se habla de los grupos de clases clásicos o logarítmicos, cualquier cosa puede pasar. Sin embargo, como ya hemos visto, el teorema de Iwasawa (\cref{thm:iwasawa-ThmdIwa}) tiene su análogo logarítmico (\cref{thm:iwasawa-ThmdIwaLOG}). Será que ¿Los invariantes serán iguales o totalmente diferentes? Aún no hay una respuesta general a esta pregunta. En esta sección hacemos un compendio de los resultados conocidos al respecto. Recomendamos al lector que compare y se auxilie de la \cref{sec:invariantesmulam}.

Históricamente, los primeros resultados entre la relación de los invariantes fueron descubiertos por Jaulent. En su tesis de doctorado demostró que 
$$\mu(K^{c}/K)=\mut(K^{c}/K).$$
Es decir, el $p$-subgrupo de Sylow y el grupo de clases logarítmicas con respecto a $p$ crecen al mismo ritmo exponencial. No fue sino años más tarde que este resultado recobró una fuerza extraordinaria gracias a los trabajos de Ferrero y Washington (\cref{thm:ferrero-washington}).

\begin{thm}[Ferrero-Washington logarítmico] Sea $K$ una extensión abeliana de $\Q$, y sea $K^{c}/K$ su $\Zp$-extensión ciclotómica. Entonces
$$\mut(K^{c}/K) = 0.$$
\end{thm}

Cuando $K_{\infty}/K$ es una extensión no ciclotómica, nos inspiramos en el trabajo de Greenberg para dar una respuesta parcial (ver \cref{sec:invariantesmulam}).

Sea $\Delta(K)$ el conjunto de todas las $\Zp$ extensiones de $K$. Recordemos que la conjetura de Leopoldt (\cref{conj:Leopoldt}) afirma que este conjunto consta de una sola extensión si $K$ es totalmente real, de lo contrario es infinito y el compuesto es generado por $c+1$ $\Zp$-extensiones, donde $c$ es el número de encajes complejos de $K$. Resulta que $\Delta(K)$ contiene un subconjunto denso $\Delta^{0}(K)$ que consiste de las $\Zp$-extensiones en las cuales las plazas $\p\in \Pl_{K}$ arriba de $p$ son finitamente descompuestas. Es decir, los subgrupos de descomposición asociados tienen índice finito (\cite{MR332712}).  

En \cite{Villanueva18}, demostramos que si $K_{\infty}\in \Delta^{0}(K)$ entonces 
$$\mu(K_{\infty}/K)=\mut(K_{\infty}/K).$$
Es decir, en un conjunto denso del conjunto de las $\Zp$-extensiones se replica el comportamiento de los invariantes $\mu$ y $\mut$. Es conjeturado que en el resto de las extensiones debe de suceder un fenómeno similar, pero esto aún no ha sido demostrado.

Por otro lado, el comportamiento del invariante $\lambda$ es más errante. Como demostramos en \cite{Villanueva18} para un campo cuadrático imaginario $K=\Q(\sqrt{d})$ cuyo discriminante no es dividido por $p$, tenemos
$$\lambda(K^{c}/K) =   \left\lbrace\begin{array}{cl}
\lat(K^{c}/K)+1 & \text{ si } \left(\frac{d}{p}\right)=1,\\
\lat(K^{c}/K)   & \text{ si } \left(\frac{d}{p}\right)=-1.
      \end{array}\right.
$$

Además, en el caso de que $K_{\infty}/K$ no sea la $\Zp$-extensión ciclotómica y $K_{\infty}\in\Delta^{0}(K)$, demostramos que $\lambda$ y $\lat$ difieren en función de los factores ciclotómicos que aparecen en sus respectivos $\LL$-módulos noetherianos y de torsión. \\

Para decir más acerca de los invariantes de Iwasawa y sus relaciones en los dos contextos, tenemos que profundizar en las ideas de Greenberg y sumergirnos en el universo de las ideas topológicas de Kleine (ver \cite{Villanueva19}). 

El conjunto $\Delta(K)$ de las $\Zp$-extensiones de un campo de números $K$, admite una topología. Dados una $\Zp$-extensión y un $r\geq 0$, los conjuntos 
$$\Delta(K_{\infty},r):=\{K'_{\infty}\in\Delta(K)\,|\, [K_{\infty}\cap K'_{\infty}:K]\geq p^{r} \}$$
forman una base para la topología de Greenberg. El siguiente teorema es el análogo logarítmico al \cref{thm:Greenbound}.

\begin{thm} Sea $K_{\infty}$ una $\Zp$-extensión en $\Delta^{0}(K)$, es decir las plazas arriba de $p$ son finitamente descompuestas. Entonces con respecto a la topología de Greenberg:
\begin{itemize}
\item[(i)] El invariante $\mut$ es acotado en una vecindad de $K_{\infty}$.
\item[(ii)] Si $\mut(K_{\infty}/K)=0$, entonces en una vecindad de $K_{\infty}$ los invariantes $\mut$ son nulos y los invariantes $\lat$ acotados.
\end{itemize}
\end{thm}

Modificando la topología de Kleine al contexto logarítmico, obtenemos que los invariantes no son solamente acotados sino localmente máximos. 

Por último, no está demás decir que suponiendo la conjetura de Gross-Kuz'min, al igual que en el caso clásico (\cref{thm:baba-monky}), el invariante $\mut$ está acotado en $\Delta(K)$. A pesar de que la igualdad entre los invariantes $\mu(K_{\infty}/K)$ y $\mut(K_{\infty}/K)$, sólo está demostrada para $K_{\infty}\in\Delta^{0}(K)$.

\subsection{Perspectivas}

La aritmética logarítmica presenta preguntas nuevas en la teoría de números. Es una rama relativamente nueva con muchas preguntas que contestar, brevemente discutiremos algunas de estas. 

En el \cref{ex:clogvarios} calculamos el grupo de clases logarítmicas para los primeros dos campos $p$-ciclotómicos irregulares, es decir, cuyo grupo de clases tiene una $p$-componente no trivial. Los cálculos muestran que sus grupos de clases logarítmicas en $p$ tampoco son triviales. ¿Será que este es un fenómeno general? y de ser el caso ¿Cuál es su relación con los números de Bernoulli? Es decir, ¿El criterio de Kummer (\cref{thm:kummer-crit}) se puede expresar en términos logarítmicos? 

Es conjeturado que el producto de los grupos de clase logarítmicos $\prod_{p} \Clog{K}$ de un campo de números cuadrático real $K$ es finito. Si la conjetura es cierta, ¿Por qué se da este fenómeno? Además, ¿Será que este fenómeno se replica a otras familias de campos de números? El lector especializado, podrá coincidir que el problema tiene cierta similaridad con la conjetura de Tate-Shafarevich.

Con respecto a la teoría de Iwasawa clásica, sería interesante estudiar el comportamiento del grupo de clases logarítmico en extensiones de Lie $p$-ádicas, es decir extensiones de Galois $E/K$, con $\Gal(E/K)$ un grupo de Lie $p$-ádico. Por ejemplo, cuando el grupo de Galois de la torre $E/K$ es isomorfo a $\Zp\rtimes\Zp$. Además sería interesante, estudiar las relaciones de los invariantes de Iwasawa clásicos y logarítmicos más a fondo. Como lo mencionamos anteriormente, los invariantes $\mu$ y $\mut$ coinciden en un subconjunto denso de las $\Zp$-extensiones, pero ¿Qué pasa fuera de este conjunto denso? Por otro lado los invariantes $\lambda$ y $\lat$ difieren en función de ciertos polinomios ciclotómicos en el caso de que la extensión esté en $\Delta^{0}$, de nuevo la pregunta es ¿Qué pasa fuera de $\Delta^{0}$?. En particular trabajo en curso del segundo autor con Kleine, estudia fenómenos de ramificación mixta.

Finalmente, es de interés estudiar la conjetura principal en su contexto logarítmico. Dado el $\LL$-módulo $\widetilde{X}_{\infty}$ de naturaleza logarítmica, es decir que corresponde a la máxima extension abeliana logarítmicamente no ramificada de una $\Zp$-extensión $K_{\infty}/K$, ¿Existirán elementos $h$ y $\mu$, como en la conjetura principal, tal que exista un pseudo-isomorfismo de $\LL$-módulos
$$\LL/(h\mu)\rightarrow \widetilde{X}_{\infty}$$
y tal que $\mu$ satisfaga cierta fórmula de interpolación? En particular, esto representa tener una intuición y conocimiento profundo de la teoría de Iwasawa y la aritmética logarítmica. 

El lector podrá darse cuenta de que muchos de los avances de la teoría de números clásica que se presentarán en la siguiente sección pueden llevarse al contexto logarítmico y quizás arrojar relaciones sorprendentes.

\newpage
\section[Generalizaciones de la Conjetura Principal]{Generalizaciones de la Conjetura Principal: Curvas elípticas, Motivos y la Conjetura Equivariante de los Números de Tamagawa}\label{sec:generalisaciones-conjetura-princial}

En esta sección final contamos de qué manera los fenómenos alrededor de la Conjetura
Principal que conocimos en los capítulos precedentes se extienden a nuevos terrenos. Paso a paso vamos a ver como
reinterpretarlas y generalizarlas hasta alcanzar una ampliación sustancial. En este camino
no siempre podemos ser completamente exactos porque esto nos obligaría a introducir bastantes nociones nuevas que probablemente llenarían otro libro. Nuestro objetivo es que el lector pueda
tener una intuición del camino a seguir.

En (casi) toda esta sección sea $K_r=\Q(\mu_{p^r})$ para $r\in\Ncero\cup\{\infty\}$ y
$G=\Gal(K_\infty/\Q)$ como en el \cref{sec:mc}.

\subsection{Proemio}

La Conjetura Principal conecta una función \enquote{$L$} compleja -- la función zeta de
Riemann -- a un objeto aritmético -- los grupos de clases. Existen muchas funciones
complejas más que se llaman funciones $L$ \dots\ ¿Tendrán estas funciones también un par
$p$-ádico? ¿Estarán relacionadas con objetos aritméticos? En caso afirmativo
¿Cuáles son estos objetos aritméticos?

Una de las primeras generalizaciones propuestas es la formulación de una Conjetura
Principal para curvas elípticas hecha por Mazur en los años 1970. En este caso tenemos la función $L$ de
Hasse y Weil, que de hecho tiene un análogo $p$-ádico que veremos más adelante. En el lado
algebraico, el papel de los grupos de clases será jugado por grupos de Selmer de la curva
elíptica. Esta formulación, que además es paralela a la Conjetura Principal original de
Iwasawa, resultó ser correcta: fue demostrada en muchos casos por Skinner y Urban en 2003. Este proceso no se detuvo en las curvas elípticas, en realidad se podía formular una
Conjetura Principal para objetos mucho más generales: los denominados motivos, que vamos a
mencionar a grandes rasgos más adelante. En el lado algebraico, Greenberg, Bloch y Kato
definieron grupos de Selmer en gran generalidad alrededor de 1990. Por otro lado, en 1988
Coates y Perrin-Riou formularon conjeturas sobre la existencia de funciones $L$ $p$-ádicas
para ciertos motivos. Estos desarrollos llevaron a la formulación de una Conjetura Principal
para motivos hecha por Greenberg y a una versión mucho más general hecha por Fukaya y Kato en
2006. Sin embargo, de la mayoría de estas generalizaciones hasta ahora no tenemos
demostraciones.

En este capítulo primero vamos a estudiar grupos de Selmer, que son los protagonistas del
lado algebraico de la Teoría de Iwasawa, y explicaremos su conexión con objetos clásicos como
grupos de clases. Luego hablaremos sobre funciones $L$ complejas en gran generalidad. Como primer contacto con estos conceptos expondremos la aplicación a curvas elípticas, que son un objeto de gran
interés en la geometría aritmética, y veremos cómo formular una Conjetura Principal para
ellas. Después de esto discutiremos las funciones $L$ $p$-ádicas en esta situación porque
aquí entrarán las formas modulares en el cuadro. Finalmente esbozamos las generalizaciones de todo lo anterior y la manera en que sumergen en un océano aún más grande. 

Hay que tener un poco de cuidado con algunos encajes, por eso aclaramos esto inmediatamente.  Sea $K$ un campo de números
con un encaje en $\Qbar$, y recuerde que ya fijamos encajes
$\Qbar\hookrightarrow\overline\Q_\ell$ para cada primo $\ell$ y
$\Qbar\hookrightarrow\C$. Estas elecciones inducen varias más, como explicamos
ahora. Primero, ya mencionamos que vía restricción obtenemos inclusiones de grupos
$\G\R\hookrightarrow\GQ$ y $\G{\Q_\ell}\hookrightarrow\GQ$. Esto entonces fija un subgrupo
de inercia de $\GQ$ para cada primo $\ell$, que es el núcleo de la aplicación de
$\G{\Q_\ell}$ al grupo absoluto de Galois del campo residual $\F_\ell$. Lo denotamos como
$I_\ell$; su campo fijo es la máxima extensión no ramificada $\Q_\ell^\nr$. El encaje
$K\hookrightarrow\Qbar$ fija encajes $K\hookrightarrow\C$ y $K\hookrightarrow\overline\Q_\ell$, además para cada primo $\ell$ el núcleo de 
\[ \O_K\hookrightarrow\O_{\overline\Q_\ell}
\twoheadrightarrow\overline{\F}_{\ell} \] es un ideal primo de $\O_{K}$ y por lo tanto una plaza $\lambda\mid\ell$ de $K$. Los encajes
también fijan inclusiones de grupos $\Gal(\Qbar/K)\hookrightarrow\GQ$ y
$\Gal(\overline\Q_\ell/K_\lambda)\hookrightarrow\G{\Q_\ell}$. Dicho esto, a partir de
ahora siempre asumiremos que cada campo de números está encajado en $\Qbar$, y recordaremos
estas elecciones.

\subsection{Representaciones de Galois y grupos de Selmer}

Los grupos de Selmer fueron originalmente introducidos para curvas elípticas (o
variedades abelianas), pero ahora existe una definición bastante general. Al especializar esta definición con los parámetros correctos recobramos la definición de los grupos de clases, esto motiva naturalmente el uso de
grupos de Selmer para las generalizaciones de la Conjetura Principal.

A partir de aquí suponemos algunos conocimientos extras, por ejemplo asumimos que el lector
conoce la teoría de cohomología de Galois, es decir la cohomología de cocadenas continuas
de grupos de Galois con coeficientes en módulos topológicos, como es explicada por ejemplo
en \cite[Cap.\ I, II, esp.\ II.7]{NSW}. Si $K$ es un campo entonces escribimos la
cohomología de su grupo absoluto de Galois $\G K$ como $\HL^*(K,-)$ en lugar de
$\HL^*(\G K,-)$. Si $K$ es una extensión finita de $\Q$ o de $\Q_\ell$
entonces denotamos $K^\nr$ su máxima extensión no ramificada. Es decir, en el caso de una
extensión $K_\lambda/\Q_\ell$ la cohomología $\HL^*(K_\lambda^\nr,-)$ es la cohomología del
subgrupo de inercia $I_\lambda$ en $\G{K_\lambda}$.

\begin{defi}
  Sea $V$ un espacio vectorial de dimensión finita sobre algún campo topológico $L$ con una
  acción continua de $\GQ$. Entonces $V$ se llama una \emph{representación de Galois}\index[def]{representación!de Galois}
  $L$-lineal (o con coeficientes en $L$). Después de escoger una base, esto puede ser visto
  como un morfismo continuo de grupos \[ \rho\colon\GQ\rightarrow\GL_n(L) \] con $n=\dim V$.

  Sea $L/\Qp$ una extension finita con anillo de enteros $\O$. Un $\O$-submódulo $T$ de $V$
  se llama \define{retículo estable} si $T$ es compacto, estable bajo la acción de $\GQ$ y
  $V=L\tensor_\O T$. La representación $V$ se llama \emph{no ramificada}\index[def]{representación!no ramificada} en una plaza $v$
  de $K$ si el subgrupo de inercia $I_v$ actúa trivialmente en $V$. Se llama \emph{no
    ramificada en casi todos lados}\index[def]{representación!no
    ramificada en casi todos lados} si es ramificada solo en una cantidad finita de plazas
  de $K$.
\end{defi}

Por continuidad, cada representación de Galois con coeficientes en una extensión finita de
$\Qp$ contiene un retículo estable (\cref{ejer:reticulo-estable})

En esta sección en la mayoría de los casos tendremos $L=\Qp$. Además, normalmente el retículo
es canónicamente dado -- de hecho, primero conocemos el retículo $T$ y
luego definimos $V=L\tensor_\O T$.

\begin{remark}
  Vía restricción obtenemos acciones de $\G{\Q_\ell}$ y de $\G\R$ en $V$ y $T$.
\end{remark}

\begin{ex}
  Un ejemplo muy importante es el siguiente. Sea
  \[ \Zp(1)=\varprojlim_{r\in\Nuno}\mu_{p^r} \] como en la \cref{defi:zpuno}, que es un
  $\Zp$-módulo compacto que es no canónicamente isomorfo a $\Zp$ con una acción continua de
  $\GQ$. Definimos $\Qp(1):=\Qp\tensor_\Zp\Zp(1)$. Entonces $V=\Qp(1)$ y $T=\Zp(1)$ es un
  ejemplo de una representación de Galois $\Qp$-lineal y un retículo estable.  Esta
  representación es ramificada solo en $p$, así que es no ramificada en casi todos lados.
\end{ex}

Ahora introducimos los grupos de Selmer. La idea es que queremos un grupo que mide
\enquote{la diferencia entre el comportamiento local y global de un objeto aritmético}. Por
ejemplo, si $K$ es un campo de números entonces cada ideal de $\O_K$ se vuelve un ideal
principal en cada completación $\O_{K_v}$ para los primos $v$ de $K$, pero en general no
debe ser un ideal principal en $\O_K$. La diferencia entre la situación local y global es
medida por el grupo de clases (que esencialmente es un grupo de Selmer, véase la
\cref{prop:selmer-class-group} abajo). Los objetos aritméticos generales para cuales
introducimos grupos de Selmer son representaciones de Galois, y usamos la cohomología de
Galois local y global para definirlos.

\begin{defi}\label{defi:selmer-general}
  Sea $L/\Q$ una extensión finita con anillo de enteros $\O$. Sea $V$ una representación de
  Galois con coeficientes en $L$ que es no ramificada en casi todos lados y $T$ un retículo
  estable. Para un campo de números $K$, encajado en $\Qbar$ de tal manera que su grupo absoluto de
  Galois $\G K$ sea un subgrupo de $\GQ$, definimos los siguientes grupos.
  \begin{enumerate}
  \item Sea $v$ una plaza de $K$. Definimos un subgrupo del grupo $\HL^1(K_v,V)$ de la
    cohomología local de Galois como
    \[ \Hf^1(K_v,V) =
      \begin{cases}
        \ker(\HL^1(K_v,V)\rightarrow\HL^1(K^\nr_v,V)), & v\nmid p\infty,\\
        \ker(\HL^1(K_v,V)\rightarrow\HL^1(K_v,V\tensor_\Qp\Bcris)), & v\mid p,\\
        0, & v\mid\infty
      \end{cases}
    \]
    los mapeos siendo la restricción.  Aquí, $\Bcris$ es un cierto anillo sobre cual no
    podemos decir mucho.\footnote{El anillo $\Bcris$ fue definido por Fontaine y es uno de
      los \define{anillos de períodos $p$-ádicos} de la teoría de Hodge $p$-ádica, que no
      explicamos aquí (remitimos a \cite{BrinonConradPAdicHodgeTheory} para esto). Es una
      $\Qp$-álgebra topológica completa con una acción continua de $\GQp$ y algunas
      estructuras más, de manera que $V\tensor_\Qp\Bcris$ tiene una acción de $\GQp$
      diagonalmente.}  Usando esto definimos un subgrupo de $\HL^1(K_v,V/T)$ para cada plaza
    $v$ como
    \[ \Hf^1(K_v,V/T)=\im(\Hf^1(K_v,V)\rightarrow\HL^1(K_v,V/T)). \]
  \item Definimos un subgrupo del grupo $\HL^1(K,V/T)$ de la cohomología global de Galois como
    \begin{align*}
      \Sel(K,V/T)&=\{c\in\HL^1(K,V/T) \mid \forall v\text{ plaza}\colon
                    \mathrm{res}_v(c)\in\Hf^1(K_v,V/T)\}\\
                  &= \ker\Big(\HL^1(K,V/T)\longrightarrow\prod_{v\text{
                    plaza}}\frac{\HL^1(K_v,V/T)}{\Hf^1(K_v,V/T)}\Big)
    \end{align*}
    con $\mathrm{res}_v\colon\HL^1(K,V/T)\rightarrow\HL^1(K_v,V/T)$ el mapeo de restricción.
    Esto es un $\O$-módulo discreto que se llama el \define{grupo de Selmer} de $V/T$ (sobre
    $K$).
  \end{enumerate}
  Más precisamente estos grupos de Selmer a veces se llaman \define{grupos de Selmer de
    Bloch y Kato} porque ellos los introdujeron en \cite{MR1086888}; también existen otras
  variantes con diferentes subgrupos locales en lugar de los $\Hf^1$, pero los de Bloch y
  Kato se comportan bien y son los que se usan en las generalizaciones de la Conjetura
  Principal.
\end{defi}

 Notemos que aunque los llamamos
\enquote{grupos de Selmer}, realmente son $\O$-módulos.

\begin{remark}\label{nota:selmer-warum-so-komisch-def}
  En la definición de arriba, se podría preguntar por qué no definimos
  \[ \tilde\HL_{\mathrm f}^1(K_v,V/T)=\ker(\HL^1(K_v,V/T)\rightarrow\HL^1(K_v^\nr,V/T)) \]
  en lugar de
  \[ \Hf^1(K_v,V/T)=\im(\Hf^1(K_v,V)\rightarrow\HL^1(K_v,V/T)) \] para $v\nmid
  p\infty$. Estos dos grupos en general son diferentes. El segundo es un subgrupo del
  primero y siempre es divisible porque $\Hf^1(K_v,V)$ es un $L$-espacio vectorial. De hecho
  se cumple que
  \[ \Hf^1(K_v,V/T) = \tilde\HL_{\mathrm f}^1(K_v,V/T)_{\text{div}}, \] donde escribimos
  $(\cdot)_{\text{div}}$ como los elementos $p$-divisibles en un $\Zp$-módulo abeliano
  \cite[Lem.\ I.3.5 (i)]{MR1749177}. En particular, si $\tilde\HL_{\mathrm f}^1(K_v,V/T)$ es
  divisible entonces sí tenemos la igualdad
  \[  \Hf^1(K_v,V/T)=\ker(\HL^1(K_v,V/T)\rightarrow\HL^1(K_v^\nr,V/T)). \]
  En este caso podemos describir el grupo de Selmer como
  \begin{multline}\label{eqn:selmer-si-divisible}
    \Sel(K,V/T)= \ker\bigg(\HL^1(K,V/T)\longrightarrow \\
    \prod_{v\nmid
      p\infty}{\HL^1(I_v,V/T)}\times\prod_{v\mid\infty}{\HL^1(K_v,V/T)}\times\prod_{v\mid
      p}\frac{\HL^1(K_v,V/T)}{\Hf^1(K_v,V/T)}\bigg).
  \end{multline}
\end{remark}

Para explicar las relaciones de todo esto con lo anterior, empecemos con calcular el grupo
de Selmer en el caso más simple posible. Sea $\O=\Zp$, $L=\Qp$, $T=\Zp$ y $V=\Qp$ con la
acción trivial de $\GQ$.

\begin{prop}\label{prop:selmer-class-group}
  Sea $K$ un campo de números. Entonces tenemos un isomorfismo canónico
  \[ \Sel(K,\Qp/\Zp) \isom \Cl(K)(p)^\vee ,\]
  donde $\Cl(K)(p)$ es la $p$-parte del grupo de clases de $K$ y $(-)^\vee$ denota el dual
  de Pontryagin.
\end{prop}
\begin{proof}
  Usamos el \cref{thm:campo-de-hilbert} que dice que el grupo de clases $\Cl(K)$ es
  canónicamente isomorfo al grupo de Galois $\Gal(H/K)$ del campo de clases de Hilbert $H$
  de $K$, que es la extensión máxima abeliana no ramificada. Por definición tenemos una
  sucesión exacta de grupos profinitos 
  \begin{equation*}
    \bigoplus_{v\text{ plaza}} I_v\rightarrow\G K^\ab\rightarrow\Gal(H/K)\rightarrow 0
  \end{equation*}
  donde $I_v$ es el grupo de inercia en la plaza $v$ de $K$ (si $v$ es una plaza
  arquimediana entonces $I_v$ es de orden $2$ generado por la conjugación compleja en esta
  plaza).

  Aplicamos el funtor $\Hom_{\Zp}(-,\Qp/\Zp)$ a esta sucesión y usamos que esto es
  lo mismo que $\HL^1(-,\Qp/\Zp)$ porque la acción de todos los grupos en $\Qp/\Zp$ es
  trivial. Además usamos que para cada abeliano grupo finito $A$ tenemos que
  $\Hom(A,\Qp/\Zp)=\Hom(A(p),\Qp/\Zp)$.
  Así obtenemos
  \begin{equation}\label{eqn:sucesion-cl-selmer}\tag{$*$}
    0\rightarrow\Gal(H/K)(p)^\vee\rightarrow\HL^1(K,\Qp/\Zp)\rightarrow\prod_v\HL^1(I_v,\Qp/\Zp). 
  \end{equation}
  
  Queremos usar la descripción en \eqref{eqn:selmer-si-divisible}. Para esto tenemos que
  verificar que $\ker(\HL^1(K_v,\Qp/\Zp)\rightarrow\HL^1(K_v^\nr,\Qp/\Zp))$ es divisible
  para los primos $v\nmid p$. Pero como $\HL^1(-,\Qp/\Zp)=\Hom_{\Zp}(-,\Qp/\Zp)$, los
  elementos en este núcleo son los homomorfismos de $\G{K_v}$ a $\Qp/\Zp$ que son triviales en
  $I_v$, es decir son los que se factorizan a través de $\G{K_v}/I_v\cong\widehat\Z$. Pero
  $\Hom_{\Zp}(\widehat\Z,\Qp/\Zp)\cong\Qp/\Zp$ porque cada tal homomorfismo es únicamente
  determinado por la imagen de $1\in\widehat\Z$. Este grupo es claramente divisible.
  
  Es decir, la sucesión \eqref{eqn:sucesion-cl-selmer} es casi la misma que
  \eqref{eqn:selmer-si-divisible}, las diferencias están en las plazas $v\mid p$ y
  $v\mid \infty$.  Para las plazas $v\mid\infty$ observamos que en el producto a la derecha
  en ambas sucesiones los grupos son triviales porque $I_v$ es de orden $2$ en este caso y
  $p\neq2$, y $\G{K_v}$ es de orden $1$ o $2$. Falta ver que para las plazas $v\mid p$
  tenemos
  \begin{equation*}
    \ker(\HL^1(K_v,\Qp)\rightarrow\HL^1(I_v,\Qp)) =
    \ker(\HL^1(K_v,\Qp)\rightarrow\HL^1(K_v,\Bcris)).
  \end{equation*}
  Una demostración de esto se encuentra en \cite[Ex.\ 3.9]{MR1086888} (teniendo en cuenta el
  \cref{ejer:unramified-morphisms}); la omitimos aquí porque ni siquiera explicamos que es
  $\Bcris$.
\end{proof}

\begin{defi}\label{defi:x-selmer}
  Sea $(K_r)_r$ la torre infinita de campos que usamos también en el \cref{sec:mc}, es decir
  $K_r=\Q(\mu_{p^r})$ para $r\in\Ncero\cup \{\infty\}$, y sea $G=\lim_r\Gal(K_r/\Q)$ su grupo de
  Galois. Fijamos encajes compatibles de todos los $K_r$ en $\Qbar$. Además sea $V$ una
  representación de Galois con coeficientes en $L$ y $T$ un retículo estable. Entonces definimos
  \[ \mathrm X(V/T) = \varprojlim_{r\in\Nuno}\Sel(K_r,V/T)^\vee =
    (\varinjlim_{r\in\Nuno}\Sel(K_r,V/T))^\vee \] donde los mapeos
  $\Sel(K_r,V/T)\rightarrow\Sel(K_{r+1},V/T)$ son las restricciones (véase el
  \cref{ejer:selmer-restriccion}).  Esto es un $\LL(G)$-módulo compacto (aquí $\LL(G)$
  es el álgebra de Iwasawa con coeficientes en $\O$, los enteros de $L$).
\end{defi}

De la \cref{prop:selmer-class-group} obtenemos inmediatamente:

\begin{cor}\label{cor:x-es-x}
  $\mathrm X(\Qp/\Zp)=X_\infty$ es el módulo de la
  \cref{sec:conjetura-principal}.
\end{cor}
\begin{proof}
  Lo que falta verificar es que los mapeos con respeto a los cuales tomamos los límites son los
  mismos, es decir, que el diagrama
  \begin{equation*}
    \begin{tikzcd}[column sep=0pt, row sep=0pt]
      \HL^1(K_{r+1},\Qp/\Zp) & \supseteq & \Sel(K_{r+1},\Qp/\Zp) & \isom & \Cl(K_{r+1})(p)^\vee \\
      =\Hom(\G{K_{r+1}},\Qp/\Zp) \arrow[dddd, "\mathrm{res}"] & & & & \isom\Hom(\Gal(H_{r+1}/K_{r+1}),\Qp/\Zp)
      \arrow[dddd, "\mathrm{res}"] \\
      & \; \\
      & \; \\
      & \; \\
      \Hom(\G{K_{r}},\Qp/\Zp) & & & & \Hom(\Gal(H_{r}/K_{r}),\Qp/\Zp)\isom\\
      =\HL^1(K_{r},\Qp/\Zp) & \supseteq & \Sel(K_{r},\Qp/\Zp) & \isom & \Cl(K_{r})(p)^\vee \\
    \end{tikzcd}
  \end{equation*}
  conmuta para $r$ suficientemente grande (donde $H_r$ es el campo de clases de Hilbert de
  $K_r$). Lo dejamos como ejercicio.
\end{proof}

También podemos obtener el otro módulo $Y_\infty$ de la \cref{sec:conjetura-principal} como
un módulo $\mathrm X(-)$ de esta forma. En la siguiente proposición escribimos $\Qp/\Zp(1)$
para $\Qp(1)/\Zp(1)$.

\begin{prop}\label{prop:y-infty-selmer}
  Existe un isomorfismo canónico
  \[ \mathrm X(\Qp/\Zp(1))=Y_\infty. \]
\end{prop}
\begin{proof}
  Para hacer todo más concreto, empecemos con describir el grupo de cohomología
  $\HL^1(F,\mu_m(F))$ para cualquier campo $F$ y cada $m\in\Nuno$: Existe un isomorfismo
  canónico
  \begin{equation*}
    \HL^1(F,\mu_m(\overline F))\isom F^\times/(F^\times)^m=F^\times\tensor_\Z\Z/p^m\Z,
  \end{equation*}
  véase \cite[p.\ 344, antes de (6.2.2)]{NSW}. 
  Porque tomar cohomología es compatible con límites directos, esto implica que
  \begin{equation*}
    \HL^1(F,\Qp/\Zp(1))\isom\varinjlim_{m\ge1} F^\times/(F^\times)^m=F^\times\tensor_\Z\Qp/\Zp
  \end{equation*}
  si $F$ es un campo de números o una completación de tal. Sea ahora $K$ un campo de números
  y $v$ una plaza de $K$. Si $v$ es arquimediana entonces $\Hf^1(K_v,\Qp/\Zp(1))=0$ por
  definición. Si $v\nmid p$ entonces también tenemos que $\Hf^1(K_v,\Qp/\Zp(1))=0$; los
  argumentos para ver esto los esbozamos en el \cref{ejer:hf-kv-unram}. El caso
  interesante es entonces el caso $v\mid p$: en este caso tenemos que 
  \begin{equation*}
    \Hf^1(K_v,\Qp/\Zp(1))\isom\varinjlim_{m\ge1} \O_v^\times/(\O_v^\times)^m=\O_v^\times\tensor_\Z\Qp/\Zp
  \end{equation*}
  según \cite[Ex.\ 3.9]{MR1086888}. Es decir, el grupo de Selmer lo podemos describir como
  \begin{equation*}
    \Sel(K,\Qp/\Zp(1))\isom\{x\in K^\times\tensor\Qp/\Zp : \forall v\nmid p\colon
    \operatorname{res}_v(x)=0; \forall \mathfrak p\mid p\colon \operatorname{res}_{\mathfrak
      p}(x)\in\O_{\mathfrak p}^\times\tensor\Qp/\Zp \}
  \end{equation*}
  donde
  $\operatorname{res}_v\colon K^\times\tensor\Qp/\Zp\rightarrow K_v^\times\tensor\Qp/\Zp$ es
  la aplicación canónica. Si denotamos
  $\operatorname{res}_{v,m}\colon K^\times\tensor\Z/p^m\Z\rightarrow K_v^\times\tensor\Z/p^m\Z$
  y definimos para $m\ge 1$
  \begin{equation*}
    S_m(K)=\{x\in K^\times\tensor\Z/p^m\Z : \forall v\nmid p\colon
    \operatorname{res}_{v,m}(x)=0; \forall \mathfrak p\mid p\colon \operatorname{res}_{\mathfrak
      p,m}(x)\in\O_{\mathfrak p}^\times\tensor\Z/p^m\Z \}
  \end{equation*}
  entonces es fácil ver que $\Sel(K,\Qp/\Zp(1))\isom\varinjlim_{m\ge1}S_m(K)$.

  Tenemos que comparar esto con el módulo $Y_\infty$. Por definición,
  $Y_\infty=\Gal(M_\infty/K_\infty)$. Si escribimos
  \begin{equation*}
     D_r=\{\alpha\in K_r^\times : (\alpha)=\mathfrak a^{p^r} \text{ para un ideal fraccional
    }\mathfrak a\text{ primo a }\mathfrak p_r \}
  \end{equation*}
  para $r\in\Nuno$ (dónde $\mathfrak p_r$ es el único ideal primo de $\O_{K_r}$ arriba de
  $p$) como en la \cref{prop:m-infty-explicito} entonces esta implica que
  $Y_\infty=\varprojlim_{r\ge1}\Gal(N_r/K_\infty)$ con $N_r=K_\infty(\sqrt[p^r]{D_r})$. La
  extensión $N_r/K_\infty$ es una extensión de Kummer, y el resultado principal de la teoría
  de Kummer permite describir el dual de Pontryagin de su grupo de
  Galois: la versión citada en la demostración del \cref{thm:kummer} dice que
  \begin{equation*}
    \Gal(N_r/K_\infty)^\vee=D_r/(K_r^\times)^{p^r}.
  \end{equation*}
      
  Se puede verificar que la aplicación canónica
  \begin{equation*}
    \varinjlim_{r\ge1}S_r(K_r)\rightarrow
    \varinjlim_{r\ge1}\varinjlim_{m\ge1}S_m(K_r)\isom\varinjlim_{r\ge1}\Sel(K_r,\Qp/\Zp(1))
  \end{equation*}
  es un isomorfismo. Con esto, lo único que falta ver es que
  $D_r/(K_r^\times)^{p^r}=S_r(K_r)$ como subgrupos de $K_r^\times/(K_r^\times)^{p^r}$.  Esto
  resulta de las definiciones de estos subgrupos y lo dejamos como el
  \cref{ejer:detalles-y-selmer}.
\end{proof}

Estas observaciones ya insinúan cómo la Conjetura Principal podría ser generalizada. Sin
embargo, quedan muchas preguntas: ¿Para cuáles tipos de representaciones de Galois podemos
esperar una generalización? ¿Qué tendría que ser la función $L$ $p$-ádica? Y antes que nada,
¿cómo definimos una función $L$ compleja? Vamos a indicar las respuestas a estas preguntas
en las siguientes secciones. 

\ejercicios

\begin{ejer}\label{ejer:reticulo-estable}
  Demuestre que cada representación de Galois con coeficientes en una extensión finita de
  $\Qp$ contiene un retículo estable. Para esto tome cualquier retículo, no necesariamente
  estable, y considere sus trasladados bajo la acción de $\GQ$. Use la continuidad de la
  representación y el hecho de que $\O$ es compacto.
\end{ejer}

\begin{ejer}\label{ejer:unramified-morphisms}
  Demuestre que para cada plaza $v\mid p$ de un campo de números $K$ el espacio vectorial
  \[ \ker(\HL^1(K_v,\Qp)\rightarrow\HL^1(I_v,\Qp)) \]
  tiene dimensión $1$ sobre $\Qp$. Para esto use que $\HL^1(-,\Qp)=\Hom(-,\Qp)$ (ya que la
  acción en $\Qp$ es trivial) y la teoría local de campos de clases para describir estos
  homomorfismos.
\end{ejer}

\begin{ejer}\label{ejer:selmer-restriccion}
  Sea $V$ una representación de Galois con coeficientes en una extensión finita $L/\Qp$ y
  $T$ un retículo estable.  Demuestre que si $K\subseteq K'$ son campos de números entonces
  el mapeo de restricción $\HL^1(K,V/T)\rightarrow\HL^1(K',V/T)$ envía $\Sel(K,V/T)$ a
  $\Sel(K',V/T)$.
\end{ejer}

\begin{ejer}
  Verifique que el diagrama en la demostración del \cref{cor:x-es-x} es conmutativo.
  Use la conmutatividad del diagrama \eqref{eqn:diagrama-x-c} para esto.
\end{ejer}

\begin{ejer}\label{ejer:hf-kv-unram}
  Sea $K$ un campo de números y  $v\nmid p$ una plaza. Escribimos $K_v$ para la completación
  y $K_v^\nr$ para la extensión máxima no ramificada de $K_v$.
  \begin{enumerate}
  \item Verifique (o busque una referencia) que
    \[ \HL^1(F,\Qp(1))=\varprojlim_{r\ge1}F^\times/(F^\times)^{p^r}\tensor_\Zp\Qp \]
    para cualquier campo $F$. Concluya que \[
      \Hf^1(K_v,\Qp(1))=\varprojlim_{r\ge1}(K_v^\nr)^\times/((K_v^\nr)^\times)^{p^r}\tensor_\Zp\Qp. \]
  \item Verifique que la extension $K(\sqrt[p^r]{\alpha})/K$ es no ramificada en $v$ para
    cualquier $\alpha\in K^\times$ y $r\in\Nuno$.
  \item Concluya que $\Hf^1(K_v,\Qp(1))=0$ y $\Hf^1(K_v,\Qp/\Zp(1))=0$.
  \end{enumerate}
\end{ejer}



\begin{ejer}\label{ejer:detalles-y-selmer}
  Usamos la notación de la demostración de la \cref{prop:y-infty-selmer}, es decir
  \begin{align*}
    S_m(K_r)&=\{x\in K_r^\times\tensor\Z/p^m\Z : \forall v\nmid p\colon
    \operatorname{res}_{v,m}(x)=0; \forall \mathfrak p\mid p\colon \operatorname{res}_{\mathfrak
      p,m}(x)\in\O_{\mathfrak p_r}^\times\tensor\Z/p^m\Z \},\\
    D_r&=\{\alpha\in K_r^\times : (\alpha)=\mathfrak a^{p^r} \text{ para un ideal fraccional
    }\mathfrak a\text{ primo a }\mathfrak p_r \}
  \end{align*}
  para $m,r\in\Nuno$, dónde $\mathfrak p_r$ es el único ideal primo de $\O_{K_r}$ arriba de $p$.
  \begin{enumerate}
  \item Demuestre que la aplicación canónica
    \[ \varinjlim_{r\ge1}S_r(K_r)\rightarrow
      \varinjlim_{r\ge1}\varinjlim_{m\ge1}S_m(K_r)  \]
    es un isomorfismo.
  \item Demuestre que los subgrupos $D_r/(K_r^\times)^{p^r}$ y $S_r(K_r)$ de
    $K_r^\times/(K_r^\times)^{p^r}$ coinciden.
  \end{enumerate}
\end{ejer}

\subsection{Funciones $L$ para sistemas compatibles de representaciones}

Recordemos algunos aspectos de las funciones $L$ complejas. Las que conocimos en la \cref{sec:proemio-l} tenían
un producto de Euler, es decir un producto de la forma
\[ L(s)= \prod_{\ell\text{ primo}}P_\ell(\ell^{-s})^{-1} \quad (s\in\C,\ \Re s\gg0) \]
donde $P_\ell\in\Qbar[T]$ es un polinomio: en el caso de la función zeta de Riemann
tenemos $P_\ell=1-T$ para cada primo $\ell$ y en el caso de la función $L$ de un carácter de
Dirichlet $\chi$ tenemos $P_\ell=1-\chi(\ell)T\in\Q(\chi)[T]$ -- véase la
\cref{defi:l-chi}. Notemos que la función zeta de Riemann es un caso especial de una función
$L$ de Dirichlet, es decir, la del carácter trivial. El camino para generalizar la conexión de
los polinomios $P_\ell$ y el carácter $\chi$ se aclara cuando vemos el carácter como una
representación de Galois.

Para explicar esto necesitamos usar los \define{elementos de Frobenius} en $\GQ$. Si
$\ell$ es un primo entonces tenemos el subgrupo (gracias a los encajes que fijamos)
$\G{\Q_\ell}\subseteq\GQ$ que tiene una sobreyección al grupo de Galois absoluto del campo
residual $\F_\ell$ de $\Q_\ell$. Este grupo de Galois $\G{\F_\ell}$ es canónicamente
isomorfo a $\widehat\Z$ con $1\in\widehat\Z$ correspondiendo al automorfismo Frobenius
(\cref{ejer:gf-finito}). Como $\G{\Q_\ell}\twoheadrightarrow\G{\F_\ell}$ es sobreyectivo
podemos escoger un levantamiento que llamamos $\Frob_\ell\in\G{\Q_\ell}\subseteq\GQ$. Este
elemento no es único, pero está bien definido salvo multiplicación por el grupo de inercia
$I_\ell$, que es el núcleo de la aplicación $\G{\Q_\ell}\rightarrow\G{\F_\ell}$; esto será
suficiente para lo que queremos hacer. A partir de ahora fijemos elementos de
Frobenius $\Frob_\ell\in\GQ$ para cada primo $\ell$.

Ahora tomamos un carácter de Dirichlet $\chi$ de conductor $N\in\Nuno$ y lo vemos como una
representación de Galois $\C$-lineal vía
\begin{equation}
  \label{eqn:chi-gal}
  \GQ\twoheadrightarrow\Gal(\Q(\mu_N)/\Q)\isomarrow(\Z/N\Z)^\times\labeledarrow{\chi}
  \overline\Q^\times\subset\C^\times
\end{equation}
(llamamos este mapeo $\GQ\rightarrow\C^\times$ también
$\chi$), que define una acción de $\GQ$ en el $\C$-espacio vectorial $V=\C$ en que
$g\in\GQ$ actúa por multiplicación con
$\chi(g)$.
Esta representación es ramificada en un primo $\ell$ si y solo si $\ell\mid
N$, es decir si y solo si $\chi(\ell)=0$.
La acción de un elemento de Frobenius $\Frob_\ell$ en $V$ en general no está bien definida
porque $\Frob_\ell$ solo está bien definido módulo elementos de $I_\ell$ y este último podría
actuar no trivialmente. Pero si escribimos $V^{I_\ell}$ como el subespacio de $V$ donde
$I_\ell$ actúa trivialmente, entonces la acción de $\Frob_\ell$ en $V^{I_\ell}$ sí está bien
definida; en particular, su polinomio característico lo es. Si definimos $P_\ell(\chi,T)$
como el polinomio\footnote{Aquí y en lo siguiente, la notación $\det(\varphi, V)$ significa
  el determinante de un endomorfismo $\varphi$ de un espacio vectorial $V$.}
\[ P_\ell(\chi,T):= \det(1-\chi(\Frob_\ell)T, V^{I_\ell} ) \] entonces de hecho tenemos
\[ P_\ell(\chi,T)=1-\chi(\ell)T, \] es decir ¡Reconstruimos los polinomios que definen el
producto de Euler a partir de la representación de Galois! Le sugerimos urgentemente al lector
que verifique todas estas afirmaciones.

Las representaciones que estudiamos al principio de esta sección tuvieron coeficientes en
$\Qp$ o una extensión finita $L$. Si usamos la misma fórmula como arriba para definir
polinomios $P_\ell$ para ellas tenemos un problema -- los coeficientes estarán en $L$, que no
podemos encajar en $\C$ para definir una función $L$ compleja. Sin embargo, para las
representaciones $\Qp$ (con la acción trivial) y $\Qp(1)$ que estudiamos antes este problema
no aparece: por supuesto, para la representación trivial tenemos
\[ P_\ell(\Qp,T):= \det(1-T, \Qp^{I_\ell})=1-T \]
para cada primo $\ell$ y en el otro caso tenemos
\[ P_\ell(\Qp(1),T):= \det(1-T, \Qp(1)^{I_\ell})=1-\ell T \] al menos para los primos
$\ell\neq p$, mientras para $\ell=p$ obtenemos $P_p(\Qp(1),T)=1$. ¡Otra vez el lector
debería verificar estas afirmaciones! Es decir, estos polinomios de hecho tienen
coeficientes en $\Qbar\subseteq\Qpbar$, que encajamos en $\C$. Por supuesto hay
representaciones para las cuales esto no es verdad, por eso sólo vamos a usar
representaciones donde los coeficientes de los polinomios que obtengamos estén en $\Qbar$
(véase la siguiente definición).

Si ahora definimos una función $L$ para las representaciones $\Qp$ y $\Qp(1)$ usando la
fórmula del producto de Euler obtenemos la función zeta de Riemann para $\Qp$ y obtenemos
\[ (1-p^{-(s+1)})\zeta(s+1) \] para $\Qp(1)$, es decir la función zeta trasladada por $1$ y
con un factor de Euler faltante. Esto es un poco raro \dots\ queremos tener también el
factor de Euler en $p$. Notemos que también existen las representaciones $\Q_q(1)$ para todos
los otros primos $q\neq p$, y si usamos éstas en lugar de $\Qp(1)$ para definir el polinomio
$P_\ell$ entonces obtenemos lo mismo para $\ell\neq p,q$, pero para $\ell=p$ obtenemos el
polinomio que corresponde al factor de Euler que hacía falta. Es decir, las
representaciones $\Q_q(1)$ para todos los primos $q$ son compatibles de alguna manera, y
para definir $P_\ell$ podemos usar cualquiera de ellas salvo $\Q_\ell(1)$. Esto nos lleva a la siguiente definición.

\begin{defi}\label{defi:sistema-compatible}
  Un \define{sistema compatible de representaciones de Galois} es una colección $V=(V_q)_q$ de
  representaciones de $\GQ$, donde $V_q$ es un espacio vectorial sobre $\Q_q$, para cada
  primo $q$, tal que las siguientes condiciones sean ciertas.
  \begin{enumerate}
  \item Existe un conjunto finito $S$ de primos tal que cada $V_q$ no es ramificado fuera de
    $S\cup\{q\}$. 
  \item\label{defi:sistema-compatible:compatibilidad} Para cada primo $\ell$ y todo
    primo $q\neq\ell$ el polinomio
    \[ P_\ell(V,T):= \det(1-\Frob_\ell T, V_q^{I_\ell} ) \]
    que a priori tiene coeficientes en $\Q_q$ de hecho tiene coeficientes en $\Q$ y no
    depende de $q$.
  \end{enumerate}

  Un poco más general, si $K$ es un campo de números entonces definimos un sistema
  compatible de representaciones de Galois con coeficientes en $K$ como es una colección de
  representaciones 
  $V=(V_{\mathfrak q})_{\mathfrak q}$ de $\GQ$ indexada por todos los primos de $K$ , donde $V_{\mathfrak q}$ es un espacio vectorial sobre
  $K_{\mathfrak q}$ con las condiciones análogas (en este caso los polinomios $P_\ell(V,T)$
  deben tener coeficientes en $K$).

  Si $V$ es un sistema compatible de representaciones de Galois entonces definimos su
  función $L$ como
  \[ L(V,s) := \prod_{\ell\text{ primo}}P_\ell(V,\ell^{-s})^{-1}. \]
  De momento, esto sólo es una expresión formal, porque todavía no sabemos nada sobre
  convergencia. 
\end{defi}

Queremos dar una heurística por qué esta definición es interesante. El teorema de Brauer y
Nesbitt dice que dos representaciones de un grupo profinito son iguales (salvo a
semisimplificación, que no vamos a explicar aquí) si y solo si los polinomios
característicos de los imágenes de todos elementos del grupo bajo las dos representaciones
son iguales. Además, el teorema de densidad de Chebotarev implica que los elementos de
Frobenius son densos en $\GQ$, así que por continuidad la igualdad de los polinomios
característicos de ellos ya es suficiente. Es decir, en la formula que define la función
$L(V,s)$ multiplicamos expresiones que juntas determinan $V$ únicamente (salvo a
semisimplificación). Véase \cite[Thm. 30.16]{MR2215618}\footnote{El teorema de Brauer y
  Nesbitt allá es formulado para un grupo finito, pero examinando la demostración se puede
  ver que la demostración sigue funcionando para un grupo profinito.} y \cite[Thm.\
13.4]{MR1697859}.

\begin{ex}
  \begin{enumerate}
  \item Poniendo $V_q=\Q_q$ con la acción trivial de $\GQ$ para cada primo $q$ nos da un
    sistema compatible de representaciones de Galois. Su función $L$ es la función zeta de
    Riemann.
  \item Poniendo $V_q=\Q_q(1)$ nos da también un sistema compatible de representaciones de
    Galois. Su función $L$ es $\zeta(s+1)$.
  \item Sea $\chi$ un carácter de Dirichlet y $K=\Q(\chi)$. Entonces si ponemos
    $V_{\mathfrak q}=K_{\mathfrak q}$ para cada primo $\mathfrak q$ de $K$ con $g\in\GQ$
    actuando en $V_{\mathfrak q}$ como multiplicación con $\chi(g)$ entonces esto también es
    un sistema compatible de representaciones de Galois, esta vez con coeficientes en
    $K$. Su función $L$ es la función $L$ de Dirichlet $L(\chi,s)$.
  \end{enumerate}
\end{ex}

Un sistema compatible de representaciones de Galois todavía no es la noción final a la cual
se puede generalizar los fenómenos de la Teoría de Iwasawa, pero para nuestros ejemplos
básicos es suficiente (la verdadera noción son los motivos, que mencionamos en la
\cref{sec:etnc}). Resumiendo, tenemos lo siguiente para el sistema compatible de
representaciones de Galois $(\Q_q)_q$:
\begin{itemize}
\item La función $L$ del sistema (la función zeta de Riemann) es meromorfa en todo de $\C$ y
  algunos valores especiales son algebraicos.
\item Fijamos un primo $p$. Entonces estos valores especiales pueden ser interpolados
  $p$-ádicamente, esto conduce a la existencia de una función $L$ $p$-ádica, que (esencialmente)
  es un elemento del álgebra de Iwasawa $\LL(G)$, dónde $G$ es el grupo de Galois de la
  torre infinita $(K_r)_r=(\Q(\mu_{p^r}))_r$ de campos de números.
\item Para el mismo primo $p$, si consideramos el módulo $\mathrm X(\Qp/\Zp)$ para el
  miembro en $p$ del sistema de representaciones y un retículo estable obtenemos un módulo
  noetheriano de torsión sobre $\LL(G)$, y por lo tanto tiene un ideal característico gracias
  a la teoría de estructura de tales módulos.
\item Este ideal característico es generado por la función $L$ $p$-ádica del sistema -- esto
  es la Conjetura Principal.\footnote{Salvo el hecho de que tenemos que restringir a la parte
    $(\cdot)^-$ del módulo y multiplicar la función $L$ $p$-ádica con su denominador. Vamos
    a ignorar estos detalles aquí y en lo que sigue. De hecho, en las generalizaciones estos
    dos fenómenos no ocurren.}
\end{itemize}
Algo similar ocurre para el sistema $(\Q_q(1))_q$, dándonos la otra versión de la
Conjetura Principal. La relación entre estas dos formulaciones y el hecho de que son
equivalentes la discutiremos más tarde (\cref{footnote:ep-eq-x-y} en la
\cpageref{footnote:ep-eq-x-y}).

\ejercicios

\begin{ejer}
  Para los siguientes campos de números $K$, ¿Cuáles primos de $\Q$ son ramificados y cuáles son
  los elementos de Frobenius?
  \begin{enumerate}
  \item $K=\Q(\sqrt d)$ para $d\in\Z$ libre de cuadrados;
  \item $K=\Q(\mu_m)$ con $m\in\Nuno$.
  \end{enumerate}
  Use los resultados para deducir el ley de reciprocidad cuadrática
  \[ \left(\frac p q\right)\left(\frac q p\right)=(-1)^{\frac{(p-1)(q-1)}{4}} \]
  para primos impares $p,q$.
\end{ejer}

\begin{ejer}
  Para un carácter de Dirichlet que vemos como representación de Galois como en
  \eqref{eqn:chi-gal}, ¿En cuáles primos es ramificado?
\end{ejer}

\begin{ejer}
  Demuestre que si $V$ es un espacio vectorial de dimensión $1$ sobre $\C$ con una acción
  de $\GQ$ dada por un carácter de Dirichlet $\chi$ como en \eqref{eqn:chi-gal} entonces
  tenemos
  \[ \det(1-\chi(\Frob_\ell)T, V^{I_\ell} ) = 1-\chi(\ell)T, \]
  para cada primo $\ell$ (no importa si es ramificado o no).
\end{ejer}

\begin{ejer}
  Sea $V=(V_q)_q$ un sistema compatible de representaciones de Galois. Demuestre que
  entonces \[ V(1):=(V_q\tensor_{\Q_q}\Q_q(1))_q \] también es un sistema compatible de
  representaciones de Galois y que
  \[ L(V(1),s)=L(V,s+1). \]
\end{ejer}

\subsection{La Conjetura Principal para curvas elípticas}

Como prometido ahora explicamos cómo las curvas elípticas se insertan en la imagen que
esbozamos hasta ahora. Para esto asumimos que el lector está familiarizado con la teoría básica de
curvas elípticas como es explicada por ejemplo en \cite{MR2514094}. Para toda la sección
fijamos una curva elíptica $E$ sobre $\Q$ y un primo $p$, además, suponemos que la curva tiene
buena reducción ordinaria en $p$.

Al lector interesado en aprender más sobre la Teoría de Iwasawa de curvas elípticas le
recomendamos el texto \cite{MR1860044} de Greenberg y también \cite[§2]{MR3586809}.

Escribimos $E(K)$ para los puntos de $E$ con coeficientes en un campo $K/\Q$. Si $m\in\Nuno$
entonces escribimos $E(K)[m]$ para los elementos de $E(K)$ cuyo orden es divisible por $m$,
es decir los que son anulados por la multiplicación por $m$, que denotamos $[m]$. Se sabe
que $E(\Qbar)[m]$ es isomorfo a $(\Z/m\Z)^2$, aunque el isomorfismo no es canónico. Para cada primo $q$, el
\define{módulo de Tate} $q$-ádico de $E$ está definido como
\[ T_qE := \varprojlim_{l\in\Nuno}E(\Qbar)[q^l], \]
los mapeos $E(\Qbar)[q^{l+1}]\rightarrow E(\Qbar)[q^l]$ siendo $P\mapsto[q]P$. Esto es no
canónicamente isomorfo a $\Z_q^2$. También definimos
\[ V_qE=\Q_q\tensor_{\Z_q}T_pE \]
que es un espacio vectorial de dimensión $2$ sobre $\Q_q$. La acción de $\GQ$ en $E(\Qbar)$
induce acciones continuas en $E(\Qbar)[q^l]$, $T_qE$ y $V_qE$, y por eso $V_qE$ es una
representación de Galois $\Q_q$-lineal y $T_qE$ es un retículo estable.

Sea $p$ un primo.
A continuación resumimos cómo se define el grupo de Selmer de una curva elíptica clásicamente. Fijamos un
campo de números y $l\in\Nuno$ y seguimos \cite[§X.4]{MR2514094} poniendo allá $E=E'$ y
$\phi=[p^l]$. Para cada plaza $v$ de $K$ tenemos una sucesión exacta
\[ 0\rightarrow E[p^l](\overline K_v)\rightarrow E(\overline K_v)\labeledarrow{[p^l]}
  E(\overline K_v) \rightarrow 0 \]
de grupos abelianos discretos con acción de $\G{K_v}$. Tomando la sucesión exacta larga en
cohomología obtenemos
\[  0 \rightarrow E[p^l](K_v)\rightarrow E(K_v)\labeledarrow{[p^l]} E(K_v) \labeledarrow{\partial} \HL^1(K_v,E[p^l](\overline K_v))\rightarrow \dotsm. \]
El morfismo de borde $\partial$ induce un morfismo inyectivo
\[ \kappa_{r,v}\colon E(K_v)/p^l E(K_v) \hookrightarrow\HL^1(K_v,E[p^l](\overline K_v)) \]
que se llama el \define{morfismo de Kummer} (en \cite[§X.4]{MR2514094} lo denotan $\delta$).
Entonces se define el grupo de Selmer de nivel $p^l$ sobre $K$ como las clases en
cohomología de Galois global que localmente están en la imagen del morfismo de Kummer, es
decir
\[ \mathrm S^{(p^l)}(E/K) := \{ c\in\HL^1(K,E[p^l](\overline K)) \mid \forall v\nmid\infty\colon \operatorname{res}_v(c) \in
  \im\kappa_{l,v} \} \]
y luego
\[ \mathrm S^{(p^\infty)}(E/K):=\varinjlim_{l\in\Nuno}\mathrm S^{(p^l)}(E/K)\subseteq
  \varinjlim_{l\in\Nuno}\HL^1(K,E[p^l](\overline K))=\HL^1(K,E[p^\infty](\overline K)) \] (aquí
la última igualdad sigue de \cite[(1.2.5)]{NSW} y el hecho de que cada $E[p^l](\overline K)$
es finito, así que su estabilizador en $\G K$ es abierto).
Además, se puede ver fácilmente que $E[p^\infty](\overline K)\isom V_pE/T_pE$ canónicamente
(véase el \cref{ejer:a-qp-zp-colim}). La siguiente proposición dice que, usando esta
identificación, el grupo de Selmer clásico de hecho es lo mismo que el grupo de Selmer de la
\cref{defi:selmer-general} en esta situación.

\begin{prop}
  Para cada campo de números $K$ tenemos una igualdad
  \[ \mathrm S^{(p^\infty)}(E/K) = \Sel(K,V_pE/T_pE) \]
  de subgrupos de $\HL^1(K,V_pE/T_pE)$.
\end{prop}
\begin{proof}
  \cite[Prop.\ 1.6.7]{MR1749177}
\end{proof}

Aplicando la \cref{defi:x-selmer} en nuestra situación obtenemos un $\LL(G)$-módulo compacto
$\mathrm X(V_pE/T_pE)$ que simplemente denotamos $\mathrm X(E)$. Se sabe que este módulo es
noetheriano \cite[Lem.\ 13]{MR3586809}, lo que es importante para poder aplicar
la teoría de estructura a este módulo. Como los módulos que aparecen en la Conjetura
Principal clásica tenemos incluso el siguiente resultado, que era sospechado por Mazur:

\begin{thm}[Kato]\label{thm:kato-x-torsion}
  El $\LL(G)$-módulo $\mathrm X(E)$ es de torsión.
\end{thm}
\begin{proof}
  \cite[Thm.\ 17.4 (1)]{MR2104361}
\end{proof}

Como consecuencia de esto se puede aplicar las mismas técnicas que llevaron a la
demostración del \cref{thm:iwasawa-intro} de Iwasawa sobre el crecimiento de los grupos de
clases para obtener un resultado análogo sobre los ordenes de las $p$-partes de los grupos
de Tate-Shafarevich, véase \cite[Prop.\ 14]{MR3586809}.

Ahora toca el turno de hablar de la función $L$ asociada. Para la curva elíptica $E$ tenemos la función $L$ de Hasse
y Weil
\[ L(E,s) = \prod P_\ell(E,\ell^{-s})^{-1} \]
con
\[ P_\ell(E,s)=\begin{cases}
            (1-(\ell+1-\#E(\F_\ell))\ell^{-s}+\ell^{1-2s}) \text{ si } E\text{ tiene buena reducción en }\ell, \\
            (1-\ell^{-s})  \text{ si }  E\text{ tiene reducción multiplicativa escindida en }\ell, \\
            (1+\ell^{-s})  \text{ si }  E\text{ tiene reducción multiplicativa no escindida en }\ell, \\
            1  \text{ si }  E\text{ tiene reducción aditiva en }\ell \\
            \end{cases}
            \]         

que es introducida en \cite[§C.16]{MR2514094} (por ahora, esto es sólo una expresión formal;
hablamos de convergencia en la siguiente sección). Esta función $L$ es un caso especial
de la función $L$ general de la \cref{defi:sistema-compatible}: El sistema compatible de
representaciones de Galois que hay que usar es $(V_qE)_q$, es decir está formado por todos los
módulos de Tate. La primera condición sobre la ramificación de las representaciones es
cierta por el criterio de Néron, Ogg y Shafarevich \cite[Thm.\ VII.7.1]{MR2514094}. En
\cite[§2.7.2]{MR2894984} es demostrado que los polinomios de la
\cref{defi:sistema-compatible} coinciden con los de la definición arriba, esto explica las
fórmulas anteriores; en particular, la segunda condición sobre la compatibilidad también es
cierta.

La discusión de esta función $L$ y la pregunta sobre la existencia de una función $L$
$p$-ádica asociada la posponemos a la siguiente sección para la claridad de la
exposición. Por ahora mencionamos el resultado sólo en una forma provisional, la versión
precisa se encuentra en el \cref{thm:palf-modulformen}.

\begin{thm}[Mazur/Swinnerton-Dyer]
  Existe un elemento $\mu_E\in\LL(G)$ que interpola valores especiales de la función
  $L(E,s)$ $p$-ádicamente.
\end{thm}

Es decir, las curvas elípticas (sobre $\Q$ con buena reducción ordinaria en $p$) también tienen
una función $L$ $p$-ádica. Con estas preparaciones debería ser claro como formular la
Conjetura Principal para curvas elípticas.

\begin{conj}[Conjetura Principal de Mazur para curvas elípticas]\label{conj:mc-ec}
  Tenemos la igualdad de ideales en $\LL(G)$
  \[ \charideal_{\LL(G)}\mathrm X(E)=(\mu_E). \]
\end{conj}

Al día de hoy, esta conjetura es demostrada en muchos casos. El caso de una curva con multiplicación
compleja la conjetura fue demostrada por Rubin \cite{MR1079839}. En el caso general, suponiendo
algunas condiciones técnicas menudas, Kato \cite{MR2104361} demostró la inclusión
\enquote{$\supseteq$} en la igualdad y Skinner y Urban \cite{MR3148103} demostraron la
otra. Véase también \cite[§2.3.3, §2.4.4, §2.5.5]{MR2334196} para una discusión de las ideas
detrás de estas demostraciones.

\ejercicios

\begin{ejer}\label{ejer:a-qp-zp-colim}
  Sea $A$ un grupo abeliano. Use el hecho de que el producto tensorial es compatible con
  colímites para demostrar que canónicamente
  \[ A\tensor_\Z(\Qp/\Zp)\isom\varinjlim_{l\in\Nuno} A/p^lA \]
  donde los mapeos $A/p^lA\rightarrow A/p^{l+1}A$ en el colímite a la derecha son $a\mapsto
  pa$. Aplique esto al módulo de Tate $A=T_pE$ de una curva elíptica $E$ para demostrar que
  \[  V_pE/T_pE\isom E[p^\infty](\overline K). \]
\end{ejer}

\subsection{Funciones $L$ $p$-ádicas para formas modulares}

En la sección anterior dejamos pendiente la pregunta de cómo debe ser la función $L$ $p$-ádica
de una curva elíptica. En el caso de la función zeta $p$-ádica, esta interpolaba valores
especiales de la función zeta en enteros negativos, que eran algebraicos. ¡Recuerde que en
los enteros negativos la serie que define la función zeta no converge! Además, no obtuvimos estos
valores sino hasta que continuamos la función zeta analíticamente. Por lo tanto, antes de que
podamos pensar en una función de Hasse y Weil $p$-ádica, tenemos que preguntarnos:
\begin{itemize}
\item ¿Dónde converge la serie que define la función $L$ de Hasse y Weil?
\item ¿Tiene una continuación meromorfa o analítica a todo $\C$?
\item ¿También tenemos valores especiales que son algebraicos? ¿Cuáles son estos valores?
\end{itemize}

Estas preguntas se pueden formular también para funciones $L$ más generales como las de la
\cref{defi:sistema-compatible}, y lo que pasa es que en esta generalidad no se sabe casi nada
(aunque hay algunas conjeturas que mencionaremos en la \cref{sec:etnc}). El problema es que
las funciones $L$ son definidas por polinomios de origen aritmético (esencialmente
polinomios característicos de elementos de Frobenius en una representación de Galois), y
esta definición no dice mucho sobre el comportamiento analítico de dichas
funciones. Necesitamos algunas herramientas para sobreponernos a estos obstáculos.

Aquí entra el famoso Programa de Langlands en el escenario. Esto es un tema enorme sobre el que
se podría escribir otro libro, así que sólo indicamos algunas pocas ideas detrás del Programa de Langlands. La filosofía es que para cada \enquote{objeto aritmético} (más precisamente, un
motivo\footnote{Por ahora nos imaginamos un motivo como un sistema compatible de
  representaciones de Galois aunque en realidad es algo más especifico: conjeturalmente,
  cada motivo lleva a un tal sistema, pero no al revés. Diremos un poco más sobre esto en la
  siguiente \cref{sec:etnc}.}) debería existir un objeto \enquote{automorfo} (más
precisamente, una representación automorfa), que es de origen analítico, con la misma
función $L$. Sobre las funciones $L$ de origen automorfo sabemos mucho más y en casos
favorables tenemos por ejemplo la continuación analítica. Esto debería generalizar la teoría
de campos de clases, es decir, la teoría de campos de clases debería de ser una especialización del Programa de Langlands. Sobre $\Q$ tenemos lo siguiente: El teorema de Kronecker y Weber dice
que cada extensión abeliana está contenida en una extensión ciclotómica $\Q(\mu_m)$ para un
$m\in\Nuno$, y el isomorfismo de Artin de la teoría de campos de clases en este caso
simplemente es $\Gal(\Q(\mu_m)/\Q)\isom(\Z/m\Z)^\times$. En particular, las representaciones
de Galois que se factorizan a traves de $\Gal(\Q(\mu_m)/\Q)$ corresponden únicamente a los
caracteres de Dirichlet de conductor un divisor de $m$. Bajo esta correspondencia, la función
$L$ de tal representación de Galois es la función $L$ del carácter de Dirichlet
correspondiente, como mencionamos antes. Esto es el caso más sencillo de una correspondencia
tipo Langlands: los objetos aritméticos son las representaciones de $\GQ$
de orden finito y de una dimensión, y los objetos automorfos son los caracteres de
Dirichlet. Tienen la misma función $L$ y esta afirmación es equivalente al isomorfismo de
Artin. Debido a que conocemos propiedades como la continuación analítica de las funciones $L$ de
Dirichlet, obtenemos esta correspondencia también para las funciones $L$ de dichas representaciones de
Galois. Aunque esto parece casi trivial en este caso básico, de hecho es una encarnación del
Programa de Langlands.

En el caso de curvas elípticas el objeto automorfo que corresponde a ellas son ciertas
formas modulares. En este caso tenemos el siguiente famoso resultado que implica el Último
Teorema de Fermat. A partir de ahora hasta el fin de la sección suponemos que el lector
conoce la teoría básica de formas modulares como es explicada por ejemplo en
\cite{MR2112196} o \cite{MR1291394}. Antes de citar el resultado resumimos la definición de
la función $L$ de formas modulares.

Sea $f\in\mathrm S_k(N,\chi)$ una forma modular cuspidal nueva\footnote{\textit{cuspidal newform}} de peso $k\ge2$, nivel
$N\in\Nuno$ y nebentipo $\chi$; en particular $f$ es una forma propia para todos los operadores
de Hecke. Si la vemos como función en el semiplano superior $\mathbb H$ entonces se escribe
como una serie de Fourier
\[ f(z)=\sum_{n=1}^\infty a_nq^n,\quad q=\e^{2\pi\mathrm i z} \quad(z\in\mathbb H). \] La función
$L$ en este caso está definida como \[ L(f,s)=\sum_{n=1}^\infty a_nn^s. \] Más generalmente,
si $\psi$ es un carácter de Dirichlet cualquiera entonces se define
\[ L(f,\psi,s)=\sum_{n=1}^\infty a_n\psi(n)n^s. \]
Esta función se llama la función de $f$ chanfleada por $\psi$. El siguiente resultado
clásico describe sus propiedades básicas analíticas.

\begin{prop}
  La serie que define $L(f,\psi,s)$ converge absolutamente para $\Re(s)>\frac k2+1$ y la
  función holomorfa que define tiene una continuación analítica a todo $\C$. Además
  tiene un producto de Euler
  \[ L(f,\psi,s) = \prod_{\ell\text{ primo}} (1-\psi(\ell) a_\ell\ell^{-s} +
    \chi\psi^2(\ell)\ell^{k-1-2s})^{-1} \qquad(\Re(s)>\textstyle\frac k2+1) \]
\end{prop}
\begin{proof}
  \cite[Thm.\ 3.66]{MR1291394}, \cite[Thm.\ 5.9.2]{MR2112196}
\end{proof}

El Teorema de Modularidad entonces es el siguiente.

\begin{thm}[Wiles, Taylor, Breuil, Conrad, Diamond]
  Sea $E/\Q$ una curva elíptica. Entonces existe una forma modular nueva $f$ de peso $2$ con
  nebentipo trivial cuya función $L$ es la misma que la de $E$, es decir
  \[ L(E,s)=L(f,s) \quad\forall s\in\C. \] En particular $L(E,-)$ tiene
  una continuación analítica a todo de $\C$.
\end{thm}

Este resultado también es una encarnación del Programa de Langlands, y nos provee de las
herramientas deseadas para responder las preguntas planteadas del inicio de esta sección. En general, no hay
ninguna esperanza de responder estas preguntas directamente sin pasar por el mundo
automorfo, y esto ilustra la enorme importancia del Programa de Langlands para la Teoría de
Iwasawa.

Es decir, la existencia de una función $L$ $p$-ádica para una curva elíptica ahora es
equivalente a la existencia de tal función para formas modulares nuevas de peso $2$ con
nebentipo trivial. De hecho, no es mucho más difícil explicar esto para formas modulares
nuevas en general, así que desde ahora ya no hablaremos de curvas elípticas sino de formas
modulares.

A partir de ahora fijamos una forma modular cuspidal nueva $f\in\mathrm S_k(N,\chi)$ como
arriba. La pregunta de algebraicidad de valores especiales es contestada por el siguiente
resultado.

\begin{thm}[Shimura]\label{thm:shimura-algebraicidad}
  Sea $K_f$ el campo de números generado por los coeficientes de Fourier de $f$.
  Existen dos números $\Omega_f^\pm\in\C^\times$ tal que para cada carácter de Dirichlet
  $\psi$ y para $n=1,\dotsc,k-1$ tenemos
  \[ \frac{\mathrm G(\psi^{-1})}{(2\pi\mathrm i)^n \Omega_f^\pm} L(f,\psi,n) \in K_f(\psi). \]
  Aquí el superíndice de $\Omega_f^\pm$ debe ser el signo de $(-1)^n\psi(-1)$ y
  \[ \mathrm G(\psi^{-1}) =\sum_{j=1}^c\psi^{-1}(j)\e^{2\pi\mathrm i j/c} \]
  es la suma de Gauß (con $c$ siendo el conductor de $\psi$).
\end{thm}
\begin{proof}
  \cite[Thm. 1 (ii)]{MR0463119}
\end{proof}

Como se puede ver de este teorema, los valores especiales en general ya no son algebraicos
(porque los números $\Omega_f^\pm$ en general son transcendentes), pero al menos podemos
describir y controlar su parte transcendente con dos números. Otra diferencia, con el resultado
análogo de la \cref{prop:numeros-de-bernoulli-gen} para caracteres de Dirichlet, es que sólo
tenemos valores algebraicos para una cantidad finita de enteros $n$, pero
para una cantidad infinita de chanfles por caracteres. Este fenómeno tiene explicación
en una conjetura de Deligne que abordaremos en la siguiente sección.

Para el siguiente resultado sobre la función $L$ $p$-ádica tenemos que suponer que $f$ es
\define[forma modular ordinaria]{ordinaria} en $p$, esto significa que $|a_p|_p=1$ (donde $|\cdot|_p$ es el valor absoluto
$p$-ádico). Es fácil ver que en este caso el polinomio
\[ X^2-a_pX+\chi(p)p^{k-1}, \]
que se llama el \emph{polinomio de Hecke}\index[def]{polinomio!de Hecke} $p$-ésimo de $f$, tiene una única raíz
$\alpha\in\Qbar$ tal que $|\alpha|_p=1$. Este $\alpha$ aparece en el siguiente resultado.
La condición ordinaria aquí es análoga a la condición en la \cref{sec:stickelberger} que pide
que el conductor del carácter $\chi$ no sea divisible por $p$.

La forma modular asociada a una curva elíptica de buena reducción ordinaria en $p$ es
ordinaria. Mencionamos además que en este caso el campo $K_f$ es $\Q$.

Aquí y en los siguientes resultados, $\LL(G)$ denota el álgebra de Iwasawa con coeficientes
en el anillo de enteros $\O$ de la completación de $K_f$ en la plaza arriba de $p$ fijada
por nuestros encajes (es decir, para una curva elíptica $\O=\Zp$).

\begin{thm}[Mazur/Swinnerton-Dyer, Mazur/Tate/Teitelbaum]\label{thm:palf-modulformen}
  Sean $K_f$ y $\Omega_f^\pm$ como arriba. 

  Entonces existe un único elemento $\mu_f\in\LL(G)$ tal que para
  $n=1,\dotsc,k-1$ y cada carácter de Dirichlet $\psi$ de conductor $p^m$ con $m\in\Ncero$
  tenemos
    \begin{multline*}
      \int_{G}\psi^{-1}\kappa^n\integrald\mu_f = \\
      (n-1)!(1-\alpha^{-1}\psi^{-1}(p)p^{n-1})(1-\alpha^{-1}\chi\psi(p)p^{k-n})\frac{p^{m(n-1)}\mathrm
        G(\psi^{-1})}{\alpha^{m}(2\pi\mathrm
        i)^{n}\Omega_f^\pm}L(f,\psi,n),
  \end{multline*}
  donde otra vez el superíndice de $\Omega_f^\pm$ es el signo de $(-1)^n\psi(-1)$.
\end{thm}
\begin{proof}
  \cite{MR0830037} (y \cite{MR0354674} para $k=2$)
\end{proof}

La fórmula de interpolación en este teorema se ve mucho menos elegante que la del
\cref{thm:palf-stickelberger-no-trivial} -- no simplemente quitamos un factor de Euler, sino que
aparecen muchas \enquote{expresiones de corrección} que la desproveen de belleza. Sin
embargo, hay una explicación satisfactoria por la cual cada una de estas expresiones debe
aparecer y por qué esta fórmula de hecho es consistente con la del
\cref{thm:palf-stickelberger-no-trivial}. Esbozaremos esta explicación en la siguiente
sección.

Finalmente podemos formular una Conjetura Principal también para formas
modulares. Un resultado famoso de Deligne asocia a una forma modular nueva como antes un
sistema compatible de representaciones de Galois cuya función $L$ es la función $L$ de
$f$. Más precisamente, tenemos lo siguiente.

\begin{thm}[Deligne]
  Sea $f$ una forma modular nueva como antes y $K_f$ el campo generado por los coeficientes
  de Fourier. Entonces para cada primo ${\mathfrak q}$ de $K_f$ existe una representación de Galois
  \[ \rho_{f,{\mathfrak q}}\colon\GQ\rightarrow\GL_2(K_{f,{\mathfrak q}}) \]
  que es no ramificada fuera de $N$, $\infty$ y el primo $q$ de $\Q$ abajo de ${\mathfrak q}$ y tal que
  para todo primo $\ell\neq q$ tenemos
  \[ \det(1-\rho_{f,{\mathfrak q}}(\Frob_\ell) T, V_{\mathfrak q}^{I_\ell} ) = 1- a_\ell T +
    \chi(\ell)\ell^{k-1} T^2 \] (donde $V_{\mathfrak q}=K_{f,{\mathfrak q}}^2$ con la acción
  de $\GQ$ definida por $\rho_{f,\p}$)
\end{thm}

En el caso en que $f$ corresponde a una curva elíptica $E$ la representación del teorema
anterior es la misma que la representación $V_qE$ construida con los módulos de Tate.

Sea $\p\mid p$ el primo de $K_f$ arriba de $p$ fijado por nuestros encajes. Después de
escoger un retículo estable $T_\p$ en $V_\p$ tenemos el grupo de Selmer $\Sel(V_\p/T_\p)$ y
el módulo $\mathrm X(V_\p/T_\p)$ de la \cref{defi:x-selmer}, que denotamos como
$\mathrm X(f)$. El \cref{thm:kato-x-torsion} de Kato de hecho todavía es válido en
esta situación más general y dice que el $\LL(G)$-módulo $\mathrm X(f)$ es noetheriano y de
torsión.  La Conjetura Principal para formas modulares entonces es la afirmación siguiente
(note que la \cref{conj:mc-ec} para curvas elípticas es un caso especial).

\begin{conj}[Conjetura Principal para formas modulares]
  Sea $f$ como antes. Entonces tenemos la igualdad de ideales en $\LL(G)$
  \[ \charideal_{\LL(G)}\mathrm X(f)=(\mu_f). \]
\end{conj}

Los resultados de Kato, Skinner y Urban que citamos después de la \cref{conj:mc-ec} de
hecho son más generales: no sólo funcionan para curvas elípticas (es decir con $k=2$) sino
para cualquier forma modular $f$ como arriba (con las mismas pequeñas condiciones técnicas
que mencionamos). Es decir, también esta conjetura es demostrada en muchos casos.

\ejercicios

\begin{ejer}
  Sea $f$ una forma modular ordinaria. Demuestre que el polinomio de Hecke
  $p$-ésimo de $f$ \[ X^2-a_pX+\chi(p)p^{k-1}, \] tiene una única raíz $\alpha\in\Qbar$ tal
  que $|\alpha|_p=1$.
\end{ejer}

\begin{ejer}
  Demuestre la unicidad en el \cref{thm:palf-modulformen}. Es decir, aunque la fórmula de
  interpolación sólo es verdad para una cantidad finita de $n\in\Nuno$ (al contrario del
  \cref{thm:palf-stickelberger-no-trivial}) aún así el elemento $\mu_f\in\LL(G)$ es
  determinado únicamente por la fórmula. Esta demostración es análoga a la de la
  \cref{prop:unicidad-palf}.
\end{ejer}

\begin{ejer}
  ¿Que significa la existencia de $\mu_f$ del \cref{thm:palf-modulformen} para los valores
  especiales de la función $L$? Use las proposiciones
  \ref{prop:congruencias-de-kummer-general} y \ref{prop:unidad-solo-depende-de-i} para
  deducir resultados análogos a los corolarios \ref{cor:congruencias-de-kummer-dirichlet} y
  \ref{cor:unidad-depende-solo-de-i}.
\end{ejer}

\begin{ejer}
  ¿Que significa la Conjetura Principal para los grupos de Selmer? Use los lemas
  \ref{lem:wert-einheit-bedeutung} y \ref{lem:cp-tamano} para deducir resultados análogos a
  las proposiciones \ref{prop:cp-implica-kummer} y \ref{prop:consecuencia-numero-de-clases}
  (también para curvas elípticas).
\end{ejer}

\subsection{Motivos y la Conjetura Equivariante de Números de Tamagawa}
\label{sec:etnc}

Hasta ahora hemos conocido varias \enquote{Conjeturas Principales}: dos versiones de la clásica,
la de curvas elípticas y más generalmente aquella para formas modulares. ¿Cómo se pueden
uniformar estas en una sola afirmación general? En esta última sección tratamos de esbozar
algunas ideas de cómo hacerlo.

Aquí entran los motivos que ya mencionamos varias veces. Ellos son de alguna manera
\enquote{el objeto más general de interés aritmético}. En la introducción a este texto
escribimos que la pregunta que queremos estudiar es \enquote{cómo se comportan los números
  enteros o algebraicos y ecuaciones entre ellos en diferentes situaciones}. Ecuaciones (de
polinomios) entre números enteros o algebraicos llevan a un objeto geométrico: una variedad
algebraica.\footnote{Para nosotros, una \define{variedad} es un esquema integral y separado
  de tipo finito sobre un campo.} Entonces, los motivos son como \enquote{pedazos} o
\enquote{componentes} de tales variedades que son aritméticamente interesantes. Son
objetos muy generales que incluyen los objetos que estudiamos hasta ahora y muchos más.
Intentamos a explicar esto un poco más, pero el lector está en libertad de simplemente
pensar en variedades (o ecuaciones).

De manera ligeramente más precisa, la idea es la siguiente. Escribimos
$\mathpzc{Var}(\Q)$ para la categoría de variedades lisas proyectivas sobre $\Q$ (esto se
puede hacer para campos más generales como por ejemplo campos de números, pero para simplificar sólo tratamos el caso de $\Q$). De una variedad se puede considerar su
cohomología de diferentes maneras: existen varios funtores de $\mathpzc{Var}(\Q)$ a
categorías de espacios vectoriales con estructuras adicionales que tienen propiedades
similares aunque son construidos de maneras muy diferentes. Por ejemplo, tomando los puntos
en $\C$ de una variedad algebraica lleva a una variedad analítica compleja, de la cual
podemos tomar la cohomología singular de la topología algebraica. También existe la
cohomología de de Rham (algebraica). Además existe para cada primo $q$ la cohomología étale
$q$-ádica. Todas estas cohomologías tienen propiedades comunes (por ejemplo, sus dimensiones
son iguales, existen resultados de dualidad, \dots), aunque ni siquiera son espacios
vectoriales sobre el mismo campo. El deseo que estimuló la teoría de motivos era el de una
\enquote{teoría de cohomología universal}, que debería existir un funtor
$h\colon\mathpzc{Var}(\Q)\rightarrow\mathpzc{Mot}(\Q)$ a una categoría de motivos sobre $\Q$
tal que todos los funtores de cohomología se factoricen a través de $h$, y tal que
$\mathpzc{Mot}(\Q)$ se encuentre \enquote{más cerca} de las categorías de espacios vectoriales --
idealmente debería ser una categoría abeliana. Si algo así existe, entonces la categoría
$\mathpzc{Mot}(\Q)$ claramente tiene que tener más objetos que $\mathpzc{Var}(\Q)$ que no es
una categoría abeliana -- por eso el motivo $h(X)$ asociado a una variedad $X$ se puede
descomponer aunque $X$ sea irreducible. Esto es lo que quisimos decir cuando hablamos de
\enquote{pedazos} de variedades. Vamos a ilustrar este fenómeno en el \cref{ex:motivos}~\ref{ex:dedekind-zeta}.

Hasta hoy no se ha podido demostrar que una categoría $\mathpzc{Mot}(\Q)$ con las propiedades
deseadas existe, pero se han construido algunos candidatos. Para que estas realmente tengan
dichas propiedades falta que demostrar algunas conjeturas difíciles (las \enquote{conjeturas
  estándar} de Grothendieck). No vamos a explicar más detalles aquí. Simplemente nos
imaginamos que la categoría existe, como de hecho lo hacen muchos textos en la
literatura. Lo que pasa es que, aunque no sea posible hacer todo lo que Grothendieck soñaba
que se pudiera hacer con los motivos, lo que necesitamos para la Teoría de Iwasawa sí es
posible, por eso este enfoque pragmático funciona bien. Lo que necesitamos es
\enquote{tomar cohomología de un motivo}, es decir aplicar los funtores que mencionamos
antes como cohomología singular (que se llama también cohomología de Betti), de Rham o étale
$q$-ádica a un motivo para obtener espacios vectoriales. Estos se llaman las
\define{realizaciones} del motivo. Si $M$ es un motivo, las denotamos como $M_\betti$,
$M_\dR$ y $M_q$, respectivamente.

Al lector que quiera aprender más (y de una manera precisa) sobre los motivos le
recomendamos el maravilloso texto \cite{MR3204952} y para más detalles el libro
\cite{MR2115000}.

Entonces ¿cómo extendemos las ideas de las secciones anteriores a los motivos? Para esto lo
más importante son las realizaciones étales de un motivo. Si $X$ es una variedad lisa proyectiva
sobre $\Q$ y $q$ es un primo entonces la cohomología étale $q$-ádica
\[ \HL^i_\et(X\times_\Q\Qbar,\Q_q) \]
es un espacio vectorial sobre $\Q_q$ con una acción continua de $\GQ$, así que es una
representación de Galois. Por eso, si $M$ es un motivo sobre $\Q$, su realización étale
$q$-ádica $M_q$ también es una representación de Galois. Por supuesto, podemos hacer esto
para cada primo $q$, y la siguiente conjetura salta a la vista.

\begin{conj}
  Sea $M$ un motivo sobre $\Q$. Entonces $(M_q)_q$ es un sistema compatible de
  representaciones de Galois.
\end{conj}

Suponemos que esta conjetura es cierta.\footnote{En general esta conjetura está abierta,
  pero algunos casos especiales son conocidos (por ejemplo los de curvas elípticas o formas
  modulares). Además, las famosas Conjeturas de Weil (véase por ejemplo \cite{MR926276})
  implican que al menos para \emph{casi todas las plazas} la compatibilidad de la
  \cref{defi:sistema-compatible}~\ref{defi:sistema-compatible:compatibilidad} es cierta.}
En esta sección vamos a explicar más y más conjeturas cada una planteada sobre las
anteriores. Esto quizás parezca un poco tambaleante, pero todas estas conjeturas sí son
conocidas en algunos casos (por ejemplo para curvas elípticas), así que lo peor sería que
fueran válidas solamente en estos casos, pero al menos sabemos que no son vacías.

Con el sistema compatible podemos definir (casi) todos los objetos que necesitamos para una
Conjetura Principal, como la función $L$ y los grupos de Selmer. No se cree que cada
sistema compatible viene de un motivo -- para esto conjeturalmente se necesita una condición
extra al sistema, que en este caso se llama \enquote{geométrico}; esto es el contenido de la conjetura de
Fontaine y Mazur \cite{MR1363495}.
La Conjetura Principal hipotéticamente es verdad para motivos, y esto
explica para cuales sistemas compatibles esperamos una Conjetura Principal: los que vienen
de motivos, o si la conjetura de Fontaine y Mazur es verdad, equivalentemente los
geométricos.

\begin{defi}
  La función $L$ de un motivo $M$ sobre $\Q$ es la función $L$ del sistema compatible
  $(M_q)_q$. La denotamos $L(M,-)$.
\end{defi}

La siguiente conjetura entonces sería una consecuencia de las conjeturas en el Programa de
Langlands.

\begin{conj}
  La función $L(M,-)$ tiene una continuación meromorfa a todo $\C$.
\end{conj}

Antes de continuar mencionamos algunos ejemplos.

\begin{ex}\label{ex:motivos}
  \begin{enumerate}
  \item El caso más sencillo es el de la variedad $X=\Spec\Q$, que geométricamente es un
    punto. Su cohomología étale $q$-ádica es un espacio vectorial de dimensión $1$ con la
    acción trivial de $\GQ$. Por eso, su función $L$ es la función zeta de Riemann. Este
    motivo lo denotamos como $\Q$.
  \item\label{ex:tate-motive} El sistema de representaciones $(\Q_q(1))_q$ también viene de
    un motivo, que se denota $\Q(1)$. Por supuesto, su función $L$ es
    $L(\Q(1),s)=\zeta(s+1)$ ($s\in\C$).
  \item\label{ex:artin-motive} Mencionamos que para cada carácter de Dirichlet $\chi$ existe
    un motivo $[\chi]$ cuya función $L$ es la de $\chi$. Más generalmente para cada
    \define{representación de Artin} $\rho$, es decir una representación de $\GQ$ con imagen
    finita, existe un motivo $[\rho]$ cuya función $L$ es la función $L$ de Artin de $\rho$.
  \item Una curva elíptica $E$ sobre $\Q$ es una variedad proyectiva lisa y por eso da lugar
    a un motivo.\footnote{Más precisamente, aquí usamos el motivo $h^1(E)$.} Aquí la
    cohomología étale $q$-ádica $\HL^1_\et(E\times_\Q\Qbar,\Q_q)$ \emph{no} es lo mismo que
    $V_qE$ (que usamos para definir la función $L$) sino dual a esto, que gracias al
    apareamiento de Weil en la curva es isomorfo a $V_qE\tensor_{\Q_q}\Q_q(1)$. Por eso la
    función $L$ del motivo asociado a $E$ es $s\mapsto L(E,s+1)$ ($s\in\C$).  
    \item\label{ex:dedekind-zeta} Un ejemplo bastante interesante es el asociado a la variedad
    $\Spec K$ sobre $\Q$ para un campo de números $K$. Se puede verificar que la función $L$
    asociada a su motivo es la función zeta de Dedekind $\zeta_K$ del campo, que es también
    la función $L$ de Artin de la representación trivial de $\Gal(K/K)$.  Supongamos que
    $K/\Q$ es Galois. Entonces el formalismo de Artin (es decir, el hecho de que la función de
    Artin es invariante bajo inducción de representaciones, véase \cite[Prop.\
    3.8]{MR0349635}) dice que $\zeta_K$ es la función $L$ de la representación regular de
    $\Gal(K/\Q)$. De la teoría de representaciones de grupos finitos sabemos que esta
    representación se descompone en irreducibles: si $\rho_1,\dotsc,\rho_k$ son todas las
    representaciones irreducibles de $\Gal(K/\Q)$ entonces la representación regular es
    isomorfa a $\rho_1^{\dim\rho_1}\oplus\dotsm\oplus\rho_k^{\dim\rho_k}$. Su función $L$
    entonces se escribe como un producto de funciones de $L$ de Artin
    \[ \zeta_K(s)=\prod_{i=1}^k L(\rho_i^{\dim\rho_i},s). \] Esta descomposición de la
    función $L$ del motivo $\Spec K$ es la presencia, en el lado de las funciones $L$, del
    fenómeno de que el motivo $\Spec K$ se puede descomponer en $\mathpzc{Mot}(\Q)$
    aunque geométricamente es irreducible (¡es un punto!): cada una de las funciones
    $L(\rho_i,s)$ de hecho es la función $L$ de un submotivo. Esto ilustra la idea de que
    un motivo es un \enquote{pedazo} de una variedad e ilustra la utilidad de motivos:
    realmente nos dan un panorama más fino que las variedades.
  \item Finalmente mencionamos que para formas modulares (más precisamente, formas nuevas)
    también existe un motivo cuya función $L$ es la asociada a la forma modular. Esto es un teorema
    de Scholl \cite{MR1047142}.\footnote{\label{footnote:motivos-con-coef}Estrictamente,
      para formas modulares necesitamos más generalmente motivos con coeficientes en un
      campo de números (en lugar de sólo $\Q$), que en el caso de una forma nueva $f$ es el
      campo $K_f$. Para estos, las realizaciones son espacios vectoriales sobre $K_f$ o sus
      completaciones en lugar de sobre $\Q$ o $\Q_q$.}
  \end{enumerate}
\end{ex}

Ahora, la pregunta importante es la de algebraicidad de valores especiales, que
necesitamos para funciones $L$ $p$-ádicas. Esto es el contenido de una conjetura de Deligne
\cite[Conj.\ 1.8]{MR0546622}. Antes de poder explicarla tenemos que hablar sobre
\define{isomorfismos de comparación} para motivos. Las diferentes maneras de tomar
la cohomología de una variedad proyectiva lisa están relacionadas de manera tal que los
espacios vectoriales que obtenemos son canónicamente isomorfos después de tensarlos con un
campo más grande. Por ejemplo, la cohomología singular y de Rham son canónicamente isomorfas
al aplicar el producto tensorial con $\C$. De manera similar, la cohomología singular y la $p$-ádica son
canónicamente isomorfas al tensar la primera con $\Q_p$. Entre la de de Rham y la
$p$-ádica también existe un tal isomorfismo después de tensar con $\BdR$, que es un campo
de la teoría $p$-ádica de Hodge que no vamos a explicar aquí (remitimos a
\cite{BrinonConradPAdicHodgeTheory} para esto). Estos isomorfismos se extienden
a los motivos, así que tenemos lo siguiente para cada motivo $M$:
Existen isomorfismos canónicos
\begin{align*}
{\mathrm{cp}_\infty}\colon& M_\betti\tensor_\Q\C\isomarrow M_\dR\tensor_\Q\C,\\
{\mathrm{cp}_\et}\colon& M_\betti\tensor_\Q\Q_p\isomarrow M_p,\\
{\mathrm{cp}_\dR}\colon& M_p\tensor_{\Qp} \BdR\isomarrow M_\dR\tensor_\Q \BdR
\end{align*}
(los últimos dos para cada primo $p$) que son compatibles con las varias estructuras extras
que tenemos en ambos lados.

Parte de estas estructuras extras son subespacios canónicos $M_\betti^+\subseteq M_\betti$ y
$M_\dR^0\subseteq M_\dR$.\footnote{Más precisamente, en $M_\betti$ tenemos una acción de
  $\G\R$ y $M_\betti^+$ es el espacio fijo por esta acción, y en $M_\dR$ tenemos una
  filtración descendente que se llama la filtración de Hodge, y $M_\dR^0$ es el paso $0$ de
  esta filtración. Estas estructuras vienen directamente de las estructuras análogas en la
  cohomología de una variedad.} Con esto podemos definir:

\begin{defi}\label{defi:critico}
  Un motivo $M$ se llama \emph{crítico}\index[def]{motivo@motivo!crítico} si la composición
  \[ M_\betti^+\tensor_\Q\C\hookrightarrow
    M_\betti\tensor_\Q\C\labeledarrow{\mathrm{cp}_\infty}M_\dR\tensor_\Q\C\twoheadrightarrow M_\dR/M_\dR^0\tensor_\Q\C \]
  es un isomorfismo. En este caso, si escogemos bases de los $\Q$-espacios vectoriales
  $M_\betti^+$ y $M_\dR/M_\dR^0$ definimos $\Omega_\infty(M)\in\C^\times$ como el
  determinante de este morfismo con respeto a estas bases. Por supuesto, esto depende de
  las bases, es decir $\Omega_\infty(M)$ solo está bien definido salvo a elementos de
  $\Q^\times$, pero esto será suficiente. El número $\Omega_\infty(M)$ se llama el
  \define{período complejo} de $M$.
\end{defi}

La conjetura de Deligne entonces es la siguiente. Obviamente su validez no depende de las
bases que escogemos.

\begin{conj}[Deligne]
  Si $M$ es un motivo crítico entonces \[ \frac{L(M,0)}{\Omega_\infty(M)}\in\Q^\times. \]
\end{conj}

El \cref{thm:shimura-algebraicidad} para formas modulares es una instancia de esta conjetura
que está demostrada.\footnote{Como mencionamos en la \cref{footnote:motivos-con-coef} en la
  \cpageref{footnote:motivos-con-coef}, para formas modulares necesitamos motivos con
  coeficientes en un campo de números $K$ en lugar de $\Q$. La conjetura de Deligne que
  mencionamos es la versión para motivos con coeficientes en $\Q$, la versión más general
  afirma que la expresión está en $K^\times$. El \cref{thm:shimura-algebraicidad} es la
  conjetura de Deligne para el motivo $M(f)(\psi)(n)$ con la notación que introducimos en la \cref{def:motQ1}.}  También las proposiciones \ref{zeta-continuacion} y
\ref{prop:numeros-de-bernoulli-gen} son casos especiales de esta conjetura.

Ahora estamos casi listos para explicar cómo debería
ser en general una función $L$ $p$-ádica. Antes de esto tenemos que introducir la siguiente notación, que usa el hecho de que la
categoría de motivos tiene un producto tensorial y duales (los cuales en realizaciones se
convierten en el producto tensorial y el dual usual).

\begin{defi}\label{def:motQ1}
  Sea $M$ un motivo.
  \begin{enumerate}
  \item Sea $n\in\Z$. Si $n\ge1$ entonces definimos $\Q(n):=\Q(1)^{\otimes n}$. Para $n=-1$
    definimos $\Q(-1)=\Q(1)^*$ el dual de $\Q(1)$ y para $n\le1$ definimos
    $\Q(n)=\Q(-1)^{\otimes(-n)}$. (Aquí $\Q(1)$ es como en el
    \cref{ex:motivos}~\ref{ex:tate-motive}.)
  \item Definimos $M(n):=M\tensor\Q(n)$ para $n\in\Z$.
  \item Si $\psi$ es un carácter de Dirichlet entonces definimos
    $M(\psi):=M\tensor[\psi]$. Más generalmente, si $\rho$ es una representación de Artin
    entonces definimos $M(\rho):=M\tensor[\rho]$. (Aquí $[\psi]$ y $[\rho]$ son como en el
    \cref{ex:motivos}~\ref{ex:artin-motive}.)
  \end{enumerate}
\end{defi}

Se ve fácilmente que $L(M(n),s)=L(M,s+n)$ para $s\in\C$. La respuesta a la pregunta ¿Cuáles
valores de una función $L$ compleja pueden ser interpolados $p$-ádicamente?
Es presumible: los valores $L(M(n)(\chi),0)=L(M(\chi),n)$ para $n\in\Z$ y caracteres de
Dirichlet $\chi$ tal que el motivo $M(n)(\chi)$ es crítico. Por ejemplo, se puede verificar
que para una forma modular nueva $f$ de peso $k$, un entero $n\in\Z$ y un carácter de
Dirichlet $\chi$ el motivo $M(f)(\chi)(n)$ es crítico si y solo si $1\le n\le k-1$, lo cual es
compatible con el \cref{thm:shimura-algebraicidad}.

En las conjeturas sobre funciones $L$ $p$-ádicas los motivos deben cumplir una condición
adicional que muchas veces se llama la \define{condición de Panchishkin}. No la explicamos
aquí, se encuentra por ejemplo en \cite[§3]{MR1265554} o \cite[§4.2.3]{MR2276851}. Para una
forma modular, esta condición es equivalente a la condición que la forma sea ordinaria.

De manera similar a como definimos el período complejo en la \cref{defi:critico} se puede definir
también un período $p$-ádico, que esencialmente es el determinante de
\[ M_\betti^+\tensor_\Q\BdR\hookrightarrow
  M_\betti\tensor_\Q\BdR
  \labeledarrow{\mathrm{cp}_\et}
  M_p\tensor_\Qp\BdR
  \labeledarrow{\mathrm{cp}_\dR}
  M_\dR\tensor_\Q\BdR\twoheadrightarrow M_\dR/M_\dR^0\tensor_\Q\BdR \] salvo a una pequeña
modificación que omitimos aquí. Siempre suponemos que calculamos este determinante
con respeto a las mismas bases que usamos para calcular $\Omega_\infty(M)$. Esto lleva a un
número $\Omega_p(M)\in(\hat\Q_p^\nr)^\times$, es decir en la completación de la máxima
extensión no ramificada de $\Qp$.\footnote{La modificación que mencionamos garantiza
  que de hecho $\Omega_p(M)\neq0$; como lo definimos arriba puede pasar que la composición
  de mapeos no es un isomorfismo. Además la modificación garantiza también que el
  período de hecho está en el subcampo $\hat\Q_p^\nr\subseteq\BdR$.} Con esta definición
podemos esbozar la conjetura. En ella escribimos
$\tilde\LL(G):=\tilde\O\llbracket G\rrbracket$ para el álgebra de Iwasawa con coeficientes
en $\tilde\O$, que denota el anillo de enteros en $\hat\Q_p^\nr$, y escribimos $\tilde\Phi(G)$
para el anillo de cocientes de $\tilde\LL(G)$ (eso es la localización donde invertimos todos
los elementos que no son divisores de cero). En general las funciones $L$ $p$-ádicas están
en $\tilde\Phi(G)$ y no en $\tilde\LL(G)$, como vimos en el ejemplo de la función zeta $p$-ádica
de Riemann. La función $L$ $p$-ádica conjeturalmente tiene valores en $\tilde\O$.

\begin{conj}\label{conj:palf-motivos}
  Sea $M$ un motivo que cumple la condición de Panchishkin. Entonces existe un único
  elemento $\mu_M\in\tilde\Phi(G)$ tal que para cada carácter de Dirichlet $\psi$ cuyo
  conductor es una potencia de $p$ y cada $n\in\Z$ tal que $M(\chi)(n)$ es crítico tenemos
  \[
    \int_G\psi^{-1}\kappa^n\integrald\mu_M=(\text{\normalfont factores de
      corrección})\cdot\frac{\Omega_p(M(\chi)(n))}{\Omega_\infty(M(\chi)(n))}L(M(\chi)(n),0). \]
\end{conj}

En esta forma la conjetura es un caso especial de una conjetura de Fukaya y Kato de
\cite{MR2276851}, pero ya antes conjeturas similares han sido formuladas por Coates y
Perrin-Riou \cite{MR1097608,MR1129081}.

\begin{remark}\label{nota:ll-vs-ll-tilde}
  El anillo $\tilde\LL(G)$ que aparece aquí es mucho más grande que $\LL(G)$, donde las
  funciones $L$ $p$-ádicas anteriores vivían. En las versiones de esta conjetura de Coates y
  Perrin-Riou la función $L$ $p$-ádica vive en $\LL(G)$ o su anillo de cocientes (que
  denotamos $\Phi(G)$), pero en su
  fórmula de interpolación no aparece el período $p$-ádico. La versión de aquí de Fukaya y
  Kato incluye el período $p$-ádico y por eso es un poco más elegante y conceptual (como
  explicamos al final de esta sección), pero para incluirla hay que aumentar el anillo.

  Sin embargo, mencionamos que el elemento conjetural $\mu_M$ se puede escribir como
  $\mu_M=\lambda\mu_M'$ con $\lambda\in\tilde\LL(G)$ y $\mu_M'\in\Phi(G)$ de una manera
  esencialmente única, es decir si $\mu_M=\lambda'\mu''_M$ es otra tal descomposición
  entonces $\mu_M'$ y $\mu_M''$ sólo difieren por una unidad en $\LL(G)^\times$.
\end{remark}

Es decir, la conjetura tiene la forma
\[ \frac{\text{valor de la función $L$ $p$-ádica}}{\text{período $p$-ádico}} =
  (\text{factores de corrección})\cdot\frac{\text{valor de la función $L$
      compleja}}{\text{período complejo}}. \]
Los factores de corrección son expresiones sencillas como factores de Euler y
factoriales. En particular son algebraicos, así que lo de arriba es una igualdad en $\Qbar$
y el período $p$-ádico describe la parte transcendente del valor de la función $L$
$p$-ádica, justamente como el período complejo describe la parte transcendente del valor de
la función $L$ compleja.

Notemos que, aunque los períodos no están unívocamente definidos (sólo salvo a un factor
en $\Q^\times$), si cambiamos la base con respeto a cual los definimos entonces ambos
cambian por el mismo factor. Es decir, el valor de la integral en la conjetura arriba no
depende de las bases que escogemos.

En general, no es nada fácil verificar que una función $L$ $p$-ádica construida sea
compatible con esta conjetura. Esto se debe a que por definición los períodos son
difíciles de calcular, mientras que para construir una función $L$ $p$-ádica no existe una
manera estándar y las fórmulas de interpolación que uno obtiene muchas veces contienen
expresiones que son mas bien artefactos del método de construcción y a priori no tienen
un significado conceptual. Por lo tanto, no queda claro si estas expresiones de hecho están relacionadas con los
períodos. Sin embargo, para todas las funciones $L$ $p$-ádicas que hemos visto hasta ahora,
esto sí es cierto. Para las de caracteres de Dirichlet esto es fácil de ver porque los
motivos y los isomorfismos de comparación tienen descripciones muy explícitas. Para las
formas modulares esto es mucho más difícil, pero también es conocido (una demostración se
encuentra en \cite{MichaelDiss}).

Ahora podemos enunciar la Conjetura Principal para motivos. Una conjetura de este estilo fue
formulada por Greenberg en \cite{MR1097613} y \cite[§3, 3.1]{MR1265554}. Para esto sea $M$
un motivo cumpliendo la condición de Panchishkin. Además escojamos un retículo estable $T_p$
en la realización $M_p$, con el cual podemos definir el módulo
$\mathrm X(M):=\mathrm X(M_p/T_p)$, que según un resultado de
\cite{MR1265554} es noetheriano\footnote{En los textos de Greenberg la
  definición del grupo de Selmer es ligeramente diferente, el subgrupo $\Hf^1(K_v,-)$
  para las plazas $v\mid p$ difiere de nuestra definición; sin embargo, se puede demostrar
  que bajo la condición de Panchishkin la versión de Greenberg de hecho es la misma que la
  de Bloch y Kato.} y conjeturalmente es de torsión.
Escribimos $\mu_M=\lambda\mu_M'$ con $\lambda\in\tilde\LL(G)$ y $\mu_M'\in\Phi(G)$ como en
la \cref{nota:ll-vs-ll-tilde}. Según esta nota la veracidad de la siguiente afirmación no
depende de esta descomposición.

\begin{conj}[Conjetura Principal para motivos]\label{conj:cp-motivos}
  Existe $h\in\LL(G)$ tal que tenemos la igualdad de ideales en $\LL(G)$
  \[ \charideal_{\LL(G)}\mathrm X(M)=(h\mu'_M) \]
  y $(h)$ también es el ideal característico de un cierto módulo (que omitimos aquí).
  
  En particular, si $\mu_M\in\tilde\LL(G)$ entonces tenemos la igualdad de ideales en
  $\LL(G)$
  \[ \charideal_{\LL(G)}\mathrm X(M)=(\mu'_M). \]
\end{conj}

El lector debería comparar esta Conjetura Principal (en el caso
$\mu'_M\notin\tilde\LL(G)$) con la interpretación de la Conjetura Principal clásica que damos
en la \cref{nota:cp-cociente-de-ideales-car}. En cuyo caso explicamos el módulo cuyo
ideal característico es generado por $h$.

Esta conjetura claramente generaliza todas las que hemos visto antes, pero todavía no es el fin de
la historia.  Quedan dos direcciones principales en las cuales podemos generalizar.

Primero, podemos cambiar el campo de base $\Q$ a un campo de números $F$. Esto
significaría estudiar (sistemas de) representaciones de $\G F$ y motivos sobre $F$
(como por ejemplo curvas elípticas). Esto es posible y las generalizaciones en este caso han
sido formuladas, pero esto no cambia mucho conceptualmente, sólo hace que la notación es
menos clara. Por eso ninguneamos esta dirección.

Segundo, es posible cambiar la torre de campo de números $(K_r)_r$ a otra torre, por ejemplo
una más grande. Esto cambia el grupo $G$ y también el álgebra de Iwasawa $\LL(G)$. Incluso
se puede estudiar el caso en que $G$ no es conmutativo. La Teoría de Iwasawa no
conmutativa fue fundada por Harris, Coates, Howson, Ochi y Venjakob y trae algunos nuevos
fenómenos. Por ejemplo, en el lado algebraico la teoría de estructura ya no funciona porque entonces el
anillo $\LL(G)$ se comporta peor generalmente. En particular, ya no tenemos ideales
característicos, que eran esenciales para formular la Conjetura Principal. Se puede arreglar
esto usando la teoría K algebraica para definir \define{elementos característicos} -- estos
son elementos de una modificación de $\mathrm K_1$ del álgebra de Iwasawa que son enviados
al módulo por el morfismo de borde a $\mathrm K_0$. Para formular una Conjetura Principal, las funciones $L$ $p$-ádicas también son elementos de este $\mathrm K_1$ y se
sospecha que son elementos característicos de grupos de Selmer. En su fórmula de
interpolación entonces aparecen no sólo los chanfles del motivo por caracteres de Dirichlet
sino de representaciones de Artin del grupo $G$ (que para $G\isom\Z_p^\times$ como antes,
son justamente los caracteres de Dirichlet que usamos). Para una curva elíptica la teoría es
descrita en \cite{MR2217048}. También queremos mencionar \cite{MR2185787} que es una
introducción muy bonita a estas ideas. Para motivos más generales, la Teoría de Iwasawa no
conmutativa, es un caso especial de las ideas que indicamos enseguida.

Las últimas generalizaciones de las que hablaremos son la denominada \emph{Conjetura
  Equivariante de Números de Tamagawa}\index[def]{conjetura!equivariante de Números de Tamagawa} y la \emph{conjetura equivariante de los
  $\varepsilon$-isomorfismos}\index[def]{conjetura!equivariante de los
  $\varepsilon$-isomorfismos}, formuladas por Burns, Flach, Fukaya y Kato después de
trabajos de Deligne, Beilison, Bloch, Kato, Perrin-Riou, Fontaine, Huber, Kings y otros. En
la introducción mencionamos dos resultados sobre conexiones entre valores especiales de
funciones $L$ y la aritmética que todavía relucen en el cuadro: la fórmula analítica de
números de clases y la Conjetura de Birch y Swinnerton-Dyer. Estas dos afirmaciones de hecho
no son \enquote{$p$-ádicas}, por eso su relación con las Conjeturas Principales no es clara
a primera vista. La Conjetura Equivariante de Números de Tamagawa y la conjetura
equivariante de los $\varepsilon$-isomorfismos están diseñadas como una generalización común
de estas conexiones. Describen de una manera satisfactoria el significado de valores de
funciones $L$ de motivos mediante invariantes cohomológicas del motivo.

Explicamos las ideas de estas conjeturas; aquí seremos muy vagos. La Conjetura de Números
de Tamagawa (todavía no equivariante), también conocida como la conjetura de Bloch y Kato,
es una generalización común de la fórmula analítica de números de clases y la Conjetura de
Birch y Swinnerton-Dyer. En estas dos afirmaciones tenemos una función $L$ de un motivo $M$,
meromorfa en $\C$, que podemos escribir como serie de potencias
\[ L(M,s)=L^*(M)s^{r(M)} + \text{expresiones de orden más alto} \] donde
$L^*(M)\in\C^\times$ es el coeficiente principal y $r(M)\in\Z$ es el orden del cero de la función
en $s=0$ (para la fórmula analítica de números de clases tomamos $M=K(1)$ para un campo de
números $K$ y para la Conjetura de Birch y Swinnerton-Dyer tomamos un motivo construido de la
curva elíptica). Las afirmaciones entonces describen los números $r(M)$ y $L^*(M)$ usando
invariantes del motivo (el número de clases, la orden del grupo de Tate y Shafarevich, el
rango del grupo de Mordell y Weil, \dots). La Conjetura de Números de Tamagawa es una
generalización de esto para cada motivo, expresando $r(M)$ y $L^*(M)$ con invariantes del
motivo.  Entonces la Conjetura de Números de Tamagawa Equivariante filosóficamente predice que los valores de $r(M)$ y $L^*(M)$ \enquote{varían continuamente con el
  motivo $M$}. Finalmente, conjeturalmente la función $L$ de un motivo $M$ tiene una
ecuación funcional que relaciona $L(M,s)$ con $L(M^*(1),-s)$ para $s\in\C$. Esta ecuación
funcional da una relación entre $L^*(M)$ y $L^*(M^*(1))$ y también $r(M)$ y $r(M^*(1))$, por
eso debería haber también una relación entre las Conjeturas de Números de Tamagawa para $M$
y $M^*(1)$. La conjetura de los $\varepsilon$-isomorfismos (no equivariante) describe tal relación de una manera muy sutil y la conjetura equivariante de los
$\varepsilon$-isomorfismos afirma que estas relaciones también varían continuamente con el
motivo. Una introducción a estas ideas se encuentra en \cite{MR2392359} o \cite{MR2088713}.

¿Qué tiene todo esto que ver con la Conjetura Principal? Esto es explicado por el siguiente
teorema. Aquí ya no aparecen grupos de Selmer sino complejos de Selmer,
que son generalizaciones de grupos de Selmer apropiados para esta situación general. El
\enquote{otro módulo} que apareció en la \cref{conj:cp-motivos} ahora está incluido en el
complejo de Selmer.

\begin{thm}[Fukaya/Kato]\label{thm:fk}
  Supongamos que la Conjetura Equivariante de los Números de Tamagawa y la conjetura
  equivariante de los $\varepsilon$-isomorfismos son ciertas.

  Sea $M$  un motivo que cumple la condición de Panchishkin. Fijemos una torre
  $K_\infty/\Q$ tal que $G=\Gal(K_\infty/\Q)$ tenga un subgrupo normal de índice finito que
  es pro-$p$ y topológicamente finitamente generado.

  Entonces:
  \begin{enumerate}
  \item\label{thm:fk:palf} Existe una función $L$ $p$-ádica, que es un elemento de un cierto
    grupo $\mathrm K_1$, que interpola los valores $L(M(\rho)(n),0)$, donde $\rho$ es una
    representación de Artin de $G$ y $n\in\Z$ tal que $M(\rho)(n)$ es crítico.
  \item\label{thm:fk:mc} Esta función es un elemento característico del complejo de Selmer
    de $M$ (esto es la Conjetura Principal).
  \end{enumerate}
\end{thm}

Para más detalles sobre esto remitimos al texto original \cite{MR2276851} de Fukaya y Kato y
también al texto introductorio \cite{MR2392359} de Venjakob que explica algunas de las ideas
con más detalles.

El elemento conjetural $\mu_M$ de la \cref{conj:palf-motivos} es el de la afirmación
\ref{thm:fk:palf} del \cref{thm:fk} en el caso especial $K_\infty=\Q(\mu_{p^\infty})$. La
\cref{conj:cp-motivos} entonces es equivalente a la afirmación \ref{thm:fk:mc} en este
caso especial.
Es decir, ¡La Conjetura Equivariante de los Números de Tamagawa y la conjetura equivariante
de los $\varepsilon$-isomorfismos implican toda la Teoría de Iwasawa! En particular,
implican la existencia de funciones $L$ $p$-ádicas para muchos motivos, y también
proporcionan la fórmula de interpolación.
Así le dan una explicación más
profunda: es una consecuencia de estas dos conjeturas, y en particular es compatible con la
fórmula analítica de números de clases y la Conjetura de Birch y Swinnerton-Dyer. Además, la
Conjetura Principal, incluso no conmutativa, también es una consecuencia; por ende, son
todas las Conjeturas Principales que aparecieron en este texto. Así la Conjetura
Equivariante de los Números de Tamagawa y la conjetura equivariante de los
$\varepsilon$-isomorfismos unifican toda la Teoría de Iwasawa e implican esencialmente todo
lo que uno podría querer saber sobre funciones $L$ y sus conexiones a la
aritmética.\footnote{\label{footnote:ep-eq-x-y}Mencionamos que la equivalencia de las dos
  formulaciones de la Conjetura Principal que damos en la \cref{sec:conjetura-principal} es
  una encarnación de la conjetura equivariante de los $\varepsilon$-isomorfismos -- note que
  la ecuación funcional en este caso conecta $\zeta(s)$ y $\zeta(1-s)$!}  Por supuesto queda
un largo camino por demostrar, no obstante, la amplia impresión dada por estas ideas
es de una elegancia peculiar, que ojalá persuada al lector de continuar estudiando la
Teoría de Iwasawa.

\newpage
\backmatter
\emergencystretch=1em
\printbibliography
\printindex[def]

\end{document}